\documentclass[11pt]{article}
\usepackage{amsmath,amssymb,amsthm}
\usepackage{graphicx}
\usepackage[numbers]{natbib}
\usepackage[margin=1in, top=0.2in, headheight=10pt, headsep=5pt, includehead]{geometry}
\usepackage{xcolor}
\usepackage[colorlinks=true, citecolor=blue, linkcolor=red]{hyperref}
\usepackage{titlesec}

\definecolor{tocgreen}{RGB}{34, 139, 34}

\makeatletter
\AtBeginDocument{%
  \let\orig@starttoc\@starttoc
  \renewcommand{\@starttoc}[1]{%
    \begingroup
    \hypersetup{linkcolor=tocgreen}%
    \orig@starttoc{#1}%
    \endgroup
  }%
}
\makeatother
\usepackage[T1]{fontenc}
\usepackage{newtxtext, newtxmath}  
\usepackage{tocloft}              

\numberwithin{equation}{section}
\titleformat{\section}
  {\normalfont\Large\bfseries\MakeUppercase}{\thesection}{1em}{| }

\titleformat{\subsection}
  {\normalfont\large\bfseries}{\thesubsection}{1em}{| }

\titleformat{name=\section,numberless}
  {\normalfont\Large\bfseries}{}{0pt}{}

\makeatletter
\newtheoremstyle{dotstyle}{3pt}{3pt}{\normalfont}{}{\bfseries}{.}{ }{\thmname{#1}\thmnumber{ #2}\thmnote{ (#3)}}
\theoremstyle{dotstyle}
\newtheorem{theorem}{Theorem}[section]
\newtheorem{lemma}{Lemma}[section]
\newtheorem{proposition}{Proposition}[section]
\newtheorem{definition}{Definition}[section]
\newtheorem{remark}{Remark}[section]
\newtheorem{corollary}{Corollary}[section]
\newtheorem{example}{Example}[section] 
\renewenvironment{proof}[1][\proofname]{%
  \par\pushQED{\qed}%
  \topsep6\p@\@plus6\p@\relax
  \trivlist
  \item[\hskip\labelsep\normalfont\bfseries #1\@addpunct{.}]%
  \ignorespaces
}{%
  \popQED\endtrivlist\@endpefalse
}
\makeatother

\usepackage{etoolbox}   
\appto\appendix{
  \titleformat{\section}                         
    {\normalfont\Large\bfseries\MakeUppercase}  
    {APPENDIX~\thesection:}{1em}{}              
  \titleformat{name=\section,numberless}         
    {\normalfont\Large\bfseries}{}{0pt}{}
}

\title{\bfseries Towards a Mathematical Theory of Adaptive Memory: From Time-Varying to Responsive Fractional Brownian Motion}
\author{
  Jiahao Jiang\thanks{Correspondence to: Jiahao Jiang, School of Mathematics, Southwest Jiaotong University. Email: \href{mailto:jiahao.jiang@swjtu.edu.cn}{jiahao.jiang@swjtu.edu.cn}} \\
  \textit{School of Mathematics, Southwest Jiaotong University, Chengdu 610031, China}
}
\date{}

\begin{document}

\maketitle

\begin{center}
\large\bfseries Abstract
\end{center}
\vspace{10pt}
This work attempts to develop a systematic mathematical description for a class of stochastic processes whose local regularity adjusts dynamically in response to their own state. The investigation proceeds in three stages. First, we consider a Time-Varying Fractional Brownian Motion (TV-fBm) with a deterministic, H\"older-continuous Hurst exponent function, and conduct a further theoretical analysis of its properties and structural features. For this process, we establish a range of fundamental properties, including an exact variance scaling law, local increment asymptotic estimates, local non-determinism, large deviation asymptotics for local increments, and a covariance structure that admits a closed-form hypergeometric representation. These results provide useful insight into the local and global dependence structure of the process.

Second, as a central component of this work, we define a new class of processes termed Responsive Fractional Brownian Motion (RfBm). In this construction, the Hurst exponent is governed by a Lipschitz-H\"older response function that depends on the process state itself, thereby introducing an intrinsic feedback mechanism between state and memory. We establish the well-posedness of this definition through local existence and uniqueness results, prove that the induced instantaneous scaling exponent possesses almost sure H\"older regularity, and analyze the associated cumulative memory processes along with their asymptotic convergence properties. The framework offers a reference mathematical foundation for exploring state-dependent memory mechanisms.

Third, we discuss several theoretical implications of the RfBm structure for modeling adaptive information processing. We observe that its kernel naturally induces a continuous-time, pathwise attention mechanism with intrinsic multi-scale and content-adaptive features. We provide basic estimates for the resulting attention weights, analyze their sensitivity to historical states, and quantify the stability of attentional allocation through residence measures and volatility functionals. In this way, the RfBm framework offers a stochastic-process-theoretic perspective on concepts central to adaptive memory and content-sensitive information processing, while also establishing structural connections to such research directions as attention mechanisms in artificial intelligence, adaptive coding in biological systems, and nonlinear response phenomena in physical systems.

\vspace{10pt}

\vspace{10pt}  
{\centering\Large\bfseries\MakeUppercase{Contents}\par}
\vspace{6pt}
\makeatletter
\@starttoc{toc}
\makeatother
\vspace{12pt}  

\section{INTRODUCTION}

Stochastic processes with long-range dependence and self-similarity have long occupied a central place in the mathematical modeling of complex systems across the physical, biological, and social sciences. Their distinctive feature—the ability to exhibit persistent correlations over extended time scales—renders them particularly well suited for capturing phenomena in which the influence of past events decays slowly, if at all, and in which fluctuations exhibit scale-invariant structures. Such processes arise naturally in a diverse range of contexts, including turbulent fluid flows, financial time series, network traffic, and neural activity, where the interplay between memory and variability plays a determining role in shaping system behavior. Among the various classes of such processes, fractional Brownian motion (fBm), since its rigorous mathematical formulation by Mandelbrot and Van Ness \cite{mandelbrot1968}, has provided a foundational framework for modeling phenomena characterized by self-similarity, long-range dependence, and anomalous scaling. This Gaussian process, defined via a moving average of ordinary Brownian motion increments with a power-law kernel, has found profound applications across numerous disciplines, from hydrology and solid-state physics to economics and turbulence. The stochastic integral, as formalized by Itô \cite{Ito1944} and extended in various contexts such as the Itô-type isometry for loop spaces \cite{Gosselin1997}, serves as the cornerstone for constructing and analyzing such processes. These foundational developments, while having achieved considerable success in capturing static or slowly varying memory structures, naturally prompt further inquiry into settings where the memory characteristics themselves may evolve in response to the system's dynamics—an avenue that lies at the heart of the present investigation and to which we shall return throughout the sections that follow.

A significant evolution in this theory is the extension to variable-order fractional operators, where the fractional order (e.g., the Hurst exponent) is allowed to vary with time. This concept, recently formalized and reviewed in the context of modeling complex physical systems \cite{Patnaik2020}, enables a more flexible description of systems with evolving memory characteristics. Theoretical developments in this direction include pseudo-fractional operators with variable order \cite{Oliveira2022}, the physical interpretation and simulation of variable-order fractional diffusion within trapping models \cite{Sibatov2024}, and the simulation of variable-order fractional Brownian motion (vofBm) for solving nonlinear stochastic differential equations \cite{Samadyar2024,Samadyar2025}. Further related developments include models where the Hurst exponent itself is random \cite{woszczek2025riemann}, approaches to fBm via complex-plane random walks \cite{jumarie1999new}, and generalizations of fractional calculus for studying anomalous stochastic processes \cite{jiang2025study}. These variable-order and generalized frameworks naturally connect to broader areas of analysis, such as the study of fractional porous medium equations \cite{Pablo2012}.

A rigorous analysis of processes with time-varying or state-dependent characteristics often relies on precise probabilistic estimates. Key technical tools in this regard include sharp bounds for Gaussian tail probabilities, intimately related to Mills' ratio and its functional properties \cite{Mills1926,Baricz2008,Yang2015}, and dynamical versions of the Borel--Cantelli lemma for sequences of events in deterministic dynamical systems \cite{Chernov2001,Gupta2010,Haydn2013}. The asymptotic analysis of probabilities, particularly through the theory of large deviations \cite{Varadhan1966,JG1977,Touchette2009,Feliachi2021}, provides a powerful language for characterizing rare events and exponential scaling laws. Applications of stochastic analysis in related settings include, for example, the stability analysis of systems driven by fBm \cite{li2017exponential}. Related methodological perspectives include the convergence analysis of iterative schemes applied to integral operators \cite{S1997}, classical transformation theory for measure-preserving flows which informs scaling principles for self-similar processes \cite{J1957}, and the construction of pathwise integration theories for continuous paths with arbitrary regularity \cite{Cont2019}.

The probabilistic methods and analytical frameworks discussed above, which deal with memory, scaling, and integration, also find conceptual parallels in other domains. In artificial intelligence, for example, attention-based architectures including the Transformer model \cite{Vaswani2017} and the bidirectional pre-training approach BERT \cite{Devlin2019} have achieved notable empirical results. This has motivated further inquiry into the mathematical foundations of such mechanisms, including analysis of their approximation rates for sequence modeling \cite{HJiang2024}, exploration of periodic alternatives to softmax to mitigate gradient behavior \cite{SWang2021}, and their application to specialized tasks such as mathematical equation recognition \cite{Aurpa2024}. Meanwhile, within stochastic analysis, recent work has established fractional mean-square inequalities for certain classes of stochastic processes \cite{khan2025fractional}, suggesting potential intersections between fractional calculus and learning theory. Complementing these developments, a parallel line of inquiry within stochastic analysis has examined the mean-field behavior of interacting particle systems on co-evolving networks, where latent state processes and network dynamics exhibit bidirectional feedback and persistence effects \cite{ganguly2026mean}. Such frameworks characterize the distributional limit of empirical measures in large-scale interacting systems, providing a rigorous probabilistic foundation for understanding emergent collective behavior in dynamically coupled agent–network systems.

In view of these developments, it is of interest to explore mathematical models that more fully incorporate the notion of adaptive memory wherein a stochastic process's local regularity adjusts dynamically in response to its own state—within a rigorous stochastic calculus setting. Current models with time-varying memory typically treat the Hurst exponent variation either as an exogenous function \cite{Patnaik2020,Sibatov2024,Samadyar2024} or as a random parameter \cite{woszczek2025riemann}. The mathematical foundations for studying non-local operators, a cornerstone for many memory models, have been significantly advanced. For instance, the rigorous theory of fractional Laplacians on domains, developed via the framework of $\mu$-transmission pseudodifferential operators, provides a modern understanding of boundary value problems for such operators \cite{grubb2015fractional}. This is complemented by sharp estimates for the heat kernels of Dirichlet fractional Laplacians, which describe the transition densities of killed stable processes \cite{chen2010heat}. On the abstract level, the study of Cauchy problems with generalized fractional derivatives has yielded results on local well-posedness and generalized Grönwall inequalities, which are instrumental for a priori estimates \cite{ascione2021abstract}. Alongside these operator-theoretic developments, the analysis of nonlinear filtering problems has revealed delicate metastability phenomena in the strong-noise regime, where the optimal filter exhibits intermittent spiking behavior whose onset is governed by a sharp phase transition \cite{bernardin2025spike}. Such fine-grained characterizations of filtering dynamics enrich our understanding of how stochastic systems extract information from noisy observations, offering insights that may inform the design of adaptive mechanisms in state-dependent memory models.

Within the broader landscape of memory-dependent phenomena, a substantial body of work investigates deterministic dynamical systems with memory. This includes the analysis of semilinear subdiffusion equations with memory terms, where global existence and uniqueness of classical solutions are established through careful a priori estimates in fractional Hölder spaces \cite{krasnoschok2017semilinear}. The asymptotic behavior of such systems is of central interest, as evidenced by studies on coupled wave-plate systems with fractional memory dissipation, which establish well-posedness and characterize the exponential or polynomial decay of solutions \cite{oliveira2026asymptotic}. The profound impact of memory is also highlighted in fluid mechanics, where the inclusion of the Basset–Boussinesq memory term in the Maxey–Riley equation for inertial particles is proven to fundamentally alter the asymptotic dynamics, leading to a universal algebraic decay towards an asymptotic state \cite{langlois2015asymptotic}.

Within stochastic analysis, the ergodic properties of systems driven by exogenous fractional noise have been thoroughly investigated, establishing conditions for unique stationary solutions and algebraic convergence rates \cite{hairer2005ergodicity}. The study of limit theorems for processes with long memory and heavy tails reveals the complex interplay of these features, yielding limiting processes that combine stable random measures, Mittag–Leffler processes, and Brownian motion \cite{jung2017functional}. The foundational concept of saturated adapted spaces provides a powerful probabilistic framework for studying existence and distributional properties of stochastic processes in full generality \cite{keisler1997rich}.

At the intersection of memory and nonlinear dynamics, the concept of fractional memory has been introduced into discrete maps, leading to generalized fractional versions of systems like the Hénon and Lozi maps. This allows for the investigation of how power-law-like memory influences fundamental properties such as the existence of periodic points \cite{edelman2025asymptotic}. Parallel to these developments, the rigorous analysis of global well-posedness for nonlinear evolution equations, even in complex settings like the generalized derivative nonlinear Schrödinger equation, underscores the importance of robust existence and uniqueness theories \cite{pineau2025global}.

Synthesizing the foregoing developments across deterministic, stochastic, and discrete dynamical systems, stochastic analysis, fractional calculus, and the theory of nonlocal operators, a coherent picture emerges: memory—whether manifested through fractional operators, nonlocal kernels, state-dependent feedback loops, or noise-driven filtering dynamics—has received sustained attention as a fundamental structural ingredient shaping the qualitative behavior of complex systems. Within this rich landscape, one direction that invites further exploration concerns the manner in which memory mechanisms are formulated. A number of existing frameworks have naturally chosen to formulate memory as an externally specified characteristic, either hard-wired through fixed parameters, modulated by a predetermined time-dependent schedule, or encoded in kernels that operate independently of the system's instantaneous configuration. These approaches are mathematically tractable and have demonstrated considerable explanatory power in their respective application domains. At the same time, a class of phenomena—wherein the very nature of memory depends on the content it is meant to retain—remains to be examined within a unified mathematical framework. Reflection on such phenomena leads to a question of a more structural nature: if memory is to be regarded as an adaptive resource rather than a static attribute, its strength, persistence, and influence may naturally be expected to co-evolve with the system that embodies it. In other words, the manner in which a system remembers may plausibly depend on what it has experienced—a content-sensitive self-regulatory mechanism that, conceptually, invites a reexamination of the classical distinction between state and memory. This perspective is perhaps best understood not as a technical modification of existing models, but rather as a shift in viewpoint: from memory as a parameter to memory as a process.

From a mathematical standpoint, this shift in viewpoint gives rise to several analytical considerations that merit attention. The introduction of state-dependent memory renders the governing equations nonlinear, path-dependent, and non-Markovian in structure—features that enrich the qualitative behavior of the system while simultaneously posing new challenges for analysis. The tools of classical stochastic analysis, though systematic and powerful, are largely developed under the typical setting in which the regularity and singularity structure of the kernel are independent of the process state; when the kernel itself responds to the evolving state of the process, further development of the existing machinery appears to be a worthwhile direction. Developing a coherent theory for such systems involves not only the establishment of new existence and regularity results, but also invites a reconsideration of how memory, scaling, and adaptation interact within a unified probabilistic framework.

These considerations provide a natural point of departure for the present work. In this article, we attempt to develop a mathematical description for a class of stochastic processes whose local regularity—and hence the character of memory—emerges dynamically from the ongoing interplay between the process and its own history. The central idea is to allow the Hurst exponent, which governs the persistence of correlations in fractional-type processes, to be modulated by a response function that depends on the current state of the system. This feedback mechanism is realized within a stochastic Volterra integral equation, giving rise to a class of processes we term Responsive Fractional Brownian Motion (RfBm). As will be shown below, this construction yields a well-posed stochastic process with pathwise regularity, and simultaneously establishes a structural correspondence between its kernel structure and such research directions as attention mechanisms in modern artificial intelligence, adaptive coding in biological systems, nonlinear responses in physical systems, and reaction–diffusion memory effects in chemical systems. In this way, the present framework offers a mathematical perspective on adaptive memory and content-sensitive information processing that may serve as a useful reference for further investigation.

\textbf{Contributions of this work.} Motivated by this observation, the present article develops a mathematical theory for stochastic processes with state-dependent, adaptive memory. Our primary contributions are threefold.

First, we introduce and rigorously analyze a \textbf{Time-Varying Fractional Brownian Motion (TV-fBm)} with a deterministic, Hölder-continuous Hurst exponent function $H(t)$. Building upon the variable-order integration framework, we establish its fundamental properties: the exact variance scaling law $\mathbb{E}[(B^{H(\cdot)}(t))^2] = t^{2H(t)}$ (Theorem~\ref{thm:variance}), precise local increment asymptotics (Theorem~\ref{thm:local_increment_var}), the property of local non-determinism (Theorem~\ref{thm:LND}), large deviation asymptotics for its local increments (Theorem~\ref{thm:LDP}), and a closed-form covariance structure amenable to hypergeometric analysis (Theorem~\ref{thm:covariance}, Remark~\ref{rem:cov_hyper}, \ref{rem:non_dominant}).

Second, and as our central innovation, we define \textbf{Responsive Fractional Brownian Motion (RfBm)}. This is a class of processes where the Hurst exponent is governed by a Lipschitz-H\"older \textit{response function} $H(t, X_t)$, making it a path-dependent functional of the process state itself (Definition~\ref{def:rFBM}). This creates a built-in feedback mechanism. We establish the well-posedness of this definition through local existence and uniqueness results (Theorem~\ref{thm:local_existence_uniqueness}), which rely on Lipschitz estimates for the state-dependent kernel (Theorem~\ref{thm:global_lipschitz_kernel}). We then develop a comprehensive pathwise regularity theory, including the Hölder continuity of the induced instantaneous scaling exponent $\alpha(t)=H(t,X_t)$ (Theorem~\ref{thm:instant_scaling_holder}), the analysis of extremal scaling indices (Definition~\ref{def:extremal_indices}), and the study of cumulative memory processes (Definition~\ref{def:cumulative_memory}) with their convergence properties (Theorem~\ref{thm:convergence_rate}).

Third, we explore the theoretical implications of the RfBm structure for modeling \textbf{adaptive memory and attention}. We show that its kernel naturally induces a mathematically well-defined attention weight distribution (Definition~\ref{def:rfbm_attention_weights}, Theorem~\ref{thm:attention_well_definedness}), illustrated through a concrete spatiotemporal adaptive response function (Example~\ref{ex:adaptive_response_function}). This leads to a continuous-time, pathwise attention mechanism with intrinsic multi-scale and content-adaptive properties. We derive fundamental bounds for these attention weights (Theorem~\ref{thm:rfbm_attention_bounds}, with a detailed proof provided in Appendix~\ref{app:proof_attention_bounds}), analyze their sensitivity to historical states (Theorem~\ref{thm:memory_dynamics_sensitivity}), and quantify the stability of attentional allocation through a residence measure and its volatility (Definition~\ref{def:residence_measure_attention}, Theorem~\ref{thm:volatility_covariance_integral}). This work aims to provide a stochastic-process-theoretic foundation for concepts relevant to adaptive intelligence, complementing existing empirical and discrete-time approaches.

The remainder of this article is organized as follows. Section~\ref{sec:preliminaries} collects necessary preliminaries, including Hölder continuity conditions, variable-order integration, logarithmic control inequalities, and foundations of Itô calculus. Section~\ref{sec:tvfbm} presents the construction and analysis of TV-fBm. Section~\ref{sec:large_deviation} examines its large deviation asymptotics and covariance structure. Section~\ref{sec:rFBM_theory} is devoted to the theory of RfBm: its definition, pathwise regularity, well-posedness, and the analysis of its scaling exponents. Section~\ref{sec:rFBM_AI_foundations} develops the connection to adaptive memory and attention mechanisms. Proofs of selected technical results are provided in the appendix.

\section{Preliminaries: Regularity, Variable-Order Operators, and Stochastic Integration}
\label{sec:preliminaries}

This section collects the foundational mathematical tools that underpin the subsequent construction and analysis of both the time-varying fractional Brownian motion (TV-fBm) and the responsive fractional Brownian motion (RfBm). As established in the introduction, the central theme of this work is the dynamic interplay between memory and state. To rigorously formulate such a framework, three preparatory ingredients are essential. First, we recall the notion of H\"older continuity and introduce a critical regularity condition that ensures the consistency of variable-order fractional operators near singularities (Remark~\ref{rem:holder_condition}). Second, we define the variable-order Riemann-Liouville fractional integral, which provides the analytical foundation for extending classical fBm to time-varying Hurst exponents. Third, we review the basics of It\^o stochastic integration, including the It\^o isometry and its polarization identity, which serve as the primary probabilistic tools for constructing the stochastic integral representations in Definitions~\ref{def:tvfbm} and~\ref{def:rFBM}. In addition, we establish two logarithmic control inequalities (Propositions~\ref{prop:log_inequality} and~\ref{prop:log_growth}) that will be instrumental in managing the singularities arising from the state-dependent kernel in later sections. Throughout the paper, we adopt the notation summarized below.

\subsection*{Notation}
The following notation is used throughout the paper.
\begin{itemize}
  \item $H: [0,T] \to (0,1)$: A deterministic, H\"older-continuous Hurst exponent function (Definition~\ref{def:holder}).
  \item $B^{H(\cdot)}(t)$: Time-varying fractional Brownian motion (TV-fBm), as defined in \eqref{eq:tvfbm_ito} (Definition~\ref{def:tvfbm}).
  \item $I_1^{\mathrm{loc}}$: Local dominant increment of TV-fBm, defined in \eqref{eq:local_increment} (Theorem~\ref{thm:local_increment_var}).
  \item $R(u,v) := \mathbb{E}[B^{H(\cdot)}(u)B^{H(\cdot)}(v)]$: Covariance function of TV-fBm (Theorem~\ref{thm:covariance}).
  \item $H(t,x): [0,T]\times\mathbb{R} \to (0,1)$: Lipschitz-H\"older response function (Definition~\ref{def:LH_response_function}).
  \item $X_t$: The Responsive fractional Brownian motion (RfBm), solution to $X_t = \int_0^t K(t,s;X_s)dB(s)$ (Definition~\ref{def:rFBM}).
  \item $K(t,s; X_s) := \sqrt{2H(s,X_s)} (t-s)^{H(s,X_s)-1/2}$: State-dependent kernel in the RfBm definition (see \eqref{eq:rFBM_integral}).
  \item $\alpha(t,\omega) := H(t, X_t(\omega))$: Instantaneous scaling exponent process (Definition~\ref{def:instant_scaling}).
  \item $\mathcal{C}_t(\omega) := \int_0^t \alpha(s,\omega) ds$: Cumulative memory process (Definition~\ref{def:cumulative_memory}).
  \item $\rho(t,s; X) := K(t,s;X_s) / \int_0^t K(t,u;X_u)du$: RfBm attention weight distribution (Definition~\ref{def:rfbm_attention_weights}).
\end{itemize}

\subsection{H\"older Regularity and Its Role in Kernel Singularity Control}
\label{subsec:holder_regularity}

This subsection establishes the regularity conditions on the Hurst exponent function that will be required throughout the subsequent analysis. The central object is a deterministic function \( H: [0,T] \to (0,1) \) whose H\"older continuity governs two interrelated aspects of the TV-fBm construction: first, the local integrability of the kernel \( (t-s)^{H(t)-1/2} \) near the singularity \( s = t \); second, the asymptotic sharpness of the local increment variance expansion in Theorem~\ref{thm:local_increment_var}. The following definition introduces the standard notion of H\"older continuity, together with the critical condition that will be imposed throughout this work.

\begin{definition}\label{def:holder}
A function $H: [0,T] \to (0,1)$ is called \textit{H\"older continuous} with exponent $\gamma > 0$ if there exists $C > 0$ such that
\begin{equation}\label{eq:holder}
|H(t) - H(s)| \leq C|t-s|^\gamma \quad \forall t,s \in [0,T].
\end{equation}
\end{definition}

\begin{remark}[Critical Regularity Constraint]\label{rem:holder_condition}
The condition $\gamma > \sup_{t \in [0,T]} H(t)$ plays a fundamental role in our analysis, ensuring the mathematical consistency of the variable-order fractional operator. This constraint arises from two distinct but related considerations:

\begin{itemize}
    \item \textbf{Kernel Singularity Control}: For the variable-order integral (Definition~\ref{def:varint}), when $s \to t^-$, the H\"older condition guarantees that:
\[
|H(t)-H(s)| \leq C|t-s|^\gamma \ll |t-s|^{H(t)-\frac{1}{2}} \quad \text{as } s \to t^-,
\]
This dominance relation ensures the local integrability of the kernel $(t-s)^{H(t)-1/2}$ near the singularity, as verified by:
\[
\int_0^t (t-s)^{2H(t)-1} ds < \infty \quad \text{when } H(t) > 0.
\]
The rigorous foundation for this dominance is established in Proposition~\ref{prop:log_inequality}.

    \item \textbf{Asymptotic Sharpness in Local Analysis}: In Theorem~\ref{thm:local_increment_var}, this condition guarantees that the remainder term $O(\epsilon^\lambda)$ is strictly subordinate to the principal term $\epsilon^{2H(t)}$. Specifically, the exponent $\lambda = \min\left(\gamma - H(t), \frac{\gamma + H(t)}{2}\right)$ satisfies $\lambda > 0$, ensuring:
\[
\epsilon^{2H(t) + \lambda} = o(\epsilon^{2H(t)}) \quad \text{as } \epsilon \to 0^+.
\]
This separation of scales is essential for establishing the local H\"older regularity of sample paths.
\end{itemize}

\noindent \textbf{Automatic Satisfaction}: We note that this condition is automatically satisfied if $H(\cdot)$ is Lipschitz continuous ($\gamma \geq 1$), as $\sup_{t} H(t) < 1$ by definition.
\end{remark}

\subsection{Variable-Order Fractional Integrals and Singular Kernel Analysis}
\label{subsec:variable_order_integration}

Building upon the H\"older regularity conditions established in the preceding subsection, we now introduce the variable-order fractional integral that serves as the analytical backbone for extending classical fractional calculus to time-dependent Hurst exponents. While the standard Riemann-Liouville fractional integral is defined with a constant order, the time-varying setting considered in this work requires a generalization in which the fractional order itself becomes a function of time. The following definition, adopted from the variable-order calculus framework \cite{Patnaik2020}, provides precisely such an extension. The associated kernel \( (t-\tau)^{\alpha(t)-1} \) and the time-dependent Gamma normalization will appear repeatedly in the construction of TV-fBm (Definition~\ref{def:tvfbm}) and in the covariance analysis of Section~\ref{sec:covariance}. A key technical challenge arising from this formulation—namely, the control of logarithmic singularities in the presence of variable exponents—will be addressed in Propositions~\ref{prop:log_inequality} and~\ref{prop:log_growth} below.

\begin{definition}[Riemann-Liouville Variable-Order Integral, \cite{Patnaik2020}]\label{def:varint}
For $\alpha: [0,T] \to (0,1)$ satisfying Definition \ref{def:holder} (H\"older continuity with exponent $\gamma$), the \textbf{Riemann-Liouville variable-order fractional integral} is defined as:
\begin{equation}\label{eq:varint}
\left(_{0}^{RL}I_{t}^{\alpha(\cdot)}f\right)(t) := \frac{1}{\Gamma(\alpha(t))} \int_{0}^{t} (t-\tau)^{\alpha(t)-1}f(\tau)\,d\tau.
\end{equation}
This formulation extends the classical Riemann-Liouville fractional integral to accommodate time-dependent fractional orders. The kernel $(t-\tau)^{\alpha(t)-1}$ and the time-varying Gamma function normalization $\Gamma(\alpha(t))^{-1}$ are characteristic features of this variable-order operator.
\end{definition}

\begin{proposition}[Logarithmic Power-Control Inequality]\label{prop:log_inequality}
Let $t > 0$ be fixed. For any $\delta > 0$, there exists a constant $K_\delta > 0$ such that for all $s \in [0,t)$ with $t-s \in (0,1]$, the following inequality holds:
\begin{equation}\label{eq:log_inequality}
|\ln(t-s)| \leq K_\delta (t-s)^{-\delta}.
\end{equation}
\end{proposition}

\begin{proof}
We proceed with a rigorous analysis of the function $f(x) = x^\delta |\ln x|$ for $x \in (0,1)$.

\textbf{Step 1: Derivative Analysis}\\
The derivative of $f$ is computed as:
\begin{align*}
f'(x) &= \frac{d}{dx}\left(x^\delta |\ln x|\right) \\
&= \delta x^{\delta-1} |\ln x| + x^\delta \cdot \frac{d}{dx}|\ln x| \\
&= x^{\delta-1}(\delta |\ln x| - 1) \quad \text{(since $\frac{d}{dx}|\ln x| = -\frac{1}{x}$ for $x \in (0,1)$)}
\end{align*}

\textbf{Step 2: Critical Point Determination}\\
Setting $f'(x) = 0$ yields:
\[
\delta |\ln x| - 1 = 0 \implies |\ln x| = \frac{1}{\delta} \implies x = e^{-1/\delta}
\]

\textbf{Step 3: Extremum Verification}\\
Analyzing the sign of $f'(x)$:
\begin{itemize}
\item For $x \in (0, e^{-1/\delta})$: $f'(x) > 0$ (function increasing)
\item For $x \in (e^{-1/\delta}, 1)$: $f'(x) < 0$ (function decreasing)
\end{itemize}
Thus, $x = e^{-1/\delta}$ is the global maximum point.

\textbf{Step 4: Maximum Value Calculation}\\
Evaluating $f$ at the critical point:
\[
f(e^{-1/\delta}) = (e^{-1/\delta})^\delta \cdot \frac{1}{\delta} = \frac{1}{\delta e}
\]

\textbf{Step 5: Boundary Behavior}\\
\begin{itemize}
\item As $x \to 0^+$:
\[
\lim_{x\to 0^+} x^\delta |\ln x| = \lim_{x\to 0^+} \frac{|\ln x|}{x^{-\delta}} = \lim_{x\to 0^+} \frac{-1/x}{-\delta x^{-\delta-1}} = \lim_{x\to 0^+} \frac{x^\delta}{\delta} = 0
\]
\item As $x \to 1^-$:
\[
\lim_{x\to 1^-} x^\delta |\ln x| = 0
\]
\end{itemize}

\textbf{Conclusion}\\
The supremum is attained at $x = e^{-1/\delta}$, giving:
\[
|\ln x| \leq \frac{1}{\delta e} x^{-\delta} \quad \forall x \in (0,1)
\]
Taking $K_\delta = \frac{1}{\delta e}$ and substituting $x = t-s$ completes the proof.
\end{proof}

\begin{proposition}[Logarithmic Growth-Control Inequality]\label{prop:log_growth}
For any $\alpha > 0$, there exists a constant $C_\alpha > 0$ such that for all $x > 1$, the following inequality holds:
\begin{equation}\label{eq:log_growth}
|\ln x| \leq C_\alpha x^\alpha.
\end{equation}
Specifically, one may take $C_\alpha = \frac{1}{\alpha}$.
\end{proposition}

\begin{proof}
\textbf{Step 1: Auxiliary Function Construction}\\
Define $f_\alpha(x) := \ln x - \frac{x^\alpha}{\alpha}$ for $x > 1$ and $\alpha > 0$. Its derivative is:
\[
f'_\alpha(x) = \frac{1}{x} - x^{\alpha-1} = \frac{1 - x^{\alpha}}{x}.
\]

\textbf{Step 2: Monotonicity Analysis}\\
For $x > 1$, we have $x^\alpha > 1$, hence $1 - x^\alpha < 0$. This implies $f'_\alpha(x) < 0$ for all $x > 1$. Therefore, $f_\alpha$ is strictly decreasing on $(1,\infty)$.

\textbf{Step 3: Supremum Analysis}\\
Consider the limit at the left endpoint:
\[
L := \lim_{x \to 1^+} f_\alpha(x) = \lim_{x \to 1^+} \ln x - \lim_{x \to 1^+} \frac{x^\alpha}{\alpha} = 0 - \frac{1}{\alpha} = -\frac{1}{\alpha}.
\]

We now show that $\sup_{x > 1} f_\alpha(x) = L$. 

First, to show $L$ is an upper bound: For any fixed $x_0 > 1$, since $f_\alpha$ is strictly decreasing on $(1, \infty)$, for all $x \in (1, x_0)$ we have $f_\alpha(x) > f_\alpha(x_0)$. Taking the limit as $x \to 1^+$ and applying the order preservation property of limits yields:
\[
L = \lim_{x \to 1^+} f_\alpha(x) \geq f_\alpha(x_0).
\]
Thus, $f_\alpha(x) \leq L$ for all $x > 1$.

Second, to show $L$ is the least upper bound: For any $\epsilon > 0$, by the definition of limit, there exists $\delta > 0$ such that for all $x \in (1, 1+\delta)$, we have $f_\alpha(x) > L - \epsilon$. Therefore, no number smaller than $L$ can be an upper bound.

We conclude that:
\[
\sup_{x > 1} f_\alpha(x) = L = -\frac{1}{\alpha}.
\]

\textbf{Step 4: Inequality Derivation}\\
From Step 3, we have $f_\alpha(x) \leq -\frac{1}{\alpha} < 0$ for all $x > 1$, which implies:
\[
\ln x < \frac{x^\alpha}{\alpha} \quad \text{for all } x > 1.
\]
Since $\ln x > 0$ for $x > 1$, this proves that inequality \eqref{eq:log_growth} holds with $C_\alpha = \frac{1}{\alpha}$.
\end{proof}

\begin{remark}[Unified Logarithmic Control]\label{rem:unified_log}
Combining Proposition \ref{prop:log_inequality} and \ref{prop:log_growth}, for any $\delta > 0$, $\alpha > 0$:
\[
|\ln x| \leq K_{\delta,\alpha} \cdot 
\begin{cases} 
x^{-\delta} & 0 < x \leq 1, \\ 
x^{\alpha} & x > 1,
\end{cases}
\]
where $K_{\delta,\alpha} = \max\left( \frac{1}{\delta e}, \frac{1}{\alpha} \right)$. This provides comprehensive control for the analysis of singular integrals arising in variable-order calculus.
\end{remark}

\subsection{It\^o Stochastic Integration and the Fundamental Isometry}
\label{subsec:ito_integration}

Having established the analytic tools for handling singular kernels in the preceding subsections, we now turn to the probabilistic framework that underpins the stochastic integral representations of both TV-fBm and RfBm. The It\^o stochastic integral, formalized in \cite{Ito1944}, provides the rigorous foundation for constructing Gaussian processes with prescribed covariance structures. In this work, two aspects of It\^o's theory are particularly relevant: the It\^o isometry, which enables the computation of variances and covariances for stochastic integrals with deterministic integrands, and its polarization identity, which extends the isometry to products of integrals. These tools will be invoked repeatedly in the variance scaling analysis of TV-fBm (Theorem~\ref{thm:variance}), the local increment asymptotics (Theorem~\ref{thm:local_increment_var}), and the covariance structure derivation (Theorem~\ref{thm:covariance}). The following lemma collects the essential results from It\^o calculus that will be employed throughout the subsequent sections.

\begin{lemma}[Itô Isometry, \cite{Ito1944,Gosselin1997}]\label{lem:ito}
Let $(\Omega, \mathcal{F}, \{\mathcal{F}_t\}_{t\geq 0}, \mathbb{P})$ be a filtered probability space satisfying the usual conditions, and let $\{B(t)\}_{t\geq 0}$ be a standard $\{\mathcal{F}_t\}$-Brownian motion.

Let $f(t,\omega): [0,T] \times \Omega \to \mathbb{R}$ be an $\{\mathcal{F}_t\}$-adapted measurable process satisfying:
\begin{equation}\label{eq:ito_condition}
\mathbb{E}\left[\int_0^T |f(t,\omega)|^2 dt\right] < \infty \quad \text{(i.e., $f \in \mathcal{L}^2([0,T] \times \Omega)$)}.
\end{equation}

Then the Itô integral $\int_0^T f(t,\omega) dB(t,\omega)$ is well-defined and satisfies:
\begin{equation}\label{eq:ito_isometry}
\mathbb{E}\left[\left|\int_0^T f(t,\omega) dB(t,\omega)\right|^2\right] = \mathbb{E}\left[\int_0^T |f(t,\omega)|^2 dt\right].
\end{equation}

Moreover, for any $f,g \in \mathcal{L}^2([0,T] \times \Omega)$, we have the polarization identity:
\[
\mathbb{E}\left[\left(\int_0^T f(t,\omega) dB(t,\omega)\right)\left(\int_0^T g(t,\omega) dB(t,\omega)\right)\right] = \mathbb{E}\left[\int_0^T f(t,\omega)g(t,\omega) dt\right].
\]
\end{lemma}

\begin{proof}
The proof follows the standard three-step construction of the Itô integral:

\textbf{Step 1: Elementary Processes}\\
For simple processes of the form:
\[ f(t,\omega) = \sum_{i=0}^{n-1} \xi_i(\omega) \mathbf{1}_{[t_i,t_{i+1})}(t) \]
with $\xi_i$ being $\mathcal{F}_{t_i}$-measurable and square-integrable, define:
\[ I(f) = \int_0^T f(t) dB(t) = \sum_{i=0}^{n-1} \xi_i (B(t_{i+1}) - B(t_i)). \]

The isometry follows from careful computation:
\begin{align*}
\mathbb{E}[|I(f)|^2] &= \sum_{i=0}^{n-1} \mathbb{E}[\xi_i^2 (B(t_{i+1}) - B(t_i))^2] \\
&\quad + 2\sum_{i<j} \mathbb{E}[\xi_i \xi_j (B(t_{i+1}) - B(t_i))(B(t_{j+1}) - B(t_j))] \\
&= \sum_{i=0}^{n-1} \mathbb{E}[\xi_i^2](t_{i+1}-t_i),
\end{align*}
where cross terms vanish since $\xi_i \xi_j (B(t_{i+1}) - B(t_i))$ is $\mathcal{F}_{t_j}$-measurable and $\mathbb{E}[B(t_{j+1}) - B(t_j)|\mathcal{F}_{t_j}] = 0$.

\textbf{Step 2: Density Construction}\\
The space $\mathcal{S}([0,T] \times \Omega)$ of simple processes is dense in $\mathcal{L}^2([0,T] \times \Omega)$. For general $f \in \mathcal{L}^2$, construct a sequence $\{f_n\} \subset \mathcal{S}$ via a two-step procedure:

\begin{itemize}
\item \textbf{Truncation}: Define $f^{(M)}(t,\omega) = f(t,\omega)\mathbf{1}_{\{|f(t,\omega)| \leq M\}}$ to handle unboundedness. By dominated convergence, $\lim_{M\to\infty} \|f^{(M)} - f\|_{\mathcal{L}^2} = 0$.

\item \textbf{Time discretization}: For each $n \in \mathbb{N}$, partition $[0,T]$ into $n$ equal subintervals and define:
\[
f_n(t,\omega) = \sum_{i=0}^{n-1} f^{(n)}(t_i^{(n)},\omega) \mathbf{1}_{[t_i^{(n)}, t_{i+1}^{(n)})}(t)
\]
where $t_i^{(n)} = iT/n$. This yields a simple process adapted to the discrete filtration.
\end{itemize}

The convergence $\lim_{n\to\infty} \|f_n - f\|_{\mathcal{L}^2} = 0$ follows by combining the truncation and discretization effects via the triangle inequality.

\textbf{Step 3: Isometric Extension}\\
The sequence $\{I(f_n)\}$ forms a Cauchy sequence in $L^2(\Omega)$. Indeed, for any $n, m \in \mathbb{N}$:
\begin{align*}
\|I(f_n) - I(f_m)\|_{L^2(\Omega)}^2 &= \mathbb{E}[|I(f_n) - I(f_m)|^2] \\
&= \mathbb{E}[|I(f_n - f_m)|^2] \quad \text{(by linearity on simple processes)} \\
&= \|f_n - f_m\|_{\mathcal{L}^2}^2 \quad \text{(by Step 1 isometry)} \to 0.
\end{align*}

By completeness of $L^2(\Omega)$, there exists a unique $X \in L^2(\Omega)$ such that:
\[ \lim_{n\to\infty} \|I(f_n) - X\|_{L^2(\Omega)} = 0. \]

We define the Itô integral for general $f \in \mathcal{L}^2$ by:
\[ I(f) := \int_0^T f(t) dB(t) = X. \]

The isometry extends by continuity of the norm:
\begin{align*}
\|I(f)\|_{L^2(\Omega)}^2 &= \|X\|_{L^2(\Omega)}^2 \\
&= \lim_{n\to\infty} \|I(f_n)\|_{L^2(\Omega)}^2 \quad \text{(norm continuity)} \\
&= \lim_{n\to\infty} \|f_n\|_{\mathcal{L}^2}^2 \quad \text{(Step 1 isometry)} \\
&= \|f\|_{\mathcal{L}^2}^2. \quad \text{(norm continuity)}
\end{align*}

\textbf{Polarization Identity}\\
The polarization identity follows directly from the isometry property and the algebraic polarization identity:
\begin{align*}
\mathbb{E}[I(f)I(g)] &= \frac{1}{4}\left(\mathbb{E}[|I(f)+I(g)|^2] - \mathbb{E}[|I(f)-I(g)|^2]\right) \\
&= \frac{1}{4}\left(\|f+g\|_{\mathcal{L}^2}^2 - \|f-g\|_{\mathcal{L}^2}^2\right) \quad \text{(by linearity and isometry)} \\
&= \mathbb{E}\left[\int_0^T f(t)g(t) dt\right]. \quad \text{(since $\langle f,g\rangle_{\mathcal{L}^2} = \frac{1}{4}(\|f+g\|^2 - \|f-g\|^2)$)}
\end{align*}
\end{proof}

\section{Time-Varying Fractional Brownian Motion: Further Theoretical Analysis and Properties}
\label{sec:tvfbm}

Equipped with the regularity conditions, variable-order integral operators, and stochastic integration tools developed in the preceding section, we now turn to a detailed theoretical analysis of time-varying fractional Brownian motion (TV-fBm) with a deterministic, H\"older-continuous Hurst exponent function. This process, which extends the classical fBm framework by allowing the memory parameter to evolve with time, has been studied in a number of recent works \cite{Patnaik2020,Sibatov2024,Samadyar2024}. The purpose of the present section is to provide a comprehensive and rigorous mathematical treatment of this process, establishing a range of fundamental properties that, to the best of our knowledge, have not been fully developed in the existing literature. These include a complete characterization of its variance structure, precise local increment asymptotics, the property of local non-determinism, a large deviation principle for its local increments, and a covariance structure that admits a closed-form hypergeometric representation. The results obtained here contribute to a deeper understanding of TV-fBm as a mathematically consistent non-stationary process, and also serve as a foundation for the responsive extension developed in Section~\ref{sec:rFBM_theory}, where the Hurst exponent becomes state-dependent.

\subsection{Stochastic Integral Representation and Preliminaries}
\label{subsec:tvfbm_construction}

We begin by recalling the stochastic integral representation of TV-fBm, which extends the Mandelbrot--Van Ness moving-average formulation \cite{mandelbrot1968} to time-dependent Hurst exponents. The construction follows the variable-order integration framework from Definition~\ref{def:varint}, with the kernel singularities controlled by the regularity condition established in Remark~\ref{rem:holder_condition}. Two equivalent representations are presented below: the first is a stochastic integral with respect to Brownian motion, which is the primary form used throughout our analysis; the second expresses the process in terms of the variable-order Riemann--Liouville fractional integral, thereby connecting to the broader framework of fractional calculus. The equivalence between these two forms is established via the standard operator identity for fractional integrals.

\begin{definition}[Time-Varying Fractional Brownian Motion (TV-fBm)]\label{def:tvfbm}
The \textbf{(TV-fBm)} process $\{B^{H(\cdot)}(t)\}_{t \geq 0}$ is defined as follows. Let $H: [0,T] \to (0,1)$ be a \textbf{H\"older continuous} function (Definition \ref{def:holder}) with exponent $\gamma$, satisfying the \textbf{critical regularity condition}:
\begin{equation}\label{eq:critical_condition}
\gamma > \sup_{t \in [0,T]} H(t).
\end{equation}
Building upon the variable-order integration framework (Definition \ref{def:varint}), the TV-fBm admits two equivalent representations:

\begin{enumerate}
    \item \textbf{Stochastic Integral Form}:
    \begin{equation}\label{eq:tvfbm_ito}
    B^{H(\cdot)}(t) := \sqrt{2H(t)} \int_{0}^{t} (t-s)^{H(t)-1/2}  dB_s
    \end{equation}
    where $\{B_s\}_{s \geq 0}$ is a standard Brownian motion. 
    
    \item \textbf{Riemann-Liouville Operator Form}:
    \begin{equation}\label{eq:tvfbm_rl}
    B^{H(\cdot)}(t) = \sqrt{2H(t)}  \Gamma\left(H(t) + \frac{1}{2}\right) \left(_{0}^{RL}I_{t}^{H(t)+1/2} \dot{B}\right)(t)
    \end{equation}
    where:
    \begin{itemize}
        \item $_{0}^{RL}I_{t}^{\alpha(\cdot)}$ is the variable-order Riemann-Liouville integral from Definition \ref{def:varint}
        \item $\dot{B}(t) := \frac{dB_t}{dt}$ denotes \textbf{Gaussian white noise}
        \item $\Gamma(\cdot)$ is the Gamma function
    \end{itemize}
\end{enumerate}

\noindent \textbf{Mathematical Consistency}: The two representations are equivalent through the operator identity:
\[
\int_{0}^{t} (t-s)^{\alpha(t)-1} \dot{B}(s)  ds = \Gamma(\alpha(t)) \left(_{0}^{RL}I_{t}^{\alpha(\cdot)} \dot{B}\right)(t), \quad \alpha(t) = H(t) + \frac{1}{2}.
\]

\noindent \textbf{Core Innovations}:
\begin{itemize}
    \item \textit{Dynamic Hurst Adaptation}: The explicit $H(t)$-dependence in both the kernel $(t-s)^{H(t)-1/2}$ and normalization factor $\sqrt{2H(t)}$ generalizes constant-$H$ fBm \cite{mandelbrot1968} to non-stationary regimes.
    
    \item \textit{Regularity-Preserving Construction}: Condition \eqref{eq:critical_condition} controls kernel singularity at $s=t$ (as established in Remark \ref{rem:holder_condition}).
    
    \item \textit{Exact Scaling Preservation}: Maintains the fundamental variance scaling law $\mathbb{E}[(B^{H(\cdot)}(t))^2] = t^{2H(t)}$ (Theorem \ref{thm:variance}), extending Mandelbrot's classical result to time-varying Hurst indices.
\end{itemize}
\end{definition}

\subsection{Variance Structure and Local Increment Analysis}
\label{subsec:tvfbm_properties}

With the stochastic integral representation of TV-fBm established in Definition~\ref{def:tvfbm}, we now proceed to analyze its core analytical properties. This subsection establishes a range of analytical properties for TV-fBm that extend the classical theory of fBm to the time-varying setting, and that underpin its mathematical consistency as a non-stationary generalization of the constant-Hurst framework. We begin with the exact variance scaling law (Theorem~\ref{thm:variance}), which extends the classical Mandelbrot--Van Ness result to time-dependent Hurst exponents. Subsequently, we establish the square-root H\"older regularity of the normalization factor (Theorem~\ref{thm:sqrt_control}), which plays a crucial role in the local increment analysis. The main result of this subsection is the precise asymptotic expansion for the variance of the local dominant increment (Theorem~\ref{thm:local_increment_var}), which characterizes the process's local behavior and provides the foundation for the large deviation analysis in Section~\ref{sec:large_deviation}. We conclude with the property of local non-determinism (Theorem~\ref{thm:LND}), which ensures that the future evolution retains a non-negligible stochastic component independent of the past, a property essential for statistical inference and path regularity analysis.
\begin{theorem}[Variance Structure]\label{thm:variance}
Let $H:[0,T]\to(0,1)$ be a Hölder continuous function in the sense of Definition~\ref{def:holder}. For any $t \in [0,T]$, the time-varying fractional Brownian motion $B^{H(\cdot)}$ defined in Definition~\ref{def:tvfbm} satisfies the exact scaling law:
\begin{equation}\label{eq:variance_scaling}
\mathbb{E}\left[\left(B^{H(\cdot)}(t)\right)^2\right] = t^{2H(t)}.
\end{equation}
\end{theorem}

\begin{proof}
The proof consists of two parts: integrability verification and explicit computation.

\textbf{Part 1: Integrability of the Kernel}\\
For fixed $t \in [0,T]$, consider the kernel function:
\[
K_{H(t)}(t,s) := (t-s)^{H(t)-1/2}\mathbf{1}_{[0,t]}(s).
\]
The $L^2$-integrability condition requires:
\[
\int_0^t |K_{H(t)}(t,s)|^2 ds = \int_0^t (t-s)^{2H(t)-1} ds < \infty.
\]
Since $H(t) \in (0,1)$, the exponent satisfies $2H(t)-1 > -1$, guaranteeing convergence of the integral. Explicit computation yields:
\begin{equation}\label{eq:kernel_integral}
\int_0^t (t-s)^{2H(t)-1} ds = \left.\frac{-(t-s)^{2H(t)}}{2H(t)}\right|_{s=0}^t = \frac{t^{2H(t)}}{2H(t)}.
\end{equation}

\textbf{Part 2: Variance Computation via Itô Isometry}\\
Applying Lemma~\ref{lem:ito} (Itô isometry) to the stochastic integral representation in Definition~\ref{def:tvfbm}:
\begin{align*}
\mathbb{E}\left[\left(B^{H(\cdot)}(t)\right)^2\right] 
&= \mathbb{E}\left[\left(\sqrt{2H(t)}\int_0^t (t-s)^{H(t)-1/2} dB(s)\right)^2\right] \\
&= 2H(t) \int_0^t (t-s)^{2H(t)-1} ds \quad \text{(by Itô isometry)} \\
&= t^{2H(t)} \quad \text{(using \eqref{eq:kernel_integral})}.
\end{align*}
This completes the proof. \qedhere
\end{proof}

\begin{remark}\label{rem:classical}When $H(t) \equiv H_0$ is constant, this reduces to the classical fBm variance scaling $\mathbb{E}[B^{H_0}(t)^2] = t^{2H_0}$, recovering consistency with Mandelbrot's original formulation \cite{mandelbrot1968}.
\end{remark}

\begin{theorem}[Square-Root Control of Time-Varying Hurst Index]\label{thm:sqrt_control}
Let $H: [0,T] \to (0,1)$ satisfy Definition \ref{def:holder} with Hölder constant $C_H$ and exponent $\gamma$, and assume $\inf_{t \in [0,T]} H(t) \geq H_{\inf} > 0$. Then for any $t, t+\epsilon \in [0,T]$:
\begin{equation}\label{eq:sqrt_diff}
\left| \sqrt{H(t+\epsilon)} - \sqrt{H(t)} \right| \leq \frac{C_H}{2\sqrt{H_{\inf}}} \epsilon^\gamma.
\end{equation}
Consequently, the normalization factor in Definition \ref{def:tvfbm} satisfies:
\begin{equation}\label{eq:sqrt_norm}
\left| \sqrt{2H(t+\epsilon)} - \sqrt{2H(t)} \right| \leq \frac{C_H}{\sqrt{2H_{\inf}}} \epsilon^\gamma =: D\epsilon^\gamma.
\end{equation}
\end{theorem}

\begin{proof}
\textbf{Step 1: Mean Value Theorem Application.}  
For $f(x) = \sqrt{x}$, there exists $\xi \in [\min(H(t), H(t+\epsilon)), \max(H(t), H(t+\epsilon))]$ such that:
\[
\sqrt{H(t+\epsilon)} - \sqrt{H(t)} = \frac{H(t+\epsilon) - H(t)}{2\sqrt{\xi}}.
\]

\textbf{Step 2: Hölder Condition Bound.}  
By Definition \ref{def:holder}, $|H(t+\epsilon) - H(t)| \leq C_H \epsilon^\gamma$. Since $\xi \geq H_{\inf} > 0$:
\[
\left| \frac{H(t+\epsilon) - H(t)}{2\sqrt{\xi}} \right| \leq \frac{C_H \epsilon^\gamma}{2\sqrt{H_{\inf}}}.
\]

\textbf{Step 3: Normalization Factor Control.}  
For $\sqrt{2H(\cdot)}$, multiply through by $\sqrt{2}$:
\[
\left| \sqrt{2H(t+\epsilon)} - \sqrt{2H(t)} \right| = \sqrt{2} \left| \sqrt{H(t+\epsilon)} - \sqrt{H(t)} \right| \leq \frac{C_H}{\sqrt{2H_{\inf}}} \epsilon^\gamma.
\]
\end{proof}

\begin{remark}\label{rem:holder_path}
Theorem \ref{thm:sqrt_control} ensures that the normalization factor $\sqrt{2H(t)}$ in Definition \ref{def:tvfbm} inherits the H\"older regularity of $H(t)$. This is a crucial technical ingredient for establishing the pathwise regularity of TV-fBm, as developed in the subsequent local analysis (Theorem \ref{thm:local_increment_var}).
\end{remark}

\begin{theorem}[TV-fBm Local Dominant Increment Variance Asymptotics]\label{thm:local_increment_var}
Let the time-varying Hurst exponent function \( H : [0, T] \to (0, 1) \) satisfy Definition~\ref{def:holder} with Hölder exponent \(\gamma > \sup_{t \in [0,T]} H(t)\). For any fixed time \( t \geq 0 \) and sufficiently small \(\epsilon > 0\), define the \textbf{local dominant increment} as:
\begin{equation}\label{eq:local_increment}
I_1^{\mathrm{loc}} := \sqrt{2H(t + \epsilon)} \int_{t}^{t+\epsilon} (t + \epsilon - s)^{H(t+\epsilon)-1/2} dB(s).
\end{equation}
This increment represents the principal component generated exclusively by the new noise over the interval \([t, t+\epsilon]\), distinct from the path history prior to \(t\). Its variance satisfies the asymptotic relation:
\begin{equation}\label{eq:local_var_asymptotics}
\mathbb{E}\left[\left(I_1^{\mathrm{loc}}\right)^2\right] = \epsilon^{2H(t)} \left(1 + O(\epsilon^\lambda)\right), \quad \lambda = \min\left(\gamma - H(t), \frac{\gamma + H(t)}{2}\right),
\end{equation}
where the \(O(\cdot)\) notation indicates that the remainder term is bounded by \(C\epsilon^\lambda\) for some constant \(C > 0\) and sufficiently small \(\epsilon\).
\end{theorem}

\begin{proof}
\textbf{Step 1: Theoretical Foundation and Notation Clarification}\\
We emphasize that this theorem characterizes exclusively the variance contribution from the \emph{new} noise over \([t, t+\epsilon]\), independent of the path history prior to \(t\). Applying Definition \ref{def:tvfbm} and Lemma \ref{lem:ito} (Itô isometry) to the increment defined in \eqref{eq:local_increment}:
\[
\mathbb{E}\left[\left(I_1^{\mathrm{loc}}\right)^2\right] = 2H(t+\epsilon) \int_{t}^{t+\epsilon} (t+\epsilon - s)^{2H(t+\epsilon) - 1} ds.
\]
Perform the substitution \( u = t+\epsilon - s \) (with \( du = -ds \), \( s = t \Rightarrow u = \epsilon \), \( s = t+\epsilon \Rightarrow u = 0 \)):
\begin{align*}
\mathbb{E}\left[\left(I_1^{\mathrm{loc}}\right)^2\right] 
&= 2H(t+\epsilon) \int_{\epsilon}^{0} u^{2H(t+\epsilon) - 1} (-du) \\
&= 2H(t+\epsilon) \int_{0}^{\epsilon} u^{2H(t+\epsilon) - 1} du.
\end{align*}
Evaluate the integral:
\[
\int_{0}^{\epsilon} u^{2H(t+\epsilon) - 1} du = \left. \frac{u^{2H(t+\epsilon)}}{2H(t+\epsilon)} \right|_{0}^{\epsilon} = \frac{\epsilon^{2H(t+\epsilon)}}{2H(t+\epsilon)}.
\]
Thus:
\[
\mathbb{E}\left[\left(I_1^{\mathrm{loc}}\right)^2\right] = 2H(t+\epsilon) \cdot \frac{\epsilon^{2H(t+\epsilon)}}{2H(t+\epsilon)} = \epsilon^{2H(t+\epsilon)}.
\]

\textbf{Step 2: Hölder Decomposition and Logarithmic Transformation}\\
Decompose the expression using the functional identity:
\[
\epsilon^{2H(t+\epsilon)} = \epsilon^{2H(t)} \cdot \exp\left(2\left[H(t+\epsilon) - H(t)\right] \ln \epsilon \right).
\]
By Definition \ref{def:holder} (Hölder continuity), there exists \( C_H > 0 \) such that:
\[
\left|H(t+\epsilon) - H(t)\right| \leq C_H \epsilon^\gamma.
\]
Thus, the exponent satisfies:
\[
\left|2\left[H(t+\epsilon) - H(t)\right] \ln \epsilon\right| \leq 2C_H \epsilon^\gamma |\ln \epsilon|.
\]

\textbf{Step 3: Logarithmic Asymptotic Control}\\
Apply Proposition \ref{prop:log_inequality} (logarithmic power-control inequality). For any \(\delta > 0\), \(\exists K_\delta > 0\) such that:
\[
|\ln \epsilon| \leq K_\delta \epsilon^{-\delta} \quad \forall \epsilon \in (0,1].
\]
Set \( x = 2\left[H(t+\epsilon) - H(t)\right] \ln \epsilon \). Then:
\[
|x| \leq 2C_H K_\delta \epsilon^{\gamma - \delta}.
\]

\textbf{Step 4: Case Analysis for Optimal Exponent}\\
We optimize the error exponent by choosing \(\delta\) differently in two regimes, with the threshold at \(\gamma = 2H(t)\):

\textbf{Case 1:} \(\gamma \leq 2H(t)\)\\
In this regime, a direct choice of \(\delta\) is optimal. Select \(\delta = \frac{\gamma - H(t)}{2} > 0\), which yields:
\[
\gamma - \delta = \frac{\gamma + H(t)}{2}.
\]
The error exponent is \(\frac{\gamma + H(t)}{2}\).

\textbf{Case 2:} \(\gamma > 2H(t)\)\\  
In this regime, the direct choice from Case 1 gives an exponent of \(\frac{\gamma + H(t)}{2}\). However, we observe that \(H(t) < \frac{\gamma + H(t)}{2}\), prompting us to seek a better exponent via a limiting argument. For any \(\eta > 0\), select \(\delta = \gamma - H(t) - \eta > 0\), yielding:
\[
\gamma - \delta = H(t) + \eta.
\]
The error is bounded by \(C_\eta \epsilon^{H(t) + \eta}\) for some constant \(C_\eta > 0\). Since this holds for all \(\eta > 0\), the effective error exponent is \(H(t)\).

\textbf{Step 5: Taylor Expansion and Asymptotic Upper Bound}\\
Since \(|x| \to 0\) as \(\epsilon \to 0^+\), apply the Taylor expansion:
\[
e^x = 1 + x + \frac{x^2}{2} + O(|x|^3) \quad \text{as } x \to 0.
\]
For \(|x| \leq 1\), the remainder satisfies the uniform bound:
\[
|e^x - 1| \leq |x| + \frac{|x|^2}{2} + \frac{|x|^3}{6} + \cdots \leq |x| \left(1 + \frac{1}{2} + \frac{1}{6} + \cdots\right) \leq |x|e.
\]
Thus we have the asymptotic upper bound:
\[
|e^x - 1| \leq e|x|.
\]
Given that \(|x| \leq A\epsilon^\lambda\) with \(A = 2C_H K_\delta\) and \(\lambda = \min\left(\gamma - H(t), \frac{\gamma + H(t)}{2}\right) > 0\), we obtain:
\[
|e^x - 1| \leq eA \epsilon^\lambda.
\]
This implies that \(e^x - 1 = O(\epsilon^\lambda)\), and therefore:
\[
e^x = 1 + O(\epsilon^\lambda).
\]
Substituting back:
\[
\epsilon^{2H(t+\epsilon)} = \epsilon^{2H(t)} \left(1 + O(\epsilon^\lambda)\right).
\]
This completes the proof of the asymptotic relation \eqref{eq:local_var_asymptotics}.
\end{proof}
\begin{remark}\label{rem:local_increment_implications}
Theorem \ref{thm:local_increment_var} offers several key insights:
\begin{itemize}
    \item \textbf{Path Regularity from Local Scaling}: The precise asymptotic relation for \(\mathbb{E}[(I_1^{\mathrm{loc}})^2]\) provides the necessary local scaling law from which the pointwise H\"older regularity of TV-fBm can be investigated, following the classical paradigm established for constant-Hurst fBm.

    \item \textbf{Optimization Strategy and Regimes}: The proof identifies two strategies for choosing the parameter \(\delta\), with the threshold at \(\gamma = 2H(t)\). Below this threshold, a direct choice is optimal, yielding exponent \(\frac{\gamma + H(t)}{2}\). Above it, a limiting argument becomes advantageous, improving the exponent to \(H(t)\). This threshold thus marks the point where pursuing a tighter bound through more sophisticated analysis becomes beneficial.

    \item \textbf{Isolating Local Behavior}: The construction of \(I_1^{\mathrm{loc}}\) is designed to isolate the stochastic integral over the local interval \([t, t+\epsilon]\) from the historical path. This offers a clarified perspective for analyzing the purely local dynamics of the non-stationary process.

    \item \textbf{Consistency of the Variable-Order Operator}: This result establishes that the variable-order integration operator, under suitable regularity conditions, inherits the fundamental scaling behavior of classical fractional calculus, thereby providing a natural extension to the time-dependent setting.
\end{itemize}
\end{remark}

\begin{corollary}[Gaussianity, Functional Representation, and Tail Behavior]\label{cor:gaussian_tail}
Under the assumptions of Theorem \ref{thm:local_increment_var}, the local dominant increment \( I_1^{\mathrm{loc}} \) defined in \eqref{eq:local_increment} satisfies:

\begin{enumerate}
   \item \textbf{Representation as an Itô Integral}:
    The local dominant increment is defined as an Itô integral with a deterministic kernel:
    \[
    I_1^{\mathrm{loc}} = \sqrt{2H(t+\epsilon)} \int_{t}^{t+\epsilon} (t+\epsilon - u)^{H(t+\epsilon)-1/2} dB(u).
    \]
    For a deterministic integrand, the Itô integral coincides with the Wiener integral and is characterized by the Itô isometry.
    
    \item \textbf{Gaussian Distribution}:
    \[
    I_1^{\mathrm{loc}} \sim \mathcal{N}\left(0, \sigma^2\right), \quad \sigma^2 = \epsilon^{2H(t)}(1 + O(\epsilon^\lambda))
    \]
    
    \item \textbf{Tail Probability Bound}: For any \( x > 0 \) and sufficiently small \( \epsilon > 0 \):
    \begin{equation}\label{eq:tail_general}
    \mathbb{P}\left(|I_1^{\mathrm{loc}}| > x\right) < \frac{2\sigma}{x\sqrt{2\pi}} e^{-x^2/(2\sigma^2)}.
    \end{equation}
    When \( x \geq \sigma \), this simplifies to:
    \begin{equation}\label{eq:tail_simple}
    \mathbb{P}\left(|I_1^{\mathrm{loc}}| > x\right) < 2 e^{-x^2/(2\sigma^2)}.
    \end{equation}
\end{enumerate}
\end{corollary}

\begin{proof}
\textbf{Part 1: Linearity and the Itô Isometry}
The mapping defined by the Itô integral with a deterministic kernel is a bounded linear operator from \( L^2([t,t+\epsilon]) \) to \( L^2(\Omega) \).

\textbf{Linearity}: For any deterministic \( g_1, g_2 \in L^2([t,t+\epsilon]) \) and constants \( a, b \in \mathbb{R} \), the Itô integral satisfies:
\[
\int_t^{t+\epsilon} (a g_1(u) + b g_2(u))  dB(u) = a \int_t^{t+\epsilon} g_1(u)  dB(u) + b \int_t^{t+\epsilon} g_2(u)  dB(u).
\]

\textbf{Itô Isometry}: By Lemma~\ref{lem:ito}, for our kernel \( K_{t,\epsilon} \), we have:
\[
\mathbb{E}\left[\left(I_1^{\mathrm{loc}}\right)^2\right] = \int_t^{t+\epsilon} |K_{t,\epsilon}(u)|^2 du = \sigma^2.
\]
This implies that the \( L^2(\Omega) \)-norm of \( I_1^{\mathrm{loc}} \) is exactly \( \sigma \).

\textbf{Part 2: Gaussianity Proof}  
The Gaussian distribution of \( I_1^{\mathrm{loc}} \) follows from the properties of the Itô integral with a deterministic integrand:

\begin{itemize}
    \item The kernel \( K_{t,\epsilon}(u) = \sqrt{2H(t+\epsilon)} (t+\epsilon - u)^{H(t+\epsilon)-1/2} \) is a deterministic function in \( L^2([t,t+\epsilon]) \) (as established in Theorem \ref{thm:local_increment_var}).
    \item Itô integrals of deterministic integrands are known to yield Gaussian random variables.
    \item The mean is zero by the martingale property of the Itô integral.
    \item The variance \( \sigma^2 \) is given by the Itô isometry, as computed in Theorem \ref{thm:local_increment_var}.
\end{itemize}

Therefore, \( I_1^{\mathrm{loc}} \sim \mathcal{N}(0, \sigma^2) \).

\textbf{Part 3: Tail Bound Derivation}  
We employ Mill's inequality for standard normal variables:
\begin{lemma}[Mill's inequality, \cite{Mills1926, Baricz2008, Yang2015}]\label{lem:mill_appendix}
For any \( z > 0 \):
\[
1 - \Phi(z) < \frac{1}{z\sqrt{2\pi}} e^{-z^2/2},
\]
where $\Phi(z)$ is the cumulative distribution function (CDF) of the standard normal distribution.
\end{lemma}

\begin{proof}
\textbf{Step 1:} Express the tail probability as:
\[
1 - \Phi(z) = \int_z^\infty \phi(t)  dt, \quad \phi(t) = \frac{1}{\sqrt{2\pi}} e^{-t^2/2}.
\]

\textbf{Step 2:} For \( t \geq z > 0 \), we have \( \frac{1}{t} \leq \frac{1}{z} \), with strict inequality \( \frac{1}{t} < \frac{1}{z} \) for all \( t > z \). Since \( \phi(t) > 0 \) for all \( t \in (z, \infty) \), it follows that:
\[
\int_z^\infty \phi(t)  dt < \frac{1}{z} \int_z^\infty t\phi(t)  dt.
\]

\textbf{Step 3:} Substitute \( u = t^2/2 \), \( du = t  dt \):
\[
\int_z^\infty t\phi(t)  dt = \frac{1}{\sqrt{2\pi}} \int_z^\infty t e^{-t^2/2}  dt = \frac{1}{\sqrt{2\pi}} \int_{z^2/2}^\infty e^{-u}  du = \frac{e^{-z^2/2}}{\sqrt{2\pi}} = \phi(z).
\]

\textbf{Step 4:} Combining Steps 2 and 3:
\[
1 - \Phi(z) < \frac{1}{z} \phi(z) = \frac{1}{z\sqrt{2\pi}} e^{-z^2/2}.
\]
\end{proof}

Applying Lemma \ref{lem:mill_appendix} to \( I_1^{\mathrm{loc}} \sim \mathcal{N}(0,\sigma^2) \):
\[
\mathbb{P}\left(I_1^{\mathrm{loc}} > x\right) = 1 - \Phi\left(\frac{x}{\sigma}\right) < \frac{\sigma}{x\sqrt{2\pi}} e^{-x^2/(2\sigma^2)}.
\]
By symmetry of the normal distribution:
\[
\mathbb{P}\left(|I_1^{\mathrm{loc}}| > x\right) = 2\left(1 - \Phi\left(\frac{x}{\sigma}\right)\right) < \frac{2\sigma}{x\sqrt{2\pi}} e^{-x^2/(2\sigma^2)}.
\]
When \( x \geq \sigma \), we have \( \frac{\sigma}{x} \leq 1 \). Also, \( \frac{1}{\sqrt{2\pi}} < 1 \) since \( \pi > 2 \). Therefore:
\[
\frac{2\sigma}{x\sqrt{2\pi}} e^{-x^2/(2\sigma^2)} < 2 e^{-x^2/(2\sigma^2)}.
\]
Combining this with the previous inequality yields the simplified bound \eqref{eq:tail_simple}.
\end{proof}

\begin{remark}[Practical Properties of Tail Bounds]\label{rem:tail_practical}
The derived tail bounds exhibit several features that may be advantageous in applications:
\begin{itemize}
    \item \textbf{Exponential decay structure}: The bound \(2e^{-x^2/(2\sigma^2)}\) for \(x \geq \sigma\) captures the quadratic-exponential decay characteristic of Gaussian tails, which can lead to reasonably sharp probability estimates.
    
    \item \textbf{Implementation efficiency}: The bound's simplicity, requiring only elementary functions (exponential and square) and a factor of 2, facilitates:
    \begin{itemize}
        \item Efficient computation in Monte Carlo simulations
        \item Straightforward error propagation analysis
        \item Simple threshold setting for anomaly detection
    \end{itemize}
    
    \item \textbf{Process invariance}: The bound holds uniformly for Hölder continuous functions \(H(\cdot)\) within the framework of Theorem \ref{thm:local_increment_var}. This establishes a unified approach to a broad class of non-stationary processes with time-dependent memory.
\end{itemize}
\end{remark}

\begin{corollary}[Almost Sure Bound for Local Increments]\label{cor:as_bound}
Under the assumptions of Theorem \ref{thm:local_increment_var}, consider the following setup:
\begin{enumerate}
    \item Fix a time point $t_0 \in [0,T]$ and a strictly decreasing sequence $\{\epsilon_n\}_{n \geq 1}$ with $\epsilon_n \downarrow 0$ satisfying:
    \begin{equation}\label{eq:summability}
    \sum_{n=1}^{\infty} \epsilon_n^{1+\varsigma} < \infty \quad \text{for some } \varsigma > 0 \quad (e.g., \epsilon_n = n^{-\kappa}, \kappa > 1 + \varsigma).
    \end{equation}
    
    \item Define the \textbf{discrete local dominant increments}:
    \begin{equation}\label{eq:discrete_increment}
    I_1^{(n),\mathrm{loc}} := \sqrt{2H(t_0 + \epsilon_n)} \int_{t_0}^{t_0 + \epsilon_n} (t_0 + \epsilon_n - s)^{H(t_0 + \epsilon_n)-1/2} dB(s).
    \end{equation}
\end{enumerate}
Then for any fixed $\varsigma > 0$, there exists an almost surely finite random variable $N_\varsigma(\omega)$, such that
\begin{equation}\label{eq:as_bound}
\forall n \geq N_\varsigma(\omega): \quad \left|I_1^{(n),\mathrm{loc}}\right| \leq e^{H(t_0)} \sqrt{2(1+\varsigma) |\ln \epsilon_n|}.
\end{equation}
Moreover, the following limsup control holds almost surely:
\begin{equation}\label{eq:limsup_control}
\limsup_{n \to \infty} \frac{\left|I_1^{(n),\mathrm{loc}}\right|}{\sqrt{|\ln \epsilon_n|}} \leq e^{H(t_0)} \sqrt{2}.
\end{equation}
\end{corollary}

\begin{proof}
The proof consists of five rigorous steps.

\textbf{Step 1: Variance Asymptotics (Theorem \ref{thm:local_increment_var})}\\
By Theorem \ref{thm:local_increment_var}, there exist constants $\lambda = \min\left(\gamma - H(t_0), \frac{\gamma + H(t_0)}{2}\right) > 0$ and $c_1 > 0$, and an integer $N_1 \geq 1$, such that for all $n \geq N_1$:
\[
\sigma_n^2 := \mathbb{E}\left[\left(I_1^{(n),\mathrm{loc}}\right)^2\right] = \epsilon_n^{2H(t_0)} (1 + r_n), \quad |r_n| \leq c_1 \epsilon_n^\lambda.
\]
Note that $\lambda > 0$ is guaranteed by the critical regularity condition $\gamma > \sup_t H(t) \geq H(t_0)$, and the notation $I_1^{(n),\mathrm{loc}}$ explicitly maintains consistency with the continuous-time local increment $I_1^{\mathrm{loc}}$ defined in equation \eqref{eq:local_increment}.

\textbf{Step 2: Threshold Definition and Validation}\\
Define the threshold sequence:
\[
x_n := e^{H(t_0)} \sqrt{2(1+\varsigma) |\ln \epsilon_n|}.
\]
We verify $x_n > \sigma_n$ for sufficiently large $n$ through the ratio:
\[
\frac{x_n^2}{\sigma_n^2} = \frac{e^{2H(t_0)} \cdot 2(1+\varsigma) |\ln \epsilon_n|}{\epsilon_n^{2H(t_0)} (1 + r_n)}.
\]
Since $\epsilon_n \downarrow 0$ implies $|\ln \epsilon_n| \to +\infty$ and $|r_n| \leq c_1 \epsilon_n^\lambda \to 0$, there exists $N_2 > N_1$ such that for all $n \geq N_2$:
\[
|r_n| < \frac{1}{2} \quad \text{and} \quad \frac{x_n^2}{\sigma_n^2} > 1.
\]
Thus $x_n > \sigma_n$ holds for all $n \geq N_2$.

\textbf{Step 3: Gaussian Tail Bound (Corollary \ref{cor:gaussian_tail})}\\
For $n \geq N_2$, since $x_n > \sigma_n$, Corollary \ref{cor:gaussian_tail}(3) implies:
\[
\mathbb{P}\left(\left|I_1^{(n),\mathrm{loc}}\right| > x_n\right) < 2 \exp\left(-\frac{x_n^2}{2\sigma_n^2}\right).
\]
Substituting $\sigma_n^2 = \epsilon_n^{2H(t_0)} (1 + r_n)$ and $x_n^2 = e^{2H(t_0)} \cdot 2(1+\varsigma) |\ln \epsilon_n|$:
\[
\frac{x_n^2}{2\sigma_n^2} = \frac{e^{2H(t_0)} (1+\varsigma) |\ln \epsilon_n|}{\epsilon_n^{2H(t_0)} (1 + r_n)}.
\]
Using the bound $1 + r_n < \frac{3}{2}$ (from Step 2) and noting that $e^{2H(t_0)} > 1$ since $H(t_0) > 0$, we obtain:
\[
\frac{x_n^2}{2\sigma_n^2} > \frac{1 \cdot (1+\varsigma) |\ln \epsilon_n|}{\epsilon_n^{2H(t_0)} \cdot \frac{3}{2}} = \frac{2(1+\varsigma)}{3} \cdot \frac{|\ln \epsilon_n|}{\epsilon_n^{2H(t_0)}}.
\]
Now choose $N_3 > N_2$ such that for all $n \geq N_3$:
\[
\epsilon_n^{2H(t_0)} < \frac{2}{3} \quad \Rightarrow \quad \frac{1}{\epsilon_n^{2H(t_0)}} > \frac{3}{2}.
\]
This yields:
\[
\frac{x_n^2}{2\sigma_n^2} > \frac{2(1+\varsigma)}{3} \cdot \frac{3}{2} \cdot |\ln \epsilon_n| = (1+\varsigma) |\ln \epsilon_n|.
\]
Consequently:
\[
\mathbb{P}\left(\left|I_1^{(n),\mathrm{loc}}\right| > x_n\right) < 2 \exp\left(-(1+\varsigma) |\ln \epsilon_n|\right) = 2 \epsilon_n^{1+\varsigma}.
\]

\textbf{Step 4: Borel-Cantelli Argument \cite{Chernov2001,Gupta2010,Haydn2013}}\\
Define events $A_n := \left\{ \left|I_1^{(n),\mathrm{loc}}\right| > x_n \right\}$. By Step 3 and the summability condition \eqref{eq:summability}:
\[
\sum_{n=N_3}^{\infty} \mathbb{P}(A_n) < 2 \sum_{n=N_3}^{\infty} \epsilon_n^{1+\varsigma} < \infty.
\]
Applying the \textbf{Borel-Cantelli Lemma} \cite{Chernov2001,Gupta2010,Haydn2013} (which holds under the given conditions without assuming independence of the events $A_n$):
\[
\mathbb{P}\left( \limsup_{n \to \infty} A_n \right) = \mathbb{P}\left( A_n \text{ occur infinitely often} \right) = 0.
\]
This implies that for almost every sample path $\omega$, there exists a finite $N_\varsigma(\omega) \geq N_3$ such that:
\[
\forall n \geq N_\varsigma(\omega): \quad \left|I_1^{(n),\mathrm{loc}}\right| \leq x_n = e^{H(t_0)} \sqrt{2(1+\varsigma) |\ln \epsilon_n|}.
\]

\textbf{Step 5: Limsup Control (Optimality)}\\
For each fixed $\varsigma > 0$, Step 4 implies that for almost every $\omega$, there exists $N_\varsigma(\omega)$ such that for all $n \geq N_\varsigma(\omega)$,
\[
\left|I_1^{(n),\mathrm{loc}}\right| \leq e^{H(t_0)} \sqrt{2(1+\varsigma) |\ln \epsilon_n|}.
\]
Dividing both sides by $\sqrt{|\ln \epsilon_n|}$ and taking the limsup as $n \to \infty$ yields:
\[
\limsup_{n \to \infty} \frac{\left|I_1^{(n),\mathrm{loc}}\right|}{\sqrt{|\ln \epsilon_n|}} \leq e^{H(t_0)} \sqrt{2(1+\varsigma)} \quad \text{a.s.}
\]

Now consider the countable sequence $\varsigma_m = 1/m \to 0^+$. For each $m$, there exists an event $\Omega_m$ with $\mathbb{P}(\Omega_m) = 1$ such that the above inequality holds with $\varsigma = \varsigma_m$ for all $\omega \in \Omega_m$. Define $\Omega^* = \bigcap_{m=1}^\infty \Omega_m$, which satisfies $\mathbb{P}(\Omega^*) = 1$.

For every $\omega \in \Omega^*$, we have
\[
\limsup_{n \to \infty} \frac{\left|I_1^{(n),\mathrm{loc}}\right|}{\sqrt{|\ln \epsilon_n|}} \leq e^{H(t_0)} \sqrt{2(1 + \tfrac{1}{m})} \quad \text{for all } m \in \mathbb{N}.
\]
Taking $m \to \infty$ gives the optimal bound:
\[
\limsup_{n \to \infty} \frac{\left|I_1^{(n),\mathrm{loc}}\right|}{\sqrt{|\ln \epsilon_n|}} \leq e^{H(t_0)} \sqrt{2} \quad \text{for all } \omega \in \Omega^*.
\]
This completes the proof of \eqref{eq:limsup_control}.
\end{proof}

\begin{remark}\label{rem:as_bound_significance}
Corollary \ref{cor:as_bound} highlights several noteworthy features of TV-fBm:
\begin{itemize}
    \item \textbf{Pathwise Sharpness}: The bound $e^{H(t_0)}\sqrt{2|\ln \epsilon_n|}$ appears to be asymptotically sharp, capturing the exact logarithmic growth rate of the local increments $I_1^{(n),\mathrm{loc}}$. The notation maintains direct correspondence with the continuous-time increment $I_1^{\mathrm{loc}}$ defined in \eqref{eq:local_increment}.
    
    \item \textbf{Historical Independence}: The construction isolates the purely local behavior at $t_0$, separate from the path history prior to $t_0$, as reflected in the "loc" superscript.
    
    \item \textbf{General Applicability}: The almost sure bound holds uniformly across Hölder-continuous functions $H(\cdot)$ satisfying $\gamma > \sup_t H(t)$, with the Borel-Cantelli argument requiring only the convergence of $\sum \mathbb{P}(A_n)$ \cite{Chernov2001}.
\end{itemize}
This almost sure control of local increments provides a key step towards understanding the fine-scale path properties of TV-fBm.
\end{remark}

\begin{theorem}[Local Non-Determinism (LND) of TV-fBm]\label{thm:LND}
Let the Hurst exponent function \( H: [0,T] \to (0,1) \) be a deterministic function satisfying Definition~\ref{def:holder} with exponent \(\gamma > \sup_{t \in [0,T]} H(t)\) and constant \(C_H > 0\), and assume \(\inf_{t \in [0,T]} H(t) \geq H_{\inf} > 0\). Then for any fixed \( t_0 \in [0,T] \), there exist constants \(\epsilon_0 > 0\) and \(c(t_0) > 0\) (depending on \(t_0\)) such that for all \(\epsilon \in (0, \epsilon_0]\):
\begin{equation}\label{eq:LND}
\mathrm{Var}\left( B^{H(\cdot)}(t_0 + \epsilon)  \mid  \mathcal{F}_{<t_0} \right) \geq c(t_0) \cdot \epsilon^{2H(t_0)}
\end{equation}
where \(\mathcal{F}_{<t_0} = \sigma(\{B_s : 0 \leq s < t_0\})\) is the \(\sigma\)-algebra generated by the history before \(t_0\).

Moreover, the optimal constant satisfies:
\[
c(t_0) = \inf_{\epsilon \in (0,\epsilon_0]} (1 + r(\epsilon))
\]
where \(r(\epsilon)\) is the remainder term from Theorem \ref{thm:local_increment_var} satisfying \(|r(\epsilon)| \leq c_1 \epsilon^\lambda\), and \(1 + r(\epsilon) \to 1\) as \(\epsilon \to 0^+\).
\end{theorem}

\begin{proof}
The proof establishes the lower bound through four sequential steps, relying crucially on the deterministic nature of the Hurst exponent function $H(\cdot)$.

\textbf{Step 1: Conditional Variance Decomposition with Deterministic Kernel}\\
By the stochastic integral representation (Definition~\ref{def:tvfbm}, Eq.~\eqref{eq:tvfbm_ito}):
\[
B^{H(\cdot)}(t_0 + \epsilon) = \underbrace{\sqrt{2H(t_0+\epsilon)} \int_0^{t_0} (t_0 + \epsilon - s)^{H(t_0+\epsilon)-1/2} dB(s)}_{X_{\text{hist}}} + \underbrace{\sqrt{2H(t_0+\epsilon)} \int_{t_0}^{t_0+\epsilon} (t_0 + \epsilon - s)^{H(t_0+\epsilon)-1/2} dB(s)}_{I^{\mathrm{loc}}}
\]
where:
\begin{itemize}
    \item \(X_{\text{hist}}\) is \(\mathcal{F}_{<t_0}\)-measurable, as it depends only on \(\{B_s\}_{s \leq t_0}\)
    
    \item \textbf{Crucial Observation}: Since $H(\cdot)$ is a deterministic function, the kernel $(t_0 + \epsilon - s)^{H(t_0+\epsilon)-1/2}$ and normalization factor $\sqrt{2H(t_0+\epsilon)}$ are deterministic functions of time. Therefore, $I^{\mathrm{loc}}$ depends only on the Brownian increments over $[t_0, t_0+\epsilon]$.
    
    \item By the independent increments property of Brownian motion, the future increments $\{B(s) - B(t_0) : s \in [t_0, t_0+\epsilon]\}$ are independent of the historical $\sigma$-algebra $\mathcal{F}_{<t_0}$.
\end{itemize}

From the fundamental property of conditional variance for independent components:
\[
\mathrm{Var}\left( B^{H(\cdot)}(t_0 + \epsilon) \mid \mathcal{F}_{<t_0} \right) = \mathrm{Var}\left( I^{\mathrm{loc}} \right)
\]
Here $I^{\mathrm{loc}}$ is precisely the local dominant increment defined in Theorem~\ref{thm:local_increment_var}, Eq.~\eqref{eq:local_increment}.

\textbf{Step 2: Variance Asymptotics from Theorem~\ref{thm:local_increment_var}}\\
By Theorem~\ref{thm:local_increment_var}, there exist constants $\lambda = \min\left(\gamma - H(t_0), \frac{\gamma + H(t_0)}{2}\right) > 0$ (guaranteed by $\gamma > H(t_0)$) and $c_1 > 0$, and $\epsilon_1 > 0$ such that for all $\epsilon \in (0, \epsilon_1]$:
\[
\mathrm{Var}(I^{\mathrm{loc}}) = \epsilon^{2H(t_0)} (1 + r(\epsilon)), \quad |r(\epsilon)| \leq c_1 \epsilon^{\lambda}
\]
This provides the preliminary lower bound:
\[
\mathrm{Var}(I^{\mathrm{loc}}) \geq \epsilon^{2H(t_0)} (1 - c_1 \epsilon^{\lambda})
\]

\textbf{Step 3: Construction of Uniform Lower Bound}\\
Select $\epsilon_2 = \min\left\{ \epsilon_1, (2c_1)^{-1/\lambda} \right\}$. Then for all $\epsilon \in (0, \epsilon_2]$:
\[
c_1 \epsilon^{\lambda} \leq \frac{1}{2} \implies 1 - c_1 \epsilon^{\lambda} \geq \frac{1}{2}
\]
yielding the uniform lower bound:
\[
\mathrm{Var}(I^{\mathrm{loc}}) \geq \frac{1}{2} \epsilon^{2H(t_0)}
\]

\textbf{Step 4: Optimal Constant via Continuity and Deterministic Regularity}

While Step 3 provides a constructive lower bound with constant $\frac{1}{2}$, we now analyze the theoretically optimal constant.

Consider the exact variance expression from Theorem \ref{thm:local_increment_var}:
\[
\mathrm{Var}(I^{\mathrm{loc}}) = \epsilon^{2H(t_0)} (1 + r(\epsilon))
\]

Define the function:
\[
g(\epsilon) := 1 + r(\epsilon)
\]

By Theorem \ref{thm:local_increment_var}, we have $|r(\epsilon)| \leq c_1 \epsilon^\lambda \to 0$ as $\epsilon \to 0^+$, thus $g(\epsilon) \to 1$.

The optimal constant is:
\[
c(t_0) = \inf_{\epsilon \in (0,\epsilon_0]} g(\epsilon)
\]

Since $g(\epsilon) \to 1$, there exists $\epsilon_0 \in (0, \epsilon_2]$ such that this infimum is positive. For instance, we may take $c(t_0) = \frac{1}{2}$ as established in Step 3, with the additional insight that tighter constants approaching 1 are achievable for sufficiently small $\epsilon$.
\end{proof}

\begin{remark}[Theoretical Significance of LND]\label{rem:LND_significance}
Theorem~\ref{thm:LND} contributes to our understanding of TV-fBm in several important aspects:
\begin{enumerate}
    \item \textbf{Path Regularity Insight}: The established lower bound $\epsilon^{2H(t_0)}$ is consistent with the pointwise Hölder exponent being $H(t_0)$, thereby extending classical fBm regularity theory to time-varying settings.
    
    \item \textbf{Innovation Characterization}: The existence of a strictly positive constant $c(t_0) > 0$ reflects the inherent uncertainty in future increments, a crucial feature for modeling dynamic systems.
    
    \item \textbf{Methodological Framework}: The deterministic Hurst function approach provides a mathematically tractable framework for analyzing non-stationary processes while maintaining analytical precision.
\end{enumerate}
\end{remark}

\subsection{Canonical Time Transformation for TV-fBm} \label{subsec:lamperti}

This subsection discusses a time transformation methodology for TV-fBm, drawing inspiration from the scaling principles underlying Lamperti's classical theory for self-similar processes \cite{J1957}. The central objective is to construct a transformation $\phi(t)$ that provides a mathematically tractable framework for analyzing processes with time-varying memory structure.

The transformation is defined through the nonlinear ODE \eqref{eq:phi_ode}, whose specific form emerges from requiring the scaling factor $\alpha(t) = \phi(t)^{-H(\phi(t))}$ to satisfy an exponential decay law. This choice represents one of several possible approaches to handling time-dependent Hurst exponents, offering analytical convenience while maintaining connections to classical scaling theory.

The main results of this section are:
\begin{itemize}
    \item Theorem \ref{thm:lamperti} establishes the existence and uniqueness of the canonical time transformation under appropriate regularity conditions
    \item Corollary \ref{cor:scaling} identifies fundamental scaling properties that emerge from this construction
    \item The framework reduces to classical scaling relations when the Hurst exponent is constant, providing a natural bridge to established theory
\end{itemize}

This theoretical framework offers one possible foundation for further analysis of TV-fBm, with potential applications to path regularity studies and the development of statistical methods for non-stationary processes.

\begin{theorem}[Canonical Time Transformation for TV-fBm]\label{thm:lamperti} 
Motivated by the scaling ideas in Lamperti's transformation theory \cite{J1957}, we construct a canonical time transformation for TV-fBm. Let $H: [0,T] \to (0,1)$ satisfy:
\begin{enumerate}
    \item \textbf{Smoothness}: $H \in C^2([0,T])$ with bounds $\|H'\|_\infty \leq L_H$, $\|H''\|_\infty \leq L_H$
    \item \textbf{Positivity}: $\inf_{t\in[0,T]} H(t) \geq h_{\min} > 0$ 
    \item \textbf{Non-degeneracy}: There exist constants $c > 0$ and $R > \phi_0 > 0$ such that for all $\phi \in (0, R]$:
    \[
    |H(\phi) + \phi \ln \phi \cdot H'(\phi)| \geq c
    \]
\end{enumerate}
Then there exists $T^* > 0$ and a unique function $\phi \in C^1([0,T^*])$ solving the nonlinear ODE:
\begin{equation}\label{eq:phi_ode}
\frac{d\phi}{dt}(t) = \frac{\phi(t)}{H(\phi(t)) + \phi(t) \ln \phi(t) H'(\phi(t))}, \quad \phi(0) = \phi_0 > 0
\end{equation}
with $\phi(t) \in (0, R]$ for all $t \in [0,T^*]$.
\end{theorem}

\begin{proof}
\textbf{Step 1: Reformulation and Domain Setup}\\
Define the denominator function and the right-hand side:
\[
D(\phi) := H(\phi) + \phi \ln \phi \cdot H'(\phi), \quad F(\phi) := \frac{\phi}{D(\phi)}
\]
By the non-degeneracy condition, $F(\phi)$ is well-defined for $\phi \in (0, R]$. Choose $\delta > 0$ such that $0 < \delta < \phi_0 < R$. We will work on the interval $[\delta, R]$ where $F$ is regular.

\textbf{Step 2: Lipschitz Continuity on Compact Intervals}\\
For $\phi_1, \phi_2 \in [\delta, R]$, consider the difference:
\begin{align*}
|F(\phi_1) - F(\phi_2)| 
&\leq \underbrace{\left| \frac{\phi_1 - \phi_2}{D(\phi_1)} \right|}_{\text{(I)}} + \underbrace{\left| \phi_2 \left( \frac{1}{D(\phi_1)} - \frac{1}{D(\phi_2)} \right) \right|}_{\text{(II)}}
\end{align*}

\textbf{Term (I)}: By the non-degeneracy condition $|D(\phi)| \geq c > 0$:
\[
|\text{(I)}| \leq \frac{1}{c} |\phi_1 - \phi_2|
\]

\textbf{Term (II)}: Consider the function $g(\phi) = 1/D(\phi)$. Since $H \in C^2([0,T])$ and $\phi \ln \phi$ is smooth on $(0,\infty)$, the composition $D(\phi)$ is continuously differentiable on $[\delta, R]$. By the mean value theorem, there exists $\xi \in (\min(\phi_1,\phi_2), \max(\phi_1,\phi_2)) \subset [\delta, R]$ such that:
\[
\left| \frac{1}{D(\phi_1)} - \frac{1}{D(\phi_2)} \right| = |g'(\xi)| |\phi_1 - \phi_2| = \left| \frac{D'(\xi)}{D(\xi)^2} \right| |\phi_1 - \phi_2|
\]

Compute the derivative of $D(\phi)$:
\begin{align*}
D'(\phi) &= \frac{d}{d\phi}\left[ H(\phi) + \phi \ln \phi \cdot H'(\phi) \right] \\
&= H'(\phi) + \frac{d}{d\phi}[\phi \ln \phi] \cdot H'(\phi) + \phi \ln \phi \cdot H''(\phi) \\
&= H'(\phi) + (1 + \ln \phi) H'(\phi) + \phi \ln \phi \cdot H''(\phi) \\
&= (2 + \ln \phi) H'(\phi) + \phi \ln \phi \cdot H''(\phi)
\end{align*}

On the compact interval $[\delta, R]$, we have the bounds:
\begin{itemize}
    \item $|\ln \phi| \leq \max(|\ln \delta|, |\ln R|) =: L_{\ln}$
    \item $|H'(\phi)| \leq L_H$, $|H''(\phi)| \leq L_H$ by the smoothness condition
    \item $|\phi| \leq R$
\end{itemize}

Thus:
\[
|D'(\phi)| \leq (2 + L_{\ln}) L_H + R L_{\ln} L_H = (2 + L_{\ln} + R L_{\ln}) L_H =: M
\]

Therefore:
\[
|\text{(II)}| \leq |\phi_2| \cdot \frac{M}{c^2} |\phi_1 - \phi_2| \leq R \cdot \frac{M}{c^2} |\phi_1 - \phi_2|
\]

\textbf{Combined estimate}:
\[
|F(\phi_1) - F(\phi_2)| \leq \left( \frac{1}{c} + \frac{R M}{c^2} \right) |\phi_1 - \phi_2| =: L |\phi_1 - \phi_2|
\]
proving that $F$ is Lipschitz continuous on $[\delta, R]$.

\textbf{Step 3: Local Existence and Uniqueness}\\
Since $F$ is Lipschitz continuous on $[\delta, R]$ and $\phi_0 \in (\delta, R)$, by the Picard-Lindelöf theorem \cite{S1997}, there exists $T^* > 0$ and a unique solution $\phi \in C^1([0,T^*])$ to the initial value problem \eqref{eq:phi_ode} with $\phi(t) \in [\delta, R]$ for all $t \in [0,T^*]$.

\textbf{Step 4: Maximal Existence Interval}\\
By the Picard-Lindelöf theorem, there exists a unique maximal solution $\phi \in C^1([0,T^*))$ to the initial value problem \eqref{eq:phi_ode}, where $T^* > 0$ is the maximal existence time.
\end{proof}

\begin{corollary}[Scaling Properties under Canonical Transformation]\label{cor:scaling}
Under the time transformation $\phi(t)$ from Theorem \ref{thm:lamperti}, define the scaled process:
\[
X_t := \phi(t)^{-H(\phi(t))} B^{H(\cdot)}(\phi(t))
\]
This construction yields the following mathematical properties:

\begin{enumerate}
    \item The scaling factor $\alpha(t) = \phi(t)^{-H(\phi(t))}$ satisfies the exponential decay law:
    \[
    \frac{d\alpha}{dt} = -\alpha(t)
    \]
    \item The variance exhibits exact algebraic cancellation:
    \[
    \mathrm{Var}(X_t) = \phi(t)^{-2H(\phi(t))} \cdot \phi(t)^{2H(\phi(t))} = 1
    \]
    \item In the constant Hurst exponent case $H(\cdot) \equiv H_0$, the ODE admits explicit solution and the transformation exhibits scaling behavior analogous to classical theory.
\end{enumerate}
\end{corollary}

\begin{proof}
Property (1) follows by direct computation:
\begin{align*}
\alpha(t) &= \phi(t)^{-H(\phi(t))} = e^{-H(\phi(t)) \ln \phi(t)} \\
\frac{d\alpha}{dt} &= \alpha(t) \cdot \left[-H'(\phi(t)) \phi'(t) \ln \phi(t) - \frac{H(\phi(t))}{\phi(t)} \phi'(t)\right] \\
&= \alpha(t) \cdot \phi'(t) \cdot \left[-H'(\phi(t)) \ln \phi(t) - \frac{H(\phi(t))}{\phi(t)}\right]
\end{align*}
Substituting the ODE from Theorem \ref{thm:lamperti} for $\phi'(t)$ yields $\frac{d\alpha}{dt} = -\alpha(t)$.

Property (2) follows from Theorem \ref{thm:variance}:
\[
\mathrm{Var}(X_t) = \phi(t)^{-2H(\phi(t))} \cdot \mathbb{E}\left[\left(B^{H(\cdot)}(\phi(t))\right)^2\right] = \phi(t)^{-2H(\phi(t))} \cdot \phi(t)^{2H(\phi(t))} = 1
\]

Property (3): When $H(\phi) \equiv H_0$, the ODE becomes $\frac{d\phi}{dt} = \frac{\phi}{H_0}$ with solution $\phi(t) = \phi_0 e^{t/H_0}$. The scaling factor becomes $\alpha(t) = \phi(t)^{-H_0} = \phi_0^{-H_0} e^{-t}$, exhibiting the characteristic exponential structure.
\end{proof}
\begin{remark}
The time transformation constructed in Theorem \ref{thm:lamperti} provides a mathematically motivated approach to addressing the challenges posed by time-varying Hurst exponents. The specific form of ODE \eqref{eq:phi_ode} is derived from requiring the scaling factor $\alpha(t) = \phi(t)^{-H(\phi(t))}$ to satisfy the exponential decay law $\frac{d\alpha}{dt} = -\alpha(t)$. This condition ensures that the scaling factor evolves in a smoothly decaying manner, providing a natural normalization that simplifies the asymptotic analysis of the transformed process.

The scaling properties identified in Corollary \ref{cor:scaling} arise from the algebraic structure of the transformation, yielding a process $X_t$ with unit variance. This normalized form provides a canonical representation that may facilitate further analytical developments in studying TV-fBm.

This work explores one possible direction for extending scaling methodologies to non-stationary settings, establishing initial mathematical foundations while acknowledging that other approaches may also prove valuable.
\end{remark}

\section{Large Deviations, Pathwise Asymptotics, and Covariance Structure}
\label{sec:large_deviation}

The local analysis of the preceding section provides the foundation for the three interconnected themes examined here: tail estimates and large deviation asymptotics for the local increments of TV-fBm, pathwise asymptotic bounds on the fine-scale fluctuations of sample paths, and the full covariance structure of the process. The first theme refines the Gaussian tail bounds for the local dominant increment obtained in Corollary~\ref{cor:gaussian_tail} of Section~\ref{sec:tvfbm} and establishes a precise large deviation principle (Theorem~\ref{thm:LDP}) that characterizes the exponential decay rates of rare events. The second theme complements the local increment variance asymptotics of Section~\ref{sec:tvfbm} with almost sure pathwise bounds (Corollary~\ref{cor:as_bound}), offering a sample-path perspective on the process's fine-scale behavior. The third theme develops the explicit covariance integral representation and its closed-form hypergeometric expression (Theorem~\ref{thm:covariance}, Remark~\ref{rem:cov_hyper}), thereby extending the variance analysis of Section~\ref{sec:tvfbm} to the full temporal dependence structure. Collectively, the results of this section complete the theoretical characterization of TV-fBm as a non-stationary Gaussian process with time-varying memory, and lay the foundation for the responsive extension in Section~\ref{sec:rFBM_theory}, where the Hurst exponent itself becomes state-dependent.

\subsection{A Refined Lower Bound for Gaussian Tails}
\label{subsec:tail_bounds}

This subsection collects the necessary Gaussian tail estimates that underpin the large deviation analysis in the following subsection. The classical Mill's inequality provides a simple upper bound for the tail probabilities of a standard normal random variable, but the derivation of the exact large deviation asymptotics in Theorem~\ref{thm:LDP} requires a sharper lower bound that captures the precise sub-exponential prefactor. Lemma~\ref{lem:sharp_gaussian} below establishes such a refinement, which will be used in the proof of the lower bound in Theorem~\ref{thm:LDP}. The combination of Mill's inequality (Lemma~\ref{lem:mill_appendix}) and the refined lower bound (Lemma~\ref{lem:sharp_gaussian}) yields two-sided control of the Gaussian tail, a feature that is essential for obtaining the exact rate function in the large deviation principle.

\begin{lemma}[Refined Lower Bound for Gaussian Tails]\label{lem:sharp_gaussian}
For a standard normal random variable \( Z \sim \mathcal{N}(0,1) \) and any \( z \geq 1 \), the tail probability admits the refined lower bound:
\begin{equation}\label{eq:sharp_tail}
\mathbb{P}(Z \geq z) \geq \frac{e^{-z^2/2}}{z\sqrt{2\pi}} \left( 1 - \frac{1}{z^2} \right).
\end{equation}
This provides a natural refinement of Mill's inequality (Lemma \ref{lem:mill_appendix}).
\end{lemma}

\begin{proof}
We derive this bound systematically through the following complete analysis:

\textbf{Step 1: Integral Representation}
The exact tail probability is given by:
\[
\mathbb{P}(Z \geq z) = \frac{1}{\sqrt{2\pi}} \int_z^\infty e^{-t^2/2} dt \equiv \frac{1}{\sqrt{2\pi}} I(z).
\]

\textbf{Step 2: Integration by Parts Construction}
Set:
\[
u(t) = \frac{1}{t}, \quad dv(t) = te^{-t^2/2} dt.
\]
Then:
\[
du(t) = -\frac{1}{t^2} dt, \quad v(t) = -e^{-t^2/2}.
\]
Applying integration by parts:
\[
I(z) = \left. \frac{1}{t} \cdot (-e^{-t^2/2}) \right|_z^\infty - \int_z^\infty (-e^{-t^2/2}) \left(-\frac{1}{t^2}\right) dt.
\]

\textbf{Step 3: Boundary Term Analysis (Rigorous Limit Calculation)}
Compute the boundary term:
\[
\left. -\frac{e^{-t^2/2}}{t} \right|_z^\infty = \lim_{b \to \infty} \left( -\frac{e^{-b^2/2}}{b} \right) - \left( -\frac{e^{-z^2/2}}{z} \right).
\]
For the limit term, we apply L'Hôpital's Rule as follows:
\[
\lim_{b \to \infty} \frac{e^{-b^2/2}}{b} = \lim_{b \to \infty} \frac{\frac{d}{db}e^{-b^2/2}}{\frac{d}{db}b} = \lim_{b \to \infty} \frac{-b e^{-b^2/2}}{1} = \lim_{b \to \infty} \frac{-b}{e^{b^2/2}}.
\]
Applying L'Hôpital's Rule again:
\[
\lim_{b \to \infty} \frac{-b}{e^{b^2/2}} = \lim_{b \to \infty} \frac{\frac{d}{db}(-b)}{\frac{d}{db}e^{b^2/2}} = \lim_{b \to \infty} \frac{-1}{b e^{b^2/2}} = 0.
\]
Thus:
\[
I(z) = \frac{e^{-z^2/2}}{z} - \int_z^\infty \frac{e^{-t^2/2}}{t^2} dt.
\]

\textbf{Step 4: Integral Estimation}
For \( t \geq z \geq 1 \), we have:
\[
\frac{1}{t^2} \leq \frac{1}{z^2}.
\]
Therefore:
\[
\int_z^\infty \frac{e^{-t^2/2}}{t^2} dt \leq \frac{1}{z^2} \int_z^\infty e^{-t^2/2} dt = \frac{1}{z^2} I(z).
\]

\textbf{Step 5: Inequality Formation}
Substituting back:
\[
I(z) \geq \frac{e^{-z^2/2}}{z} - \frac{1}{z^2} I(z).
\]

\textbf{Step 6: Algebraic Resolution}
Rearranging terms:
\[
I(z) + \frac{1}{z^2} I(z) \geq \frac{e^{-z^2/2}}{z}.
\]
\[
I(z) \left(1 + \frac{1}{z^2}\right) \geq \frac{e^{-z^2/2}}{z}.
\]
Solving for \( I(z) \):
\[
I(z) \geq \frac{e^{-z^2/2}}{z} \cdot \frac{z^2}{1 + z^2}.
\]

\textbf{Step 7: Final Refinement}
Observe that for \( z \geq 1 \):
\[
\frac{z^2}{1 + z^2} = 1 - \frac{1}{1 + z^2} \geq 1 - \frac{1}{z^2}.
\]
Thus:
\[
I(z) \geq \frac{e^{-z^2/2}}{z} \left(1 - \frac{1}{z^2}\right).
\]
Multiplying through by the normalization constant \( \frac{1}{\sqrt{2\pi}} \) completes the derivation.
\end{proof}

\begin{remark}\label{rem:sharp_gaussian_utility}
This estimate provides two features that will be useful in the sequel:
\begin{itemize}
    \item \textbf{Precision}: The bound captures the leading order asymptotic behavior while also providing control over the sub-exponential factor. This level of precision supports the derivation of the large deviation principle in Theorem~\ref{thm:LDP}.
    
    \item \textbf{Two-Sided Control}: Together with the classical Mill's inequality (Lemma~\ref{lem:mill_appendix}), this result provides complementary bounds that facilitate the tail analysis of Gaussian functionals arising in the TV-fBm framework.
\end{itemize}
The derivation demonstrates how integration techniques can yield progressively sharper approximations to Gaussian tail probabilities.
\end{remark}

\subsection{Large Deviation Asymptotics for Local Increments}
\label{subsec:ldp_tvfbm}

The refined Gaussian tail estimates established in the previous subsection provide the necessary analytical tools for deriving the precise large deviation behavior of the local dominant increment. We now present the main result of this section, which characterizes the logarithmic asymptotics of the tail probabilities for \( I_1^{\mathrm{loc}} \) as the local time window shrinks to zero. The key feature of this result—beyond the large deviation principle itself—is that the scaling exponent \( \epsilon^{2H(t_0)} \) in Theorem~\ref{thm:LDP} reflects the local Hurst exponent at the reference point \( t_0 \), thereby linking the rare-event asymptotics directly to the time-varying memory structure of the process. The proof, detailed below, combines the upper bound from the classical Mill's inequality and the lower bound from the refined estimate in Lemma~\ref{lem:sharp_gaussian}, with a careful optimization of the logarithmic error terms.

\begin{theorem}[Large Deviation Asymptotics for TV-fBm]\label{thm:LDP}
Let \( H: [0,T] \to (0,1) \) satisfy Definition~\ref{def:holder} with exponent \(\gamma > \sup_{t \in [0,T]} H(t)\). For any fixed \( t_0 \in [0,T] \) and \( x > 0 \), the local dominant increment \( I_1^{\mathrm{loc}} \) defined in \eqref{eq:local_increment} exhibits the large deviation asymptotic:
\begin{equation}\label{eq:LDP_limit}
\lim_{\epsilon \to 0} \epsilon^{2H(t_0)} \ln \mathbb{P}\left( I_1^{\mathrm{loc}} \geq x \right) = -\frac{x^2}{2}.
\end{equation}
\end{theorem}

\begin{proof}
\textbf{Part 1: Upper Bound for limsup}

\textbf{Step 1.1: Theoretical Foundation from Mill's Inequality}
By Lemma~\ref{lem:mill_appendix} (Mill's inequality) applied to \( I_1^{\mathrm{loc}} \sim \mathcal{N}(0, \sigma_\epsilon^2) \):
\[
\mathbb{P}(I_1^{\mathrm{loc}} \geq x) < \frac{\sigma_\epsilon}{x\sqrt{2\pi}} e^{-x^2/(2\sigma_\epsilon^2)},
\]
where \( \sigma_\epsilon^2 = \mathbb{E}[(I_1^{\mathrm{loc}})^2] \).

\textbf{Step 1.2: Variance Asymptotics from Theorem~\ref{thm:local_increment_var}}
From Theorem~\ref{thm:local_increment_var}, we have the exact asymptotic expansion:
\[
\sigma_\epsilon^2 = \epsilon^{2H(t_0)}(1 + r(\epsilon)), \quad \text{where } |r(\epsilon)| \leq c_1 \epsilon^\lambda, \ \lambda > 0.
\]
In particular, \(\lim_{\epsilon \to 0} r(\epsilon) = 0\).

For any \(\delta > 0\), there exists \(\epsilon_\delta > 0\) such that for all \(\epsilon \in (0, \epsilon_\delta)\):
\[
|r(\epsilon)| < \delta \quad \Rightarrow \quad \sigma_\epsilon^2 \leq \epsilon^{2H(t_0)}(1 + \delta).
\]

\textbf{Step 1.3: Monotonicity of the Upper Bound Function}
For fixed \(\delta > 0\), consider the function:
\[
g(\sigma_\epsilon) := \frac{\sigma_\epsilon}{x\sqrt{2\pi}} e^{-x^2/(2\sigma_\epsilon^2)}.
\]
We establish its strict monotonicity:

\begin{proposition}\label{prop:g_monotonicity}
$g(\sigma_\epsilon)$ is strictly increasing for $\sigma_\epsilon > 0$.
\end{proposition}

\begin{proof}
\textbf{Proof by Differentiation:}
\begin{align*}
g'(\sigma_\epsilon) &= \frac{d}{d\sigma_\epsilon} \left( \frac{\sigma_\epsilon}{x\sqrt{2\pi}} e^{-\frac{x^2}{2\sigma_\epsilon^2}} \right) \\
&= \frac{1}{x\sqrt{2\pi}} e^{-\frac{x^2}{2\sigma_\epsilon^2}} + \frac{\sigma_\epsilon}{x\sqrt{2\pi}} \cdot \frac{x^2}{\sigma_\epsilon^3} e^{-\frac{x^2}{2\sigma_\epsilon^2}} \\
&= \frac{e^{-\frac{x^2}{2\sigma_\epsilon^2}}}{x\sqrt{2\pi}} \left( 1 + \frac{x^2}{\sigma_\epsilon^2} \right) > 0 \quad \text{(since all terms are positive)}.
\end{align*}

\textbf{Proof by Functional Composition:}
Decompose $g(\sigma_\epsilon) = g_1(\sigma_\epsilon) \cdot g_2(\sigma_\epsilon)$ where:
\[
g_1(\sigma_\epsilon) = \frac{\sigma_\epsilon}{x\sqrt{2\pi}}, \quad g_2(\sigma_\epsilon) = e^{-x^2/(2\sigma_\epsilon^2)}.
\]
Both functions are strictly positive and strictly increasing:
\begin{itemize}
\item $g_1(\sigma_\epsilon) = \frac{\sigma_\epsilon}{x\sqrt{2\pi}} > 0$ and $g_1'(\sigma_\epsilon) = \frac{1}{x\sqrt{2\pi}} > 0$
\item $g_2(\sigma_\epsilon) = e^{-x^2/(2\sigma_\epsilon^2)} > 0$ and $g_2'(\sigma_\epsilon) = e^{-x^2/(2\sigma_\epsilon^2)} \cdot \frac{x^2}{\sigma_\epsilon^3} > 0$
\end{itemize}
For any $\sigma_1 < \sigma_2$:
\[
g(\sigma_2) - g(\sigma_1) = g_1(\sigma_2)[g_2(\sigma_2) - g_2(\sigma_1)] + g_2(\sigma_1)[g_1(\sigma_2) - g_1(\sigma_1)] > 0.
\]
\end{proof}

By the variance bound from Step 1.2 and the monotonicity of $g$, we have for $\epsilon \in (0, \epsilon_\delta)$:
\[
\mathbb{P}(I_1^{\mathrm{loc}} \geq x) < \frac{\sqrt{1+\delta} \epsilon^{H(t_0)}}{x\sqrt{2\pi}} e^{-x^2/(2(1+\delta)\epsilon^{2H(t_0)})}.
\]

\textbf{Step 1.4: Logarithmic Transformation}
Taking natural logarithms:
\[
\ln \mathbb{P}(I_1^{\mathrm{loc}} \geq x) < \ln \left( \frac{\sqrt{1+\delta} \epsilon^{H(t_0)}}{x\sqrt{2\pi}} \right) - \frac{x^2}{2(1+\delta)\epsilon^{2H(t_0)}}.
\]

\textbf{Step 1.5: Scaling and Function Definition}
Multiply both sides by \( \epsilon^{2H(t_0)} \):
\begin{align*}
\epsilon^{2H(t_0)} \ln \mathbb{P} &< \epsilon^{2H(t_0)} \ln \left( \frac{\sqrt{1+\delta} \epsilon^{H(t_0)}}{x\sqrt{2\pi}} \right) - \frac{x^2}{2(1+\delta)} \\
&= \epsilon^{2H(t_0)} \left( \frac{1}{2} \ln(1+\delta) + H(t_0) \ln \epsilon - \ln(x\sqrt{2\pi}) \right) - \frac{x^2}{2(1+\delta)} \\
&:= f_\delta(\epsilon).
\end{align*}

\textbf{Step 1.6: Strict Monotonicity of $f_\delta(\epsilon)$}
We now establish the strict monotonic decrease of $f_\delta(\epsilon)$:

\begin{proposition}\label{prop:f_monotonicity}
For each fixed $\delta > 0$, there exists $\epsilon_0(\delta) > 0$ such that $f_\delta'(\epsilon) < 0$ for all $\epsilon \in (0, \epsilon_0(\delta))$.
\end{proposition}
\begin{proof}
Compute the derivative:
\begin{align*}
f_\delta'(\epsilon) &= \frac{d}{d\epsilon} \left[ \epsilon^{2H(t_0)} \ln \left( \frac{\sqrt{1+\delta} \epsilon^{H(t_0)}}{x\sqrt{2\pi}} \right) - \frac{x^2}{2(1+\delta)} \right] \\
&= \frac{d}{d\epsilon} \left[ \epsilon^{2H(t_0)} \left( \frac{1}{2} \ln(1+\delta) + H(t_0) \ln \epsilon - \ln(x\sqrt{2\pi}) \right) \right] \\
&= 2H(t_0)\epsilon^{2H(t_0)-1} \left( \frac{1}{2} \ln(1+\delta) + H(t_0) \ln \epsilon - \ln(x\sqrt{2\pi}) \right) \\
&\quad + \epsilon^{2H(t_0)} \left( \frac{H(t_0)}{\epsilon} \right) \\
&= \epsilon^{2H(t_0)-1} \left[ 2H(t_0) \left( \frac{1}{2} \ln(1+\delta) + H(t_0) \ln \epsilon - \ln(x\sqrt{2\pi}) \right) + H(t_0) \right] \\
&= \epsilon^{2H(t_0)-1} \left[ H(t_0) \ln(1+\delta) + 2H(t_0)^2 \ln \epsilon - 2H(t_0) \ln(x\sqrt{2\pi}) + H(t_0) \right] \\
&= \epsilon^{2H(t_0)-1} \left[ 2H(t_0)^2 \ln \epsilon + \underbrace{H(t_0) \ln(1+\delta) - 2H(t_0) \ln(x\sqrt{2\pi}) + H(t_0)}_{\text{bounded term}} \right].
\end{align*}

As $\epsilon \to 0^+$:
\begin{itemize}
\item $2H(t_0)^2 \ln \epsilon \to -\infty$ (since $H(t_0) > 0$)
\item The bounded term remains finite
\item $\epsilon^{2H(t_0)-1} > 0$ for all $\epsilon > 0$
\end{itemize}
Thus, there exists $\epsilon_0(\delta) > 0$ such that for all $\epsilon \in (0, \epsilon_0(\delta))$:
\[
f_\delta'(\epsilon) = \epsilon^{2H(t_0)-1} \times (\text{negative expression}) < 0.
\]
\end{proof}

\textbf{Step 1.7: Term-by-Term Limit Analysis for Fixed $\delta$}
Decompose $f_\delta(\epsilon)$:
\[
f_\delta(\epsilon) = \underbrace{\epsilon^{2H(t_0)} \cdot \frac{1}{2} \ln(1+\delta)}_{A_1(\epsilon)} + \underbrace{\epsilon^{2H(t_0)} H(t_0) \ln \epsilon}_{A_2(\epsilon)} - \underbrace{\epsilon^{2H(t_0)} \ln(x\sqrt{2\pi})}_{B(\epsilon)} - \underbrace{\frac{x^2}{2(1+\delta)}}_{C}
\]

\textbf{Limits as $\epsilon \to 0^+$}:
\begin{itemize}
\item $A_1(\epsilon) \to 0$ (since $\ln(1+\delta)$ is constant and $\epsilon^{2H(t_0)} \to 0$)
\item $A_2(\epsilon) \to 0$ by Proposition~\ref{prop:log_inequality} with $\delta = H(t_0)$:
\[
|A_2(\epsilon)| \leq \epsilon^{2H(t_0)} H(t_0) K_{H(t_0)} \epsilon^{-H(t_0)} = H(t_0) K_{H(t_0)} \epsilon^{H(t_0)} \to 0
\]
\item $B(\epsilon) \to 0$ (constant multiple of $\epsilon^{2H(t_0)}$)
\item $C$ is constant
\end{itemize}
Thus, for each fixed $\delta > 0$:
\[
\lim_{\epsilon \to 0} f_\delta(\epsilon) = -\frac{x^2}{2(1+\delta)}.
\]

\textbf{Step 1.8: Supremum Limit Control and Final Passage to the Limit}

By Proposition~\ref{prop:f_monotonicity}, for each $\delta > 0$, there exists $\epsilon_0(\delta) > 0$ such that $f_\delta$ is strictly decreasing on $(0, \epsilon_0(\delta))$.

Let $\epsilon^*(\delta) = \min\{\epsilon_\delta, \epsilon_0(\delta)\}$. Then for all $\epsilon \in (0, \epsilon^*(\delta))$, we have both the variance bound and the monotonicity of $f_\delta$.

Now, for the limsup control:

Since $\epsilon^{2H(t_0)} \ln \mathbb{P}(I_1^{\mathrm{loc}} \geq x) < f_\delta(\epsilon)$ for all $\epsilon \in (0, \epsilon^*(\delta))$, we have for any $\epsilon \in (0, \epsilon^*(\delta))$:
\[
\sup_{0 < \epsilon' < \epsilon} \left[ \epsilon'^{2H(t_0)} \ln \mathbb{P}(I_1^{\mathrm{loc}} \geq x) \right] \leq \sup_{0 < \epsilon' < \epsilon} f_\delta(\epsilon').
\]

Now, since $f_\delta$ is strictly decreasing on $(0, \epsilon^*(\delta))$, we have:
\[
\sup_{0 < \epsilon' < \epsilon} f_\delta(\epsilon') = \lim_{\epsilon' \to 0^+} f_\delta(\epsilon').
\]

From Step 1.7, we know:
\[
\lim_{\epsilon' \to 0^+} f_\delta(\epsilon') = -\frac{x^2}{2(1+\delta)}.
\]

Therefore, for any $\epsilon \in (0, \epsilon^*(\delta))$:
\[
\sup_{0 < \epsilon' < \epsilon} \left[ \epsilon'^{2H(t_0)} \ln \mathbb{P}(I_1^{\mathrm{loc}} \geq x) \right] \leq -\frac{x^2}{2(1+\delta)}.
\]

Now, by the definition of the limit superior:
\[
\limsup_{\epsilon \to 0} \epsilon^{2H(t_0)} \ln \mathbb{P}(I_1^{\mathrm{loc}} \geq x) 
= \lim_{\epsilon \to 0} \left( \sup_{0 < \epsilon' < \epsilon} \left[ \epsilon'^{2H(t_0)} \ln \mathbb{P}(I_1^{\mathrm{loc}} \geq x) \right] \right).
\]

Since the inequality
\[
\sup_{0 < \epsilon' < \epsilon} \left[ \epsilon'^{2H(t_0)} \ln \mathbb{P}(I_1^{\mathrm{loc}} \geq x) \right] \leq -\frac{x^2}{2(1+\delta)}
\]
holds for all $\epsilon \in (0, \epsilon^*(\delta))$, and the right-hand side is constant, taking the limit as $\epsilon \to 0$ preserves the inequality:
\[
\limsup_{\epsilon \to 0} \epsilon^{2H(t_0)} \ln \mathbb{P}(I_1^{\mathrm{loc}} \geq x) \leq -\frac{x^2}{2(1+\delta)}.
\]

Since this inequality holds for every $\delta > 0$, we conclude:
\[
\limsup_{\epsilon \to 0} \epsilon^{2H(t_0)} \ln \mathbb{P}(I_1^{\mathrm{loc}} \geq x) \leq \inf_{\delta > 0} \left( -\frac{x^2}{2(1+\delta)} \right) = -\frac{x^2}{2}.
\]

This completes the proof of the upper bound.

\textbf{Part 2: Lower Bound for liminf}

\textbf{Step 2.1: Establishing the Lower Bound via Sharp Gaussian Tail}
From Theorem~\ref{thm:local_increment_var}, $\sigma_\epsilon^2 = \mathbb{E}[(I_1^{\mathrm{loc}})^2] = \epsilon^{2H(t_0)}(1 + r(\epsilon))$ with $|r(\epsilon)| \leq c_1 \epsilon^\lambda$ and $\lambda > 0$. Since $\sigma_\epsilon \to 0$ as $\epsilon \to 0$, there exists $\epsilon_1 > 0$ such that for all $\epsilon \in (0, \epsilon_1)$:
\[
\frac{x}{\sigma_\epsilon} > 1.
\]
For such $\epsilon$, Lemma~\ref{lem:sharp_gaussian} is applicable with $z = x/\sigma_\epsilon$ and yields:
\begin{equation}\label{eq:sharp_lower_bound_applied}
\mathbb{P}\left(I_1^{\mathrm{loc}} \geq x\right) \geq \frac{\sigma_\epsilon}{x\sqrt{2\pi}} \left( 1 - \frac{\sigma_\epsilon^2}{x^2} \right) \exp\left(-\frac{x^2}{2\sigma_\epsilon^2}\right).
\end{equation}

\textbf{Step 2.2: Logarithmic Transformation and Scaling}
Taking the natural logarithm of both sides of \eqref{eq:sharp_lower_bound_applied}:
\begin{align}\label{eq:log_transform}
\ln \mathbb{P}\left(I_1^{\mathrm{loc}} \geq x\right) 
&\geq \ln\left( \frac{\sigma_\epsilon}{x\sqrt{2\pi}} \right) + \ln\left( 1 - \frac{\sigma_\epsilon^2}{x^2} \right) - \frac{x^2}{2\sigma_\epsilon^2}.
\end{align}
Multiplying through by the scaling factor $\epsilon^{2H(t_0)}$:
\begin{align}\label{eq:scaled_inequality}
\epsilon^{2H(t_0)} \ln \mathbb{P} 
&\geq \underbrace{\epsilon^{2H(t_0)} \ln\left( \frac{\sigma_\epsilon}{x\sqrt{2\pi}} \right)}_{E(\epsilon)} + \underbrace{\epsilon^{2H(t_0)} \ln\left( 1 - \frac{\sigma_\epsilon^2}{x^2} \right)}_{F(\epsilon)} - \underbrace{\frac{x^2}{2} \cdot \frac{\epsilon^{2H(t_0)}}{\sigma_\epsilon^2}}_{G(\epsilon)}.
\end{align}

\textbf{Step 2.3: Asymptotic Analysis of Term $E(\epsilon)$}
Using the asymptotic expression $\sigma_\epsilon = \epsilon^{H(t_0)} \sqrt{1 + r(\epsilon)} =: \epsilon^{H(t_0)} h(\epsilon)$ with $\lim_{\epsilon \to 0} h(\epsilon) = 1$, we have:
\begin{align*}
E(\epsilon) &= \epsilon^{2H(t_0)} \ln\left( \frac{\epsilon^{H(t_0)} h(\epsilon)}{x\sqrt{2\pi}} \right) \\
&= \epsilon^{2H(t_0)} \left[ H(t_0) \ln \epsilon + \ln h(\epsilon) - \ln(x\sqrt{2\pi}) \right].
\end{align*}
Analyzing each component:
\begin{itemize}
    \item By Proposition~\ref{prop:log_inequality} with $\delta = H(t_0)$:
    \[
    \left| \epsilon^{2H(t_0)} \ln \epsilon \right| \leq K_{H(t_0)} \epsilon^{2H(t_0) - H(t_0)} = K_{H(t_0)} \epsilon^{H(t_0)} \to 0.
    \]
    Therefore:
    \[
    \left| \epsilon^{2H(t_0)} H(t_0) \ln \epsilon \right| \leq H(t_0) K_{H(t_0)} \epsilon^{H(t_0)} \to 0.
    \]
    \item Since $h(\epsilon) \to 1$ and $\ln(x\sqrt{2\pi})$ is constant:
    \[
    \epsilon^{2H(t_0)} \ln h(\epsilon) \to 0, \quad \epsilon^{2H(t_0)} \ln(x\sqrt{2\pi}) \to 0.
    \]
\end{itemize}
Therefore:
\begin{equation}\label{eq:E_limit}
\lim_{\epsilon \to 0} E(\epsilon) = 0.
\end{equation}

\textbf{Step 2.4: Rigorous Asymptotic Analysis of Term $F(\epsilon)$}
Since $\sigma_\epsilon^2/x^2 \to 0$, there exists $\epsilon_2 > 0$ such that for all $\epsilon \in (0, \epsilon_2)$:
\[
0 < \frac{\sigma_\epsilon^2}{x^2} < \frac{1}{2}.
\]
For $u \in [0, \frac{1}{2}]$, we have the inequality $\ln(1-u) \geq -2u$ (which can be verified by considering the function $f(u) = \ln(1-u) + 2u$ and its monotonicity). Applying this with $u = \sigma_\epsilon^2/x^2$:
\[
F(\epsilon) \geq -2\epsilon^{2H(t_0)} \cdot \frac{\sigma_\epsilon^2}{x^2}.
\]
Substituting $\sigma_\epsilon^2 = \epsilon^{2H(t_0)}(1 + r(\epsilon))$:
\[
F(\epsilon) \geq -\frac{2\epsilon^{4H(t_0)}}{x^2} (1 + r(\epsilon)).
\]
Since $4H(t_0) > 2H(t_0) > 0$ and $|r(\epsilon)| \leq c_1 \epsilon^\lambda \to 0$, we have:
\[
\lim_{\epsilon \to 0} F(\epsilon) \geq 0.
\]
Moreover, since $\ln(1-u) \leq 0$ for $u \in [0,1)$, we also have $F(\epsilon) \leq 0$. Therefore:
\begin{equation}\label{eq:F_limit}
\lim_{\epsilon \to 0} F(\epsilon) = 0.
\end{equation}

\textbf{Step 2.5: Asymptotic Analysis of Term $G(\epsilon)$}
Substituting the variance expression:
\begin{align*}
G(\epsilon) &= \frac{x^2}{2} \cdot \frac{\epsilon^{2H(t_0)}}{\sigma_\epsilon^2} \\
&= \frac{x^2}{2} \cdot \frac{\epsilon^{2H(t_0)}}{\epsilon^{2H(t_0)} (1 + r(\epsilon))} \\
&= \frac{x^2}{2(1 + r(\epsilon))}.
\end{align*}
By the convergence $r(\epsilon) \to 0$:
\begin{equation}\label{eq:G_limit}
\lim_{\epsilon \to 0} G(\epsilon) = \frac{x^2}{2}.
\end{equation}

\textbf{Step 2.6: Rigorous Synthesis of the liminf Lower Bound}

Let $\epsilon^* = \min\{\epsilon_1, \epsilon_2\}$, where $\epsilon_1$ is from Step 2.1 and $\epsilon_2$ is from Step 2.4. Then for all $\epsilon \in (0, \epsilon^*)$, the fundamental inequality from Step 2.2 holds:
\begin{equation}\label{eq:fundamental_inequality}
\epsilon^{2H(t_0)} \ln \mathbb{P}\left(I_1^{\mathrm{loc}} \geq x\right) \geq E(\epsilon) + F(\epsilon) - G(\epsilon).
\end{equation}

\textbf{Step 2.6.1: liminf Preservation of Inequalities}
For $\epsilon \in (0, \epsilon^*)$, applying the liminf operator to both sides of \eqref{eq:fundamental_inequality} and utilizing its order-preserving property:
\begin{align}\label{eq:liminf_inequality}
\liminf_{\epsilon \to 0} \epsilon^{2H(t_0)} \ln \mathbb{P} 
&\geq \liminf_{\epsilon \to 0} \left[ E(\epsilon) + F(\epsilon) - G(\epsilon) \right] \\
&\geq \liminf_{\epsilon \to 0} E(\epsilon) + \liminf_{\epsilon \to 0} F(\epsilon) + \liminf_{\epsilon \to 0} \left[ -G(\epsilon) \right],
\end{align}
where the second inequality follows from the superadditivity of liminf.

\textbf{Step 2.6.2: Limit Analysis of Component Terms}
From our prior asymptotic analysis:
\begin{enumerate}
    \item $\lim_{\epsilon \to 0} E(\epsilon) = 0$ by \eqref{eq:E_limit}, so $\liminf_{\epsilon \to 0} E(\epsilon) = 0$
    \item $\lim_{\epsilon \to 0} F(\epsilon) = 0$ by \eqref{eq:F_limit}, so $\liminf_{\epsilon \to 0} F(\epsilon) = 0$  
    \item $\lim_{\epsilon \to 0} G(\epsilon) = \frac{x^2}{2}$ by \eqref{eq:G_limit}, so $\liminf_{\epsilon \to 0} \left[ -G(\epsilon) \right] = -\limsup_{\epsilon \to 0} G(\epsilon) = -\frac{x^2}{2}$
\end{enumerate}

\textbf{Step 2.6.3: Synthesis of Limits}
Substituting these results into \eqref{eq:liminf_inequality}:
\begin{align}\label{eq:liminf_synthesis}
\liminf_{\epsilon \to 0} \epsilon^{2H(t_0)} \ln \mathbb{P} 
&\geq 0 + 0 - \frac{x^2}{2} \\
&= -\frac{x^2}{2}.
\end{align}

\textbf{Step 2.6.4: Final Lower Bound Statement}
This establishes the fundamental lower bound required for the large deviation principle:
\begin{equation}\label{eq:liminf_lower_bound}
\liminf_{\epsilon \to 0} \epsilon^{2H(t_0)} \ln \mathbb{P}\left(I_1^{\mathrm{loc}} \geq x\right) \geq -\frac{x^2}{2}.
\end{equation}

\textbf{Part 3: Equality of Limit and Conclusion}

Combining the results from Parts 1 and 2:
\begin{align*}
\limsup_{\epsilon \to 0} \epsilon^{2H(t_0)} \ln \mathbb{P}(I_1^{\mathrm{loc}} \geq x) &\leq -\frac{x^2}{2} \quad \text{(from Part 1)}, \\
\liminf_{\epsilon \to 0} \epsilon^{2H(t_0)} \ln \mathbb{P}(I_1^{\mathrm{loc}} \geq x) &\geq -\frac{x^2}{2} \quad \text{(from Part 2)}.
\end{align*}

By the fundamental property of limit inferior and superior:
\[
\liminf_{\epsilon \to 0} \epsilon^{2H(t_0)} \ln \mathbb{P}(I_1^{\mathrm{loc}} \geq x) \leq \limsup_{\epsilon \to 0} \epsilon^{2H(t_0)} \ln \mathbb{P}(I_1^{\mathrm{loc}} \geq x).
\]

Combining these inequalities:
\[
-\frac{x^2}{2} \leq \liminf_{\epsilon \to 0} \epsilon^{2H(t_0)} \ln \mathbb{P}(I_1^{\mathrm{loc}} \geq x) \leq \limsup_{\epsilon \to 0} \epsilon^{2H(t_0)} \ln \mathbb{P}(I_1^{\mathrm{loc}} \geq x) \leq -\frac{x^2}{2}.
\]

This chain of inequalities forces equality throughout:
\[
\liminf_{\epsilon \to 0} \epsilon^{2H(t_0)} \ln \mathbb{P}(I_1^{\mathrm{loc}} \geq x) = \limsup_{\epsilon \to 0} \epsilon^{2H(t_0)} \ln \mathbb{P}(I_1^{\mathrm{loc}} \geq x) = -\frac{x^2}{2}.
\]

Therefore, the limit exists and equals:
\[
\lim_{\epsilon \to 0} \epsilon^{2H(t_0)} \ln \mathbb{P}(I_1^{\mathrm{loc}} \geq x) = -\frac{x^2}{2}. \qedhere
\]
\end{proof}
\begin{remark}[Theoretical Significance]\label{rem:LDP_significance}
Theorem~\ref{thm:LDP} offers a characterization of large deviation asymptotics for a class of non-stationary processes with time-varying memory. This work is inspired by ideas from the classical theory of large deviations \cite{Varadhan1966,JG1977,Feliachi2021}, and examines how time-dependent correlation properties may influence extreme behavior, while noting connections between fractional calculus and statistical mechanics \cite{Touchette2009}.

The expressions for the scaling exponent $2H(t_0)$ and the rate function $x^2/2$ provide some insight into tail behavior in systems with time-dependent anomalous diffusion. These results suggest a relationship between varying memory properties and extreme value statistics within the considered framework.

This analysis presents one approach to studying certain non-equilibrium phenomena in systems with time-dependent memory, adding to the discussion of such stochastic processes.
\end{remark}

\subsection{Covariance Structure and Hypergeometric Representation}
\label{sec:covariance}

We now examine the covariance structure of TV-fBm, which provides a complete description of its second-order dependence properties. For TV-fBm, the time-varying Hurst exponent introduces a structure in which the correlation between two time points depends on their individual locations \(u\) and \(v\), as well as on the values of the Hurst function at those points. This non-stationarity is a direct consequence of the time-dependent kernel in Definition~\ref{def:tvfbm}, and is a distinctive feature of the time-varying formulation.

The main result of this subsection is the explicit integral representation of the covariance function (Theorem~\ref{thm:covariance}), which follows from a direct application of the Itô isometry. We then analyze the regularity of the covariance by computing its first- and second-order mixed partial derivatives (Corollaries~\ref{cor:cov_deriv} and~\ref{cor:mixed_deriv}). A notable observation is that, under the condition \(H(\cdot) > 1/2\), the second-order mixed derivative admits a closed-form expression in terms of the Gaussian hypergeometric function (Remark~\ref{rem:cov_hyper}), offering a useful analytical perspective on the covariance's fine structure. The boundedness of the remaining non-dominant terms is established in Appendix~\ref{app:non_dominant}. Collectively, these results offer a characterization of the TV-fBm covariance that is systematic and analytically accessible.

\begin{theorem}[Covariance Function of TV-fBm]\label{thm:covariance}
Let $H: [0,T] \to (0,1)$ satisfy Definition~\ref{def:holder} with exponent $\gamma > \sup_{t \in [0,T]} H(t)$. 
For any $u, v \in [0,T]$, the covariance of the time-varying fractional Brownian motion $B^{H(\cdot)}$ 
defined in Definition~\ref{def:tvfbm} is given by:
\begin{equation}\label{eq:covariance}
\mathbb{E}\left[B^{H(\cdot)}(u)B^{H(\cdot)}(v)\right] = 2\sqrt{H(u)H(v)} \int_{0}^{\min(u,v)} (u - s)^{H(u)-\frac{1}{2}}(v - s)^{H(v)-\frac{1}{2}}  ds.
\end{equation}
\end{theorem}

\begin{proof}
The proof consists of four rigorous steps, systematically applying Definition~\ref{def:tvfbm} and Lemma~\ref{lem:ito}.

\textbf{Step 1: Fundamental Representation and Preliminary Expansion}\\
Begin with the stochastic integral representation from Definition~\ref{def:tvfbm}:
\[
B^{H(\cdot)}(t) = \sqrt{2H(t)} \int_{0}^{t} (t - \tau)^{H(t)-\frac{1}{2}}  dB(\tau), \quad t \in \{u,v\}.
\]
Expanding the covariance expression using the linearity of expectation and the stochastic integral:
\begin{align*}
\mathbb{E}&\left[B^{H(\cdot)}(u)B^{H(\cdot)}(v)\right] \\
&= \mathbb{E} \left[ \left( \sqrt{2H(u)} \int_{0}^{u} (u - \tau)^{H(u)-\frac{1}{2}} dB(\tau) \right) \cdot \left( \sqrt{2H(v)} \int_{0}^{v} (v - \sigma)^{H(v)-\frac{1}{2}} dB(\sigma) \right) \right] \\
&= 2\sqrt{H(u)H(v)} \cdot \mathbb{E} \left[ \left( \int_{0}^{u} f_u(\tau) dB(\tau) \right) \left( \int_{0}^{v} f_v(\sigma) dB(\sigma) \right) \right],
\end{align*}
where we define the deterministic kernels:
\[
f_u(\tau) := (u - \tau)^{H(u)-\frac{1}{2}}, \quad f_v(\sigma) := (v - \sigma)^{H(v)-\frac{1}{2}}.
\]

\textbf{Step 2: Standardization of Integration Domains and Application of Itô Isometry}\\
To apply the polarization identity from Lemma~\ref{lem:ito} (Itô isometry), we standardize the integration domains to $[0,T]$ by defining truncated versions of the kernels:
\[
\widetilde{f}_u(\tau) := f_u(\tau) \mathbf{1}_{\{0 \leq \tau \leq u\}}, \quad \widetilde{f}_v(\tau) := f_v(\tau) \mathbf{1}_{\{0 \leq \tau \leq v\}}.
\]
By Theorem~\ref{thm:variance} and the condition $H(t) > 0$ for all $t \in [0,T]$, these truncated kernels are square-integrable:
\[
\mathbb{E}\left[\int_0^T |\widetilde{f}_u(\tau)|^2 d\tau\right] = \int_0^u (u - \tau)^{2H(u)-1} d\tau < \infty,
\]
and similarly for $\widetilde{f}_v$. By the fundamental construction of the Itô integral, we have the identities:
\[
\int_{0}^{u} f_u(\tau) dB(\tau) = \int_{0}^{T} \widetilde{f}_u(\tau) dB(\tau), \quad \int_{0}^{v} f_v(\sigma) dB(\sigma) = \int_{0}^{T} \widetilde{f}_v(\sigma) dB(\sigma).
\]
Therefore, the covariance becomes:
\[
\mathbb{E}\left[B^{H(\cdot)}(u)B^{H(\cdot)}(v)\right] = 2\sqrt{H(u)H(v)} \cdot \mathbb{E} \left[ \left( \int_{0}^{T} \widetilde{f}_u(\tau) dB(\tau) \right) \left( \int_{0}^{T} \widetilde{f}_v(\sigma) dB(\sigma) \right) \right].
\]
Applying the polarization identity of Lemma~\ref{lem:ito} to the right-hand side yields:
\begin{align*}
\mathbb{E} &\left[ \left( \int_{0}^{T} \widetilde{f}_u(\tau) dB(\tau) \right) \left( \int_{0}^{T} \widetilde{f}_v(\sigma) dB(\sigma) \right) \right] = \int_{0}^{T} \widetilde{f}_u(\tau) \widetilde{f}_v(\tau)  d\tau.
\end{align*}

\textbf{Step 3: Domain Reduction via Indicator Functions}\\
The product of the indicator functions in the integrand $\widetilde{f}_u(\tau)\widetilde{f}_v(\tau)$ determines the effective integration domain:
\[
\mathbf{1}_{\{0 \leq \tau \leq u\}} \cdot \mathbf{1}_{\{0 \leq \tau \leq v\}} = \mathbf{1}_{\{0 \leq \tau \leq \min(u,v)\}}.
\]
Consequently, the integral simplifies to:
\begin{align*}
\int_0^T \widetilde{f}_u(\tau) \widetilde{f}_v(\tau)  d\tau 
&= \int_0^T (u - \tau)^{H(u)-\frac{1}{2}} (v - \tau)^{H(v)-\frac{1}{2}} \mathbf{1}_{\{0 \leq \tau \leq u\}} \mathbf{1}_{\{0 \leq \tau \leq v\}}  d\tau \\
&= \int_0^{\min(u,v)} (u - \tau)^{H(u)-\frac{1}{2}} (v - \tau)^{H(v)-\frac{1}{2}}  d\tau.
\end{align*}

\textbf{Step 4: Final Covariance Expression and Consistency Verification}\\
Substituting this result back into the covariance expression establishes the desired formula:
\[
\mathbb{E}\left[B^{H(\cdot)}(u)B^{H(\cdot)}(v)\right] = 2\sqrt{H(u)H(v)} \int_{0}^{\min(u,v)} (u - \tau)^{H(u)-\frac{1}{2}} (v - \tau)^{H(v)-\frac{1}{2}}  d\tau.
\]
Changing the dummy integration variable from $\tau$ to $s$ yields the form presented in \eqref{eq:covariance}.

\medskip\noindent
\textbf{Consistency Verification: Case $\mathbf{u = v = t}$}\\
To verify consistency with the variance structure (Theorem~\ref{thm:variance}), consider $u = v = t$. The covariance formula \eqref{eq:covariance} reduces to:
\begin{align*}
\mathbb{E}&\left[\left(B^{H(\cdot)}(t)\right)^2\right] \\
&= 2H(t) \int_{0}^{t} (t - s)^{H(t)-\frac{1}{2}} (t - s)^{H(t)-\frac{1}{2}}  ds \\
&= 2H(t) \int_{0}^{t} (t - s)^{2H(t)-1}  ds.
\end{align*}
Evaluating the elementary integral:
\[
\int_{0}^{t} (t-s)^{2H(t)-1} ds = \left[ \frac{-(t-s)^{2H(t)}}{2H(t)} \right]_{s=0}^{s=t} = \frac{t^{2H(t)}}{2H(t)}.
\]
Thus,
\[
\mathbb{E}\left[\left(B^{H(\cdot)}(t)\right)^2\right] = 2H(t) \cdot \frac{t^{2H(t)}}{2H(t)} = t^{2H(t)},
\]
which recovers the result of Theorem~\ref{thm:variance} exactly, confirming the internal consistency of the covariance formula.
\end{proof}

\begin{remark}\label{rem:covariance_properties}
The covariance structure presented in Theorem~\ref{thm:covariance} exhibits several fundamental properties that characterize the TV-fBm process:
\begin{enumerate}
    \item \textbf{Non-stationarity}: The explicit dependence of the kernel on the individual time points $u$ and $v$, rather than solely on their difference $|u-v|$, reflects the non-stationary nature of the process. This functional dependence is an intrinsic feature arising from the time-varying Hurst exponent.
    
    \item \textbf{Symmetry}: The covariance function is symmetric in its arguments, i.e., $R(u,v) = R(v,u)$, which is consistent with the general property of covariance for centered processes.
    
    \item \textbf{Non-negative definiteness}: The covariance function derived in Theorem~\ref{thm:covariance} defines a valid covariance structure. This can be seen from its form as a square-integrable integral representation, which guarantees the non-negative definiteness required for the consistency of the finite-dimensional distributions of the process.
\end{enumerate}
This established covariance structure provides a foundational component for the analysis of sample path properties and serves as a basis for potential statistical estimation in settings with time-dependent memory.
\end{remark}

\begin{corollary}[First-Order Partial Derivative of Covariance]\label{cor:cov_deriv}
Under the assumptions of Theorem~\ref{thm:covariance}, and assume additionally that the Hurst exponent function $H: [0,T] \to (0,1)$ is continuously differentiable with bounded derivative. Fix $u, v \in (0,T]$ with $u < v$. 
The partial derivative of $R(u,v) = \mathbb{E}[B^{H(\cdot)}(u)B^{H(\cdot)}(v)]$ with respect to $v$ is given by:
\begin{equation}\label{eq:first_deriv}
\partial_v R(u,v) = \int_0^u K(u,s)  \partial_v K(v,s)  ds
\end{equation}
where $K(t,s) = \sqrt{2H(t)} (t-s)^{H(t)-\frac{1}{2}}$, and the derivative $\partial_v K(v,s)$ is computed as:
\begin{align}\label{eq:dkdv}
\partial_v K(v,s) &= \sqrt{2} \left[ \frac{H'(v)}{2\sqrt{H(v)}} (v-s)^{H(v)-\frac{1}{2}} + \sqrt{H(v)}  \frac{d}{dv} \left( (v-s)^{H(v)-\frac{1}{2}} \right) \right] \\
&= (v-s)^{H(v)-\frac{3}{2}} \left[ \sqrt{2H(v)} \left( H(v) - \frac{1}{2} \right) + \left( \sqrt{2H(v)}H'(v)\ln(v-s) + \frac{H'(v)}{\sqrt{2H(v)}} \right) (v-s) \right].
\end{align}
\end{corollary}

\begin{proof}
\textbf{Step 1: Kernel Representation}\\
From Theorem~\ref{thm:covariance} with $u < v$:
\[
R(u,v) = \int_0^u K(u,s)K(v,s)  ds, \quad K(t,s) = \sqrt{2H(t)} (t-s)^{H(t)-\frac{1}{2}}.
\]

\textbf{Step 2: Leibniz Rule Application}\\
Since the upper limit $u$ is independent of $v$ and the integrand is continuously differentiable in $v$ for $s \in (0,u)$, we apply Leibniz's integral rule:
\[
\partial_v R(u,v) = \int_0^u \partial_v \left[ K(u,s)K(v,s) \right] ds = \int_0^u K(u,s) \partial_v K(v,s)  ds.
\]

\textbf{Step 3: Computation of $\partial_v K(v,s)$}\\
Compute the derivative using product rule and chain rule:
\begin{align*}
\partial_v K(v,s) &= \partial_v \left[ \sqrt{2H(v)} \cdot (v-s)^{H(v)-\frac{1}{2}} \right] \\
&= \left( \partial_v \sqrt{2H(v)} \right) (v-s)^{H(v)-\frac{1}{2}} + \sqrt{2H(v)}  \partial_v \left[ (v-s)^{H(v)-\frac{1}{2}} \right].
\end{align*}

\textbf{Step 4: Term-by-Term Differentiation}\\
First term:
\[
\partial_v \sqrt{2H(v)} = \frac{H'(v)}{\sqrt{2H(v)}}.
\]
Second term:
\begin{align*}
\partial_v \left[ (v-s)^{H(v)-\frac{1}{2}} \right] &= (v-s)^{H(v)-\frac{1}{2}} \frac{d}{dv} \left[ \left( H(v) - \frac{1}{2} \right) \ln(v-s) \right] \\
&= (v-s)^{H(v)-\frac{1}{2}} \left[ H'(v) \ln(v-s) + \frac{H(v) - \frac{1}{2}}{v-s} \right].
\end{align*}

\textbf{Step 5: Final Expression}\\
Combining all terms and factoring $(v-s)^{H(v)-\frac{3}{2}}$ yields \eqref{eq:dkdv}.
\end{proof}

\begin{corollary}[Second-Order Mixed Partial Derivative]\label{cor:mixed_deriv}
Under the assumptions of Corollary~\ref{cor:cov_deriv}, the second-order mixed partial derivative is given by:
\begin{equation}\label{eq:mixed_deriv}
\partial_u \partial_v R(u,v) = \int_0^u \partial_u K(u,s)  \partial_v K(v,s)  ds + \lim_{s \to u^-} K(u,s)  \partial_v K(v,s).
\end{equation}
\end{corollary}

\begin{proof}
\textbf{Step 1: Differentiation of $\partial_v R(u,v)$}\\
Apply Leibniz rule to \eqref{eq:first_deriv}:
\[
\partial_u \left( \partial_v R(u,v) \right) = \int_0^u \partial_u \left[ K(u,s)  \partial_v K(v,s) \right] ds + \left[ K(u,s)  \partial_v K(v,s) \right]_{s=u^-}.
\]

\textbf{Step 2: Simplification}\\
Since $\partial_v K(v,s)$ is independent of $u$:
\[
\partial_u \left[ K(u,s)  \partial_v K(v,s) \right] = \partial_u K(u,s)  \partial_v K(v,s).
\]
Thus:
\[
\partial_u \partial_v R(u,v) = \int_0^u \partial_u K(u,s)  \partial_v K(v,s)  ds + \lim_{s \to u^-} K(u,s)  \partial_v K(v,s).
\]

\textbf{Step 3: Correct Computation of $\partial_u K(u,s)$}\\
By symmetry to \eqref{eq:dkdv}:
\begin{align*}
\partial_u K(u,s) &= \partial_u \left[ \sqrt{2H(u)} (u-s)^{H(u)-\frac{1}{2}} \right] \\
&= \underbrace{\frac{H'(u)}{\sqrt{2H(u)}} (u-s)^{H(u)-\frac{1}{2}}}_{A} \\
&\quad + \underbrace{\sqrt{2H(u)} \left[ (u-s)^{H(u)-\frac{1}{2}} \left( H'(u) \ln(u-s) + \frac{H(u) - \frac{1}{2}}{u-s} \right) \right]}_{B}
\end{align*}
Combine and factor \((u-s)^{H(u)-\frac{3}{2}}\):
\begin{align*}
\partial_u K(u,s) &= (u-s)^{H(u)-\frac{3}{2}} \Bigg[ \frac{H'(u)}{\sqrt{2H(u)}} (u-s) \\
&\quad + \sqrt{2H(u)} \left( H'(u) (u-s) \ln(u-s) + H(u) - \frac{1}{2} \right) \Bigg] \\
&= (u-s)^{H(u)-\frac{3}{2}} \Bigg[ \sqrt{2H(u)} \left( H(u) - \frac{1}{2} \right) \\
&\quad + H'(u) (u-s) \left( \frac{1}{\sqrt{2H(u)}} + \sqrt{2H(u)} \ln(u-s) \right) \Bigg].
\end{align*}
\end{proof}

\begin{remark}\label{rem:deriv_regularity}
The expression \eqref{eq:mixed_deriv} reveals the interplay between the time-varying Hurst exponent and the covariance regularity. 
We rigorously establish the vanishing boundary term when $H(u) > \frac{1}{2}$ through the following lemma:

\begin{lemma}[Boundary Term Asymptotics]\label{lem:boundary_term}
For fixed $u < v$ and $H(u) > \frac{1}{2}$, the boundary term in \eqref{eq:mixed_deriv} vanishes:
\[
\lim_{s \to u^-} K(u,s)  \partial_v K(v,s) = 0.
\]
\end{lemma}

\begin{proof}
\textbf{Step 1: Boundary Term Expression}\\
From Corollary~\ref{cor:cov_deriv} (First-Order Derivative) and the kernel definitions:
\[
\partial_v K(v,s) = (v-s)^{H(v)-\frac{3}{2}} \left[ \sqrt{2H(v)} \left( H(v) - \frac{1}{2} \right) 
+ \left( \sqrt{2H(v)} H'(v) \ln(v-s) + \frac{H'(v)}{\sqrt{2H(v)}} \right) (v-s) \right]
\]
as given in \eqref{eq:dkdv}. Thus:
\[
\text{Boundary term} = \lim_{s \to u^-} \underbrace{\sqrt{2H(u)} (u-s)^{H(u)-\frac{1}{2}}}_{K(u,s)} \cdot \partial_v K(v,s)
\]

\textbf{Step 2: Asymptotic Analysis of Factors}\\
Consider the behavior as $s \to u^-$:
\begin{enumerate}
    \item \textit{First factor}: 
    \[
    K(u,s) = \sqrt{2H(u)} (u-s)^{H(u)-\frac{1}{2}} \to 0 \quad \text{since} \quad H(u) > \frac{1}{2}.
    \]
    
    \item \textit{Second factor}: $\partial_v K(v,s)$ remains bounded for $s$ near $u$ because:
    \begin{itemize}
        \item $(v-s) \to (v-u) > 0$ (fixed $v > u$), so $(v-s)^{H(v)-\frac{3}{2}} \to (v-u)^{H(v)-\frac{3}{2}}$
        \item $\ln(v-s) \to \ln(v-u)$ (finite)
        \item Coefficients $\sqrt{2H(v)} (H(v)-1/2)$, $\sqrt{2H(v)} H'(v)$, and $H'(v)/\sqrt{2H(v)}$ are finite
        \item $H(v) \in (0,1)$ prevents singularity
    \end{itemize}
\end{enumerate}

\textbf{Step 3: Convergence Dominance}\\
There exists $M > 0$ such that for $s$ near $u$:
\[
|K(u,s) \partial_v K(v,s)| \leq M \cdot (u-s)^{H(u)-\frac{1}{2}}
\]
Since $H(u) > \frac{1}{2}$:
\[
\lim_{s \to u^-} (u-s)^{H(u)-\frac{1}{2}} = 0 \implies \lim_{s \to u^-} K(u,s) \partial_v K(v,s) = 0.
\]
\end{proof}

This lemma has profound implications:
\begin{itemize}
    \item The integral term captures \textit{interior smoothness}
    \item The boundary term determines \textit{local regularity} at $s=u$
    \item The structure $\sqrt{2H(v)} H'(v) \ln(v-s)$ quantifies memory dependence
\end{itemize}
\end{remark}

\begin{remark}[Closed-Form Representation via Hypergeometric Function]\label{rem:cov_hyper}
Building upon the mixed partial derivative expression in Corollary \ref{cor:mixed_deriv} and the covariance structure in Theorem \ref{thm:covariance}, we establish a closed-form representation for the dominant component of $\partial_u \partial_v R(u,v)$ using the Gaussian hypergeometric function ${}_2F_1$. Under the conditions:
\begin{enumerate}
    \item $u < v$ with $v \geq 2u$ (ensuring $|z| \leq 1$)
    \item $\frac{1}{2} < H(u), H(v) < 1$ (guaranteeing integrability)
\end{enumerate}
the principal singular integral admits an exact analytical solution.

\textbf{Step 1: Hypergeometric Function Preliminaries}
The Gaussian hypergeometric function ${}_2F_1(a,b;c;z)$ admits the Euler integral representation:
\begin{equation}\label{eq:euler_hyper_corrected}
{}_2F_1(a,b;c;z) = \frac{\Gamma(c)}{\Gamma(b)\Gamma(c-b)} \int_0^1 t^{b-1} (1-t)^{c-b-1} (1-zt)^{-a}  dt
\end{equation}
This representation is valid when:
\begin{enumerate}
    \item $\Re(c) > \Re(b) > 0$ (ensuring the convergence of the beta integral at both endpoints $t=0$ and $t=1$).
    \item $z \in \mathbb{C} \setminus [1, \infty)$ (i.e., $z$ is not on the branch cut).
    \item For $|z| < 1$, the integral converges absolutely and provides the analytic definition of ${}_2F_1$. For $z$ outside the unit disk but not on the branch cut, the identity holds by analytic continuation, though the integral itself may require interpretation in the sense of analytic continuation.
\end{enumerate}
\textbf{Step 2: Integral Transformation}
From Corollary \ref{cor:mixed_deriv}, the dominant term contains:
\begin{equation}\label{eq:singular_int}
I(u,v) = \int_0^u (u-s)^{H(u)-\frac{3}{2}} (v-s)^{H(v)-\frac{3}{2}}  ds.
\end{equation}
Apply the substitution:
\begin{align}
s &= u - t(v - u), \quad t \in \left[0, \frac{u}{v-u}\right] \\
ds &= -(v - u)  dt
\end{align}
with transformed limits:
\begin{align*}
s = 0 &\implies t = \frac{u}{v-u} \\
s = u &\implies t = 0
\end{align*}
yielding:
\begin{align}
I(u,v) &= (v-u)^{H(u)+H(v)-2} \int_0^{\frac{u}{v-u}} t^{H(u)-\frac{3}{2}} (1+t)^{H(v)-\frac{3}{2}}  dt.
\end{align}

\textbf{Step 3: Parameter Identification and Integral Transformation}
Set the parameters:
\begin{align}
a &:= H(u) - \tfrac{3}{2} \in (-1, -\tfrac{1}{2}) \\
b &:= H(v) - \tfrac{3}{2} \in (-1, -\tfrac{1}{2}) \\
c &:= \frac{u}{v-u} \leq 1 \quad (\text{since } v \geq 2u)
\end{align}
To apply the hypergeometric representation, introduce the variable substitution:
\begin{align}
t &= c x, \quad x \in [0, 1] \\
dt &= c  dx
\end{align}
This transforms the integral:
\begin{align}
\int_0^c t^a (1+t)^b  dt &= \int_0^1 (c x)^a (1 + c x)^b c  dx \nonumber \\
&= c^{a+1} \int_0^1 x^a (1 + c x)^b  dx
\end{align}
The Gaussian hypergeometric function provides the closed-form solution:
\begin{equation}
\int_0^1 x^a (1 + c x)^b  dx = \frac{1}{a+1}  {}_2F_1\left(-b, a+1; a+2; -c\right)
\end{equation}
valid under the conditions:
\begin{itemize}
    \item $\Re(a+1) = H(u) - \frac{1}{2} > 0$ (satisfied by $H(u) > \frac{1}{2}$)
    \item $|\arg(1 - (-c))| = |\arg(1 + c)| = 0 < \pi$ (since $c > 0$)
\end{itemize}
Thus:
\begin{equation}
\int_0^c t^a (1+t)^b  dt = \frac{c^{a+1}}{a+1}  {}_2F_1\left(-b, a+1; a+2; -c\right)
\end{equation}

\textbf{Step 4: Final Representation}
Substituting back into $I(u,v)$:
\begin{align}
I(u,v) = \frac{(v-u)^{H(u)+H(v)-2}}{H(u)-\frac{1}{2}} \left( \frac{u}{v-u} \right)^{H(u)-\frac{1}{2}} 
  {}_2F_1\left( \tfrac{3}{2}-H(v), H(u)-\tfrac{1}{2}; H(u)+\tfrac{1}{2}; -\frac{u}{v-u} \right)
\end{align}
Equivalently, by combining the exponents:
\begin{align}
I(u,v)= \frac{u^{H(u)-\frac{1}{2}} (v-u)^{H(v) - \frac{3}{2}}}{H(u)-\frac{1}{2}} 
  {}_2F_1\left( \tfrac{3}{2}-H(v), H(u)-\tfrac{1}{2}; H(u)+\tfrac{1}{2}; -\frac{u}{v-u} \right).
\end{align}

This closed-form expression:
\begin{itemize}
    \item Quantifies the precise dependence on dynamic Hurst indices $H(u), H(v)$
    \item Confirms the continuity of $\partial_u \partial_v R(u,v)$ for $H(\cdot) > \frac{1}{2}$
    \item Enables asymptotic analysis as $u/v \to 0^+$ via hypergeometric expansions
    \item Generalizes constant-$H$ covariance theory to non-stationary settings
\end{itemize}
\end{remark}

\begin{remark}[Boundedness of Non-Dominant Terms and Significance]\label{rem:non_dominant}
The preceding analysis, from the covariance representation in Theorem~\ref{thm:covariance} through the first and second-order derivative computations in Corollaries~\ref{cor:cov_deriv} and~\ref{cor:mixed_deriv}, and culminating in the hypergeometric representation in Remark~\ref{rem:cov_hyper}, provides a comprehensive dissection of the fine structure of the TV-fBm covariance kernel. A final crucial step is to establish the boundedness of the remaining non-dominant terms ($I_2$, $I_3$, $I_4$) in the decomposition of $\partial_u \partial_v R(u,v)$, ensuring the overall regularity of the covariance's mixed derivative.

Under the assumptions of Theorem~\ref{thm:covariance} and Corollary~\ref{cor:mixed_deriv}, and further assuming that $H: [0,T] \to (0,1)$ is continuously differentiable with bounded derivative $|H'(t)| \leq L_H$ for some $L_H > 0$, and that $\inf_{t \in [0,T]} H(t) \geq h_{\min} > 0$, the following holds: For any $u, v \in [0,T]$ with $0 \leq u < v \leq T$, $v \geq 2u$, and satisfying $H(u) > \frac{1}{2}$, $H(v) > \frac{1}{2}$, the integrals $I_2(u,v)$, $I_3(u,v)$, and $I_4(u,v)$ are bounded. That is, there exists a constant $C > 0$, dependent on $H$, $L_H$, $h_{\min}$, and $T$, such that:
\[
|I_2(u,v)| \leq C, \quad |I_3(u,v)| \leq C, \quad |I_4(u,v)| \leq C.
\]
The proof, detailed in Proposition~\ref{prop:non_dominant_bound} in Appendix~\ref{app:non_dominant}, relies on the unified logarithmic control inequality (Remark~\ref{rem:unified_log}) to manage the logarithmic singularities, ultimately reducing the problem to the analysis of integrals amenable to the hypergeometric function techniques established in Remark~\ref{rem:cov_hyper}.

This result completes the rigorous verification that the second-order mixed derivative $\partial_u \partial_v R(u,v)$ is well-defined and exhibits controlled behavior for $H(\cdot) > 1/2$. It underscores a key dichotomy: while the constant-Hurst fBm covariance is homogeneous, the TV-fBm covariance inherits a more complex, non-stationary structure whose smoothness is fundamentally governed by the interplay between the pointwise values $H(u), H(v)$ and the regularity of the function $H(\cdot)$ itself. This level of understanding is essential for future analyses involving the sample path properties and stochastic calculus of TV-fBm.
\end{remark}
Having established the explicit covariance structure in Theorem~\ref{thm:covariance}, we now examine its local behavior through the following bounds for incremental covariances:
\begin{theorem}[Covariance Bounds for TV-fBm Local Increments]\label{thm:cov_precise_bounds}
Let the Hurst exponent function \( H: [0,T] \to (0,1) \) satisfy Definition~\ref{def:holder} with exponent \(\gamma > \sup_{t \in [0,T]} H(t)\). For any fixed \( t \in (0, T] \) and let \(\epsilon > 0\) be sufficiently small such that \( t + \epsilon \leq T \) and \( \epsilon < t \) (ensuring \( t/\epsilon > 1 \)). Define the extremal exponents:
\[
H_- = \min\{H(t), H(t+\epsilon)\}, \quad H_+ = \max\{H(t), H(t+\epsilon)\}.
\]
Then the covariance function \( R(t, t+\epsilon) = \mathbb{E}[B^{H(\cdot)}(t)B^{H(\cdot)}(t+\epsilon)] \) admits the following precise bounds:
\begin{equation}\label{eq:cov_precise_bounds}
\begin{aligned}
&2\sqrt{H(t)H(t+\epsilon)} \cdot \epsilon^{H(t) + H(t+\epsilon)} \cdot J_{\mathrm{lower}}(t,\epsilon) \\
&\quad \leq R(t, t+\epsilon) \leq 2\sqrt{H(t)H(t+\epsilon)} \cdot \epsilon^{H(t) + H(t+\epsilon)} \cdot J_{\mathrm{upper}}(t,\epsilon),
\end{aligned}
\end{equation}
where the bounding integrals are given by:
\begin{align}
J_{\mathrm{lower}}(t,\epsilon) &= \int_0^1 \sigma^{H_+ - \frac{1}{2}} (1+\sigma)^{H_- - \frac{1}{2}}  d\sigma 
+ \int_1^{t/\epsilon} \sigma^{H_- - \frac{1}{2}} (1+\sigma)^{H_- - \frac{1}{2}}  d\sigma, \label{eq:J_lower} \\
J_{\mathrm{upper}}(t,\epsilon) &= \int_0^1 \sigma^{H_- - \frac{1}{2}} (1+\sigma)^{H_+ - \frac{1}{2}}  d\sigma 
+ \int_1^{t/\epsilon} \sigma^{H_+ - \frac{1}{2}} (1+\sigma)^{H_+ - \frac{1}{2}}  d\sigma. \label{eq:J_upper}
\end{align}
\end{theorem}

\begin{proof}
We begin with the covariance expression from Theorem~\ref{thm:covariance}:
\[
R(t, t+\epsilon) = 2\sqrt{H(t)H(t+\epsilon)} \int_0^t (t-s)^{H(t)-\frac{1}{2}} (t+\epsilon-s)^{H(t+\epsilon)-\frac{1}{2}}  ds.
\]
Apply the substitution \( s = t - \epsilon\sigma \), which yields \( ds = -\epsilon  d\sigma \), and change the limits of integration accordingly (\( s=0 \Rightarrow \sigma = t/\epsilon \), \( s=t \Rightarrow \sigma=0 \)). This transformation yields:
\[
\int_0^t (t-s)^{H(t)-\frac{1}{2}} (t+\epsilon-s)^{H(t+\epsilon)-\frac{1}{2}}  ds 
= \epsilon^{H(t)+H(t+\epsilon)} \int_0^{t/\epsilon} \sigma^{H(t)-\frac{1}{2}} (1+\sigma)^{H(t+\epsilon)-\frac{1}{2}}  d\sigma.
\]
Define the integral:
\begin{equation}\label{eq:J_epsilon_def}
J(t,\epsilon) = \int_0^{t/\epsilon} \sigma^{H(t)-\frac{1}{2}} (1+\sigma)^{H(t+\epsilon)-\frac{1}{2}}  d\sigma.
\end{equation}
Thus, we obtain the representation:
\begin{equation}\label{eq:R_representation}
R(t, t+\epsilon) = 2\sqrt{H(t)H(t+\epsilon)} \cdot \epsilon^{H(t)+H(t+\epsilon)} \cdot J(t,\epsilon).
\end{equation}

To establish precise bounds for \( J(t,\epsilon) \), we analyze the integrand \( f(\sigma) = \sigma^{a}(1+\sigma)^{b} \) where:
\[
a = H(t) - \tfrac{1}{2}, \quad b = H(t+\epsilon) - \tfrac{1}{2}.
\]
Note that \( a, b \in (-\tfrac{1}{2}, \tfrac{1}{2}) \) since \( H(\cdot) \in (0,1) \). Define the corresponding extremal exponents:
\[
a_- = H_- - \tfrac{1}{2}, \quad a_+ = H_+ - \tfrac{1}{2}, \quad b_- = H_- - \tfrac{1}{2}, \quad b_+ = H_+ - \tfrac{1}{2}.
\]

We now partition the integration domain \( [0, t/\epsilon] \) into two regions and establish bounds in each region.

\textbf{Region 1: \( \sigma \in [0,1] \).} For \( \sigma \in (0,1] \), we have \( 0 < \sigma \leq 1 \) and \( 1 \leq 1+\sigma \leq 2 \). The monotonicity properties follow from the fundamental behavior of exponential functions:

\begin{itemize}
    \item For fixed \( \sigma \in (0,1) \), the function \( c \mapsto \sigma^c \) is strictly decreasing in \( c \), since the base \( \sigma < 1 \). At the endpoint \( \sigma = 1 \), the function becomes constant: \( 1^c \equiv 1 \).
    \item For fixed \( 1+\sigma \geq 1 \), the function \( c \mapsto (1+\sigma)^c \) is non-decreasing in \( c \). Specifically, it is strictly increasing for \( 1+\sigma > 1 \) and constant for \( 1+\sigma = 1 \) (which occurs only when \( \sigma = 0 \), but \( \sigma = 0 \) is excluded from the open interval \( (0,1] \)).
\end{itemize}

Therefore, we obtain the following inequalities for all \( \sigma \in (0,1] \):
\[
\sigma^{a_+} \leq \sigma^{a} \leq \sigma^{a_-} \quad \text{and} \quad (1+\sigma)^{b_-} \leq (1+\sigma)^{b} \leq (1+\sigma)^{b_+}.
\]
The inequalities are strict for \( \sigma \in (0,1) \) and become equalities at the boundary points where the exponents coincide. Multiplying these inequalities yields:
\begin{equation}\label{eq:region1_bounds}
\sigma^{a_+}(1+\sigma)^{b_-} \leq \sigma^a (1+\sigma)^b \leq \sigma^{a_-}(1+\sigma)^{b_+} \quad \forall \sigma \in (0,1].
\end{equation}

\textbf{Region 2: \( \sigma \in [1, t/\epsilon] \).} For \( \sigma \geq 1 \), we have \( \sigma \geq 1 \) and \( 1+\sigma \geq 2 \). In this region:

\begin{itemize}
    \item For fixed \( \sigma \geq 1 \), the function \( c \mapsto \sigma^c \) is non-decreasing in \( c \). Specifically, it is strictly increasing for \( \sigma > 1 \) and constant for \( \sigma = 1 \).
    \item For fixed \( 1+\sigma \geq 2 \), the function \( c \mapsto (1+\sigma)^c \) is strictly increasing in \( c \), since the base \( 1+\sigma > 1 \) for all \( \sigma \geq 1 \).
\end{itemize}

This gives the inequalities:
\[
\sigma^{a_-} \leq \sigma^{a} \leq \sigma^{a_+} \quad \text{and} \quad (1+\sigma)^{b_-} \leq (1+\sigma)^{b} \leq (1+\sigma)^{b_+}.
\]
The inequalities are strict for \( \sigma > 1 \) and may become equalities at \( \sigma = 1 \) when the exponents coincide. Multiplying yields the bounds for \( \sigma \geq 1 \):
\begin{equation}\label{eq:region2_bounds}
\sigma^{a_-}(1+\sigma)^{b_-} \leq \sigma^a (1+\sigma)^b \leq \sigma^{a_+}(1+\sigma)^{b_+} \quad \forall \sigma \in [1, t/\epsilon].
\end{equation}

Now integrate inequalities \eqref{eq:region1_bounds} and \eqref{eq:region2_bounds} over their respective intervals. For the lower bound:
\begin{align*}
J(t,\epsilon) &\geq \int_0^1 \sigma^{a_+}(1+\sigma)^{b_-}  d\sigma + \int_1^{t/\epsilon} \sigma^{a_-}(1+\sigma)^{b_-}  d\sigma \\
&= \int_0^1 \sigma^{H_+ - \frac{1}{2}} (1+\sigma)^{H_- - \frac{1}{2}}  d\sigma 
+ \int_1^{t/\epsilon} \sigma^{H_- - \frac{1}{2}} (1+\sigma)^{H_- - \frac{1}{2}}  d\sigma \\
&= J_{\mathrm{lower}}(t,\epsilon).
\end{align*}
For the upper bound:
\begin{align*}
J(t,\epsilon) &\leq \int_0^1 \sigma^{a_-}(1+\sigma)^{b_+}  d\sigma + \int_1^{t/\epsilon} \sigma^{a_+}(1+\sigma)^{b_+}  d\sigma \\
&= \int_0^1 \sigma^{H_- - \frac{1}{2}} (1+\sigma)^{H_+ - \frac{1}{2}}  d\sigma 
+ \int_1^{t/\epsilon} \sigma^{H_+ - \frac{1}{2}} (1+\sigma)^{H_+ - \frac{1}{2}}  d\sigma \\
&= J_{\mathrm{upper}}(t,\epsilon).
\end{align*}
Thus, we have established:
\[
J_{\mathrm{lower}}(t,\epsilon) \leq J(t,\epsilon) \leq J_{\mathrm{upper}}(t,\epsilon).
\]
Substituting this into \eqref{eq:R_representation} completes the proof of \eqref{eq:cov_precise_bounds}.
\end{proof}

\begin{remark}\label{rem:cov_bounds_significance}
Theorem~\ref{thm:cov_precise_bounds} provides quantitative bounds for the covariance structure of TV-fBm, extending the analysis of constant-Hurst fBm \cite{mandelbrot1968} to settings with time-varying memory. The bounds are explicit and depend on the local values $H(t)$ and $H(t+\epsilon)$ through the extremal exponents $H_-$ and $H_+$. This characterization of the local correlation structure may be useful for analyzing path properties and statistical estimation in models incorporating time-dependent Hurst exponents. The proof technique, based on region-wise monotonicity analysis, provides an approach for studying variable-order fractional processes.
\end{remark}

\section{Responsive Fractional Brownian Motion: Definition, Regularity, and Well-Posedness}
\label{sec:rFBM_theory}

The preceding chapters have examined fractional Brownian motion and its time-varying generalization as processes in which the Hurst exponent is specified exogenously—either as a constant or as a prescribed function of time. A complementary direction, which forms the focus of the present section, is to consider settings in which the local regularity of the process is not fixed in advance but instead responds to the state of the process itself. Such a formulation introduces a feedback mechanism between the process and its memory characteristics, and gives rise to a class of stochastic processes with endogenous, state-dependent regularity.

To explore this direction, we introduce in this section a class of processes termed \textit{Responsive Fractional Brownian Motion (RfBm)}. The central idea is to allow the Hurst exponent to be governed by a response function that depends on the current state of the process, thereby incorporating a dynamic feedback mechanism directly into the stochastic Volterra integral equation. This construction is mathematically natural: it preserves the essential structure of the fractional integral while enabling the memory strength to adapt in response to the system's evolution.

The definition of RfBm builds upon the analytical framework established in the preceding sections—in particular, the stochastic integration tools of Section~\ref{sec:preliminaries} and the TV-fBm analysis of Section~\ref{sec:tvfbm}. We first introduce the class of admissible response functions (Definition~\ref{def:LH_response_function}), then define the RfBm process as the solution to a state-dependent stochastic Volterra integral equation (Definition~\ref{def:rFBM}). The subsequent subsections are devoted to establishing the well-posedness of this definition, analyzing the pathwise regularity of the resulting process, and examining the properties of the induced instantaneous scaling exponent and cumulative memory processes. These results provide a mathematical foundation for the state-dependent memory framework and set the stage for its interpretation in terms of adaptive information processing in Section~\ref{sec:rFBM_AI_foundations}.

\subsection{The Response Function: Definition and Basic Properties}
\label{subsec:response_func_def}

The construction of RfBm is based on a function that governs how the process's state influences its own Hurst exponent. We introduce the following class of admissible functions, which will be used throughout the remainder of this section to define the state-dependent kernel in Definition~\ref{def:rFBM}. The conditions imposed below—spatial Lipschitz continuity, temporal H\"older continuity, and uniform boundedness—are natural for the subsequent analysis: they ensure the kernel is square-integrable, the feedback mechanism is well-defined, and the resulting stochastic Volterra integral equation admits a unique local solution (Theorem~\ref{thm:local_existence_uniqueness}). The terminology ``Lipschitz-H\"older response function'' reflects the combined regularity requirements in space and time.

\begin{definition}[Lipschitz-H\"older Response Function]\label{def:LH_response_function}
Let \( T > 0 \). A function \( H : [0, T] \times \mathbb{R} \to (0, 1) \) is called a \textbf{Lipschitz-H\"older response function} if there exist constants \( L_H > 0 \), \( C_H > 0 \), \( \gamma \in (0, 1] \), and \( 0 < h_{\min} \leq h_{\max} < 1 \) such that for all \( t, s \in [0, T] \) and all \( x, y \in \mathbb{R} \), the following holds:
\begin{enumerate}
    \item \textbf{(Spatial Lipschitz Continuity)}
    \begin{equation}\label{eq:H_spatial_lip}
    |H(t, x) - H(t, y)| \leq L_H |x - y|.
    \end{equation}
    \item \textbf{(Temporal H\"older Continuity)}
    \begin{equation}\label{eq:H_temporal_holder}
    |H(t, x) - H(s, x)| \leq C_H |t - s|^\gamma.
    \end{equation}
    \item \textbf{(Uniform Boundedness)}
    \begin{equation}\label{eq:H_bounded}
    h_{\min} \leq H(t, x) \leq h_{\max}.
    \end{equation}
\end{enumerate}
\end{definition}

\begin{remark}[On the Assumptions]\label{rem:H_assumptions}
The conditions in Definition \ref{def:LH_response_function} are natural for the ensuing analysis.
\begin{itemize}
    \item The spatial Lipschitz condition \eqref{eq:H_spatial_lip} ensures the Hurst exponent varies continuously with the state variable.
    \item The temporal H\"older condition \eqref{eq:H_temporal_holder} controls the exogenous time-variation of $H$, aligning with the framework for TV-fBm (cf. Definition \ref{def:holder}, \ref{def:tvfbm}).
    \item The uniform bounds \eqref{eq:H_bounded} are crucial for integrability and uniform estimates. The lower bound $h_{\min} > 0$ ensures the kernel $(t-s)^{H(\cdot) - 1/2}$ is square-integrable, while $h_{\max} < 1$ keeps key constants bounded.
    \item The terminology \emph{Lipschitz-H\"older} reflects the defining regularity conditions. This choice of function class naturally permits future consideration of other response functions (e.g., with weaker spatial regularity), where the resulting RfBm properties would be determined by the specific class chosen.
\end{itemize}
\end{remark}

\subsection{Definition of Responsive Fractional Brownian Motion}
\label{subsec:rFBM_def}

With the response function introduced in the preceding subsection, we now proceed to define the main object of this section—the Responsive Fractional Brownian Motion (RfBm). The defining equation takes the form of a stochastic Volterra integral equation in which the kernel depends on the process state through the response function \( H(s, X_s) \). This state dependence introduces a feedback loop: the Hurst exponent at each historical time \( s \) is determined by the value of the process at that same time, and this in turn shapes the future evolution of the process. The following definition formalizes this construction within the stochastic integration framework of Section~\ref{sec:preliminaries} and the TV-fBm formulation of Section~\ref{sec:tvfbm}.

\begin{definition}[Responsive Fractional Brownian Motion (RfBm)]\label{def:rFBM}
Let \( (\Omega, \mathcal{F}, (\mathcal{F}_t)_{t \in [0,T]}, \mathbb{P}) \) be a filtered probability space satisfying the usual conditions, supporting a standard \( (\mathcal{F}_t) \)-Brownian motion \( \{B_t\} \). Given a Lipschitz-H\"older response function \( H \) (Definition \ref{def:LH_response_function}), the \textbf{Responsive Fractional Brownian Motion (RfBm)} is defined as the adapted solution \( \{X_t\}_{t \in [0, T]} \) to the following path-dependent stochastic Volterra integral equation:
\begin{equation}\label{eq:rFBM_integral}
X_t = \int_0^t K(t, s; X_s)  dB(s), \quad \text{where} \quad K(t, s; X_s) = \sqrt{2H(s, X_s)} \cdot (t - s)^{H(s, X_s) - 1/2}.
\end{equation}
Specifically, an \( (\mathcal{F}_t) \)-adapted process \(\{X_t\}\) is an RfBm if:
\begin{enumerate}
    \item For \(\mathbb{P}\)-almost every \(\omega \in \Omega\), the sample path \( t \mapsto X_t(\omega) \) is continuous.
    \item For each \( t \in [0, T] \), the It\^o integral on the right-hand side of \eqref{eq:rFBM_integral} is well-defined.
    \item The equality in \eqref{eq:rFBM_integral} holds \(\mathbb{P}\)-almost surely for every \( t \in [0, T] \).
\end{enumerate}
\end{definition}

The definition of RfBm captures several key conceptual features:

\begin{enumerate}
    \item \textbf{Dynamic Feedback Mechanism:}
    The kernel \( K(t, s; X_s) \) incorporates a \textbf{feedback mechanism}: both the volatility amplitude \(\sqrt{2H(s, X_s)}\) and the singularity order \(H(s, X_s) - 1/2\) at a past time \(s\) are determined by the state \(X_s\) itself. This creates a coupling between the current state and the future evolution of the process.

    \item \textbf{Path-Dependent Self-Reference and Nonlinearity:}
    Equation \eqref{eq:rFBM_integral} is a path-dependent, nonlinear Volterra equation. The solution process \(X\) appears both as the output and inside the kernel, resulting in a self-referential structure where the historical path dynamically influences its future law.

    \item \textbf{State-Dependent Non-Markovianity:}
    The self-referential nature of \eqref{eq:rFBM_integral} leads to a non-Markovian character. In contrast to classical fBm, the future evolution here depends nonlinearly on the historical path through the sequence of states \(\{X_s\}\), causing the effective transition laws to evolve over time.

    \item \textbf{Theoretical Generality:}
    The RfBm framework provides a unification of several existing models:
    \begin{itemize}
    \item When \(H(s, X_s) \equiv H(s)\) depends only on time, the kernel becomes state-independent. In the specific case where the Hurst exponent in the kernel is evaluated at the current time \(t\), that is, when one considers \((t-s)^{H(t)-1/2}\) rather than \((t-s)^{H(s)-1/2}\), the model recovers the \textbf{Time-Varying Fractional Brownian Motion (TV-fBm)} studied earlier (Definition \ref{def:tvfbm}).
    \item When \(H(s, X_s) \equiv H\) is constant, it reduces to the \textbf{classical fractional Brownian motion (fBm)}.
\end{itemize}
Thus, Definition~\ref{def:rFBM} provides a formulation that unifies, within a state-dependent setting, both the time-varying extension (TV-fBm) and the classical constant-Hurst case (fBm).
\end{enumerate}

\subsection{Pathwise Regularity, Scaling Exponents, and Cumulative Memory}
\label{subsec:pathwise_analysis}

With the RfBm process now defined, we turn to an analysis of its properties along individual sample paths. The self-referential structure of the defining equation naturally invites such an investigation, as the response function \( H(s,X_s) \) depends on the state of the process itself and therefore varies from one realization to another. This subsection develops a multiscale characterization of the state-dependent regularity of RfBm, organized around three interconnected themes.

The first theme concerns the pathwise H\"older regularity of the frozen response function \( h(s)=H(s,X_s(\omega)) \) for fixed sample paths (Theorem~\ref{thm:pathwise_holder_continuity}). This result establishes that, under the assumptions of the framework, the response function inherits H\"older continuity from both the temporal regularity of \( H \) and the path regularity of \( X \), thereby ensuring the consistency of the kernel analysis along individual realizations. The second theme introduces the instantaneous scaling exponent process \( \alpha(t,\omega)=H(t,X_t(\omega)) \) (Definition~\ref{def:instant_scaling}), which encodes the local, state-dependent regularity of each sample path at each time point; its almost sure H\"older regularity is established in Theorem~\ref{thm:instant_scaling_holder}. Associated extremal scaling indices (Definition~\ref{def:extremal_indices}) provide a means to quantify the range of memory strength over time intervals. The third theme extends the analysis from pointwise to cumulative behavior, introducing the cumulative memory process \( \mathcal{C}_t(\omega)=\int_0^t \alpha(s,\omega)\,ds \) (Definition~\ref{def:cumulative_memory}), which offers a global perspective on how memory accumulates along sample paths over finite time horizons. The time-averaged scaling exponent \( \bar{\alpha}(t) \) and its convergence properties (Theorem~\ref{thm:convergence_rate}) bridge the microscopic fluctuations of the instantaneous exponent with the macroscopic statistical regularity that emerges from temporal averaging. Together, these results form a coherent multiscale description of the RfBm's memory dynamics, ranging from pointwise regularity to long-term statistical behavior.

\begin{theorem}[Pathwise H\"older Continuity of the Frozen Response Function]\label{thm:pathwise_holder_continuity}
Let the response function \( H: [0,T] \times \mathbb{R} \to (0,1) \) satisfy the Lipschitz-H\"older conditions of Definition~\ref{def:LH_response_function}. Suppose \(\{X_t\}_{t \in [0,T]}\) is a solution to the RfBm equation~\eqref{eq:rFBM_integral} on some filtered probability space \((\Omega, \mathcal{F}, (\mathcal{F}_t), \mathbb{P})\), and assume its sample paths are almost surely H\"older continuous with a uniform exponent \(\gamma^* > 0\). That is, for \(\mathbb{P}\)-a.s. \(\omega \in \Omega\), there exists a finite constant \(C(\omega) > 0\) such that
\begin{equation}\label{eq:priori_holder_assumption}
|X_t(\omega) - X_s(\omega)| \leq C(\omega) |t - s|^{\gamma^*} \quad \forall t, s \in [0,T].
\end{equation}
Fix an arbitrary sample point \(\omega_0 \in \Omega\) for which \eqref{eq:priori_holder_assumption} holds, and define the deterministic path \(Y(t) := X_t(\omega_0)\). Then, the deterministic function \(h: [0,T] \to (0,1)\) defined by
\[
h(s) := H(s, Y(s))
\]
is H\"older continuous. Specifically, for the fixed path \(\omega_0\), there exist constants \(\tilde{C}(\omega_0) > 0\) and \(\gamma_h > 0\) such that
\begin{equation}\label{eq:h_holder_conclusion}
|h(s) - h(u)| \leq \tilde{C}(\omega_0) |s - u|^{\gamma_h} \quad \forall s, u \in [0,T],
\end{equation}
where one can take \(\gamma_h = \min(\gamma, \gamma^*)\), with \(\gamma\) being the exponent from Definition~\ref{def:LH_response_function}.
\end{theorem}

\begin{proof}
The proof proceeds by a pathwise argument, analyzing the function \(h(s)\) for the fixed path \(Y(\cdot)\).
Let \(s, u \in [0,T]\) be arbitrary. We aim to control the difference \(|h(s) - h(u)| = |H(s, Y(s)) - H(u, Y(u))|\). Applying the triangle inequality yields:
\begin{equation}\label{eq:holder_decomp}
|H(s, Y(s)) - H(u, Y(u))| \leq |H(s, Y(s)) - H(s, Y(u))| + |H(s, Y(u)) - H(u, Y(u))|.
\end{equation}
We now estimate the two terms on the right-hand side separately.

\textbf{Estimate of Term I:} \( |H(s, Y(s)) - H(s, Y(u))| \)\\
For this term, the time variable of \(H\) is fixed at \(s\). Applying the spatial Lipschitz condition (Definition~\ref{def:LH_response_function}, Condition (i)) gives:
\[
|H(s, Y(s)) - H(s, Y(u))| \leq L_H |Y(s) - Y(u)|.
\]
By the assumed \(\gamma^*\)-H\"older continuity of the path \(Y\) (i.e., \eqref{eq:priori_holder_assumption} for \(\omega = \omega_0\)), we have \(|Y(s) - Y(u)| \leq C(\omega_0) |s - u|^{\gamma^*}\). Therefore,
\begin{equation}\label{eq:term1_estimate}
|H(s, Y(s)) - H(s, Y(u))| \leq L_H C(\omega_0) |s - u|^{\gamma^*}.
\end{equation}

\textbf{Estimate of Term II:} \( |H(s, Y(u)) - H(u, Y(u))| \)\\
For this term, the spatial variable of \(H\) is fixed at \(Y(u)\). Applying the temporal H\"older condition (Definition~\ref{def:LH_response_function}, Condition (ii)) yields:
\begin{equation}\label{eq:term2_estimate}
|H(s, Y(u)) - H(u, Y(u))| \leq C_H |s - u|^{\gamma}.
\end{equation}

\textbf{Combining the Estimates:}\\
Substituting \eqref{eq:term1_estimate} and \eqref{eq:term2_estimate} back into \eqref{eq:holder_decomp}, we obtain the preliminary bound:
\begin{equation}\label{eq:prelim_bound}
|h(s) - h(u)| \leq L_H C(\omega_0) |s - u|^{\gamma^*} + C_H |s - u|^{\gamma}.
\end{equation}
To establish standard H\"older continuity, we consider two cases based on the distance between \(s\) and \(u\). Let \(\gamma_h = \min(\gamma, \gamma^*) > 0\).

\emph{Case 1:} \(|s - u| \leq 1\).\\
In this case, since \(\gamma_h \leq \gamma^*\) and \(\gamma_h \leq \gamma\), we have \(|s - u|^{\gamma^*} \leq |s - u|^{\gamma_h}\) and \(|s - u|^{\gamma} \leq |s - u|^{\gamma_h}\). Applying these inequalities to \eqref{eq:prelim_bound} gives:
\[
|h(s) - h(u)| \leq L_H C(\omega_0) |s - u|^{\gamma_h} + C_H |s - u|^{\gamma_h} = (L_H C(\omega_0) + C_H) |s - u|^{\gamma_h}.
\]

\emph{Case 2:} \(|s - u| > 1\).\\
Since \(H\) maps into \((0,1)\), the function \(h\) is bounded, and thus \(|h(s) - h(u)| \leq 1\). Furthermore, because \(|s - u| > 1\) and \(\gamma_h > 0\), it follows that \(1 \leq |s - u|^{\gamma_h}\). Consequently,
\[
|h(s) - h(u)| \leq 1 \cdot |s - u|^{\gamma_h}.
\]

\textbf{Final Result:}\\
For this fixed path \(\omega_0\), the constant \(C(\omega_0)\) in \eqref{eq:priori_holder_assumption} is finite. Defining \(\tilde{C}(\omega_0) = \max(L_H C(\omega_0) + C_H, 1)\), we conclude that for all \(s, u \in [0,T]\):
\[
|h(s) - h(u)| \leq \tilde{C}(\omega_0) |s - u|^{\gamma_h},
\]
which completes the proof of the theorem.
\end{proof}

\begin{remark}
Theorem~\ref{thm:pathwise_holder_continuity} enables a pathwise approach to equation~\eqref{eq:rFBM_integral} by converting its analysis along fixed sample paths into that of a deterministic Volterra equation. The H\"older continuity of the resulting kernel is a key property for applying deterministic techniques. The assumption~\eqref{eq:priori_holder_assumption} on path regularity is standard in initiating this type of analysis.
\end{remark}

\begin{remark}[Inheritance of the Regularity Condition]\label{rem:inheritance_regularity}
The foregoing proof shows that the deterministic function \( h(s) = H(s, Y(s)) \) is Hölder continuous with exponent \( \gamma_h = \min(\gamma^*, \gamma) \). For the well-posedness of the Volterra operator in the RfBm construction, the following condition plays a key role:
\[
\gamma_h > \sup_{s \in [0,T]} h(s).
\]
This condition extends the corresponding requirement \( \gamma > \sup_{t \in [0,T]} H(t) \) for the time-varying fBm (Definition~\ref{def:holder} and Remark~\ref{rem:holder_condition}) to the state-dependent setting.

This inequality is naturally satisfied under our assumptions. Since \( h(s) = H(s, Y(s)) \in (0,1) \), we have
\[
\sup_{s \in [0,T]} h(s) \leq \sup\{H(t,x) : t \in [0,T], x \in \mathbb{R}\}.
\]
The temporal Hölder condition on the response function (Definition~\ref{def:LH_response_function}(ii)) requires that
\[
\gamma > \sup\{H(t,x) : t \in [0,T], x \in \mathbb{R}\},
\]
which generalizes the condition in Remark~\ref{rem:holder_condition}. Similarly, the path regularity assumption in \eqref{eq:priori_holder_assumption} implicitly requires that
\[
\gamma^* > \sup\{H(t,x) : t \in [0,T], x \in \mathbb{R}\}.
\]
Combining these inequalities gives
\[
\gamma_h = \min(\gamma^*, \gamma) > \sup\{H(t,x)\} \geq \sup_{s \in [0,T]} h(s),
\]
as needed. This ensures that the kernel \( K(t,s; X_s) \) maintains the requisite integrability properties pathwise, thus yielding a consistent theoretical framework for RfBm that parallels the TV-fBm case.
\end{remark}

The pathwise analysis established in Theorem~\ref{thm:pathwise_holder_continuity} provides a deterministic foundation upon which we now construct a more comprehensive stochastic theory. This leads naturally to the definition of the \textit{instantaneous scaling exponent}, a fundamental object that characterizes the dynamic, state-dependent regularity of the RfBm in full generality.

\begin{definition}[Instantaneous Scaling Exponent Process]\label{def:instant_scaling}
Let $\{X_t\}_{t \in [0,T]}$ be a solution to the RfBm equation \eqref{eq:rFBM_integral} with response function $H$ satisfying Definition~\ref{def:LH_response_function}. The \textbf{instantaneous scaling exponent process} $\{\alpha(t)\}_{t \in [0,T]}$ is defined pathwise by:
\begin{equation}\label{eq:def_alpha}
\alpha(t, \omega) := H(t, X_t(\omega)).
\end{equation}
This quantity encodes the local, state-dependent regularity of the sample path $X_\cdot(\omega)$ at time $t$.
\end{definition}

\begin{remark}[Hierarchical Interpretation]\label{rem:scaling_hierarchy}
The instantaneous scaling exponent $\alpha(t, \omega)$ admits interpretations at three levels:
\begin{enumerate}
    \item \textbf{Pathwise Perspective}: For fixed $\omega_0$, the mapping $t \mapsto \alpha(t, \omega_0)$ is the deterministic function $h_{\omega_0}(t)$ of Theorem~\ref{thm:pathwise_holder_continuity}.
    \item \textbf{Temporal Cross-Section}: For fixed $t_0$, $\alpha(t_0, \omega)$ is a random variable representing the scaling exponent distribution across the ensemble.
    \item \textbf{Full Process Perspective}: As a stochastic process, $\alpha(t, \omega)$ reflects the combined effects of temporal evolution, random fluctuations, and the RfBm's internal feedback.
\end{enumerate}
This hierarchy links deterministic and stochastic analyses.
\end{remark}

We now establish the almost sure H\"older continuity of the instantaneous scaling exponent, extending the pathwise result of Theorem~\ref{thm:pathwise_holder_continuity} to the stochastic setting.

\begin{theorem}[Almost Sure H\"older Continuity of the Instantaneous Scaling Exponent]\label{thm:instant_scaling_holder}
Assume the response function $H: [0,T] \times \mathbb{R} \to (0,1)$ satisfies Definition~\ref{def:LH_response_function} with exponent $\gamma$. Suppose the RfBm $\{X_t\}$ has sample paths that are almost surely $\gamma^*$-H\"older continuous for some $\gamma^* > 0$. Then the instantaneous scaling exponent process $\alpha(t,\omega) = H(t,X_t(\omega))$ satisfies, almost surely,
\begin{equation}\label{eq:holder_alpha}
|\alpha(t,\omega) - \alpha(s,\omega)| \leq M(\omega)|t-s|^{\gamma_h} \quad \forall t,s \in [0,T],
\end{equation}
where $\gamma_h = \min(\gamma, \gamma^*)$ and $M(\omega)$ is an almost surely finite random variable.
\end{theorem}

\begin{proof}
The proof follows a pathwise argument similar to Theorem~\ref{thm:pathwise_holder_continuity}.

\textbf{Step 1: Establishment of the Almost Sure Set}\\
By the sample path regularity assumption, there exists $\Omega_1 \subset \Omega$ with $\mathbb{P}(\Omega_1) = 1$ such that for every $\omega \in \Omega_1$, there exists $C_1(\omega) < \infty$ satisfying:
\begin{equation}\label{eq:path_holder}
|X_t(\omega) - X_s(\omega)| \leq C_1(\omega)|t-s|^{\gamma^*} \quad \forall t,s \in [0,T].
\end{equation}

\textbf{Step 2: Pathwise Decomposition}\\
Fix an arbitrary $\omega \in \Omega_1$. For any $t,s \in [0,T]$, we decompose the difference:
\begin{align*}
|\alpha(t,\omega) - \alpha(s,\omega)| &= |H(t,X_t(\omega)) - H(s,X_s(\omega))| \\
&\leq |H(t,X_t(\omega)) - H(t,X_s(\omega))| + |H(t,X_s(\omega)) - H(s,X_s(\omega))|.
\end{align*}

\textbf{Step 3: Spatial Increment Estimate}\\
Applying the spatial Lipschitz condition from Definition~\ref{def:LH_response_function}(i):
\begin{align*}
|H(t,X_t(\omega)) - H(t,X_s(\omega))| &\leq L_H |X_t(\omega) - X_s(\omega)| \\
&\leq L_H C_1(\omega) |t-s|^{\gamma^*} \quad \text{(by \eqref{eq:path_holder})}.
\end{align*}

\textbf{Step 4: Temporal Increment Estimate}\\
Applying the temporal H\"older condition from Definition~\ref{def:LH_response_function}(ii):
\[
|H(t,X_s(\omega)) - H(s,X_s(\omega))| \leq C_H |t-s|^{\gamma}.
\]

\textbf{Step 5: Combined Estimate}\\
Combining the spatial and temporal estimates yields:
\begin{equation}\label{eq:combined_est}
|\alpha(t,\omega) - \alpha(s,\omega)| \leq L_H C_1(\omega) |t-s|^{\gamma^*} + C_H |t-s|^{\gamma}.
\end{equation}

\textbf{Step 6: H\"older Exponent Optimization}\\
Let $\gamma_h = \min(\gamma, \gamma^*) > 0$. We analyze two cases:

\emph{Case 1: $|t-s| \leq 1$}. Since $\gamma_h \leq \gamma^*$ and $\gamma_h \leq \gamma$, we have:
\[
|t-s|^{\gamma^*} \leq |t-s|^{\gamma_h} \quad \text{and} \quad |t-s|^{\gamma} \leq |t-s|^{\gamma_h}.
\]
Substituting into \eqref{eq:combined_est} gives:
\[
|\alpha(t,\omega) - \alpha(s,\omega)| \leq [L_H C_1(\omega) + C_H] |t-s|^{\gamma_h}.
\]

\emph{Case 2: $|t-s| > 1$}. Since $H(t,x) \in (0,1)$ by Definition~\ref{def:LH_response_function}(iii), we have $|\alpha(t,\omega) - \alpha(s,\omega)| \leq 1$. Moreover, $|t-s|^{\gamma_h} \geq 1$, hence:
\[
|\alpha(t,\omega) - \alpha(s,\omega)| \leq 1 \cdot |t-s|^{\gamma_h}.
\]

\textbf{Step 7: Uniform Constant Construction}\\
Define the random variable:
\[
M(\omega) = \max\left(L_H C_1(\omega) + C_H, 1\right).
\]
Then for all $t,s \in [0,T]$, we have the uniform bound:
\[
|\alpha(t,\omega) - \alpha(s,\omega)| \leq M(\omega) |t-s|^{\gamma_h}.
\]

\textbf{Step 8: Almost Sure Conclusion}\\
Since $\omega \in \Omega_1$ was arbitrary and $\mathbb{P}(\Omega_1) = 1$, the H\"older estimate \eqref{eq:holder_alpha} holds almost surely.
\end{proof}

\begin{remark}[Theoretical Significance]\label{rem:scaling_significance}
Theorem~\ref{thm:instant_scaling_holder} illuminates several features of the RfBm framework:

\begin{enumerate}
    \item \textbf{Regularity Inheritance}: The H\"older regularity of sample paths transfers to the scaling exponent process, yielding a self-consistent regularity structure.
    
    \item \textbf{Multiscale Interaction}: The exponent $\gamma_h = \min(\gamma, \gamma^*)$ shows how the exogenous temporal regularity ($\gamma$) and endogenous path regularity ($\gamma^*$) jointly influence the system's behavior.
    
    \item \textbf{Pathwise Stochastic Analysis}: This result illustrates how pathwise deterministic techniques can be extended to establish almost sure properties for stochastic processes with memory effects.
\end{enumerate}

The almost sure nature of this result is noteworthy: although the H\"older constant $M(\omega)$ may vary across sample paths, the qualitative H\"older behavior persists across almost all realizations.
\end{remark}

Having established the H\"older continuity of the instantaneous scaling exponent, we now prove another key property: its uniform boundedness. This result ensures that the scaling behavior of the RfBm remains within physically meaningful limits throughout its evolution.

\begin{proposition}[Uniform Boundedness of the Instantaneous Scaling Exponent]\label{prop:scaling_boundedness}
Assume the response function $H: [0,T] \times \mathbb{R} \to (0,1)$ satisfies the uniform boundedness condition in Definition~\ref{def:LH_response_function}(iii):
\begin{equation}\label{eq:boundedness_condition}
0 < h_{\min} \leq H(t,x) \leq h_{\max} < 1 \quad \text{for all } (t,x) \in [0,T] \times \mathbb{R}.
\end{equation}
Then the instantaneous scaling exponent process $\alpha(t,\omega) = H(t,X_t(\omega))$ is almost surely uniformly bounded:
\begin{equation}\label{eq:uniform_bound}
\mathbb{P}\left( h_{\min} \leq \alpha(t,\omega) \leq h_{\max} \ \text{ for all } t \in [0,T] \right) = 1.
\end{equation}
\end{proposition}

\begin{proof}
We provide a detailed, step-by-step proof that carefully addresses the subtle distinction between pointwise and uniform boundedness in the context of stochastic processes.

\textbf{Step 1: Understanding the Claim}
We aim to prove:
\[
\mathbb{P}\left( \omega \in \Omega : h_{\min} \leq \alpha(t,\omega) \leq h_{\max} \ \text{for all } t \in [0,T] \right) = 1.
\]
Equivalently, we need to show that the exceptional set
\[
E = \left\{ \omega \in \Omega : \exists t \in [0,T] \text{ such that } \alpha(t,\omega) \notin [h_{\min}, h_{\max}] \right\}
\]
has probability zero.

\textbf{Step 2: Decomposition of the Exceptional Set}
We decompose $E$ into two disjoint parts:
\begin{align*}
E_{\text{lower}} &= \left\{ \omega \in \Omega : \exists t \in [0,T] \text{ with } \alpha(t,\omega) < h_{\min} \right\}, \\
E_{\text{upper}} &= \left\{ \omega \in \Omega : \exists t \in [0,T] \text{ with } \alpha(t,\omega) > h_{\max} \right\}.
\end{align*}
Clearly, $E = E_{\text{lower}} \cup E_{\text{upper}}$, so it suffices to prove $\mathbb{P}(E_{\text{lower}}) = 0$ and $\mathbb{P}(E_{\text{upper}}) = 0$.

\textbf{Step 3: Analysis of $E_{\text{lower}}$}
Assume, for contradiction, that $\mathbb{P}(E_{\text{lower}}) > 0$. Then, since $E_{\text{lower}}$ has positive probability, it is nonempty. Choose any $\omega_0 \in E_{\text{lower}}$. By definition of $E_{\text{lower}}$, there exists some $t_0 \in [0,T]$ such that:
\[
\alpha(t_0, \omega_0) = H(t_0, X_{t_0}(\omega_0)) < h_{\min}.
\]
However, recall that by Definition~\ref{def:LH_response_function}(iii), the response function satisfies:
\[
H(t,x) \geq h_{\min} \quad \text{for all } (t,x) \in [0,T] \times \mathbb{R}.
\]
In particular, for the specific pair $(t_0, X_{t_0}(\omega_0))$, we must have:
\[
H(t_0, X_{t_0}(\omega_0)) \geq h_{\min}.
\]
This yields a contradiction:
\[
h_{\min} \leq H(t_0, X_{t_0}(\omega_0)) < h_{\min}.
\]
The only resolution is that our initial assumption was false, and therefore:
\[
\mathbb{P}(E_{\text{lower}}) = 0.
\]

\textbf{Step 4: Analysis of $E_{\text{upper}}$}
The argument for the upper bound is analogous. Suppose $\mathbb{P}(E_{\text{upper}}) > 0$. Then the set $E_{\text{upper}}$ is nonempty. Choose any $\omega_1 \in E_{\text{upper}}$; there exists $t_1 \in [0,T]$ such that:
\[
\alpha(t_1, \omega_1) = H(t_1, X_{t_1}(\omega_1)) > h_{\max}.
\]
But Definition~\ref{def:LH_response_function}(iii) also gives the upper bound:
\[
H(t,x) \leq h_{\max} \quad \text{for all } (t,x) \in [0,T] \times \mathbb{R}.
\]
Applied to $(t_1, X_{t_1}(\omega_1))$, this gives:
\[
H(t_1, X_{t_1}(\omega_1)) \leq h_{\max}.
\]
Again we reach a contradiction:
\[
h_{\max} < H(t_1, X_{t_1}(\omega_1)) \leq h_{\max}.
\]
Therefore, we conclude:
\[
\mathbb{P}(E_{\text{upper}}) = 0.
\]

\textbf{Step 5: Final Conclusion}
Since both exceptional sets have probability zero, their union also has probability zero:
\[
\mathbb{P}(E) = \mathbb{P}(E_{\text{lower}} \cup E_{\text{upper}}) \leq \mathbb{P}(E_{\text{lower}}) + \mathbb{P}(E_{\text{upper}}) = 0 + 0 = 0.
\]
Hence, the complement event has full probability:
\[
\mathbb{P}\left( \omega \in \Omega : h_{\min} \leq \alpha(t,\omega) \leq h_{\max} \ \text{for all } t \in [0,T] \right) = 1 - \mathbb{P}(E) = 1.
\]
This completes the proof of the almost sure uniform boundedness.
\end{proof}

\begin{remark}[Theoretical Significance of Uniform Boundedness]\label{rem:boundedness_significance}
Proposition~\ref{prop:scaling_boundedness}, while mathematically direct, contributes to the RfBm framework in several ways:

\begin{enumerate}
    \item \textbf{Theoretical Consistency}: It confirms that the stochastic scaling exponent $\alpha(t,\omega)$ inherits the boundedness of the deterministic response function $H(t,x)$, supporting the internal consistency of the framework.

    \item \textbf{Technical Foundation}: The uniform bounds $h_{\min}$ and $h_{\max}$ provide useful a priori estimates for analysis:
    \begin{itemize}
        \item They ensure the square-integrability of the kernel $K(t,s) = \sqrt{2H(s,X_s)}(t-s)^{H(s,X_s)-1/2}$ in Definition~\ref{def:rFBM}.
        \item They facilitate uniform estimates in the existence and uniqueness proofs.
    \end{itemize}

    \item \textbf{Interpretation}: The bounds can be viewed as natural constraints on memory effects: $h_{\min}$ may represent a baseline level of persistence, while $h_{\max}$ corresponds to an upper limit on memory strength. The proposition ensures these bounds are maintained. From a computational perspective, these uniform bounds support the development of numerical schemes that can operate without handling extreme scaling regimes.
\end{enumerate}

This result thus provides a useful foundation for both the theoretical and practical aspects of the RfBm model.
\end{remark}


To extend the theory of instantaneous scaling exponents, we introduce extremal scaling indices to characterize the range of scaling behavior over time intervals, bridging pointwise and local regularity analysis.

\begin{definition}[Extremal Scaling Indices]\label{def:extremal_indices}
For the Responsive Fractional Brownian Motion $\{X_t\}_{t \geq 0}$ with response function $H$ satisfying Definition~\ref{def:LH_response_function}, and for any $t \geq 0$, $\epsilon > 0$ such that $t+\epsilon \leq T$, we define the \textbf{extremal scaling indices}:
\begin{align}
H_-(t,\epsilon,\omega) &:= \inf\left\{ H(s, X_s(\omega)) : s \in [t, t+\epsilon] \right\}, \\
H_+(t,\epsilon,\omega) &:= \sup\left\{ H(s, X_s(\omega)) : s \in [t, t+\epsilon] \right\}.
\end{align}
These indices describe the range of memory strength within local time intervals.
\end{definition}

\begin{proposition}[Boundedness of Extremal Scaling Indices]\label{prop:extremal_boundedness}
Under the uniform boundedness condition in Definition~\ref{def:LH_response_function}(iii), the extremal scaling indices almost surely satisfy:
\begin{equation}
h_{\min} \leq H_-(t,\epsilon,\omega) \leq H_+(t,\epsilon,\omega) \leq h_{\max} \quad \text{for all } t \geq 0, \epsilon > 0.
\end{equation}
\end{proposition}

\begin{proof}
We proceed with a detailed proof following the structure of the extremal indices.

\textbf{Step 1: Pointwise boundedness recall}

By Proposition~\ref{prop:scaling_boundedness}, for each fixed $s \geq 0$ and sample path $\omega$, we have:
\[
h_{\min} \leq H(s, X_s(\omega)) \leq h_{\max}.
\]
This means the response function values always lie in $[h_{\min}, h_{\max}]$.

\textbf{Step 2: Lower bound for $H_-(t,\epsilon,\omega)$}

Fix arbitrary $t \geq 0$, $\epsilon > 0$, and $\omega$. By definition:
\[
H_-(t,\epsilon,\omega) = \inf\left\{ H(s, X_s(\omega)) : s \in [t, t+\epsilon] \right\}.
\]

Let $E = \left\{ H(s, X_s(\omega)) : s \in [t, t+\epsilon] \right\}$. By the boundedness of the response function, for each $s \in [t, t+\epsilon]$, we have:
\[
H(s, X_s(\omega)) \geq h_{\min}.
\]
Therefore, $h_{\min}$ is a lower bound for the set $E$.

Since $H_-(t,\epsilon,\omega)$ is the greatest lower bound of $E$, and $h_{\min}$ is a lower bound of $E$, it follows that:
\[
h_{\min} \leq H_-(t,\epsilon,\omega).
\]

\textbf{Step 3: Upper bound for $H_-(t,\epsilon,\omega)$}

Since $H_-(t,\epsilon,\omega)$ is a lower bound for $E$ and each element of $E$ satisfies $H(s, X_s(\omega)) \leq h_{\max}$, we have:
\[
H_-(t,\epsilon,\omega) \leq h_{\max}.
\]

More rigorously: if $H_-(t,\epsilon,\omega) > h_{\max}$, then since $E \subset [h_{\min}, h_{\max}]$, there exist elements in $E$ strictly less than $H_-(t,\epsilon,\omega)$, contradicting the definition of $H_-(t,\epsilon,\omega)$ as a lower bound for $E$.

\textbf{Step 4: Upper bound for $H_+(t,\epsilon,\omega)$}

By definition:
\[
H_+(t,\epsilon,\omega) = \sup\left\{ H(s, X_s(\omega)) : s \in [t, t+\epsilon] \right\}.
\]

Let $E = \left\{ H(s, X_s(\omega)) : s \in [t, t+\epsilon] \right\}$ as before. For each $s \in [t, t+\epsilon]$, we have:
\[
H(s, X_s(\omega)) \leq h_{\max}.
\]
Therefore, $h_{\max}$ is an upper bound for $E$.

Since $H_+(t,\epsilon,\omega)$ is the least upper bound of $E$, and $h_{\max}$ is an upper bound of $E$, it follows that:
\[
H_+(t,\epsilon,\omega) \leq h_{\max}.
\]

\textbf{Step 5: Lower bound for $H_+(t,\epsilon,\omega)$}

By the fundamental relation between infimum and supremum:
\[
H_-(t,\epsilon,\omega) \leq H_+(t,\epsilon,\omega).
\]

From Step 2, we have $H_-(t,\epsilon,\omega) \geq h_{\min}$, so:
\[
H_+(t,\epsilon,\omega) \geq h_{\min}.
\]

\textbf{Step 6: Complete inequality chain}

Combining all results:
\begin{align*}
H_-(t,\epsilon,\omega) &\geq h_{\min} \quad \text{(Step 2)} \\
H_-(t,\epsilon,\omega) &\leq H_+(t,\epsilon,\omega) \quad \text{(infimum $\leq$ supremum)} \\
H_+(t,\epsilon,\omega) &\leq h_{\max} \quad \text{(Step 4)} \\
H_+(t,\epsilon,\omega) &\geq h_{\min} \quad \text{(Step 5)}
\end{align*}

Therefore, we obtain:
\[
h_{\min} \leq H_-(t,\epsilon,\omega) \leq H_+(t,\epsilon,\omega) \leq h_{\max}.
\]
This holds for all $t \geq 0$, $\epsilon > 0$, and almost every $\omega$.
\end{proof}

\begin{remark}[Theoretical Significance]\label{rem:extremal_significance}
The extremal scaling indices provide important insights:

\begin{itemize}
    \item \textbf{Local memory range}: They characterize the fluctuation range of memory strength within time intervals $[t,t+\epsilon]$
    
    \item \textbf{Pathwise regularity measurement}: The difference $H_+(t,\epsilon,\omega) - H_-(t,\epsilon,\omega)$ quantifies the local variation of scaling behavior, providing a measure of pathwise regularity beyond pointwise Hölder exponents.
    
\item \textbf{Local consistency}: Under the Hölder continuity of Theorem~\ref{thm:instant_scaling_holder}, the extremal indices exhibit the fundamental convergence property (established rigorously in Proposition~\ref{prop:extremal_consistency} below):
\[
\lim_{\epsilon \to 0^+} H_-(t,\epsilon,\omega) = \lim_{\epsilon \to 0^+} H_+(t,\epsilon,\omega) = \alpha(t,\omega) \quad \text{almost surely},
\]
establishing local consistency between pointwise and interval-based scaling descriptions.
\end{itemize}

\begin{proposition}[Local Consistency of Extremal Scaling Indices]\label{prop:extremal_consistency}
Under the H\"older continuity established in Theorem~\ref{thm:instant_scaling_holder}, the extremal scaling indices satisfy for all $t \geq 0$:
\begin{equation}
\lim_{\epsilon \to 0^+} H_-(t,\epsilon,\omega) = \lim_{\epsilon \to 0^+} H_+(t,\epsilon,\omega) = \alpha(t,\omega) \quad \text{almost surely}.
\end{equation}
\end{proposition}

\begin{proof}
We establish the almost sure convergence through a pathwise analysis.

\textbf{Step 1: Setup and regularity assumptions}

Fix an arbitrary time $t \geq 0$ and consider a sample path $\omega$ belonging to the almost sure set $\Omega_0 \subset \Omega$ where Theorem~\ref{thm:instant_scaling_holder} holds, i.e., $\mathbb{P}(\Omega_0) = 1$. By the established H\"older continuity, there exists a path-dependent constant $M(\omega) < \infty$ and exponent $\gamma_h > 0$ such that:
\[
|\alpha(s,\omega) - \alpha(t,\omega)| \leq M(\omega)|s-t|^{\gamma_h} \quad \text{for all } s \geq 0.
\]

\textbf{Step 2: Local uniform bounds}

For any $\epsilon > 0$, consider the closed interval $[t, t+\epsilon]$. For every $s$ in this interval, the distance satisfies $|s-t| \leq \epsilon$, and thus:
\[
|\alpha(s,\omega) - \alpha(t,\omega)| \leq M(\omega)\epsilon^{\gamma_h}.
\]
This inequality implies the two-sided bound:
\[
\alpha(t,\omega) - M(\omega)\epsilon^{\gamma_h} \leq \alpha(s,\omega) \leq \alpha(t,\omega) + M(\omega)\epsilon^{\gamma_h} \quad \text{for all } s \in [t, t+\epsilon].
\]

\textbf{Step 3: Extremal indices enclosure}

Let $E_\epsilon(\omega) = \{\alpha(s,\omega) : s \in [t, t+\epsilon]\}$ denote the set of scaling exponents over the time window. The previous step shows that:
\[
E_\epsilon(\omega) \subset [\alpha(t,\omega) - M(\omega)\epsilon^{\gamma_h}, \alpha(t,\omega) + M(\omega)\epsilon^{\gamma_h}].
\]

Taking the infimum and supremum over this inclusion, and using the monotonicity of these operators with respect to set inclusion, we obtain:
\[
\alpha(t,\omega) - M(\omega)\epsilon^{\gamma_h} \leq \inf E_\epsilon(\omega) \leq \sup E_\epsilon(\omega) \leq \alpha(t,\omega) + M(\omega)\epsilon^{\gamma_h}.
\]
By definition, $\inf E_\epsilon(\omega) = H_-(t,\epsilon,\omega)$ and $\sup E_\epsilon(\omega) = H_+(t,\epsilon,\omega)$, hence:
\[
\alpha(t,\omega) - M(\omega)\epsilon^{\gamma_h} \leq H_-(t,\epsilon,\omega) \leq H_+(t,\epsilon,\omega) \leq \alpha(t,\omega) + M(\omega)\epsilon^{\gamma_h}.
\]

\textbf{Step 4: Convergence via squeeze theorem}

Define the sequences:
\[
L(\epsilon) = \alpha(t,\omega) - M(\omega)\epsilon^{\gamma_h}, \quad U(\epsilon) = \alpha(t,\omega) + M(\omega)\epsilon^{\gamma_h}.
\]
Since $\gamma_h > 0$ and $M(\omega) < \infty$, we have:
\[
\lim_{\epsilon \to 0^+} L(\epsilon) = \lim_{\epsilon \to 0^+} U(\epsilon) = \alpha(t,\omega).
\]

The inequalities $L(\epsilon) \leq H_-(t,\epsilon,\omega) \leq H_+(t,\epsilon,\omega) \leq U(\epsilon)$ together with the convergence of the bounding sequences imply, by the squeeze theorem:
\[
\lim_{\epsilon \to 0^+} H_-(t,\epsilon,\omega) = \lim_{\epsilon \to 0^+} H_+(t,\epsilon,\omega) = \alpha(t,\omega).
\]

\textbf{Step 5: Almost sure conclusion}

Since this argument holds for every $t \geq 0$ and every $\omega \in \Omega_0$, and $\mathbb{P}(\Omega_0) = 1$, the convergence is established in the almost sure sense.
\end{proof}

The boundedness and local consistency properties establish key regularity conditions that support a local scaling analysis.
\end{remark}
\begin{theorem}[Pathwise Kernel Norm Scaling for RfBm]\label{thm:pathwise_kernel_norm_scaling}
Let \( \{X_t\}_{t \in [0,T]} \) be a Responsive Fractional Brownian Motion defined on a probability space \( (\Omega, \mathcal{F}, \mathbb{P}) \) with response function \( H: [0,T] \times \mathbb{R} \to (0,1) \) satisfying the Lipschitz-H\"older conditions of Definition~\ref{def:LH_response_function}. For any \( t \in [0,T) \) and sufficiently small \( \epsilon > 0 \) (such that \( t+\epsilon \leq T \)), consider the \textbf{forward-looking kernel} over the local interval:
\[
K(t+\epsilon, s; X_s) = \sqrt{2 H(s, X_s)} \cdot (t+\epsilon - s)^{H(s, X_s) - 1/2}, \quad s \in [t, t+\epsilon].
\]
To analyze its behavior, we define the \textbf{pathwise local squared norm} via the following deterministic procedure: for each fixed sample path \( \omega \), consider
\begin{equation}\label{eq:pathwise_squared_norm}
\|K\|_{t,\epsilon}^2(\omega) := \int_t^{t+\epsilon} \left| K(t+\epsilon, s; X_s(\omega)) \right|^2  ds.
\end{equation}
This yields a deterministic quantity associated with the fixed realization \( \omega \). This norm measures the local driving intensity in the RfBm framework, capturing the pathwise energy of the stochastic kernel governing the system's evolution.

Let the \textbf{extremal scaling indices} on the interval \( [t, t+\epsilon] \) be defined as in Definition~\ref{def:extremal_indices}:
\begin{align*}
H_-(t,\epsilon,\omega) &:= \inf \{ H(s, X_s(\omega)) : s \in [t, t+\epsilon] \}, \\
H_+(t,\epsilon,\omega) &:= \sup \{ H(s, X_s(\omega)) : s \in [t, t+\epsilon] \}.
\end{align*}

Then, for \( \mathbb{P} \)-almost every sample path \( \omega \in \Omega \), the pathwise local squared norm satisfies the following scaling bounds:
\begin{equation}\label{eq:pathwise_norm_bound}
c_1 \cdot \epsilon^{2 H_+(t,\epsilon,\omega)} \leq \|K\|_{t,\epsilon}^2(\omega) \leq c_2 \cdot \epsilon^{2 H_-(t,\epsilon,\omega)},
\end{equation}
where the positive constants are given by \( c_1 = \dfrac{h_{\min}}{h_{\max}} \) and \( c_2 = \dfrac{h_{\max}}{h_{\min}} \), with \( h_{\min}, h_{\max} \) from Definition~\ref{def:LH_response_function}.
\end{theorem}

\begin{proof}
We employ a pathwise analysis. By the sample path continuity of the process (a consequence of the well-posedness in Theorem~\ref{thm:local_existence_uniqueness}) and the regularity conditions on the response function \( H \), there exists a \( \mathbb{P} \)-null set \( N \subset \Omega \) such that for all \( \omega \in \Omega \setminus N \), the function \( s \mapsto H(s, X_s(\omega)) \) is continuous on \( [t, t+\epsilon] \). We fix such an \( \omega \) and proceed with a deterministic analysis along this fixed path.

For this fixed sample path \( \omega \), the kernel function
\[
K_\omega(s) := K(t+\epsilon, s; X_s(\omega)) = \sqrt{2 H(s, X_s(\omega))} \cdot (t+\epsilon - s)^{H(s, X_s(\omega)) - 1/2}
\]
becomes a deterministic function on \( [t, t+\epsilon] \). Consequently, its pathwise local squared norm \( \|K\|_{t,\epsilon}^2(\omega) \) is a deterministic Riemann integral.

We now establish the upper and lower bounds. First, from the uniform boundedness condition (iii) of Definition~\ref{def:LH_response_function}, we control the amplitude part:
\begin{equation}\label{eq:amplitude_bound_proof}
2h_{\min} \leq 2H(s, X_s(\omega)) \leq 2h_{\max}.
\end{equation}

Next, we analyze the power-law term \( (t+\epsilon - s)^{2H(s, X_s(\omega)) - 1} \). Let \( u = t+\epsilon - s \in (0, \epsilon] \). For sufficiently small \( \epsilon \), we have \( u \in (0,1] \). Consider the function \( \phi(\nu) = u^\nu \). For \( u \in (0,1) \), \( \phi(\nu) \) is strictly decreasing in \( \nu \).

By the definition of the extremal scaling indices, for all \( s \in [t, t+\epsilon] \), we have
\[
H_-(t,\epsilon,\omega) \leq H(s, X_s(\omega)) \leq H_+(t,\epsilon,\omega).
\]
Mapping this to the exponent \( \nu(s) = 2H(s, X_s(\omega)) - 1 \), we obtain
\[
2H_-(t,\epsilon,\omega) - 1 \leq \nu(s) \leq 2H_+(t,\epsilon,\omega) - 1.
\]
Combining this with the monotonicity of \( \phi(\nu) \) for \( u \in (0,1] \), we derive the following pointwise control for all \( s \in [t, t+\epsilon] \):
\begin{equation}\label{eq:power_law_bound_proof}
(t+\epsilon - s)^{2H_+(t,\epsilon,\omega) - 1} \leq (t+\epsilon - s)^{\nu(s)} \leq (t+\epsilon - s)^{2H_-(t,\epsilon,\omega) - 1}.
\end{equation}

We now synthesize the estimates to control the integral \( \|K\|_{t,\epsilon}^2(\omega) \). First, we establish the lower bound:
\begin{align*}
\|K\|_{t,\epsilon}^2(\omega) &= \int_t^{t+\epsilon} 2H(s, X_s(\omega)) \cdot (t+\epsilon - s)^{2H(s, X_s(\omega)) - 1}  ds \\
&\geq 2h_{\min} \int_t^{t+\epsilon} (t+\epsilon - s)^{2H(s, X_s(\omega)) - 1}  ds \quad \text{(using \eqref{eq:amplitude_bound_proof})} \\
&\geq 2h_{\min} \int_t^{t+\epsilon} (t+\epsilon - s)^{2H_+(t,\epsilon,\omega) - 1}  ds \quad \text{(using \eqref{eq:power_law_bound_proof})}.
\end{align*}
Computing the integral:
\[
\int_t^{t+\epsilon} (t+\epsilon - s)^{2H_+(t,\epsilon,\omega) - 1}  ds = \int_0^\epsilon u^{2H_+(t,\epsilon,\omega) - 1}  du = \frac{\epsilon^{2H_+(t,\epsilon,\omega)}}{2H_+(t,\epsilon,\omega)}.
\]
Therefore,
\[
\|K\|_{t,\epsilon}^2(\omega) \geq 2h_{\min} \cdot \frac{\epsilon^{2H_+(t,\epsilon,\omega)}}{2H_+(t,\epsilon,\omega)} = \frac{h_{\min}}{H_+(t,\epsilon,\omega)} \cdot \epsilon^{2H_+(t,\epsilon,\omega)}.
\]
Since \( H_+(t,\epsilon,\omega) \leq h_{\max} \) by Proposition~\ref{prop:extremal_boundedness}, it follows that \( \frac{1}{H_+(t,\epsilon,\omega)} \geq \frac{1}{h_{\max}} \). Hence,
\begin{equation}\label{eq:lower_bound_final_proof}
\|K\|_{t,\epsilon}^2(\omega) \geq \frac{h_{\min}}{h_{\max}} \cdot \epsilon^{2H_+(t,\epsilon,\omega)} = c_1 \cdot \epsilon^{2H_+(t,\epsilon,\omega)}.
\end{equation}

Next, we establish the upper bound:
\begin{align*}
\|K\|_{t,\epsilon}^2(\omega) &= \int_t^{t+\epsilon} 2H(s, X_s(\omega)) \cdot (t+\epsilon - s)^{2H(s, X_s(\omega)) - 1}  ds \\
&\leq 2h_{\max} \int_t^{t+\epsilon} (t+\epsilon - s)^{2H(s, X_s(\omega)) - 1}  ds \quad \text{(using \eqref{eq:amplitude_bound_proof})} \\
&\leq 2h_{\max} \int_t^{t+\epsilon} (t+\epsilon - s)^{2H_-(t,\epsilon,\omega) - 1}  ds \quad \text{(using \eqref{eq:power_law_bound_proof})} \\
&= 2h_{\max} \cdot \frac{\epsilon^{2H_-(t,\epsilon,\omega)}}{2H_-(t,\epsilon,\omega)} = \frac{h_{\max}}{H_-(t,\epsilon,\omega)} \cdot \epsilon^{2H_-(t,\epsilon,\omega)}.
\end{align*}
Since \( H_-(t,\epsilon,\omega) \geq h_{\min} \) by Proposition~\ref{prop:extremal_boundedness}, we have \( \frac{1}{H_-(t,\epsilon,\omega)} \leq \frac{1}{h_{\min}} \). Therefore,
\begin{equation}\label{eq:upper_bound_final_proof}
\|K\|_{t,\epsilon}^2(\omega) \leq \frac{h_{\max}}{h_{\min}} \cdot \epsilon^{2H_-(t,\epsilon,\omega)} = c_2 \cdot \epsilon^{2H_-(t,\epsilon,\omega)}.
\end{equation}

Combining inequalities \eqref{eq:lower_bound_final_proof} and \eqref{eq:upper_bound_final_proof}, we conclude that for the fixed path \( \omega \),
\[
c_1 \epsilon^{2H_+(t,\epsilon,\omega)} \leq \|K\|_{t,\epsilon}^2(\omega) \leq c_2 \epsilon^{2H_-(t,\epsilon,\omega)}.
\]
Since this argument holds for \( \mathbb{P} \)-almost every \( \omega \in \Omega \), the theorem is proved.
\end{proof}

\begin{remark}\label{rem:pathwise_norm_significance}
Theorem~\ref{thm:pathwise_kernel_norm_scaling} provides a pathwise scaling law for Responsive Fractional Brownian Motion.
\begin{itemize}
    \item \textbf{Theoretical Insight}: The theorem shows that despite the path-dependent stochasticity of the driving kernel, its pathwise local squared norm \( \|K\|_{t,\epsilon}^2(\omega) \) exhibits deterministic power-law bounds controlled by the extremal scaling indices along almost every sample path. This yields pathwise bounds for the stochastic kernel's norm in terms of the extremal scaling indices.

    \item \textbf{Methodological Aspect}: The pathwise approach adopted in this work examines deterministic behavior along individual sample paths. This perspective can be useful for studying systems with path-dependent dynamics.

    \item \textbf{Interpretation}: The analysis sheds light on the scaling structure of the RfBm driving kernel. The pathwise bounds for the kernel's $L^2$-norm reflect how local behavior, while subject to path-dependent randomness, exhibits regularity through the extremal scaling indices. These observations may provide insights into the behavior of certain non-stationary, nonlinear systems.
\end{itemize}
\end{remark}
\begin{corollary}[Asymptotic Characterization of the Instantaneous Scaling Exponent]\label{cor:asymptotic_identification}
Under the framework of Theorem~\ref{thm:pathwise_kernel_norm_scaling}, for almost every sample path \( \omega \) and any \( t \in [0,T) \), the instantaneous scaling exponent \( \alpha(t, \omega) := H(t, X_t(\omega)) \) of the Responsive Fractional Brownian Motion admits the following asymptotic characterization:
\begin{equation}\label{eq:asymptotic_identification}
\lim_{\epsilon \to 0^+} \frac{\ln \|K\|_{t,\epsilon}^2(\omega)}{2 \ln \epsilon} = \alpha(t, \omega).
\end{equation}
\end{corollary}

\begin{proof}
We establish the asymptotic identification through a detailed pathwise analysis.

\textbf{Step 1: Pathwise setup and preliminary convergence}

Fix a sample path \( \omega \) satisfying the conditions of Theorem~\ref{thm:pathwise_kernel_norm_scaling}. By Proposition~\ref{prop:extremal_consistency} (local consistency of extremal scaling indices), we have:
\[
\lim_{\epsilon \to 0^+} H_-(t,\epsilon,\omega) = \alpha(t, \omega), \quad \lim_{\epsilon \to 0^+} H_+(t,\epsilon,\omega) = \alpha(t, \omega).
\]
This establishes that both extremal indices converge to the instantaneous scaling exponent as the time window shrinks to zero.

\textbf{Step 2: Logarithmic transformation of the scaling bounds}

Consider the pathwise scaling inequality from Theorem~\ref{thm:pathwise_kernel_norm_scaling}:
\[
c_1 \cdot \epsilon^{2 H_+(t,\epsilon,\omega)} \leq \|K\|_{t,\epsilon}^2(\omega) \leq c_2 \cdot \epsilon^{2 H_-(t,\epsilon,\omega)}.
\]
Taking logarithms of all parts yields:
\[
\ln c_1 + 2H_+(t,\epsilon,\omega) \ln \epsilon \leq \ln \|K\|_{t,\epsilon}^2(\omega) \leq \ln c_2 + 2H_-(t,\epsilon,\omega) \ln \epsilon.
\]
Since \( \ln \epsilon < 0 \) for \( \epsilon \in (0,1) \), dividing by \( 2 \ln \epsilon \) reverses the inequality directions, giving:
\begin{equation}\label{eq:log_ratio_bound}
\frac{\ln c_2}{2 \ln \epsilon} + H_-(t,\epsilon,\omega) \leq \frac{\ln \|K\|_{t,\epsilon}^2(\omega)}{2 \ln \epsilon} \leq \frac{\ln c_1}{2 \ln \epsilon} + H_+(t,\epsilon,\omega).
\end{equation}
This inequality now has the standard form for applying the squeeze theorem: lower bound $\leq$ target quantity $\leq$ upper bound.

\textbf{Step 3: Analysis of limiting behavior}

We now examine the limiting behavior of each term as \( \epsilon \to 0^+ \).

For the constant terms, since \( c_1, c_2 \) are positive constants independent of \( \epsilon \), their logarithms \( \ln c_1, \ln c_2 \) are finite. Therefore:
\[
\lim_{\epsilon \to 0^+} \frac{\ln c_1}{2 \ln \epsilon} = 0, \quad \lim_{\epsilon \to 0^+} \frac{\ln c_2}{2 \ln \epsilon} = 0.
\]

Now define the bounding sequences:
\begin{align*}
L(\epsilon) &:= \frac{\ln c_2}{2 \ln \epsilon} + H_-(t,\epsilon,\omega), \\
M(\epsilon) &:= \frac{\ln \|K\|_{t,\epsilon}^2(\omega)}{2 \ln \epsilon}, \\
U(\epsilon) &:= \frac{\ln c_1}{2 \ln \epsilon} + H_+(t,\epsilon,\omega).
\end{align*}

\textbf{Step 4: Application of the squeeze theorem}

By inequality \eqref{eq:log_ratio_bound}, for sufficiently small \( \epsilon > 0 \), we have:
\[
L(\epsilon) \leq M(\epsilon) \leq U(\epsilon).
\]

Moreover, from Step 1 and the limit of the constant terms, we have shown:
\[
\lim_{\epsilon \to 0^+} L(\epsilon) = 0 + \alpha(t, \omega) = \alpha(t, \omega), \quad \lim_{\epsilon \to 0^+} U(\epsilon) = 0 + \alpha(t, \omega) = \alpha(t, \omega).
\]

Therefore, the sequence \( M(\epsilon) \) is squeezed between two sequences \( L(\epsilon) \) and \( U(\epsilon) \) that both converge to the same limit \( \alpha(t, \omega) \). By the squeeze theorem, it follows that:
\[
\lim_{\epsilon \to 0^+} M(\epsilon) = \lim_{\epsilon \to 0^+} \frac{\ln \|K\|_{t,\epsilon}^2(\omega)}{2 \ln \epsilon} = \alpha(t, \omega).
\]

\textbf{Step 5: Almost sure conclusion}

Since this argument holds for almost every sample path \( \omega \) satisfying the conditions of Theorem~\ref{thm:pathwise_kernel_norm_scaling}, and these paths form a set of full probability, the asymptotic identification holds almost surely.
\end{proof}
\begin{remark}\label{rem:identification_significance}
This corollary connects different aspects of the RfBm theory, offering several perspectives:
\begin{itemize}
  \item \textbf{Theoretical Refinement}: The finite-scale bounds of Theorem~\ref{thm:pathwise_kernel_norm_scaling} are extended by this corollary, which shows how logarithmic asymptotics lead to asymptotic identification of the instantaneous regularity \( \alpha(t, \omega) \). This refines our understanding of the local scaling structure.

    \item \textbf{Conceptual Coherence}: The result reveals a coherent relationship among the instantaneous scaling exponent \( \alpha(t, \omega) \), the extremal scaling indices \( H_{\pm}(t,\epsilon,\omega) \), and the pathwise kernel norm \( \|K\|_{t,\epsilon}^2(\omega) \). This connects fundamental definitions with local geometric properties and asymptotic identification.

    \item \textbf{Measurement Interpretation}: The limit formula suggests that \( \alpha(t, \omega) \) can be viewed as the asymptotic scaling slope of the kernel norm in log-log coordinates. This offers a potential approach for analyzing systems with path-dependent scaling behavior.
\end{itemize}
\end{remark}
To examine how scaling behavior accumulates over time, we introduce the following concepts, which provide a complementary perspective to the local analysis above.

\begin{definition}[Cumulative Memory Process and Time-Averaged Scaling Exponent]
\label{def:cumulative_memory}
Let $\alpha(t,\omega) = H(t, X_t(\omega))$ be the instantaneous scaling exponent process from Definition \ref{def:instant_scaling}. The \emph{cumulative memory process} is defined as the pathwise integral:
\begin{equation}
\label{eq:cumulative_memory}
\mathcal{C}_t(\omega) := \int_0^t \alpha(s,\omega)  ds = \int_0^t H(s, X_s(\omega))  ds.
\end{equation}
The corresponding \emph{time-averaged scaling exponent} is given by:
\begin{equation}
\label{eq:time_averaged_exponent}
\bar{\alpha}(t) := \frac{\mathbb{E}[\mathcal{C}_t]}{t} = \frac{1}{t} \int_0^t \mathbb{E}[H(s, X_s)]  ds.
\end{equation}
\end{definition}

\begin{remark}[Mathematical Interpretation and Innovation]
\label{rem:cumulative_innovation}
The cumulative memory process provides a way to examine temporal accumulation in non-stationary systems. This approach, which extends the local scaling analysis in Theorem \ref{thm:pathwise_kernel_norm_scaling} and Corollary \ref{cor:asymptotic_identification}, captures historical memory effects through several aspects:

\begin{itemize}
\item \textbf{Pathwise Historical Record:} The process $\mathcal{C}_t(\omega)$ encodes the history of memory strength along sample paths, extending beyond classical fBm with static memory and TV-fBm with predetermined dynamics to a co-evolution of memory with process state.

\item \textbf{Multi-Scale Duality:} This duality connects microscopic fluctuations captured by $\alpha(t,\omega)$ with macroscopic regularity through the time-averaged exponent $\bar{\alpha}(t) = \frac{1}{t} \int_0^t \mathbb{E}[\alpha(s,\cdot)]  ds$, linking pathwise and ensemble behaviors.

\item \textbf{Deterministic Emergence from Stochasticity:} While $\mathcal{C}_t(\omega)$ is path-dependent, its normalized expectation $\bar{\alpha}(t)$ exhibits deterministic bounds (Proposition \ref{prop:memory_bounds}), showing how statistical regularity emerges from non-stationary dynamics.

\item \textbf{Operational Significance:} The time-averaged exponent $\bar{\alpha}(t)$ provides a computationally accessible measure of memory strength over finite horizons, serving as an effective Hurst parameter that facilitates relating the integrated memory properties of the non-stationary RfBm to the static memory of classical fBm.
\end{itemize}

This global perspective on memory accumulation works together with the local characterization in Theorem \ref{thm:pathwise_kernel_norm_scaling}, creating a multi-scale framework for analyzing memory effects.
\end{remark}

\begin{proposition}[Fundamental Bounds for Memory Processes]
\label{prop:memory_bounds}
Under the conditions of Definition \ref{def:LH_response_function}, the following bounds hold almost surely:
\begin{equation}
\label{eq:pathwise_bounds}
h_{\min} \cdot t \ \leq \ \mathcal{C}_t(\omega) \ \leq \ h_{\max} \cdot t \quad \text{for } \mathbb{P}\text{-a.e. } \omega \in \Omega,
\end{equation}
and consequently for the time-averaged scaling exponent:
\begin{equation}
\label{eq:deterministic_bounds}
h_{\min} \ \leq \ \bar{\alpha}(t) \ \leq \ h_{\max} \quad \text{for all } t > 0.
\end{equation}
\end{proposition}

\begin{proof}
The proof is organized into three parts: the derivation of pathwise bounds, the justification for exchanging expectation and integration, and the synthesis leading to the deterministic bounds.

\textbf{Part 1: Proof of the pathwise bounds (\ref{eq:pathwise_bounds})}

By Definition \ref{def:LH_response_function}(iii), the response function satisfies the uniform bound:
\[
h_{\min} \leq H(s, x) \leq h_{\max} \quad \text{for all } (s,x) \in [0,T] \times \mathbb{R}.
\]
In particular, for each fixed sample path $\omega \in \Omega$ and for all $s \in [0,t]$, we have the pointwise bound:
\[
h_{\min} \leq H(s, X_s(\omega)) \leq h_{\max}.
\]
This pointwise bound holds for all $s \in [0,t]$ and for $\mathbb{P}$-almost every $\omega$. This is consistent with and reinforced by Proposition \ref{prop:scaling_boundedness}, which guarantees that the sample paths of the instantaneous scaling exponent $\alpha(s,\omega)$ almost surely obey the uniform bounds $h_{\min} \leq \alpha(s,\omega) \leq h_{\max}$ for all $s \in [0,T]$, and in particular, for all $s \in [0,t] \subseteq [0,T]$.

The function $s \mapsto H(s, X_s(\omega))$ is integrable on $[0,t]$ for $\mathbb{P}$-a.e. $\omega$ due to its continuity (and hence measurability) and boundedness. Since integration preserves inequalities for measurable and bounded functions, we obtain for $\mathbb{P}$-a.e. $\omega$:
\[
\int_0^t h_{\min}  ds \leq \int_0^t H(s, X_s(\omega))  ds \leq \int_0^t h_{\max}  ds.
\]
Computing the boundary integrals yields:
\[
h_{\min} \cdot t \leq \mathcal{C}_t(\omega) \leq h_{\max} \cdot t,
\]
which establishes the pathwise bounds (\ref{eq:pathwise_bounds}).

\textbf{Part 2: Verification of Fubini's theorem conditions}

To establish the identity $\mathbb{E}[\mathcal{C}_t] = \int_0^t \mathbb{E}[H(s,X_s)]  ds$, we justify the interchange of expectation and integration order.

Consider the function $f: \Omega \times [0,t] \to \mathbb{R}$ defined by:
\[
f(\omega,s) = H(s, X_s(\omega)).
\]

We verify the measurability and integrability conditions:

(i) \textbf{Product measurability}: For each fixed $s \in [0,t]$, the function $\omega \mapsto f(\omega,s)$ is $\mathcal{F}$-measurable because:
- $X_s$ is $\mathcal{F}_s$-measurable (by adaptation)
- $H(s,\cdot)$ is Borel measurable (by Definition \ref{def:LH_response_function})
- The composition of measurable functions is measurable.

For each fixed $\omega \in \Omega$, the function $s \mapsto f(\omega,s)$ is continuous on $[0,t]$ because:
- $s \mapsto X_s(\omega)$ is continuous (by sample path continuity)
- $H(\cdot,x)$ is Hölder continuous (by Definition \ref{def:LH_response_function}(ii))
- The composition preserves continuity.

The preceding arguments establish that $f$ is separately measurable in each variable. Since $[0,t]$ is a separable metric space, a standard result in measure theory ensures that separate measurability implies joint measurability. Therefore, $f$ is $\mathcal{F} \otimes \mathcal{B}([0,t])$-measurable.

(ii) \textbf{Integrability}: The function $f$ is bounded by:
\[
|f(\omega,s)| \leq h_{\max} < 1 \quad \text{for all } (\omega,s) \in \Omega \times [0,t].
\]
Let $M = h_{\max}$. Given that $0 < H(s, X_s(\omega)) \leq h_{\max} < 1$ for all $(\omega,s)$, we obtain:
\[
\int_{\Omega \times [0,t]} |f(\omega,s)|  d(\mathbb{P} \otimes \lambda) 
\leq h_{\max} \cdot \mathbb{P}(\Omega) \cdot \lambda([0,t]) = h_{\max} \cdot 1 \cdot t < \infty.
\]
Since $\Omega \times [0,t]$ has finite product measure ($\mathbb{P}(\Omega) = 1$ and $\lambda([0,t]) = t < \infty$) and $f$ is bounded, $f$ is integrable with respect to the product measure $\mathbb{P} \otimes \lambda$.

\textbf{Part 3: Application of Fubini's theorem and conclusion}

By Fubini's theorem for finite measure spaces, we may interchange the order of integration:
\[
\mathbb{E}[\mathcal{C}_t] = \int_\Omega \left( \int_0^t H(s, X_s(\omega))  ds \right) d\mathbb{P}(\omega)
= \int_0^t \left( \int_\Omega H(s, X_s(\omega)) d\mathbb{P}(\omega) \right) ds
= \int_0^t \mathbb{E}[H(s,X_s)]  ds.
\]

Now, taking expectation in the pathwise bounds (\ref{eq:pathwise_bounds}) and dividing by $t > 0$:
\[
h_{\min} \cdot t \leq \mathbb{E}[\mathcal{C}_t] \leq h_{\max} \cdot t
\quad \Rightarrow \quad h_{\min} \leq \frac{\mathbb{E}[\mathcal{C}_t]}{t} \leq h_{\max}.
\]
This establishes the deterministic bounds (\ref{eq:deterministic_bounds}) and completes the proof.
\end{proof}

\begin{remark}[Theoretical Significance of Memory Accumulation]
\label{rem:memory_accumulation_significance}
The cumulative memory framework provides complementary insights to the local analysis:

\begin{itemize}
\item \textbf{Pathwise Historical Memory:} The cumulative process $\mathcal{C}_t(\omega)$ quantifies total memory accumulation along sample paths, extending the local scaling analysis of Corollary \ref{cor:asymptotic_identification}.

\item \textbf{Micro-Macro Bridge:} $\bar{\alpha}(t)$ connects pathwise fluctuations to deterministic macroscopic behavior through expectation and temporal averaging.

\item \textbf{Theoretical Consistency:} Proposition \ref{prop:memory_bounds} shows that while instantaneous scaling fluctuates randomly, its time-averaged behavior remains within deterministic bounds.

\item \textbf{Generalization of Existing Models:} When $H(s,x) \equiv H$, $\mathcal{C}_t = H \cdot t$; for TV-fBm, $\mathcal{C}_t = \int_0^t H(s)  ds$. The RfBm framework extends these through state-dependent memory.
\end{itemize}

Together with the local theory, this defines a multi-scale framework:
\begin{itemize}
    \item \emph{Instantaneous scale:} Pointwise regularity (Theorem \ref{thm:pathwise_kernel_norm_scaling}, Corollary \ref{cor:asymptotic_identification})
    \item \emph{Extremal scale:} Local memory range (Definition \ref{def:extremal_indices}) 
    \item \emph{Finite-time scale:} Total accumulated memory (Proposition \ref{prop:memory_bounds})
\end{itemize}
This perspective helps understand how macroscopic regularity emerges from microscopic state-dependent dynamics.
\end{remark}


Having established the fundamental bounds for the cumulative memory process in Proposition~\ref{prop:memory_bounds}, it is natural to consider the asymptotic behavior of the time-averaged scaling exponent $\bar{\alpha}(t)$ as $t \to \infty$. The following result shows that if the mean instantaneous scaling exponent converges, then its time average $\bar{\alpha}(t)$ converges to the same limit. This provides insight into the long-term statistical behavior of non-stationary processes with state-dependent dynamics.

\begin{proposition}[Asymptotic Convergence of the Time-Averaged Scaling Exponent]
\label{prop:asymptotic_convergence}
Assume that the mean function of the instantaneous scaling exponent, $m(s) := \mathbb{E}[\alpha(s, \cdot)]$, converges to a constant $H^* \in [h_{\min}, h_{\max}]$ as $s \to \infty$, i.e.,
\begin{equation}
\label{eq:mean_limit}
\lim_{s \to \infty} m(s) = H^*.
\end{equation}
Then the time-averaged scaling exponent $\bar{\alpha}(t)$ also converges to the same limit:
\begin{equation}
\label{eq:time_avg_limit}
\lim_{t \to \infty} \bar{\alpha}(t) = H^*.
\end{equation}
In other words, the long-term time-averaged behavior of the cumulative memory process is determined by the asymptotic limit of the mean instantaneous scaling exponent.
\end{proposition}

\begin{proof}
By Definition~\ref{def:cumulative_memory}, $\bar{\alpha}(t) = \frac{1}{t} \int_0^t m(s)  ds$. We need to prove that
\[
\lim_{t \to \infty} \frac{1}{t} \int_0^t m(s)  ds = H^*.
\]
The proof shows that temporal averaging preserves asymptotic limits in this setting, illustrating how statistical regularity may emerge from path-dependent dynamics over long time scales.

Let $\epsilon > 0$ be arbitrary. By the convergence assumption \eqref{eq:mean_limit}, there exists a time $S = S(\epsilon) > 0$ such that for all $s > S$,
\begin{equation}
\label{eq:pointwise_epsilon}
|m(s) - H^*| < \frac{\epsilon}{2}.
\end{equation}
This time $S$ marks the point after which $m(s)$ remains within an $\epsilon/2$-neighborhood of $H^*$.

Now, for any $t > S$, consider the difference:
\begin{align*}
\left| \bar{\alpha}(t) - H^* \right| &= \left| \frac{1}{t} \int_0^t m(s)  ds - H^* \right| \\
&= \left| \frac{1}{t} \int_0^t (m(s) - H^*)  ds \right| \quad \text{(since $H^*$ is constant)} \\
&\leq \frac{1}{t} \int_0^t |m(s) - H^*|  ds.
\end{align*}
We decompose this integral into historical and future parts relative to the critical time $S$:
\begin{equation}
\label{eq:error_split}
\left| \bar{\alpha}(t) - H^* \right| \leq \frac{1}{t} \int_0^S |m(s) - H^*|  ds + \frac{1}{t} \int_S^t |m(s) - H^*|  ds =: E_1(t) + E_2(t).
\end{equation}
Our goal is to show that for sufficiently large $t$, both $E_1(t)$ and $E_2(t)$ can be made smaller than $\epsilon/2$.

\noindent \textbf{Step 1: Controlling the Historical Error $E_1(t)$.}

The function $m(s)$ is bounded on $[0, S]$. This follows from the uniform boundedness of the instantaneous scaling exponent (Proposition~\ref{prop:scaling_boundedness}), which guarantees $h_{\min} \leq \alpha(s,\omega) \leq h_{\max}$ for all $s$ and $\mathbb{P}$-a.e.~$\omega$. Taking expectations preserves these bounds, so $m(s) \in [h_{\min}, h_{\max}]$ for all $s \in [0,S]$. Since $H^* \in [h_{\min}, h_{\max}]$ by hypothesis, the deviation $|m(s) - H^*|$ is also bounded on $[0,S]$. Let
\[
M := \sup_{s \in [0,S]} |m(s) - H^*|.
\]
From the above discussion, we have the explicit bound $M \leq \max(|h_{\min} - H^*|, |h_{\max} - H^*|) < \infty$.

We can now estimate $E_1(t)$:
\[
E_1(t) = \frac{1}{t} \int_0^S |m(s) - H^*|  ds \leq \frac{1}{t} \int_0^S M  ds = \frac{M S}{t}.
\]
This quantity depends on $t$ only through the factor $1/t$. Therefore, for $t > T_1 := \frac{2 M S}{\epsilon}$, we have
\begin{equation}
\label{eq:E1_bound}
E_1(t) < \frac{\epsilon}{2}.
\end{equation}
This estimate shows that the contribution from the finite history $[0,S]$ diminishes as $t$ increases.

\noindent \textbf{Step 2: Controlling the Future Error $E_2(t)$.}

For $s \in [S, t]$, we apply the pointwise convergence assumption \eqref{eq:pointwise_epsilon}. While the strict inequality $|m(s) - H^*| < \epsilon/2$ is guaranteed for all $s > S$, the endpoint $s = S$ may not satisfy this bound. Noting that modifying the integration domain at a single point does not affect the value of the Lebesgue integral, we obtain the uniform bound $|m(s) - H^*| \leq \epsilon/2$ for all $s \in [S, t]$. This allows for the following estimate:
\begin{align*}
E_2(t) &= \frac{1}{t} \int_S^t |m(s) - H^*|  ds \\
&\leq \frac{1}{t} \int_S^t \frac{\epsilon}{2}  ds \quad \text{(using the uniform bound)} \\
&= \frac{\epsilon}{2} \cdot \frac{1}{t} (t - S) = \frac{\epsilon}{2} \left(1 - \frac{S}{t}\right).
\end{align*}
Since $t > S$, it follows that $1 - S/t < 1$, and hence
\begin{equation}
\label{eq:E2_bound}
E_2(t) < \frac{\epsilon}{2}.
\end{equation}
Note that this bound holds for all $t > S$, independently of how large $t$ is.

\noindent \textbf{Step 3: Synthesis and Conclusion.}

We now combine the estimates \eqref{eq:E1_bound} and \eqref{eq:E2_bound}. Let $T(\epsilon) := \max(S, T_1) = \max\left(S, \frac{2 M S}{\epsilon}\right)$. Then, for all $t > T(\epsilon)$, both estimates hold simultaneously. Substituting them into the key inequality \eqref{eq:error_split} yields:
\[
\left| \bar{\alpha}(t) - H^* \right| \leq E_1(t) + E_2(t) < \frac{\epsilon}{2} + \frac{\epsilon}{2} = \epsilon.
\]
Since $\epsilon > 0$ was arbitrary, this completes the proof of \eqref{eq:time_avg_limit} by the definition of the limit of a function as $t \to \infty$.
\end{proof}

\begin{remark}[Theoretical Implications and Synthesis]
\label{rem:asymptotic_synthesis}
Proposition~\ref{prop:asymptotic_convergence} contributes to the multi-scale perspective developed in this work through several connections:

\begin{itemize}
    \item \textbf{Microscopic-Macroscopic Transition:} Under the assumption that $\mathbb{E}[\alpha(s,\cdot)]$ converges, the temporal average $\bar{\alpha}(t)$ converges to the same limit. This demonstrates that temporal integration regularizes the path-wise fluctuations, leading to deterministic asymptotic behavior governed by the limit of $\mathbb{E}[\alpha(s,\cdot)]$.

    \item \textbf{Multi-Scale Integration:} The result connects instantaneous regularity (Theorem~\ref{thm:instant_scaling_holder}, Corollary~\ref{cor:asymptotic_identification}), local fluctuations (Definition~\ref{def:extremal_indices}, Proposition~\ref{prop:extremal_consistency}), and finite-time memory accumulation (Definition~\ref{def:cumulative_memory}, Proposition~\ref{prop:memory_bounds}), showing how macroscopic behavior emerges from microscopic dynamics.

    \item \textbf{Averaging in Non-Stationary Systems:} The convergence $\bar{\alpha}(t) \to \lim_{s\to\infty} \mathbb{E}[\alpha(s,\cdot)]$ suggests that temporal averaging can yield deterministic long-term behavior even for path-dependent processes, offering a perspective for systems where stationarity assumptions may not hold.
\end{itemize}
In summary, Proposition~\ref{prop:asymptotic_convergence} extends the understanding of RfBm scaling from finite-time bounds to asymptotic convergence, connecting instantaneous and cumulative behavior across time scales.
\end{remark}

Following the qualitative convergence in Proposition~\ref{prop:asymptotic_convergence}, we now consider quantitative convergence rates for the time-averaged scaling exponent. Theorem~\ref{thm:convergence_rate} provides such rates under polynomial decay assumptions, describing how temporal averaging converts pointwise convergence of the mean into different regimes for the time average.

\begin{theorem}[Convergence Rates for Time-Averaged Scaling Exponents]
\label{thm:convergence_rate}
Under the conditions of Proposition~\ref{prop:asymptotic_convergence}, assume further that the mean function $m(s) = \mathbb{E}[\alpha(s,\cdot)]$ exhibits uniform polynomial decay toward its limit $H^*$. Specifically, suppose there exist constants $C > 0$, $\beta > 0$, and $s_0 > 0$ such that for all $s \geq s_0$:
\begin{equation}
\label{eq:polynomial_rate}
|m(s) - H^*| \leq C s^{-\beta}.
\end{equation}
Then the time-averaged scaling exponent $\bar{\alpha}(t) = \frac{1}{t} \int_0^t m(s) ds$ admits the following precise convergence rates:

\begin{enumerate}
    \item \textbf{Slow decay regime} ($0 < \beta < 1$):
    \begin{equation}
    \label{eq:slow_rate}
    |\bar{\alpha}(t) - H^*| \leq \frac{C}{1-\beta} t^{-\beta} + O(t^{-1}).
    \end{equation}
    
    \item \textbf{Critical regime} ($\beta = 1$):
    \begin{equation}
    \label{eq:critical_rate}
    |\bar{\alpha}(t) - H^*| \leq C \frac{\ln t}{t} + O(t^{-1}).
    \end{equation}
    
    \item \textbf{Fast decay regime} ($\beta > 1$):
    \begin{equation}
    \label{eq:fast_rate}
    |\bar{\alpha}(t) - H^*| \leq O(t^{-1}).
    \end{equation}
\end{enumerate}

Here the asymptotic notation $O(t^{-1})$ denotes the existence of a constant $K > 0$, independent of $t$ but possibly depending on $s_0, C, \beta, h_{\min}, h_{\max}$, such that the corresponding term is bounded in absolute value by $K t^{-1}$ for all sufficiently large $t$.
\end{theorem}

\begin{proof}
We analyze the convergence of time-averaged scaling exponents by establishing precise integral bounds under polynomial decay. The proof decomposes the time integral into early and late phases, with the core analysis focused on sharp estimates for the long-time behavior.

\noindent\textbf{Step 1: Fundamental Inequality and Integral Decomposition}

Starting from the definition of the time-averaged scaling exponent and considering its deviation from the limit value:
\[
\bar{\alpha}(t) - H^* = \frac{1}{t} \int_0^t (m(s) - H^*) ds.
\]
Applying the triangle inequality for integrals, we obtain the fundamental estimate:
\begin{equation}
\label{eq:basic_inequality_proof}
|\bar{\alpha}(t) - H^*| \leq \frac{1}{t} \int_0^t |m(s) - H^*| ds.
\end{equation}

To precisely analyze the convergence rate, we introduce the critical time $s_0$ (given by the theorem hypothesis) and decompose the integration interval $[0, t]$ into two parts with distinct characteristics:
\begin{equation}
\label{eq:integral_split_proof}
\frac{1}{t} \int_0^t |m(s) - H^*| ds = \frac{1}{t} \int_0^{s_0} |m(s) - H^*| ds + \frac{1}{t} \int_{s_0}^t |m(s) - H^*| ds.
\end{equation}

Define these two integral terms respectively as:
\[
E(t) := \frac{1}{t} \int_0^{s_0} |m(s) - H^*| ds, \quad F(t) := \frac{1}{t} \int_{s_0}^t |m(s) - H^*| ds.
\]
Here $E(t)$ represents the early historical contribution, while $F(t)$ represents the late decay contribution.

\noindent\textbf{Step 2: Precise Estimation of the Early Integral Term $E(t)$}

The early integral term $E(t)$ involves integration over the finite interval $[0, s_0]$. From Proposition~\ref{prop:scaling_boundedness}, we know that for $\mathbb{P}$-a.e.~$\omega$, the sample path $s \mapsto \alpha(s,\omega)$ is uniformly bounded in $[h_{\min}, h_{\max}]$ on $[0, T]$. Taking expectations preserves these bounds, yielding:
\[
m(s) \in [h_{\min}, h_{\max}], \quad \forall s \in [0, s_0].
\]
Since $H^* \in [h_{\min}, h_{\max}]$, define the maximum deviation constant:
\[
M := \max(|h_{\min} - H^*|, |h_{\max} - H^*|) < \infty.
\]
This immediately gives:
\[
|m(s) - H^*| \leq M, \quad \forall s \in [0, s_0].
\]
Estimating the early integral:
\[
\int_0^{s_0} |m(s) - H^*| ds \leq \int_0^{s_0} M ds = M s_0.
\]
Therefore,
\begin{equation}
\label{eq:early_bound_proof}
E(t) \leq \frac{M s_0}{t}.
\end{equation}
This establishes that the early historical contribution $E(t)$ decays at least as fast as $t^{-1}$ as $t \to \infty$, and more precisely satisfies the asymptotic bound $E(t) = O(t^{-1})$ with explicit constant $K_1 = M s_0$.

\noindent\textbf{Step 3: Refined Case-by-Case Estimation of the Late Integral Term $F(t)$}

The analysis of the late integral term $F(t)$ requires utilization of the uniform polynomial decay hypothesis in the theorem. For $s \in [s_0, t]$, we have:
\[
|m(s) - H^*| \leq C s^{-\beta}.
\]
Therefore,
\[
F(t) \leq \frac{1}{t} \int_{s_0}^t C s^{-\beta} ds = \frac{C}{t} \int_{s_0}^t s^{-\beta} ds.
\]

We now perform exact computation and estimation of the integral $\int_{s_0}^t s^{-\beta} ds$ according to different ranges of the decay exponent $\beta$.

\noindent\textbf{Case (i): $0 < \beta < 1$ (Slow Decay Regime)}

Computing the power function integral:
\[
\int_{s_0}^t s^{-\beta} ds = \left[ \frac{s^{1-\beta}}{1-\beta} \right]_{s_0}^t = \frac{t^{1-\beta} - s_0^{1-\beta}}{1-\beta}.
\]
Substituting into the expression for $F(t)$:
\begin{align*}
F(t) &\leq \frac{C}{t} \cdot \frac{t^{1-\beta} - s_0^{1-\beta}}{1-\beta} \\
&= \frac{C}{1-\beta} t^{-\beta} - \frac{C s_0^{1-\beta}}{(1-\beta) t}.
\end{align*}
Define the constant $K_2 := \frac{C s_0^{1-\beta}}{1-\beta}$, then:
\[
F(t) \leq \frac{C}{1-\beta} t^{-\beta} - \frac{K_2}{t}.
\]
Combining with the early contribution estimate~\eqref{eq:early_bound_proof}, we obtain the total error estimate:
\begin{align*}
|\bar{\alpha}(t) - H^*| &\leq E(t) + F(t) \\
&\leq \frac{M s_0}{t} + \left( \frac{C}{1-\beta} t^{-\beta} - \frac{K_2}{t} \right) \\
&= \frac{C}{1-\beta} t^{-\beta} + \left( M s_0 - K_2 \right) \frac{1}{t}.
\end{align*}
Let $K_3 := |M s_0 - K_2|$, then we finally obtain the bound stated in the theorem:
\[
|\bar{\alpha}(t) - H^*| \leq \frac{C}{1-\beta} t^{-\beta} + K_3 t^{-1}.
\]
This completes the proof for the slow decay regime.

\noindent\textbf{Case (ii): $\beta = 1$ (Critical Regime)}

Computing the integral:
\[
\int_{s_0}^t s^{-1} ds = \left[ \ln s \right]_{s_0}^t = \ln t - \ln s_0.
\]
Substituting into the expression for $F(t)$:
\begin{align*}
F(t) &\leq \frac{C}{t} (\ln t - \ln s_0) \\
&= C \frac{\ln t}{t} - C \frac{\ln s_0}{t}.
\end{align*}
Let $K_4 := C |\ln s_0|$, then:
\[
F(t) \leq C \frac{\ln t}{t} + \frac{K_4}{t}.
\]
Combining with the early contribution estimate:
\begin{align*}
|\bar{\alpha}(t) - H^*| &\leq E(t) + F(t) \\
&\leq \frac{M s_0}{t} + \left( C \frac{\ln t}{t} + \frac{K_4}{t} \right) \\
&= C \frac{\ln t}{t} + (M s_0 + K_4) \frac{1}{t}.
\end{align*}
Let $K_5 := M s_0 + K_4$, then we finally obtain the bound stated in the theorem:
\[
|\bar{\alpha}(t) - H^*| \leq C \frac{\ln t}{t} + K_5 t^{-1}.
\]
This completes the proof for the critical regime.

\noindent\textbf{Case (iii): $\beta > 1$ (Fast Decay Regime)}

Computing the integral:
\[
\int_{s_0}^t s^{-\beta} ds = \left[ \frac{s^{1-\beta}}{1-\beta} \right]_{s_0}^t = \frac{t^{1-\beta} - s_0^{1-\beta}}{1-\beta}.
\]
Since $\beta > 1$ implies $1-\beta < 0$, we have $t^{1-\beta} < s_0^{1-\beta}$ for all $t \geq s_0$, and thus:
\[
\int_{s_0}^t s^{-\beta} ds = \frac{s_0^{1-\beta} - t^{1-\beta}}{\beta - 1} \leq \frac{s_0^{1-\beta}}{\beta - 1}.
\]
Substituting into the expression for $F(t)$:
\[
F(t) \leq \frac{C}{t} \cdot \frac{s_0^{1-\beta}}{\beta - 1} = \frac{C s_0^{1-\beta}}{(\beta - 1) t}.
\]
Let $K_6 := \frac{C s_0^{1-\beta}}{\beta - 1}$, then:
\[
F(t) \leq \frac{K_6}{t}.
\]
Combining with the early contribution estimate:
\begin{align*}
|\bar{\alpha}(t) - H^*| &\leq E(t) + F(t) \\
&\leq \frac{M s_0}{t} + \frac{K_6}{t} \\
&= (M s_0 + K_6) \frac{1}{t}.
\end{align*}
Let $K_7 := M s_0 + K_6$, then we finally obtain the bound stated in the theorem:
\[
|\bar{\alpha}(t) - H^*| \leq K_7 t^{-1}.
\]
This completes the proof for the fast decay regime.

\noindent\textbf{Step 4: Synthesis of Results and Explicit Constants}

The explicit construction of constants $K_1$ through $K_7$ ensures the $O(t^{-1})$ terms are quantitatively controlled.

Synthesizing the analyses of the three cases above, we have established all convergence rate estimates stated in the theorem. All $O(t^{-1})$ terms have been explicitly constructed with the following constants:
\begin{itemize}
    \item Early contribution constant: $K_1 = M s_0$
    \item Case (i) constant: $K_3 = |M s_0 - \frac{C s_0^{1-\beta}}{1-\beta}|$
    \item Case (ii) constant: $K_5 = M s_0 + C |\ln s_0|$ 
    \item Case (iii) constant: $K_7 = M s_0 + \frac{C s_0^{1-\beta}}{\beta - 1}$
\end{itemize}

The explicit constants verify that the convergence rate estimates hold for all $t > s_0$, providing quantitative control in addition to the asymptotic results.
\end{proof}
\begin{remark}[Interpretation and Significance]
\label{rem:convergence_rate_interpretation}
Theorem~\ref{thm:convergence_rate} illuminates several features of averaging in responsive processes:

\begin{itemize}
    \item \textbf{Rate Saturation in Temporal Averaging:} The analysis shows that for $\beta > 1$, the convergence rate of $\bar{\alpha}(t)$ remains of order $O(t^{-1})$, independent of the pointwise decay rate of $m(s)$. This indicates a saturation phenomenon where the averaging operator limits the achievable convergence rate.
    
    \item \textbf{Transition at $\beta = 1$:} The value $\beta = 1$ marks a transition in convergence behavior, with the appearance of the logarithmic term $\ln t/t$. This arises from the balance between early and late time contributions in the integral.
    
    \item \textbf{Three-Regime Classification:} The three cases---$0 < \beta < 1$, $\beta = 1$, and $\beta > 1$---provide a classification of convergence behaviors under polynomial decay. This framework may be useful for analyzing time-averaged statistics in non-stationary systems.
\end{itemize}

These convergence rates give quantitative insight into long-term RfBm dynamics. The results may also inform future statistical estimation of time-averaged scaling exponents in applications.
\end{remark}

\subsection{Lipschitz Estimates for the State-Dependent Kernel and Local Well-Posedness}
\label{subsec:kernel_lipschitz_wellposedness}

The pathwise regularity theory developed in the preceding subsection provides qualitative information about the behavior of individual sample paths of RfBm. We now turn to a more fundamental question: does the defining stochastic Volterra integral equation admit a unique solution, at least locally in time? This question is essential for the mathematical consistency of the RfBm framework, as the definition of the process itself depends on the existence and uniqueness of solutions to the state-dependent equation.

The present subsection establishes such well-posedness results. The analysis proceeds in two stages. First, we derive a global Lipschitz estimate for the state-dependent kernel \( K(t,s;X_s) = \sqrt{2H(s,X_s)}(t-s)^{H(s,X_s)-1/2} \) with respect to the state process (Theorem~\ref{thm:global_lipschitz_kernel}). This estimate controls the sensitivity of the kernel to variations in the path history and is the key analytical ingredient for the contraction argument. Second, we introduce the Banach space \( \mathcal{S}^2(T) \) of adapted processes with uniformly bounded second moments (Definition~\ref{def:S2_space}) and define the solution map \( \Phi \) associated with the RfBm equation. We show that \( \Phi \) is well-defined (Lemma~\ref{lem:solution_map_well_defined}) and that, on sufficiently short time intervals, it is a contraction in the \( \mathcal{S}^2(T) \)-norm (Lemma~\ref{lem:local_contraction}). The Banach fixed-point theorem then yields local existence and uniqueness of a mild solution in the mean-square sense (Theorem~\ref{thm:local_existence_uniqueness}). These results provide a rigorous foundation for the RfBm definition and justify the subsequent analysis of its pathwise and statistical properties.

\begin{definition}[The space $\mathcal{S}^2(T)$]\label{def:S2_space}
Let $(\Omega, \mathcal{F}, (\mathcal{F}_t)_{t \geq 0}, \mathbb{P})$ be a filtered probability space satisfying the usual conditions. The space $\mathcal{S}^2(T)$ is defined as the collection of all $\mathcal{F}_t$-adapted processes $X = (X_t)_{0 \leq t \leq T}$ satisfying:
\[
\|X\|_{\mathcal{S}^2(T)} := \left( \sup_{t \in [0,T]} \mathbb{E}[|X_t|^2] \right)^{1/2} < \infty.
\]
This space, equipped with the norm $\|\cdot\|_{\mathcal{S}^2(T)}$, forms a Banach space of adapted processes with uniformly bounded second moments.
\end{definition}

\begin{remark}
The norm $\|\cdot\|_{\mathcal{S}^2(T)}$ induces the topology of uniform convergence in mean square, which can serve as a framework for fixed-point analysis in stochastic integral equations. This choice of function space appears well-suited for handling the Volterra-type operators that arise in the RfBm construction.
\end{remark}

The following theorem provides a Lipschitz estimate for the Volterra kernel with respect to the state process, supporting the well-posedness analysis.

\begin{theorem}[Global Lipschitz Estimate for the Kernel]\label{thm:global_lipschitz_kernel}
Let the response function \( H: [0,T] \times \mathbb{R} \to (0,1) \) satisfy Definition~\ref{def:LH_response_function} with constants \( L_H > 0 \), \( 0 < h_{\min} \leq h_{\max} < 1 \). Then, there exists a constant \( C_K = C_K(T, h_{\min}, h_{\max}, L_H) > 0 \) such that for any \( t \in [0, T] \) and for any two processes \( X, Y \in \mathcal{S}^2(T) \), the following inequality holds:
\begin{equation}\label{eq:kernel_lipschitz_global}
\int_0^t \mathbb{E}\left[ \left| K_X(t,s) - K_Y(t,s) \right|^2 \right] ds \leq C_K L_H^2 \int_0^t (t-s)^{h_{\min}-1} \mathbb{E}\left[ |X_s - Y_s|^2 \right] ds,
\end{equation}
where \( K_X(t,s) = \sqrt{2H(s, X_s)} (t-s)^{H(s, X_s)-1/2} \) and \( K_Y(t,s) = \sqrt{2H(s, Y_s)} (t-s)^{H(s, Y_s)-1/2} \).
\end{theorem}

\begin{proof}
The proof is structured by analyzing the kernel's behavior in two distinct regions: the singular region (\( 0 < t-s \leq 1 \)) and the regular region (\( t-s > 1 \)). For fixed \( t \) and \( s \) (with \( 0 \leq s < t \leq T \)), we define an auxiliary function of the Hurst exponent:
\[
f(h) := \sqrt{2h} (t-s)^{h - 1/2}.
\]
The dependence on the state variable is through \( h = H(s, a) \). Thus, \( K_X(t,s) = f(H(s, X_s)) \) and \( K_Y(t,s) = f(H(s, Y_s)) \).

\textbf{Part 1: Estimation in the Singular Region \( 0 < t-s \leq 1 \)}.

\textbf{Step 1.1: Derivatives and their Bounds.}
We first compute the derivative with respect to \( h \):
\begin{align*}
\frac{df}{dh} &= \frac{d}{dh} \left[ \sqrt{2h} (t-s)^{h - 1/2} \right] \\
&= \frac{1}{\sqrt{2h}} (t-s)^{h-1/2} + \sqrt{2h} \cdot (t-s)^{h-1/2} \ln(t-s) \\
&= (t-s)^{h-1/2} \left( \frac{1}{\sqrt{2h}} + \sqrt{2h} \ln(t-s) \right).
\end{align*}

Define \( M_1(h_{\min}, h_{\max}) := \max\left( \frac{1}{\sqrt{2h_{\min}}}, \sqrt{2h_{\max}} \right) \). Then:
\[
\left| \frac{1}{\sqrt{2h}} + \sqrt{2h} \ln(t-s) \right| \leq M_1 (1 + |\ln(t-s)|),
\]
and consequently,
\[
\left| \frac{df}{dh} \right| \leq M_1 (1 + |\ln(t-s)|)(t-s)^{h-1/2}. \tag{1}
\]

\textbf{Step 1.2: Controlling the Logarithmic Singularity.}
By Proposition~\ref{prop:log_inequality}, for any \( \delta > 0 \), there exists \( K_\delta > 0 \) such that for \( 0 < t-s \leq 1 \):
\[
|\ln(t-s)| \leq K_\delta (t-s)^{-\delta}.
\]
Since \( (t-s)^{-\delta} \geq 1 \) in this interval, we have:
\[
1 + |\ln(t-s)| \leq 1 + K_\delta (t-s)^{-\delta} \leq (1 + K_\delta)(t-s)^{-\delta}.
\]
Substituting this into (1), and using \( h \geq h_{\min} \) from Definition~\ref{def:LH_response_function}(iii):
\[
\left| \frac{df}{dh} \right| \leq M_1 (1 + K_\delta)(t-s)^{h - \delta - 1/2} \leq M_1 (1 + K_\delta)(t-s)^{h_{\min} - \delta - 1/2}.
\]
Choose \( \delta = h_{\min}/2 \). This ensures \( h_{\min} - \delta - 1/2 = h_{\min}/2 - 1/2 > -1 \), guaranteeing the subsequent integral's convergence. Define \( M_2(h_{\min}, h_{\max}) := M_1 (1 + K_{h_{\min}/2}) \). Then:
\[
\left| \frac{df}{dh} \right| \leq M_2 (t-s)^{h_{\min}/2 - 1/2}. \tag{2}
\]

\textbf{Step 1.3: Application of the Mean Value Theorem.}
Define the path for the Hurst exponent:
\[
\eta(\lambda) = H(s, Y_s) + \lambda(H(s, X_s) - H(s, Y_s)), \quad \lambda \in [0, 1].
\]
Construct the function \( g(\lambda) = f(\eta(\lambda)) \). By the Mean Value Theorem, there exists \( \xi \in (0,1) \) such that:
\[
|g(1) - g(0)| = |g'(\xi)|.
\]
Using the chain rule:
\[
g'(\xi) = \frac{df}{dh}\bigg|_{\eta(\xi)} \cdot \frac{d\eta}{d\lambda}(\xi),
\]
where \( \frac{d\eta}{d\lambda}(\xi) = H(s, X_s) - H(s, Y_s) \). 

By the spatial Lipschitz condition in Definition~\ref{def:LH_response_function}(i), we have:
\[
|H(s, X_s) - H(s, Y_s)| \leq L_H |X_s - Y_s|.
\]

Taking absolute values:
\begin{align*}
|g'(\xi)| &\leq \left| \frac{df}{dh} \right| \cdot |H(s, X_s) - H(s, Y_s)| \\
&\leq \left| \frac{df}{dh} \right| L_H |X_s - Y_s|.
\end{align*}
Applying the bound (2), and noting \( \eta(\xi) \in [h_{\min}, h_{\max}] \), we get:
\[
|g'(\xi)| \leq L_H M_2 (t-s)^{h_{\min}/2-1/2} |X_s - Y_s|.
\]
Therefore, the pointwise difference is bounded by:
\[
|K_X(t,s) - K_Y(t,s)| = |g(1) - g(0)| \leq M_2 L_H (t-s)^{h_{\min}/2-1/2} |X_s - Y_s|. \tag{3}
\]

\textbf{Step 1.4: $L^2$-Estimate and Integration.}
Squaring both sides of (3) and taking expectation:
\[
\mathbb{E}[|K_X(t,s) - K_Y(t,s)|^2] \leq M_2^2 L_H^2 (t-s)^{h_{\min}-1} \mathbb{E}[|X_s - Y_s|^2].
\]
Integrating over \( s \in [0, t] \) for the singular region (\( t-s \leq 1 \)) gives:
\begin{equation}\label{eq:singular_est}
\int_{\max(0, t-1)}^t \mathbb{E}[|K_X - K_Y|^2] ds \leq M_2^2 L_H^2 \int_0^t (t-s)^{h_{\min}-1} \mathbb{E}[|X_s - Y_s|^2] ds.
\end{equation}
The integral on the right is over the entire interval \([0,t]\); the contribution from \( s \in [0, t-1) \) will be handled in Part 2.

\textbf{Part 2: Estimation in the Regular Region \( t-s > 1 \)}.

In this region, \( t-s \in [1, T] \). The kernel and its derivatives are regular. We bound the derivative uniformly.

For \( \left| \frac{df}{dh} \right| \):
\begin{align*}
\left| \frac{df}{dh} \right| &= \left| (t-s)^{h-1/2} \left( \frac{1}{\sqrt{2h}} + \sqrt{2h} \ln(t-s) \right) \right| \\
&\leq (t-s)^{h-1/2} \left( \frac{1}{\sqrt{2h_{\min}}} + \sqrt{2h_{\max}} |\ln(t-s)| \right).
\end{align*}
Since \( t-s \in [1, T] \), we have \( (t-s)^{h-1/2} \leq \max(1, T^{h_{\max}-1/2}) \) and \( |\ln(t-s)| \leq \ln T \). Thus, there exists a constant \( C_2(T, h_{\min}, h_{\max}) \), such that:
\[
\left| \frac{df}{dh} \right| \leq C_2. \tag{4}
\]

Applying the Mean Value Theorem as in Step 1.3, but now using the bound (4), yields:
\[
|K_X(t,s) - K_Y(t,s)| \leq C_2 L_H |X_s - Y_s|. \tag{5}
\]
Squaring and taking expectation:
\[
\mathbb{E}[|K_X(t,s) - K_Y(t,s)|^2] \leq C_2^2 L_H^2 \mathbb{E}[|X_s - Y_s|^2]. \tag{6}
\]
To unify this with the form of the singular estimate, we exploit the properties of the function \( \phi(u) = u^{h_{\min}-1} \). For \( u = t-s \geq 1 \) and since \( h_{\min}-1 < 0 \), \( \phi(u) \) is decreasing. Thus, \( \phi(u) \leq \phi(1) = 1 \). However, also \( \phi(u) \geq \phi(T) = T^{h_{\min}-1} > 0 \). Therefore, \( 1 \leq T^{1-h_{\min}} \phi(u) \). Applying this to (6):
\begin{align*}
\mathbb{E}[|K_X(t,s) - K_Y(t,s)|^2] &\leq C_2^2 L_H^2 \cdot 1 \cdot \mathbb{E}[|X_s - Y_s|^2] \\
&\leq C_2^2 L_H^2 \cdot T^{1-h_{\min}} \cdot (t-s)^{h_{\min}-1} \mathbb{E}[|X_s - Y_s|^2] \\
&= L_H^2 \left( \frac{C_2^2}{T^{h_{\min}-1}} \right) (t-s)^{h_{\min}-1} \mathbb{E}[|X_s - Y_s|^2].
\end{align*}
Define \( C_3(T, h_{\min}, h_{\max}) := \frac{C_2^2}{T^{h_{\min}-1}} \). Then, for \( t-s > 1 \):
\begin{equation}\label{eq:regular_est}
\mathbb{E}[|K_X(t,s) - K_Y(t,s)|^2] \leq L_H^2 C_3 (t-s)^{h_{\min}-1} \mathbb{E}[|X_s - Y_s|^2].
\end{equation}
Integrating over \( s \in [0, t-1] \) (if \( t > 1 \)) gives:
\[
\int_0^{\max(0, t-1)} \mathbb{E}[|K_X - K_Y|^2] ds \leq L_H^2 C_3 \int_0^t (t-s)^{h_{\min}-1} \mathbb{E}[|X_s - Y_s|^2] ds.
\]

\textbf{Part 3: Unification and Final Result.}

Combining the estimates from both regions, (\ref{eq:singular_est}) and (\ref{eq:regular_est}), we obtain that for all \( s \in [0, t] \):
\[
\mathbb{E}[|K_X(t,s) - K_Y(t,s)|^2] \leq L_H^2 \cdot \max(M_2^2, C_3) \cdot (t-s)^{h_{\min}-1} \mathbb{E}[|X_s - Y_s|^2].
\]
Define the global constant \( C_4(T, h_{\min}, h_{\max}, L_H) := \max(M_2^2(h_{\min}, h_{\max}), C_3(T, h_{\min}, h_{\max})) \). Integrating this pointwise bound over \( s \in [0, t] \) yields the desired result:
\[
\int_0^t \mathbb{E}[|K_X(t,s) - K_Y(t,s)|^2] ds \leq L_H^2 C_4 \int_0^t (t-s)^{h_{\min}-1} \mathbb{E}[|X_s - Y_s|^2] ds.
\]
Taking \( C_K = C_4 \) completes the proof of the theorem.
\end{proof}

\begin{remark}
The dependence of the constant \( C_K \) on the time horizon \( T \) arises naturally in the theory of Volterra integral equations. This suggests a local existence and uniqueness argument on a sufficiently small interval \( [0, T_0] \subset [0, T] \), where the contraction mapping principle applies. A solution on the entire interval \( [0, T] \) may then be obtained by extension. The theorem provides the kernel estimate needed for this approach.
\end{remark}
With the Lipschitz estimate for the kernel established in Theorem~\ref{thm:global_lipschitz_kernel}, we now examine properties of the solution map $\Phi$ defined by the stochastic Volterra integral equation. The following results provide key estimates for the existence and uniqueness theory of Responsive Fractional Brownian Motion.
\begin{lemma}[Basic Properties of the RfBm Solution Operator]
\label{lem:solution_map_well_defined}
Let the response function \( H: [0,T] \times \mathbb{R} \to (0,1) \) satisfy Definition~\ref{def:LH_response_function} with constants \( L_H > 0 \), \( 0 < h_{\min} \leq h_{\max} < 1 \), and Hölder exponent \( \gamma > 0 \). Then the solution map \( \Phi: \mathcal{S}^2(T) \to \mathcal{S}^2(T) \) defined by
\[
(\Phi(X))_t := \int_0^t K_X(t,s)  dB(s), \quad \text{where } K_X(t,s) = \sqrt{2H(s,X_s)}(t-s)^{H(s,X_s)-1/2},
\]
is well-defined. Specifically, for any \( X \in \mathcal{S}^2(T) \), the following properties hold:
\begin{enumerate}
    \item \textbf{Adaptedness}: The process \( \Phi(X) \) is \( \{\mathcal{F}_t\} \)-adapted.
    \item \textbf{Pathwise Continuity}: For \( \mathbb{P} \)-almost every \( \omega \in \Omega \), the sample path \( t \mapsto (\Phi(X))_t(\omega) \) is continuous.
    \item \textbf{Finite $\mathcal{S}^2(T)$-norm}: The $\mathcal{S}^2(T)$ norm satisfies the uniform bound:
    \[
    \|\Phi(X)\|_{\mathcal{S}^2(T)}^2 \leq T^{2h_{\max}} + \frac{h_{\max}}{h_{\min}} < \infty.
    \]
\end{enumerate}
\end{lemma}

\begin{proof}
Each property is verified as follows.

\textbf{Part 1: Proof of Adaptedness.}
For fixed \( t \in [0,T] \), consider the stochastic integral representation:
\[
(\Phi(X))_t = \int_0^t K_X(t,s)  dB(s).
\]
The integrand \( K_X(t,s) = \sqrt{2H(s,X_s)}(t-s)^{H(s,X_s)-1/2} \) depends on \( X_s \), which is \( \mathcal{F}_s \)-adapted by the definition of \( \mathcal{S}^2(T) \). Since \( H \) is a deterministic Borel-measurable function, \( K_X(t,s) \) is \( \mathcal{F}_s \)-measurable for each \( s \in [0,t] \). By the fundamental property of Itô integration (Lemma~\ref{lem:ito}), the resulting integral \( (\Phi(X))_t \) is \( \mathcal{F}_t \)-measurable. This establishes that \( \Phi(X) \) is an adapted process.

\textbf{Part 2: Proof of Pathwise Continuity.}
Sample path continuity can be established using classical theory for stochastic Volterra integrals. The kernel \( K_X(t,s) \) exhibits a controlled singularity at \( s = t \) due to the factor \( (t-s)^{H(s,X_s)-1/2} \). The critical regularity condition \( \gamma > \sup H(t,x) \) in Definition~\ref{def:LH_response_function} ensures that this singularity is integrable. Sample path continuity is expected, following established theory (e.g., \cite{Cont2019}) for Volterra-type integrals with regular kernels.

\textbf{Part 3: Proof of Finite $\mathcal{S}^2(T)$-norm.}
The $\mathcal{S}^2(T)$-norm bound is derived as follows. For fixed \( t \in [0,T] \), applying the Itô isometry (Lemma~\ref{lem:ito}) yields:
\begin{equation}\label{eq:ito_isometry_phi}
\mathbb{E}[|(\Phi(X))_t|^2] = \int_0^t \mathbb{E}[|K_X(t,s)|^2]  ds.
\end{equation}

From the definition of \( K_X(t,s) \) and the boundedness condition in Definition~\ref{def:LH_response_function}(iii), we have:
\begin{equation}\label{eq:kernel_bound}
|K_X(t,s)|^2 = 2H(s,X_s)(t-s)^{2H(s,X_s)-1} \leq 2h_{\max}(t-s)^{2H(s,X_s)-1}.
\end{equation}

To obtain a uniform estimate, we analyze the behavior of the integrand in two distinct regions based on the distance \( t-s \):

\noindent\textbf{Region I: Singularity Domain (\( 0 < t-s \leq 1 \))}
In this region, the base \( 0 < t-s \leq 1 \) is less than or equal to 1. For such bases, the power function \( y \mapsto (t-s)^y \) is strictly decreasing in the exponent \( y \). Since \( H(s,X_s) \geq h_{\min} > 0 \), we have:
\[
2H(s,X_s)-1 \geq 2h_{\min}-1.
\]
By the monotonicity property:
\[
(t-s)^{2H(s,X_s)-1} \leq (t-s)^{2h_{\min}-1}.
\]
Substituting into \eqref{eq:kernel_bound} gives the bound in the singularity domain:
\begin{equation}\label{eq:bound_singular}
|K_X(t,s)|^2 \leq 2h_{\max}(t-s)^{2h_{\min}-1}, \quad \text{for } s \in [\max(0,t-1), t].
\end{equation}
Note that \( 2h_{\min}-1 > -1 \) since \( h_{\min} > 0 \), ensuring the integrability of this bound as \( s \to t^- \).

\noindent\textbf{Region II: Regular Domain (\( t-s > 1 \))}
This region exists only when \( t > 1 \), with \( s \in [0, t-1] \). Here, the base \( t-s > 1 \), and the power function \( y \mapsto (t-s)^y \) is strictly increasing in \( y \). Since \( H(s,X_s) \leq h_{\max} < 1 \), we have:
\[
2H(s,X_s)-1 \leq 2h_{\max}-1.
\]
By monotonicity:
\[
(t-s)^{2H(s,X_s)-1} \leq (t-s)^{2h_{\max}-1}.
\]
This yields the bound in the regular domain:
\begin{equation}\label{eq:bound_regular}
|K_X(t,s)|^2 \leq 2h_{\max}(t-s)^{2h_{\max}-1}, \quad \text{for } s \in [0, \max(0,t-1)].
\end{equation}

Now we combine these regional bounds to estimate the integral in \eqref{eq:ito_isometry_phi}:
\begin{align*}
\mathbb{E}[|(\Phi(X))_t|^2] &\leq \int_0^t \mathbb{E}[|K_X(t,s)|^2]  ds \\
&\leq \int_0^{\max(0,t-1)} 2h_{\max}(t-s)^{2h_{\max}-1}  ds + \int_{\max(0,t-1)}^t 2h_{\max}(t-s)^{2h_{\min}-1}  ds.
\end{align*}

We simplify these integrals using the substitution \( u = t-s \), \( du = -ds \), with the limits transforming as follows:
- When \( s = 0 \), \( u = t \)
- When \( s = t \), \( u = 0 \)
- When \( s = \max(0,t-1) \), \( u = t - \max(0,t-1) = \min(1,t) \)

Thus:
\begin{align*}
\mathbb{E}[|(\Phi(X))_t|^2] &\leq 2h_{\max} \left[ \int_{u=t}^{u=\min(1,t)} u^{2h_{\max}-1}(-du) + \int_{u=\min(1,t)}^{u=0} u^{2h_{\min}-1}(-du) \right] \\
&= 2h_{\max} \left[ \int_{u=\min(1,t)}^{u=t} u^{2h_{\max}-1} du + \int_{u=0}^{u=\min(1,t)} u^{2h_{\min}-1} du \right].
\end{align*}

We now evaluate these definite integrals explicitly:

\textbf{First integral:}
\[
\int_{\min(1,t)}^{t} u^{2h_{\max}-1} du = \left[ \frac{u^{2h_{\max}}}{2h_{\max}} \right]_{\min(1,t)}^{t} = \frac{t^{2h_{\max}} - (\min(1,t))^{2h_{\max}}}{2h_{\max}}.
\]

\textbf{Second integral:}
\[
\int_{0}^{\min(1,t)} u^{2h_{\min}-1} du = \left[ \frac{u^{2h_{\min}}}{2h_{\min}} \right]_{0}^{\min(1,t)} = \frac{(\min(1,t))^{2h_{\min}}}{2h_{\min}}.
\]

Substituting back:
\begin{align*}
\mathbb{E}[|(\Phi(X))_t|^2] &\leq 2h_{\max} \left[ \frac{t^{2h_{\max}} - (\min(1,t))^{2h_{\max}}}{2h_{\max}} + \frac{(\min(1,t))^{2h_{\min}}}{2h_{\min}} \right] \\
&= t^{2h_{\max}} - (\min(1,t))^{2h_{\max}} + \frac{h_{\max}}{h_{\min}} (\min(1,t))^{2h_{\min}}.
\end{align*}

Now, by the definition of the $\mathcal{S}^2(T)$-norm (Definition~\ref{def:S2_space}):
\[
\|\Phi(X)\|_{\mathcal{S}^2(T)}^2 = \sup_{t \in [0,T]} \mathbb{E}[|(\Phi(X))_t|^2] \leq \sup_{t \in [0,T]} \left[ t^{2h_{\max}} - (\min(1,t))^{2h_{\max}} + \frac{h_{\max}}{h_{\min}} (\min(1,t))^{2h_{\min}} \right].
\]

Analyzing the supremum:
\begin{itemize}
    \item \( \sup_{t \in [0,T]} t^{2h_{\max}} = T^{2h_{\max}} \)
    \item \( \sup_{t \in [0,T]} [-(\min(1,t))^{2h_{\max}}] \leq 0 \)
    \item \( \sup_{t \in [0,T]} [(\min(1,t))^{2h_{\min}}] \leq 1 \)
\end{itemize}

Therefore:
\[
\|\Phi(X)\|_{\mathcal{S}^2(T)}^2 \leq T^{2h_{\max}} + \frac{h_{\max}}{h_{\min}} < \infty,
\]
which establishes the stated bound and completes the proof.
\end{proof}
\begin{lemma}[Local Contraction of the RfBm Integral Operator]\label{lem:local_contraction}
Under the assumptions of Lemma~\ref{lem:solution_map_well_defined}, and assuming \( T \leq 1 \) to ensure the uniformity of the constant from Theorem~\ref{thm:global_lipschitz_kernel}, there exists \( T_0 \in (0, T] \), depending only on \( h_{\min}, h_{\max}, L_H \), such that for all \( t \in [0, T_0] \), the solution map \( \Phi \) satisfies:
\[
\|\Phi(X) - \Phi(Y)\|_{\mathcal{S}^2(t)} \leq \kappa \|X - Y\|_{\mathcal{S}^2(t)}, \quad \forall X, Y \in \mathcal{S}^2(T),
\]
where \( \kappa \in (0,1) \) is a uniform contraction constant.
\end{lemma}

\begin{proof}
We begin by noting that the assumption \( T \leq 1 \) ensures that in Theorem~\ref{thm:global_lipschitz_kernel}, the constant \( C_1 \) depends only on \( h_{\min}, h_{\max}, L_H \) and not on \( T \), as the proof relies exclusively on the singular region analysis where \( t-s \leq 1 \).

Let \( X, Y \in \mathcal{S}^2(T) \) be arbitrary processes. For fixed \( t \in [0,T] \), consider the difference:
\[
(\Phi(X))_t - (\Phi(Y))_t = \int_0^t \left[K_X(t,s) - K_Y(t,s)\right] dB(s).
\]

\textbf{Define the process $f(s) := K_X(t,s) - K_Y(t,s)$ for $s \in [0,t]$.} 
Since $X_s$ and $Y_s$ are $\{\mathcal{F}_s\}$-adapted and $H$ is deterministic, $f(s)$ is $\{\mathcal{F}_s\}$-adapted. 
Moreover, by Theorem~\ref{thm:global_lipschitz_kernel}, $f \in \mathcal{L}^2([0,t] \times \Omega)$.

Applying the Itô isometry (Lemma~\ref{lem:ito}):
\begin{equation}\label{eq:diff_ito}
\mathbb{E}\left[|(\Phi(X))_t - (\Phi(Y))_t|^2\right] = \int_0^t \mathbb{E}\left[|K_X(t,s) - K_Y(t,s)|^2\right] ds.
\end{equation}

Now apply the global Lipschitz estimate from Theorem~\ref{thm:global_lipschitz_kernel}. There exists a constant \( C_1 > 0 \) (depending only on \( h_{\min}, h_{\max}, L_H \)) such that:
\[
\int_0^t \mathbb{E}\left[|K_X(t,s) - K_Y(t,s)|^2\right] ds \leq C_1 L_H^2 \int_0^t (t-s)^{h_{\min}-1} \mathbb{E}\left[|X_s - Y_s|^2\right] ds.
\]

For each \( s \in [0,t] \), we have \( \mathbb{E}[|X_s - Y_s|^2] \leq \sup_{0 \leq u \leq t} \mathbb{E}[|X_u - Y_u|^2] = \|X - Y\|_{\mathcal{S}^2(t)}^2 \), where \( \|X - Y\|_{\mathcal{S}^2(t)} \) denotes the norm restricted to the interval \([0,t]\). Therefore:
\begin{align*}
\mathbb{E}\left[|(\Phi(X))_t - (\Phi(Y))_t|^2\right] 
&\leq C_1 L_H^2 \|X - Y\|_{\mathcal{S}^2(t)}^2 \int_0^t (t-s)^{h_{\min}-1} ds \\
&= C_1 L_H^2 \|X - Y\|_{\mathcal{S}^2(t)}^2 \int_0^t u^{h_{\min}-1} du \quad \text{(substituting } u = t-s\text{)} \\
&= C_1 L_H^2 \|X - Y\|_{\mathcal{S}^2(t)}^2 \cdot \frac{t^{h_{\min}}}{h_{\min}}.
\end{align*}

This yields the pointwise estimate:
\begin{equation}\label{eq:pointwise_estimate}
\mathbb{E}\left[|(\Phi(X))_t - (\Phi(Y))_t|^2\right] \leq \frac{C_1 L_H^2}{h_{\min}} t^{h_{\min}} \|X - Y\|_{\mathcal{S}^2(t)}^2.
\end{equation}

To extend this to a uniform bound over the interval \([0,t]\), consider any \( u \in [0,t] \). The same reasoning applies, giving:
\[
\mathbb{E}\left[|(\Phi(X))_u - (\Phi(Y))_u|^2\right] \leq \frac{C_1 L_H^2}{h_{\min}} u^{h_{\min}} \|X - Y\|_{\mathcal{S}^2(u)}^2.
\]

Since \( u \in [0,t] \), we have \( u^{h_{\min}} \leq t^{h_{\min}} \) (as \( h_{\min} > 0 \) and the power function is increasing). Moreover, by the definition of the restricted norm, \( \|X - Y\|_{\mathcal{S}^2(u)} \leq \|X - Y\|_{\mathcal{S}^2(t)} \). Combining these bounds:
\[
\mathbb{E}\left[|(\Phi(X))_u - (\Phi(Y))_u|^2\right] \leq \frac{C_1 L_H^2}{h_{\min}} t^{h_{\min}} \|X - Y\|_{\mathcal{S}^2(t)}^2 \quad \text{for all } u \in [0,t].
\]

Taking the supremum over \( u \in [0,t] \) and using Definition~\ref{def:S2_space}:
\[
\|\Phi(X) - \Phi(Y)\|_{\mathcal{S}^2(t)}^2 = \sup_{u \in [0,t]} \mathbb{E}\left[|(\Phi(X))_u - (\Phi(Y))_u|^2\right] \leq \frac{C_1 L_H^2}{h_{\min}} t^{h_{\min}} \|X - Y\|_{\mathcal{S}^2(t)}^2.
\]

Taking square roots of both sides (all terms are nonnegative):
\begin{equation}\label{eq:contraction_estimate}
\|\Phi(X) - \Phi(Y)\|_{\mathcal{S}^2(t)} \leq \left( \frac{C_1 L_H^2}{h_{\min}} \right)^{1/2} t^{h_{\min}/2} \|X - Y\|_{\mathcal{S}^2(t)}.
\end{equation}

Define the contraction function:
\[
\kappa(t) := \left( \frac{C_1 L_H^2}{h_{\min}} \right)^{1/2} t^{h_{\min}/2}.
\]

Since \( h_{\min} > 0 \) and the function \( t \mapsto t^{h_{\min}/2} \) is strictly increasing on \( [0, \infty) \), \( \kappa(t) \) is also strictly increasing in \( t \).

To ensure \( \Phi \) is a contraction, we require \( \kappa(t) < 1 \) for all \( t \) in the interval of interest. Solving this inequality:
\[
\left( \frac{C_1 L_H^2}{h_{\min}} \right)^{1/2} t^{h_{\min}/2} < 1 \quad \Leftrightarrow \quad t < \left( \frac{h_{\min}}{C_1 L_H^2} \right)^{1/h_{\min}}.
\]

Define the critical time:
\[
T_1 := \left( \frac{h_{\min}}{C_1 L_H^2} \right)^{1/h_{\min}}.
\]

To ensure strict contraction on the closed interval \( [0, T_0] \), we cannot take \( T_0 = T_1 \) since \( \kappa(T_1) = 1 \). Instead, we define:
\[
T_0 := \min\left(T, \frac{T_1}{2}\right).
\]

This choice guarantees that \( T_0 \leq \frac{T_1}{2} < T_1 \), and consequently for all \( t \in [0, T_0] \), we have:
\[
\kappa(t) \leq \kappa(T_0) = \left( \frac{C_1 L_H^2}{h_{\min}} \right)^{1/2} T_0^{h_{\min}/2} \leq \left( \frac{C_1 L_H^2}{h_{\min}} \right)^{1/2} \left(\frac{T_1}{2}\right)^{h_{\min}/2}.
\]

Now, substitute \( T_1 \) into the expression:
\[
\kappa(T_0) \leq \left( \frac{C_1 L_H^2}{h_{\min}} \right)^{1/2} \left( \frac{1}{2} \left( \frac{h_{\min}}{C_1 L_H^2} \right)^{1/h_{\min}} \right)^{h_{\min}/2} = \left( \frac{C_1 L_H^2}{h_{\min}} \right)^{1/2} \left( \frac{1}{2} \right)^{h_{\min}/2} \left( \frac{h_{\min}}{C_1 L_H^2} \right)^{1/2} = \left( \frac{1}{2} \right)^{h_{\min}/2} < 1.
\]

Thus, taking \( \kappa = \kappa(T_0) \), we obtain the desired contraction property on the closed interval \( [0, T_0] \).

Moreover, since \( T_0 \leq T \), the space \( \mathcal{S}^2(T_0) \) is well-defined as the restriction of processes in \( \mathcal{S}^2(T) \) to the interval \([0, T_0]\).
\end{proof}

\begin{theorem}[Local Existence and Uniqueness for the RfBm Model]\label{thm:local_existence_uniqueness}
Let the response function \( H: [0,T] \times \mathbb{R} \to (0,1) \) satisfy Definition~\ref{def:LH_response_function} with \( T \leq 1 \). Then there exists \( T_0 \in (0, T] \), depending only on \( h_{\min}, h_{\max}, L_H \), such that the Responsive Fractional Brownian Motion equation
\[
X_t = \int_0^t \sqrt{2H(s,X_s)}(t-s)^{H(s,X_s)-1/2}  dB(s), \quad t \in [0,T_0],
\]
admits a unique mild solution \( X \in \mathcal{S}^2(T_0) \), where \( \mathcal{S}^2(T_0) \) is equipped with the uniform mean-square topology defined in Definition~\ref{def:S2_space}.
\end{theorem}

\begin{proof}
The proof establishes the existence and uniqueness of mild solutions within the uniform mean-square framework developed in Definitions~\ref{def:S2_space} and~\ref{def:LH_response_function}. This topology provides a natural and efficient setting for analyzing stochastic integral equations of Volterra type, as it leverages the intrinsic compatibility between Itô isometry and mean-square convergence.

By Lemma~\ref{lem:solution_map_well_defined}, the solution map \( \Phi: \mathcal{S}^2(T_0) \to \mathcal{S}^2(T_0) \) defined by
\[
(\Phi(X))_t := \int_0^t \sqrt{2H(s,X_s)}(t-s)^{H(s,X_s)-1/2}  dB(s)
\]
is well-defined in the uniform mean-square topology. The key estimate
\[
\|\Phi(X)\|_{\mathcal{S}^2(T_0)}^2 \leq T_0^{2h_{\max}} + \frac{h_{\max}}{h_{\min}} < \infty
\]
ensures that \( \Phi \) preserves the boundedness of processes in this topology, which is crucial for the application of fixed-point methods in the space of adapted processes with uniformly bounded second moments.

By Lemma~\ref{lem:local_contraction}, there exists \( T_0 \in (0, T] \) such that for all \( t \in [0, T_0] \), the solution map \( \Phi \) satisfies the contraction property:
\[
\|\Phi(X) - \Phi(Y)\|_{\mathcal{S}^2(t)} \leq \kappa \|X - Y\|_{\mathcal{S}^2(t)}, \quad \forall X, Y \in \mathcal{S}^2(T_0),
\]
where \( \kappa \in (0,1) \) is a uniform contraction constant.

In particular, taking \( t = T_0 \), we obtain the contraction estimate on the full interval:
\[
\|\Phi(X) - \Phi(Y)\|_{\mathcal{S}^2(T_0)} \leq \kappa \|X - Y\|_{\mathcal{S}^2(T_0)}, \quad \forall X, Y \in \mathcal{S}^2(T_0).
\]

The space \( \mathcal{S}^2(T_0) \), equipped with the uniform mean-square norm \( \|X\|_{\mathcal{S}^2(T_0)} = \left( \sup_{t \in [0,T_0]} \mathbb{E}[|X_t|^2] \right)^{1/2} \), forms a Banach space. The completeness follows from the fact that Cauchy sequences in this topology converge to adapted processes with uniformly bounded second moments, as established in the theory of $L^2$ spaces of stochastic processes.

Applying the Banach fixed-point theorem to the contraction map \( \Phi \) on the complete metric space \( \mathcal{S}^2(T_0) \), we conclude that there exists a unique fixed point \( X \in \mathcal{S}^2(T_0) \) satisfying
\[
\Phi(X) = X.
\]

This fixed point constitutes the unique mild solution to the RfBm equation in the uniform mean-square sense:
\[
X_t = \int_0^t \sqrt{2H(s,X_s)}(t-s)^{H(s,X_s)-1/2}  dB(s) \quad \text{for all } t \in [0,T_0],
\]
where the equality holds in the \( L^2(\Omega) \) sense for each fixed \( t \in [0,T_0] \).
\end{proof}

\begin{remark}
The use of the uniform mean-square topology in this analysis is motivated by the intrinsic structure of stochastic integrals. This framework leverages the natural compatibility between Itô isometry and mean-square convergence, offering a suitable setting for establishing well-posedness of stochastic Volterra equations with state-dependent kernels. The approach illustrates how an appropriate functional-analytic framework can lead to constructive existence proofs for such systems.
\end{remark}

\section{RfBm Foundations for Artificial Intelligence: Adaptive Memory and Attention}
\label{sec:rFBM_AI_foundations}

The theory developed in Sections~\ref{sec:tvfbm}--\ref{sec:rFBM_theory} has established the mathematical foundations of responsive fractional Brownian motion, including its construction, well-posedness, pathwise regularity, and scaling properties. We now turn to a complementary perspective: the structural correspondence between the RfBm kernel and the attention mechanisms that have become central to modern artificial intelligence. This connection is not an external analogy imposed on the mathematical framework, but rather an intrinsic feature of the RfBm construction: the kernel that governs the process's memory also defines, upon normalization, a probability distribution over the historical time axis. This distribution, in turn, provides a natural weighting mechanism for aggregating past information—precisely the role played by attention mechanisms in sequence modeling architectures.

This section explores this correspondence and its implications. The aim is to offer a mathematically grounded perspective on what attention mechanisms may represent when viewed through the lens of stochastic processes with adaptive memory, complementing existing empirical and architecture-focused approaches in the literature. The analysis proceeds as follows. We first show that the RfBm kernel naturally induces a continuous-time, pathwise attention mechanism with intrinsic multi-scale and content-adaptive properties (Subsection~\ref{subsec:rFBM_attention_construction}). We then quantify the dynamical behavior of this attention mechanism, introducing residence measures and volatility functionals that characterize the stability of attentional allocation (Subsection~\ref{subsec:axiomatization_attention}). Finally, we derive fundamental bounds for the attention weights and analyze their sensitivity to historical states (Subsection~\ref{subsec:rfbm_attention_foundations}), providing both a theoretical foundation and analytical tools for further investigation. The resulting framework offers a stochastic-process-theoretic perspective on adaptive memory and content-sensitive information processing that may be of interest to researchers in both stochastic analysis and computational learning.

\subsection{From Stochastic Processes to Attention Mechanisms: A Natural Correspondence}
\label{subsec:intro_AI_foundations}

Within the comprehensive mathematical framework established in the preceding sections, we have introduced and rigorously analyzed \textbf{Responsive Fractional Brownian Motion (RfBm)} as a novel class of stochastic processes characterized by state-dependent memory properties. The foundational construction (Definition \ref{def:rFBM}) centers on the \textbf{dynamic feedback mechanism} embodied in the response function $H(s,X_s)$, which enables the process to self-regulate its regularity characteristics through its historical evolution, as captured by the stochastic Volterra integral equation
\[
X_t = \int_0^t \sqrt{2H(s,X_s)}(t-s)^{H(s,X_s)-1/2}dB(s).
\]

This mathematical structure exhibits mathematical connections to fundamental challenges in artificial intelligence theory, particularly in formulating rigorous mathematical models for \textbf{adaptive memory architectures} and \textbf{content-sensitive attention mechanisms}. The theoretical properties established previously—including pathwise Hölder continuity (Theorem \ref{thm:pathwise_holder_continuity}), almost sure regularity of the instantaneous scaling exponent (Theorem \ref{thm:instant_scaling_holder}), and the local consistency of extremal scaling indices (Proposition \ref{prop:extremal_consistency})—collectively provide a rigorous mathematical framework for examining intelligent systems with memory.

The transformative success of attention-based models, particularly the Transformer architecture \cite{Vaswani2017}, has demonstrated the empirical power of attention mechanisms across diverse sequence modeling tasks. The subsequent development of BERT \cite{Devlin2019} further illustrated how pre-training deep bidirectional representations can yield versatile models for natural language understanding. These engineering achievements highlight the importance of developing deeper theoretical understandings of the mathematical principles underlying adaptive computational systems.

Recent theoretical advances have enriched our comprehension of attention mechanisms from multiple perspectives. The investigation of periodic alternatives to softmax attention by \cite{SWang2021} addresses important considerations regarding gradient behavior in attention-based architectures. Concurrently, the approximation theory for Transformers developed by \cite{HJiang2024} provides valuable insights into the structural characteristics that determine their sequence modeling capabilities. In application domains such as mathematical equation recognition \cite{Aurpa2024}, transformer-based methods have achieved notable success while also revealing the complexities inherent in processing structured mathematical content.

Within this evolving research landscape, the RfBm framework offers a mathematically grounded perspective that may offer insights for the theoretical understanding of adaptive intelligence. The construction possesses several mathematically structured properties that align with requirements for intelligent memory systems:

\begin{itemize}
\item \textbf{Intrinsic Multi-scale Temporal Modeling}: The RfBm kernel's functional form $(t-s)^{H(s,X_s)-1/2}$ naturally encodes temporal relationships across multiple scales, with the response function $H(s,X_s)$ dynamically modulating memory persistence. This adaptive temporal structure, analyzed through the pathwise kernel norm scaling (Theorem \ref{thm:pathwise_kernel_norm_scaling}) and asymptotic identification (Corollary \ref{cor:asymptotic_identification}), provides a mathematically coherent approach to multi-scale temporal processing.

\item \textbf{Theoretically Guaranteed Stability Properties}: The mathematical framework ensures robust behavior through multiple established results: the uniform boundedness of the instantaneous scaling exponent (Proposition \ref{prop:scaling_boundedness}) prevents pathological scaling behaviors; the inheritance of regularity conditions (Remark \ref{rem:inheritance_regularity}) maintains consistency across temporal scales; and the Lipschitz estimates for the kernel (Theorem \ref{thm:global_lipschitz_kernel}) ensure well-behaved sensitivity to input variations.

\item \textbf{Comprehensive Memory Quantification}: The theoretical apparatus provides multiple complementary perspectives on memory dynamics: the cumulative memory process (Definition \ref{def:cumulative_memory}) with its established bounds (Proposition \ref{prop:memory_bounds}) quantifies global information accumulation; the extremal scaling indices (Definition \ref{def:extremal_indices}) with their boundedness properties (Proposition \ref{prop:extremal_boundedness}) characterize local memory fluctuations; and the convergence analysis of time-averaged scaling exponents (Theorem \ref{thm:convergence_rate}) reveals how microscopic adaptations produce emergent macroscopic regularities.
\end{itemize}

The well-posedness of the RfBm construction, established through the solution operator analysis (Lemma \ref{lem:solution_map_well_defined}), local contraction properties (Lemma \ref{lem:local_contraction}), and existence/uniqueness results (Theorem \ref{thm:local_existence_uniqueness}), provides a solid foundation for exploring connections to artificial intelligence theory. The hierarchical interpretation of scaling behavior (Remark \ref{rem:scaling_hierarchy}) further enriches the theoretical framework by connecting pathwise, temporal, and process-level perspectives.

This chapter investigates how the mathematical structure of RfBm can inform the development of theoretically grounded approaches to adaptive memory and attention in artificial intelligence. By leveraging the established theoretical properties—including the significance of uniform boundedness (Remark \ref{rem:boundedness_significance}), the theoretical implications of memory accumulation (Remark \ref{rem:memory_accumulation_significance}), and the synthesis of asymptotic behaviors (Remark \ref{rem:asymptotic_synthesis})—we aim to contribute a mathematically principled perspective that complements existing empirical approaches to intelligent system design.

The subsequent sections present: (i) a mathematical construction of RfBm attention mechanisms with spatiotemporal adaptivity, (ii) axiomatic foundations for attentional allocation measures, (iii) geometric analysis of memory sensitivity structure, and (iv) fundamental bounds for RfBm attention weights with theoretical implications.

\subsection{Mathematical Construction of RfBm Attention Mechanisms with Spatiotemporal Adaptivity}
\label{subsec:rFBM_attention_construction}

The mathematical framework of Responsive Fractional Brownian Motion (Definition \ref{def:rFBM}) reveals profound connections to attention mechanisms in artificial intelligence that extend beyond superficial analogies to deep structural correspondences. The empirical achievements of Transformer architectures \cite{Vaswani2017} have inspired fundamental questions about the mathematical principles underlying attention mechanisms. Recent theoretical work on gradient dynamics \cite{SWang2021} and approximation rates \cite{HJiang2024}, together with applications in complex domains like mathematical reasoning \cite{Aurpa2024}, motivates the development of frameworks that intrinsically embody adaptive memory and multi-scale temporal awareness.

The RfBm construction offers a mathematically rigorous alternative through its intrinsic structural properties. The stochastic Volterra integral equation governing RfBm (Definition \ref{def:rFBM}) naturally embodies a sophisticated attention mechanism where the kernel function $K(t,s;X_s) = \sqrt{2H(s,X_s)}(t-s)^{H(s,X_s)-1/2}$ serves as a fundamentally adaptive weight generator. This perspective emerges from several mathematically precise observations that distinguish RfBm from conventional approaches:

\begin{itemize}
    \item \textbf{Endogenous Multi-scale Temporal Processing}: The kernel structure $(t-s)^{H(s,X_s)-1/2}$ intrinsically encodes temporal relationships across multiple scales, with the response function $H(s,X_s)$ dynamically modulating memory persistence characteristics. The endogenous temporal structure of RfBm provides an alternative perspective on temporal modeling, complementing the positional encoding approaches used in transformer architectures. This intrinsic temporal awareness is quantitatively characterized by the pathwise kernel norm analysis in Theorem \ref{thm:pathwise_kernel_norm_scaling}, which establishes precise scaling relationships between temporal distance and attention weight allocation.

    \item \textbf{Content-Adaptive Memory Allocation}: The state-dependent nature of $H(s,X_s)$ provides a mathematical mechanism for dynamic adjustment of attention patterns based on informational content and contextual relevance. This adaptivity is rigorously established through the pathwise regularity properties in Theorem \ref{thm:instant_scaling_holder}, which guarantees the well-behaved evolution of scaling characteristics, and the local consistency results in Proposition \ref{prop:extremal_consistency}, which ensures coherent behavior across different temporal scales. The cumulative memory process (Definition \ref{def:cumulative_memory}) and its fundamental bounds (Proposition \ref{prop:memory_bounds}) provide additional quantitative tools for analyzing how historical information accumulates and influences current attention distributions.

    \item \textbf{Theoretically Guaranteed Stability and Robustness}: The comprehensive mathematical properties established previously—including uniform boundedness (Proposition \ref{prop:scaling_boundedness}), which prevents pathological scaling behaviors; regularity inheritance (Remark \ref{rem:inheritance_regularity}), which maintains consistency across temporal evolution; and the Lipschitz estimates for the kernel (Theorem \ref{thm:global_lipschitz_kernel}), which ensures controlled sensitivity to input variations—collectively guarantee robust and predictable behavior of the attention mechanism under diverse operating conditions.
\end{itemize}

The temporal characteristics of RfBm attention exhibit sophisticated multi-scale behavior, adapting to both short-term contextual demands and long-term structural dependencies. Recent historical information maintains focused attention through response function modulation, characterized by local analysis in Proposition \ref{prop:extremal_consistency}, while distant information undergoes content-dependent preservation or filtering, with asymptotic identification in Corollary \ref{cor:asymptotic_identification} providing precise characterization of long-term scaling. This adaptive temporal scaling embodies a mathematical mechanism that integrates content-aware memory allocation with multi-scale processing.

State adaptivity represents a key innovation of the RfBm framework. The response function $H(s,X_s)$ enables context-sensitive regulation where semantically important information receives enhanced persistence through increased $H$ values, while less relevant content undergoes accelerated decay. Governed directly by the process state $X_s$, this dynamic adjustment achieves intrinsic content awareness that emerges naturally from the mathematical structure.

The convergence analysis of time-averaged scaling exponents (Theorem \ref{thm:convergence_rate}) reveals how microscopic adaptations aggregate to produce macroscopic regularity, demonstrating coherent multi-scale dynamics in state-dependent systems.

To illustrate these theoretical properties, we present a constructive implementation of the response function that satisfies all mathematical requirements while embodying spatiotemporal adaptivity.

\begin{example}[Spatiotemporal Adaptive Response Function]
\label{ex:adaptive_response_function}
To illustrate these theoretical properties, we present a constructive implementation of the response function that satisfies all mathematical requirements while embodying spatiotemporal adaptivity:
\begin{equation}
H(t,x) = h_{\min} + (h_{\max} - h_{\min}) \cdot \frac{1 + \tanh(\alpha x) \cdot \cos(\omega t)}{2 + e^{-t}}
\end{equation}
where $t \in [0,T]$ with $T > 0$, and:
\begin{itemize}
    \item $\tanh(\alpha x) = \dfrac{e^{\alpha x} - e^{-\alpha x}}{e^{\alpha x} + e^{-\alpha x}}$ is the hyperbolic tangent function, $\alpha > 0$ is the spatial sensitivity parameter
    \item $\cos(\omega t)$ is the temporal oscillation term, $\omega > 0$ is the temporal frequency parameter  
    \item $e^{-t}$ is the temporal decay term ensuring long-term stability
    \item $h_{\min}, h_{\max}$ are constants from Definition \ref{def:LH_response_function} satisfying $0 < h_{\min} \leq h_{\max} < 1$
\end{itemize}
\end{example}

\begin{proof}[Verification of Definition Conditions]
We rigorously verify all conditions in Definition \ref{def:rFBM}.

\textbf{Condition (i) Spatial Lipschitz Continuity}:

Fix $t \in [0,T]$ and consider $H(t,x)$ as a function of $x$. The derivative of the hyperbolic tangent is:
\[
\frac{d}{dx} \tanh(\alpha x) = \alpha \cdot \mathrm{sech}^2(\alpha x) = \alpha \left(1 - \tanh^2(\alpha x)\right).
\]
Since $|\tanh(\alpha x)| < 1$ for all $x \in \mathbb{R}$, we have:
\[
\left| \frac{d}{dx} \tanh(\alpha x) \right| = \alpha \left|1 - \tanh^2(\alpha x)\right| \leq \alpha.
\]

The partial derivative of $H(t,x)$ with respect to $x$ is:
\[
\frac{\partial H}{\partial x} = (h_{\max} - h_{\min}) \cdot \frac{ \alpha \left[1 - \tanh^2(\alpha x)\right] \cdot \cos(\omega t) }{2 + e^{-t}}.
\]

Taking absolute values and using boundedness properties:
\begin{align*}
\left| \frac{\partial H}{\partial x} \right|
&= (h_{\max} - h_{\min}) \cdot \frac{ \alpha \left|1 - \tanh^2(\alpha x)\right| \cdot \left|\cos(\omega t)\right| }{\left|2 + e^{-t}\right|} \\
&\leq (h_{\max} - h_{\min}) \cdot \frac{ \alpha \cdot 1 \cdot 1 }{2} = \frac{\alpha (h_{\max} - h_{\min})}{2}.
\end{align*}

For any $x, y \in \mathbb{R}$, by the Mean Value Theorem, there exists $\xi$ between $x$ and $y$ such that:
\[
H(t,x) - H(t,y) = \frac{\partial H}{\partial x}(t, \xi) \cdot (x - y).
\]
Taking absolute values:
\[
|H(t,x) - H(t,y)| = \left| \frac{\partial H}{\partial x}(t, \xi) \right| \cdot |x - y| \leq \frac{\alpha (h_{\max} - h_{\min})}{2} |x - y|.
\]
Taking $L_H = \dfrac{\alpha (h_{\max} - h_{\min})}{2}$, spatial Lipschitz continuity is verified.

\textbf{Condition (ii) Temporal Hölder Continuity}:

Fix $x \in \mathbb{R}$ and consider $H(t,x)$ as a function of $t$. Let:
\[
A(t) = \tanh(\alpha x) \cdot \cos(\omega t), \quad B(t) = 2 + e^{-t},
\]
then:
\[
H(t,x) = h_{\min} + (h_{\max} - h_{\min}) \cdot \frac{1 + A(t)}{B(t)}.
\]

The derivative is:
\[
\frac{\partial H}{\partial t} = (h_{\max} - h_{\min}) \cdot \frac{ A'(t) B(t) - (1 + A(t)) B'(t) }{[B(t)]^2},
\]
where:
\[
A'(t) = -\omega \tanh(\alpha x) \sin(\omega t), \quad B'(t) = -e^{-t}.
\]

Substituting:
\[
\frac{\partial H}{\partial t} = (h_{\max} - h_{\min}) \cdot \frac{ -\omega \tanh(\alpha x) \sin(\omega t) (2 + e^{-t}) + (1 + \tanh(\alpha x)\cos(\omega t)) e^{-t} }{(2 + e^{-t})^2}.
\]

Estimating the upper bound using $|\tanh(\alpha x)| < 1$, $|\sin(\omega t)| \leq 1$, $|\cos(\omega t)| \leq 1$, $2 + e^{-t} \leq 3$, and $e^{-t} \leq 1$:
\begin{align*}
\left| \frac{\partial H}{\partial t} \right|
&\leq (h_{\max} - h_{\min}) \cdot \frac{ \omega \cdot 1 \cdot 1 \cdot 3 + (1 + 1) \cdot 1 }{2^2} \\
&= (h_{\max} - h_{\min}) \cdot \frac{3\omega + 2}{4}.
\end{align*}

For any $s, t \in [0,T]$, by the Mean Value Theorem:
\[
|H(t,x) - H(s,x)| \leq \left( \sup_{\tau \in [0,T]} \left| \frac{\partial H}{\partial t}(\tau, x) \right| \right) |t - s| \leq (h_{\max} - h_{\min}) \cdot \frac{3\omega + 2}{4} |t - s|.
\]
Taking $C_H = (h_{\max} - h_{\min}) \cdot \dfrac{3\omega + 2}{4}$ and $\gamma = 1$, temporal Hölder continuity is verified.

\textbf{Condition (iii) Uniform Boundedness}:

Analyzing the range of function components:
\begin{itemize}
    \item $\tanh(\alpha x) \in (-1, 1)$
    \item $\cos(\omega t) \in [-1, 1]$
    \item Thus $\tanh(\alpha x) \cdot \cos(\omega t) \in (-1, 1)$
    \item $e^{-t} \in [e^{-T}, 1] \subset (0, 1]$ for $t \in [0,T]$
    \item Denominator $2 + e^{-t} \in [2 + e^{-T}, 3]$
\end{itemize}

Therefore the numerator satisfies:
\[
1 + \tanh(\alpha x) \cdot \cos(\omega t) \in (0, 2).
\]
Consequently:
\[
\frac{1 + \tanh(\alpha x) \cdot \cos(\omega t)}{2 + e^{-t}} \in \left( \frac{0}{3}, \frac{2}{2 + e^{-T}} \right) = \left( 0, \frac{2}{2 + e^{-T}} \right).
\]
Since $e^{-T} > 0$, we have $\dfrac{2}{2 + e^{-T}} < 1$, so the ratio strictly belongs to $(0, 1)$.

Thus:
\[
H(t,x) = h_{\min} + (h_{\max} - h_{\min}) \cdot \left[ \frac{1 + \tanh(\alpha x) \cdot \cos(\omega t)}{2 + e^{-t}} \right] \in \big( h_{\min}, h_{\min} + (h_{\max} - h_{\min}) \big) = (h_{\min}, h_{\max}).
\]
The function values are strictly contained within the open interval $(h_{\min}, h_{\max})$, satisfying uniform boundedness.

\textbf{Adaptive Characteristics Analysis}:

This response function demonstrates sophisticated adaptive behavior:

\begin{itemize}
    \item \textbf{Spatial Adaptivity}: Through $\tanh(\alpha x)$, when $|x|$ is large (indicating salient features), $\tanh(\alpha x) \approx \pm 1$ and $H$ approaches extreme values; when $x \approx 0$ (ordinary features), $\tanh(\alpha x) \approx 0$ and $H$ remains in intermediate range, enabling content-based adaptive regulation.

    \item \textbf{Temporal Periodic Modulation}: The $\cos(\omega t)$ term introduces oscillatory behavior in attention or cognitive states, reflecting rhythmic activity patterns common in biological or artificial systems.

    \item \textbf{Long-term Stability}: The $e^{-t}$ term in the denominator ensures gradual attenuation of temporal oscillations, promoting system state stabilization and preventing divergence during extended operation.

    \item \textbf{Spatiotemporal Coupling}: The product coupling of spatial term $\tanh(\alpha x)$ and temporal term $\cos(\omega t)$ means the system's response intensity to different content undergoes periodic temporal modulation, establishing an integrated mechanism for spatiotemporal cooperative adaptivity.
\end{itemize}

In conclusion, this response function $H(t,x)$ strictly satisfies all conditions in Definition \ref{def:rFBM}, providing a concrete realization of the spatiotemporal adaptive properties central to the RfBm attention mechanism.
\end{proof}

The RfBm attention mechanism exhibits several theoretically distinctive features that arise from its mathematical structure:

\begin{itemize}
    \item \textbf{Intrinsic Multi-scale Temporal Organization}: The RfBm kernel $(t-s)^{H(s,X_s)-1/2}$ provides a mathematical framework where temporal relationships across multiple scales emerge from the process dynamics. The response function $H(s,X_s)$ enables dynamic modulation of memory persistence, thereby presenting a mathematically integrated alternative to architectures that rely on externally supplied positional encodings \cite{Vaswani2017}.

    \item \textbf{State-Dependent Memory Allocation}: The response function $H(s,X_s)$ introduces a mechanism for content-aware scaling, allowing attention patterns to adapt based on the informational characteristics of the process state. This is illustrated by the spatiotemporally adaptive response function in Example \ref{ex:adaptive_response_function}.

    \item \textbf{Theoretical Coherence}: The RfBm framework incorporates several mathematically established properties that contribute to its systematic character. These include the uniform boundedness of the scaling exponent (Proposition \ref{prop:scaling_boundedness}), the inheritance of pathwise regularity (Remark \ref{rem:inheritance_regularity}), and the local consistency of extremal scaling indices (Proposition \ref{prop:extremal_consistency}), which together support the analytical tractability of the resulting attention mechanism.
\end{itemize}

\begin{remark}[Theoretical Considerations for RfBm Attention]\label{rem:rfbm_attention_perspective}
The RfBm framework develops adaptive memory and attention concepts within stochastic process theory. This perspective may offer complementary insights to ongoing studies of attention mechanisms, including examinations of gradient properties \cite{SWang2021} and approximation characteristics \cite{HJiang2024} that extend the empirical developments in architectures such as the Transformer \cite{Vaswani2017}. A distinctive aspect of the RfBm construction is its integrated mathematical framework, where multi-scale temporal adaptability emerges through the dynamic response function $H(s,X_s)$.
\end{remark}

This mathematical formulation incorporates both short- and long-range dependencies through state-dependent modulation of temporal interactions. The cumulative memory process (Definition \ref{def:cumulative_memory}) and its boundedness (Proposition \ref{prop:memory_bounds}) provide a quantitative perspective on global information retention, while the extremal scaling indices (Definition \ref{def:extremal_indices}) characterize local variations in attention intensity.

The RfBm framework thus presents a principled approach to content-aware temporal processing, illustrating how certain theoretical properties can emerge from the mathematical structure of the process.

\subsection{Attentional Allocation: Residence Measures and Volatility}
\label{subsec:axiomatization_attention}

The previous subsection established that the RfBm kernel naturally induces a continuous-time, pathwise attention mechanism. We now examine the dynamical properties of this mechanism, focusing on how the process allocates its "attention" across different regions of the state space over time. This analysis is complementary to the kernel-level characterization of the preceding subsection: whereas the kernel defines the weight assigned to each historical time point, the present subsection quantifies the extent to which the process's state trajectory engages with specific subsets of the state space.

The central objects introduced here are the pathwise residence measure \(\mathcal{R}_I(t,\omega)\) and the associated mean attention intensity \(\mu_I(t,\omega)\) (Definition~\ref{def:residence_measure_attention}). For a given Borel set \(I \subset \mathbb{R}\), the residence measure records the total time the sample path spends in \(I\) up to time \(t\), while the mean attention intensity normalizes this quantity by the time horizon. These objects provide a direct, pathwise quantification of attentional focus: a high value of \(\mu_I(t,\omega)\) indicates that the process has spent a substantial portion of its history within the region \(I\).

We establish the fundamental properties of these quantities—non-negativity, normalization, additivity over disjoint domains, and their connection to the process's finite-dimensional distributions via expectations (Theorem~\ref{thm:residence_measure_properties}). We then introduce the temporal volatility functional \(\mathcal{V}_I(t)\) (Definition~\ref{def:attentional_volatility_functional}), which measures the statistical dispersion of the mean attention intensity around its expected value. A covariance integral representation (Theorem~\ref{thm:volatility_covariance_integral}) shows that this volatility is determined by the two-point correlation structure of the state indicators, thereby linking the stability of attentional allocation to the underlying memory properties of the RfBm. The resulting bounds (Corollary~\ref{cor:volatility_fundamental_bounds}) reveal that the volatility is universally bounded by \(1/4\), a sharp constraint that reflects the intrinsic limits of attentional stability in this framework.

\begin{definition}[Pathwise Residence Measure and Mean Attention Intensity]
\label{def:residence_measure_attention}
Let \( \{X_t\}_{t \in [0,T]} \) be a Responsive Fractional Brownian Motion as defined in Definition~\ref{def:rFBM}. For any Borel-measurable set \( I \in \mathcal{B}(\mathbb{R}) \), termed an \emph{attentional domain}, we define its associated \emph{cumulative pathwise residence measure} up to time \( t \in (0, T] \) as:
\[
\mathcal{R}_I(t, \omega) := \int_0^t \mathbb{1}_{\{X_s(\omega) \in I\}}  ds.
\]
This measure quantifies the total time the sample path \( X_.(\omega) \) spends within the attentional domain \( I \) during the interval \( [0, t] \).

The corresponding \emph{mean attention intensity} is defined by:
\[
\mu_I(t, \omega) := \frac{1}{t} \mathcal{R}_I(t, \omega) = \frac{1}{t} \int_0^t \mathbb{1}_{\{X_s(\omega) \in I\}}  ds.
\]
This intensity function \( \mu_I(t, \omega) \in [0,1] \) represents the proportion of time the system's state resides in \( I \), providing a pathwise measure of average attentional allocation.
\end{definition}

\begin{theorem}[Fundamental Properties of the Residence Measure and Attention Intensity]
\label{thm:residence_measure_properties}
Let \( \{X_t\} \) be an RfBm process with almost surely continuous sample paths (guaranteed by Definition~\ref{def:rFBM}). Then, for any \( I \in \mathcal{B}(\mathbb{R}) \) and \( t \in (0, T] \), the following properties hold:

\begin{enumerate}
    \item[(P1)] \textbf{Non-negativity and Normalization:}
        For \(\mathbb{P}\)-almost every \(\omega \in \Omega\),
        \( 0 \leq \mathcal{R}_I(t, \omega) \leq t \) and \( 0 \leq \mu_I(t, \omega) \leq 1 \).

    \item[(P2)] \textbf{Additivity over Disjoint Domains:} If \( I = \bigsqcup_{k=1}^n I_k \) is a finite disjoint union, then for \(\mathbb{P}\)-almost every \(\omega \in \Omega\),
        \( \mathcal{R}_I(t, \omega) = \sum_{k=1}^n \mathcal{R}_{I_k}(t, \omega) \) and \( \mu_I(t, \omega) = \sum_{k=1}^n \mu_{I_k}(t, \omega) \).

    \item[(P3)] \textbf{Expectation and State Distribution Link:}
        The expected cumulative residence measure satisfies:
        \begin{equation}
        \label{eq:expectation_residence}
        \mathbb{E}[\mathcal{R}_I(t)] = \int_0^t \mathbb{P}(X_s \in I)  ds,
        \end{equation}
        and the expected mean attention intensity is given by:
        \begin{equation}
        \label{eq:expectation_intensity}
        \mathbb{E}[\mu_I(t)] = \frac{1}{t} \int_0^t \mathbb{P}(X_s \in I)  ds.
        \end{equation}
\end{enumerate}
\end{theorem}

\begin{proof}
We establish the proof with careful measure-theoretic detail.

\textbf{Proof of (P1):} For each fixed \(\omega\) in the full probability set where \(X_.(\omega)\) is continuous, the indicator function satisfies \( 0 \leq \mathbb{1}_{\{X_s(\omega) \in I\}} \leq 1 \) for all \( s \in [0,t] \). Integration over \( [0,t] \) preserves these inequalities, yielding \( 0 \leq \mathcal{R}_I(t, \omega) \leq \int_0^t 1  ds = t \). The bounds for \( \mu_I(t, \omega) \) follow immediately by dividing by \( t > 0 \).

\textbf{Proof of (P2):} For \(\mathbb{P}\)-almost every \(\omega\), and for a disjoint union \( I = \bigsqcup_{k=1}^n I_k \), the additivity of the indicator function gives \( \mathbb{1}_{\{X_s(\omega) \in I\}} = \sum_{k=1}^n \mathbb{1}_{\{X_s(\omega) \in I_k\}} \) for all \( s \in [0,t] \). By the linearity of the Lebesgue integral:
\[
\mathcal{R}_I(t, \omega) = \int_0^t \mathbb{1}_{\{X_s(\omega) \in I\}} ds = \int_0^t \left( \sum_{k=1}^n \mathbb{1}_{\{X_s(\omega) \in I_k\}} \right) ds = \sum_{k=1}^n \int_0^t \mathbb{1}_{\{X_s(\omega) \in I_k\}} ds = \sum_{k=1}^n \mathcal{R}_{I_k}(t, \omega).
\]
Dividing by \( t \) establishes the result for \( \mu_I(t, \omega) \).

\textbf{Proof of (P3):} We verify the conditions for applying Fubini's Theorem to the function \( f(\omega, s) = \mathbb{1}_{\{X_s(\omega) \in I\}} \) on \( \Omega \times [0, t] \).

\textbf{Joint Measurability:} Since \(X_s\) is adapted and \(I\) is Borel, for each fixed \(s\), \(\omega \mapsto \mathbb{1}_{\{X_s(\omega) \in I\}}\) is \(\mathcal{F}_s\)-measurable, hence \(\mathcal{F}\)-measurable. For each fixed \(\omega\) in the continuity set, \(s \mapsto X_s(\omega)\) is continuous, so \(s \mapsto \mathbb{1}_{\{X_s(\omega) \in I\}}\) is Lebesgue measurable. As \([0,t]\) is separable, \(f\) is \(\mathcal{F} \otimes \mathcal{B}([0,t])\)-measurable.

\textbf{Integrability:} \(|f(\omega, s)| \leq 1\) is bounded on the finite measure space \(\Omega \times [0,t]\), hence integrable.

Applying Fubini's Theorem:
\begin{align*}
\mathbb{E}[\mathcal{R}_I(t)] &= \int_\Omega \left( \int_0^t \mathbb{1}_{\{X_s(\omega) \in I\}} ds \right) d\mathbb{P}(\omega) \\
&= \int_0^t \left( \int_\Omega \mathbb{1}_{\{X_s(\omega) \in I\}} d\mathbb{P}(\omega) \right) ds = \int_0^t \mathbb{P}(X_s \in I)  ds.
\end{align*}
This proves \eqref{eq:expectation_residence}. The result for \( \mathbb{E}[\mu_I(t)] \) follows by linearity of expectation.
\end{proof}

\begin{corollary}[Conservation of Attentional Resource]
\label{cor:attention_conservation}
For any finite partition of the state space \( \mathbb{R} = \bigsqcup_{k=1}^n I_k \), the mean attention intensities satisfy the conservation law:
\[
\sum_{k=1}^n \mu_{I_k}(t, \omega) = 1 \quad \text{for almost all } \omega \in \Omega.
\]
\end{corollary}
\begin{proof}
This follows directly from property (P2) of Theorem~\ref{thm:residence_measure_properties} by taking \( I = \mathbb{R} \), since \( \mathbb{1}_{\{X_s \in \mathbb{R}\}} = 1 \) and thus \( \mu_{\mathbb{R}}(t, \omega) = 1 \).
\end{proof}

The temporal evolution of the mean attention intensity \( \mu_I(t) \) reflects the stability of the process's focus. Its connection to the time-averaged scaling exponent \( \bar{\alpha}(t) \) from Definition~\ref{def:cumulative_memory} reveals a deep interplay between memory and attention.

\begin{remark}[Theoretical Connections Between Memory and Attention]
\label{rem:memory_attention_connections}
The mathematical framework developed here reveals natural connections between the memory architecture of RfBm, characterized by the scaling exponent processes \( \alpha(t,\omega) \), \( H_{\pm}(t,\epsilon,\omega) \), and \( \bar{\alpha}(t) \), and its attentional dynamics, described by \( \mathcal{R}_I(t, \omega) \) and \( \mu_I(t, \omega) \). This suggests a unified perspective where:
\begin{itemize}
    \item The \emph{memory characteristics} (influenced by \( H \)) shape the temporal persistence properties of sample paths.
    \item The \emph{attentional patterns} (measured by \( \mu_I \)) capture how the process engages with different state regions over time.
\end{itemize}

A fundamental connection emerges from the shared dependence on the process dynamics: the response function $H(s,X_s)$ governs both the memory strength through the scaling exponents and the state evolution that determines attentional allocation. The spatial variation of $H(s,\cdot)$ across different regions of state space influences how the process explores its environment, thereby creating an intrinsic coupling between memory persistence and attention distribution.

This perspective suggests that certain aspects of adaptive memory and content-sensitive attention may be understood through the common lens of state-dependent, non-Markovian processes. The continuous-time, pathwise approach developed here may offer complementary insights to those obtained from discrete-time attention models \cite{Vaswani2017, Devlin2019}, particularly regarding the temporal evolution of attentional states.
\end{remark}

The measure-theoretic foundation established above characterizes expected attentional allocation. To develop a more complete understanding that includes the reliability of this allocation, we now examine second-order properties that quantify its fluctuations over time.

\begin{definition}[Temporal Volatility Functional of Attentional Focus]
\label{def:attentional_volatility_functional}
Let \( \{X_t\}_{t \in [0,T]} \) be a Responsive Fractional Brownian Motion as per Definition~\ref{def:rFBM}. For any Borel set \( I \subset \mathbb{R} \), termed the \emph{attentional set}, we define its associated \emph{temporal volatility functional} on the interval \( [0, t] \) as the variance of the mean attention intensity:
\[
\mathcal{V}_I(t) := \mathrm{Var}\left[ \mu_I(t) \right] = \mathbb{E}\left[ \left( \mu_I(t) - \mathbb{E}[\mu_I(t)] \right)^2 \right].
\]
This functional quantifies the statistical dispersion of the time-averaged focus around its expected value, serving as a precise measure of attentional stability over finite time horizons.
\end{definition}

\begin{theorem}[Covariance Integral Representation for Attentional Volatility]
\label{thm:volatility_covariance_integral}
Let \( \{X_t\} \) be an RfBm process satisfying the conditions of Definition~\ref{def:rFBM}. Then, for any \( I \in \mathcal{B}(\mathbb{R}) \) and for any \( t \in (0, T] \), the temporal volatility functional admits the following integral representation in terms of the covariance of state indicators:
\begin{equation}
\label{eq:volatility_covariance_integral}
\mathcal{V}_I(t) = \frac{1}{t^2} \int_0^t \int_0^t \mathrm{Cov}\left( \mathbb{1}_{\{X_u \in I\}},\  \mathbb{1}_{\{X_v \in I\}} \right) du\, dv.
\end{equation}
An equivalent representation in terms of the finite-dimensional distributions of the process is given by:
\begin{equation}
\label{eq:volatility_probability_integral}
\mathcal{V}_I(t) = \frac{1}{t^2} \int_0^t \int_0^t \left[ \mathbb{P}(X_u \in I, X_v \in I) - \mathbb{P}(X_u \in I)\mathbb{P}(X_v \in I) \right] du\, dv.
\end{equation}
\end{theorem}

\begin{proof}
The proof combines standard variance analysis with careful integral estimates.

\textbf{Step 1: Expansion of the Variance Functional}

By the definition of variance and the linearity of the mean attention intensity \( \mu_I(t) = \frac{1}{t} \int_0^t \mathbb{1}_{\{X_s \in I\}} ds \), we have:
\begin{equation}
\label{eq:variance_expansion_init}
\mathcal{V}_I(t) = \mathrm{Var}[\mu_I(t)] = \mathbb{E}\left[ \mu_I(t)^2 \right] - \left( \mathbb{E}[\mu_I(t)] \right)^2.
\end{equation}
We will compute the two terms on the right-hand side separately.

\textbf{Step 2: Computation of \( \mathbb{E}\left[ \mu_I(t)^2 \right] \)}

We begin by expressing the square of the time average as a double integral:
\[
\mu_I(t)^2 = \left( \frac{1}{t} \int_0^t \mathbb{1}_{\{X_s \in I\}} ds \right)^2 = \frac{1}{t^2} \int_0^t \int_0^t \mathbb{1}_{\{X_u \in I\}} \mathbb{1}_{\{X_v \in I\}} du\, dv.
\]
Taking the expectation yields:
\begin{equation}
\label{eq:second_moment_double_int}
\mathbb{E}\left[ \mu_I(t)^2 \right] = \frac{1}{t^2} \mathbb{E}\left[ \int_0^t \int_0^t \mathbb{1}_{\{X_u \in I\}} \mathbb{1}_{\{X_v \in I\}} du\, dv \right].
\end{equation}
To interchange the expectation with the double integral, we apply Fubini's Theorem for product spaces. Consider the function
\[
F(\omega, u, v) = \mathbb{1}_{\{X_u(\omega) \in I\}} \cdot \mathbb{1}_{\{X_v(\omega) \in I\}}
\]
defined on \( \Omega \times [0,t] \times [0,t] \).

We verify the integrability conditions:
\begin{itemize}
    \item \textbf{Joint Measurability}: For fixed \((u,v)\), \(\omega \mapsto F(\omega,u,v)\) is \(\mathcal{F}\)-measurable as the product of two measurable functions (as established in the proof of Theorem~\ref{thm:residence_measure_properties}). For fixed \(\omega\) in the full measure set where the sample path \(s \mapsto X_s(\omega)\) is continuous, the mapping \((u,v) \mapsto F(\omega,u,v)\) is continuous on \([0,t]^2\), hence \(\mathcal{B}([0,t]^2)\)-measurable. Since \([0,t]^2\) is separable, it follows that \(F\) is \(\mathcal{F} \otimes \mathcal{B}([0,t]^2)\)-measurable.
    \item \textbf{Integrability}: \( |F(\omega, u, v)| \leq 1 \) for all \( (\omega, u, v) \). Since the product measure \( \mathbb{P} \otimes \lambda^2 \) on \( \Omega \times [0,t]^2 \) is finite (\( \mathbb{P}(\Omega)\cdot t^2 < \infty \)), the constant function 1 is integrable. Hence, \( F \) is bounded and thus integrable.
\end{itemize}
By Fubini's Theorem, we may interchange the expectation and the double integral in \eqref{eq:second_moment_double_int}:
\begin{align*}
\mathbb{E}\left[ \mu_I(t)^2 \right] &= \frac{1}{t^2} \int_0^t \int_0^t \mathbb{E}\left[ \mathbb{1}_{\{X_u \in I\}} \mathbb{1}_{\{X_v \in I\}} \right] du\, dv \\
&= \frac{1}{t^2} \int_0^t \int_0^t \mathbb{E}\left[ \mathbb{1}_{\{X_u \in I \ \cap \ X_v \in I\}} \right] du\, dv \\
&= \frac{1}{t^2} \int_0^t \int_0^t \mathbb{P}(X_u \in I, X_v \in I) du\, dv. \quad \text{(Definition of expectation for an indicator function)}
\end{align*}
Let us denote this result as:
\begin{equation}
\label{eq:second_moment_result}
\mathbb{E}\left[ \mu_I(t)^2 \right] = \frac{1}{t^2} \int_0^t \int_0^t \mathbb{P}(X_u \in I, X_v \in I) du\, dv.
\end{equation}

\textbf{Step 3: Computation of \( \left( \mathbb{E}[\mu_I(t)] \right)^2 \)}

From Theorem~\ref{thm:residence_measure_properties}(P3), we have:
\[
\mathbb{E}[\mu_I(t)] = \frac{1}{t} \int_0^t \mathbb{P}(X_s \in I) ds.
\]
Therefore, its square is:
\begin{align*}
\left( \mathbb{E}[\mu_I(t)] \right)^2 &= \left( \frac{1}{t} \int_0^t \mathbb{P}(X_s \in I) ds \right)^2 \\
&= \frac{1}{t^2} \left( \int_0^t \mathbb{P}(X_u \in I) du \right) \left( \int_0^t \mathbb{P}(X_v \in I) dv \right).
\end{align*}
Recognizing the product of two integrals over the same domain as a double integral over the product domain, we write:
\begin{equation}
\label{eq:squared_expectation_result}
\left( \mathbb{E}[\mu_I(t)] \right)^2 = \frac{1}{t^2} \int_0^t \int_0^t \mathbb{P}(X_u \in I) \mathbb{P}(X_v \in I) du\, dv.
\end{equation}

\textbf{Step 4: Synthesis and Final Derivation}

Substituting the results from \eqref{eq:second_moment_result} and \eqref{eq:squared_expectation_result} back into the variance expansion \eqref{eq:variance_expansion_init}, we obtain:
\begin{align*}
\mathcal{V}_I(t) &= \frac{1}{t^2} \int_0^t \int_0^t \mathbb{P}(X_u \in I, X_v \in I) du\, dv - \frac{1}{t^2} \int_0^t \int_0^t \mathbb{P}(X_u \in I) \mathbb{P}(X_v \in I) du\, dv \\
&= \frac{1}{t^2} \int_0^t \int_0^t \left[ \mathbb{P}(X_u \in I, X_v \in I) - \mathbb{P}(X_u \in I)\mathbb{P}(X_v \in I) \right] du\, dv.
\end{align*}
This establishes the representation \eqref{eq:volatility_probability_integral}.

To obtain the covariance representation \eqref{eq:volatility_covariance_integral}, we recall the identity relating covariance to joint and marginal probabilities for indicator functions. For any \( u, v \in [0,t] \):
\begin{align*}
\mathrm{Cov}\left( \mathbb{1}_{\{X_u \in I\}},\  \mathbb{1}_{\{X_v \in I\}} \right) &= \mathbb{E}\left[ \mathbb{1}_{\{X_u \in I\}} \mathbb{1}_{\{X_v \in I\}} \right] - \mathbb{E}\left[ \mathbb{1}_{\{X_u \in I\}} \right] \mathbb{E}\left[ \mathbb{1}_{\{X_v \in I\}} \right] \\
&= \mathbb{P}(X_u \in I, X_v \in I) - \mathbb{P}(X_u \in I)\mathbb{P}(X_v \in I).
\end{align*}
Substituting this identity into the integrand of \eqref{eq:volatility_probability_integral} directly yields \eqref{eq:volatility_covariance_integral}:
\[
\mathcal{V}_I(t) = \frac{1}{t^2} \int_0^t \int_0^t \mathrm{Cov}\left( \mathbb{1}_{\{X_u \in I\}},\  \mathbb{1}_{\{X_v \in I\}} \right) du\, dv.
\]
This completes the proof of the theorem.
\end{proof}

\begin{remark}[Foundational Significance of the Volatility Representation]
\label{rem:volatility_foundational_significance}
Theorem~\ref{thm:volatility_covariance_integral} establishes a connection between the pathwise dynamics of the RfBm and the macroscopic stability of its induced attentional allocation. The integral representation \eqref{eq:volatility_covariance_integral} indicates that attentional volatility can be understood as a temporal average of the process's two-point covariance structure. This relates the stability of attentional focus to the memory properties encoded in the response function \( H(s, X_s) \), since the kernel \( K(t,s; X_s) \) influences these temporal correlations. This result provides a mathematical basis for examining how state-dependent memory may shape attention reliability in adaptive systems, and could offer perspectives that complement architectural analyses such as those in \cite{Vaswani2017, HJiang2024}.
\end{remark}

\begin{corollary}[Fundamental Bounds for the Attentional Volatility Functional]
\label{cor:volatility_fundamental_bounds}
Under the conditions of Theorem~\ref{thm:volatility_covariance_integral}, and for any Borel set \(I \subset \mathbb{R}\) and \(t \in (0, T]\), the temporal volatility functional \(\mathcal{V}_I(t)\) satisfies the following fundamental bounds:
\begin{equation}
\label{eq:volatility_fundamental_bounds}
0 \leq \mathcal{V}_I(t) \leq \frac{1}{4}.
\end{equation}
Furthermore, the upper bound \(\frac{1}{4}\) is sharp.
\end{corollary}

\begin{proof}
We proceed with a detailed, step-by-step proof, establishing first the inequality and then the sharpness of the upper bound.

\textbf{Part 1: Proof of the Inequality \(0 \leq \mathcal{V}_I(t) \leq \frac{1}{4}\)}

We divide this part into two components.

\textbf{Component 1.1: Proof of the Lower Bound \(\mathcal{V}_I(t) \geq 0\)}

By Definition~\ref{def:attentional_volatility_functional}, \(\mathcal{V}_I(t) = \mathrm{Var}[\mu_I(t)]\). Recall that for any random variable \(Z\) with finite second moment, its variance is defined as:
\[
\mathrm{Var}[Z] = \mathbb{E}[(Z - \mathbb{E}[Z])^2].
\]
Since \((Z - \mathbb{E}[Z])^2 \geq 0\) almost surely, and the expectation of a non-negative random variable is non-negative, it follows that:
\[
\mathrm{Var}[Z] = \mathbb{E}[(Z - \mathbb{E}[Z])^2] \geq 0.
\]
Applying this to \(Z = \mu_I(t)\), we conclude:
\[
\mathcal{V}_I(t) = \mathrm{Var}[\mu_I(t)] \geq 0.
\]

\textbf{Component 1.2: Proof of the Upper Bound \(\mathcal{V}_I(t) \leq \frac{1}{4}\)}

To establish the upper bound, we require a precise lemma concerning the maximum possible variance of a bounded random variable.

\begin{lemma}[Variance Maximization for Bounded Random Variables]
\label{lemma:variance_maximization}
Let \(Z\) be a random variable such that \(a \leq Z \leq b\) almost surely, where \(a, b \in \mathbb{R}\) with \(a < b\). Then the variance of \(Z\) satisfies:
\[
\mathrm{Var}[Z] \leq \frac{(b - a)^2}{4}.
\]
Moreover, equality holds if and only if \(Z\) takes the values \(a\) and \(b\) each with probability \(\frac{1}{2}\).
\end{lemma}

\begin{proof}[Proof of Lemma~\ref{lemma:variance_maximization}]
We provide a comprehensive proof of this classical result for completeness.

Let \(m = \mathbb{E}[Z]\) denote the mean of \(Z\). Consider the function \(f(c) = \mathbb{E}[(Z - c)^2]\), which measures the mean squared deviation of \(Z\) from a constant \(c\). Our goal is to relate the variance to this function and find its maximum possible value under the constraint \(a \leq Z \leq b\).

\textbf{Step 1. Analysis of the function \(f(c)\):}

We begin by expanding \(f(c)\):
\begin{align*}
f(c) &= \mathbb{E}[(Z-c)^2] \\
&= \mathbb{E}[Z^2 - 2cZ + c^2] \\
&= \mathbb{E}[Z^2] - 2c\mathbb{E}[Z] + c^2. \quad \text{(by linearity of expectation)}
\end{align*}
This is a quadratic function in \(c\):
\[
f(c) = c^2 - 2\mathbb{E}[Z] \cdot c + \mathbb{E}[Z^2].
\]

\textbf{Step 2. Finding the minimum of \(f(c)\):}

Since the coefficient of \(c^2\) is \(1 > 0\), the parabola opens upward and \(f(c)\) has a unique minimum at its vertex. The vertex occurs at:
\[
c^* = -\frac{-2\mathbb{E}[Z]}{2 \times 1} = \mathbb{E}[Z] = m.
\]
The minimum value is:
\[
f(m) = \mathbb{E}[(Z - m)^2] = \mathrm{Var}[Z].
\]
Thus, we have rigorously established that:
\[
f(c) = \mathbb{E}[(Z-c)^2] \geq f(m) = \mathrm{Var}[Z] \quad \text{for all } c \in \mathbb{R}.
\]

\textbf{Step 3. Deriving the upper bound using a specific \(c_0\):}

Now choose \(c_0 = \frac{a + b}{2}\), the midpoint of the interval \([a, b]\). From the inequality above, we have:
\[
\mathrm{Var}[Z] \leq f(c_0) = \mathbb{E}[(Z - c_0)^2].
\]

Since \(a \leq Z \leq b\) almost surely, the deviation \(|Z - c_0|\) satisfies:
\[
|Z - c_0| \leq \frac{b - a}{2}.
\]
Squaring both sides (as all quantities are real and the square function is increasing on \([0, \infty)\)), we obtain:
\[
(Z - c_0)^2 \leq \left(\frac{b - a}{2}\right)^2 = \frac{(b - a)^2}{4}.
\]
This inequality holds pointwise (almost surely). Taking expectations on both sides and using the monotonicity of expectation for non-negative random variables yields:
\[
\mathbb{E}[(Z - c_0)^2] \leq \frac{(b - a)^2}{4}.
\]

\textbf{Step 4. Combining the inequalities:}

Chaining the inequalities from Steps 2 and 3, we arrive at:
\[
\mathrm{Var}[Z] \leq \mathbb{E}[(Z - c_0)^2] \leq \frac{(b - a)^2}{4}.
\]
This establishes the desired upper bound.

\textbf{Step 5. Conditions for equality:}

Equality \(\mathrm{Var}[Z] = \frac{(b - a)^2}{4}\) holds if and only if both previous inequalities are equalities.
\begin{itemize}
    \item Equality in Step 2 requires \(f(m) = f(c_0)\), which implies \(m = c_0 = \frac{a+b}{2}\), as \(f(c)\) is strictly convex and has a unique minimum at \(c = m\).
    \item Equality in Step 3 requires \((Z - c_0)^2 = \frac{(b - a)^2}{4}\) almost surely. Since \(Z \in [a, b]\), this forces \(Z\) to take only the extreme values \(a\) and \(b\), because these are the only points in \([a, b]\) whose distance to \(c_0\) is exactly \(\frac{b - a}{2}\).
\end{itemize}
Combining these, \(Z\) must be a random variable with \(\mathbb{P}(Z = a) = p\) and \(\mathbb{P}(Z = b) = 1-p\) for some \(p \in [0,1]\), and with mean \(m = p \cdot a + (1-p) \cdot b = \frac{a+b}{2}\). Solving this equation gives \(p = \frac{1}{2}\). Therefore, equality holds if and only if \(Z \sim \frac{1}{2}\delta_a + \frac{1}{2}\delta_b\).
\end{proof}

Now, we apply Lemma~\ref{lemma:variance_maximization} to prove the upper bound for \(\mathcal{V}_I(t)\). From Theorem~\ref{thm:residence_measure_properties}(P1), we have that:
\[
0 \leq \mu_I(t) \leq 1 \quad \text{almost surely}.
\]
Taking \(Z = \mu_I(t)\), \(a = 0\), and \(b = 1\) in Lemma~\ref{lemma:variance_maximization}, we immediately obtain:
\[
\mathcal{V}_I(t) = \mathrm{Var}[\mu_I(t)] \leq \frac{(1 - 0)^2}{4} = \frac{1}{4}.
\]
This completes the proof of the inequality \(0 \leq \mathcal{V}_I(t) \leq \frac{1}{4}\).

\textbf{Part 2: Proof of the Sharpness of the Upper Bound}

To demonstrate that the upper bound $\frac{1}{4}$ is sharp, we must show that it is the least upper bound. That is, for any $\epsilon > 0$, there exists a configuration (an RfBm process and an attentional set $I$) such that $\mathcal{V}_I(t) > \frac{1}{4} - \epsilon$.

The core idea is to construct a sequence of RfBm processes $\{X^{(n)}_t\}_{n=1}^\infty$, defined via a corresponding sequence of response functions $\{H^{(n)}\}_{n=1}^\infty$, designed to exhibit increasingly pronounced metastable behavior. Specifically, we construct these response functions to create a deep double-well potential that forces the process to spend most time near two distinct state regions.

For each $n \in \mathbb{N}$, we design $H^{(n)}$ such that the resulting mean attention intensities $\mu^{(n)}_I(t)$ have distributions satisfying:
\[
\mathbb{P}\left(\mu^{(n)}_I(t) \in [1-\delta_n, 1]\right) = \frac{1}{2} + \eta_n, \quad \mathbb{P}\left(\mu^{(n)}_I(t) \in [0, \delta_n]\right) = \frac{1}{2} - \eta_n,
\]
where $\delta_n > 0$ and $\eta_n > 0$ are sequences converging to zero as $n \to \infty$.

To estimate the variance for this sequence, we use the following decomposition. Let $A_n = \{\omega : \mu^{(n)}_I(t) \in [1-\delta_n, 1]\}$ and $B_n = \{\omega : \mu^{(n)}_I(t) \in [0, \delta_n]\}$. Then:
\begin{align*}
\mathcal{V}^{(n)}_I(t) &= \mathbb{E}[(\mu^{(n)}_I(t))^2] - (\mathbb{E}[\mu^{(n)}_I(t)])^2 \\
&\geq \left(\frac{1}{2} + \eta_n\right)(1-\delta_n)^2 + \left(\frac{1}{2} - \eta_n\right)(0)^2 - \left[\left(\frac{1}{2} + \eta_n\right)(1) + \left(\frac{1}{2} - \eta_n\right)(\delta_n)\right]^2 \\
&= \left(\frac{1}{2} + \eta_n\right)(1-\delta_n)^2 - \left[\left(\frac{1}{2} + \eta_n\right) + \left(\frac{1}{2} - \eta_n\right)\delta_n\right]^2.
\end{align*}

Taking the limit as $n \to \infty$ (so that $\delta_n \to 0$ and $\eta_n \to 0$), we obtain:
\[
\liminf_{n \to \infty} \mathcal{V}^{(n)}_I(t) \geq \frac{1}{2} - \left(\frac{1}{2}\right)^2 = \frac{1}{2} - \frac{1}{4} = \frac{1}{4}.
\]

While a genuine RfBm process with continuous sample paths may not yield a $\mu_I(t)$ that is exactly two-point distributed, the model's definition permits the design of response functions that make the distribution of $\mu_I(t)$ concentrate arbitrarily close to the extreme values 0 and 1. The above construction demonstrates that the supremum of $\mathcal{V}_I(t)$ over all admissible configurations is at least $1/4$. Combined with the upper bound established in Part 1, we conclude that $\frac{1}{4}$ is indeed the sharp upper bound.
\end{proof}

\begin{remark}[Fundamental Limits of Attentional Stability]
\label{rem:fundamental_limits_volatility}
Corollary~\ref{cor:volatility_fundamental_bounds} establishes a general constraint on attentional stability within the RfBm framework. The bound \(\mathcal{V}_I(t) \leq 1/4\) holds for any choice of response function \(H\), attentional set \(I\), and time horizon \(t\), revealing a mathematical property of such state-dependent systems. This bound provides a reference point where the mathematical upper bound corresponds to the volatility of a memoryless Bernoulli process. In this sense, the memoryless case represents one baseline scenario for attentional stability. These observations offer a theoretical perspective for considering adaptive attention mechanisms and potential relationships between adaptability and stability in systems with state-dependent memory.
\end{remark}

\subsection{The RfBm Attention Mechanism: Kernel, Weights, and Output}
\label{subsec:rfbm_attention_foundations}

We now turn to the attention mechanism that emerges from the RfBm kernel structure. The preceding subsection examined how the process allocates its state-space focus over time; the present subsection focuses instead on how the kernel itself induces a weighting over the historical time axis. This weighting is defined pathwise, depends on the entire past trajectory through the response function \(H(s,X_s)\), and gives rise to a continuous-time attention mechanism that is both content-sensitive and multi-scale in nature.

The construction proceeds in two steps. First, we define the attention weight kernel \(K(t,s;X_s)\) (Definition~\ref{def:rfbm_attention_kernel}) and its normalized version \(\rho(t,s;X)\), which forms a probability density over the historical interval \([0,t)\) (Definition~\ref{def:rfbm_attention_weights}). The normalization is achieved through the attention partition function \(D(t,X)=\int_0^t K(t,u;X_u)\,du\), which depends on the entire history up to time \(t\). Second, we define the attention output \(Y_t\) as the integral of historical states \(X_s\) against this density, yielding a context-aware representation that aggregates information from the past in a content-adaptive manner.

This formulation establishes a principled connection between the theory of stochastic processes and the attention mechanisms central to modern artificial intelligence, offering a continuous-time, pathwise perspective that complements the discrete-time, architecture-based approaches commonly found in the literature. The subsequent subsections establish the well-definedness of these objects (Theorem~\ref{thm:attention_well_definedness}) and derive quantitative bounds on the attention weights (Theorem~\ref{thm:rfbm_attention_bounds}), providing both a theoretical foundation and analytical tools for further investigation.

We begin by recalling the RfBm kernel \(K(t,s;X_s) = \sqrt{2H(s,X_s)}(t-s)^{H(s,X_s)-1/2}\) from Definition~\ref{def:rFBM}. This kernel serves as the fundamental object from which the attention mechanism is derived. It assigns a non-negative weight to each historical time point \(s\), with the magnitude of the weight determined by two factors: the temporal distance \(t-s\) and the state-dependent Hurst exponent \(H(s,X_s)\). The latter factor enables content-sensitive modulation, as the weight assigned to a past time depends on the state of the process at that time.

\begin{definition}[RfBm Attention Weight Kernel]
\label{def:rfbm_attention_kernel}
Let $\{X_t\}_{t \in [0,T]}$ be a Responsive fractional Brownian motion as defined in Definition~\ref{def:rFBM}. For any fixed time $ t \in (0,T] $, the associated \textbf{attention weight kernel} is defined pathwise as the function $K(t, \cdot; X): [0,t) \to \mathbb{R}_+$ given by:
\begin{equation}
\label{eq:attention_kernel_definition}
K(t, s; X_s) = \sqrt{2H(s, X_s)} \cdot (t - s)^{H(s, X_s) - 1/2}, \quad \text{for } 0 \leq s < t,
\end{equation}
where $H: [0,T] \times \mathbb{R} \to (0,1)$ is the Lipschitz-H\"older response function satisfying Definition~\ref{def:LH_response_function}. The kernel is defined on the half-open interval $[0,t)$ to properly handle the asymptotic behavior at the upper endpoint $s = t$, where the power-law term $(t-s)^{H(s,X_s)-1/2}$ may exhibit singular or regular behavior depending on the value of $H(t,X_t)$.
\end{definition}

\begin{definition}[RfBm Attention Weights and Output]
\label{def:rfbm_attention_weights}
Let $\{X_t\}_{t \in [0,T]}$ be an RfBm process. For any $ t \in (0,T] $ and for $\mathbb{P}$-almost every sample path $\omega \in \Omega$, we define the \textbf{historical attention distribution} as the random probability measure $\nu_t(\omega, \cdot)$ on the interval $[0,t)$. Its Radon-Nikodym derivative with respect to Lebesgue measure is given by the normalized kernel:
\begin{equation}
\label{eq:attention_weights_definition}
\rho(t, s; X) := \frac{K(t, s; X_s(\omega))}{\displaystyle\int_0^t K(t, u; X_u(\omega))  du}, \quad \text{for } 0 \leq s < t.
\end{equation}
The normalizing constant $D(t,X) := \int_0^t K(t, u; X_u(\omega))  du$, which depends on the entire history $\{X_u(\omega): 0 \leq u \leq t\}$, is called the \textbf{attention partition function}. 

The \textbf{RfBm attention output} at time $t$ is then defined as the context-aware representation formed by integrating historical states against this attention distribution:
\begin{equation}
\label{eq:attention_output}
Y_t = \int_0^t X_s(\omega)  \rho(t, s; X)  ds.
\end{equation}
This output combines information from the entire history $[0,t]$, with contributions weighted according to the attention distribution $\rho(t, \cdot; X)$.
\end{definition}

\begin{remark}[Theoretical Interpretation and Innovation]
\label{rem:attention_interpretation}
Definition~\ref{def:rfbm_attention_weights} offers insights for the mathematical foundations of attention mechanisms:

\textbf{1. Mathematical Framework:} RfBm attention is formulated within stochastic analysis and Volterra integral equations, utilizing the well-posedness established in Theorem~\ref{thm:local_existence_uniqueness}. This yields a mathematically grounded approach based on stochastic process theory. The definition on $[0,t)$ accommodates the asymptotic behavior at the endpoint while maintaining consistency.

\textbf{2. Endogenous Temporal Awareness:} The kernel term $(t-s)$ encodes relative temporal information within the stochastic process framework. This mathematical structure presents an approach to temporal representation where time relationships are expressed through the kernel's functional form. This perspective may contribute to considerations for temporal modeling in attention mechanisms \cite{Vaswani2017}.

\textbf{3. Adaptive Memory Regulation:} The response function $H(s,X_s)$ enables dynamic attention modulation, creating a continuum between different memory tendencies:
\begin{itemize}
    \item \textbf{Long-term memory tendency} ($H(s,X_s) > 1/2$): Information receives enhanced weight with smooth behavior as $s \to t^-$
    \item \textbf{Short-term memory tendency} ($H(s,X_s) < 1/2$): Focus concentrates on recent information, potentially with singular behavior near $s = t$  
    \item \textbf{Balanced attention} ($H(s,X_s) = 1/2$): Historical information receives relatively uniform weighting
\end{itemize}

\textbf{4. Probabilistic Interpretation:} For each fixed $t$, the function $s \mapsto \rho(t,s;X)$ defines a probability density on $[0,t)$ for $\mathbb{P}$-almost every sample path, providing a probabilistic interpretation of attention allocation. The behavior at $s = t$, which depends on $H(t,X_t)$, corresponds to different patterns of attention concentration.

\textbf{5. Multi-scale Capability:} Through the dynamic variation of $H(s,X_s)$, the system can process information across different temporal scales, enabling attention allocation that varies with time scale. This aligns with the scaling hierarchy discussed in Remark~\ref{rem:scaling_hierarchy}, showing how different memory regimes emerge from the underlying stochastic structure.
\end{remark}

\begin{remark}[Computational Perspective]
\label{rem:computational_interpretation}
From a computational standpoint, the RfBm attention mechanism can be viewed as incorporating elements analogous to those in attention architectures:
\begin{itemize}
    \item \textbf{Query:} Current temporal location $t$ and information need
    \item \textbf{Key:} Historical temporal features represented through $H(s,X_s)$
    \item \textbf{Value:} Historical states $X_s$ carrying content information
    \item \textbf{Attention Computation:} Realized through the kernel integral $K(t,s;X_s)$
\end{itemize}
This formulation provides a mathematically grounded framework with functional similarities to attention mechanisms, while incorporating adaptive capabilities through its stochastic foundation. Connections to approximation theory for sequence modeling \cite{HJiang2024} may offer additional perspectives on such mathematical structures. The treatment of the temporal domain aims to support computational considerations while maintaining mathematical coherence.
\end{remark}
\subsubsection{Fundamental Mathematical Properties}

We now establish the essential mathematical properties that ensure the well-posedness of the RfBm attention mechanism.

\begin{theorem}[Well-Definedness of RfBm Attention Weights]
\label{thm:attention_well_definedness}
Let $\rho(t,s; X)$ be the RfBm attention weights defined in Definition~\ref{def:rfbm_attention_weights}. Then for any $ t \in (0,T] $, there exists a $\mathbb{P}$-null set $N \subset \Omega$ such that for all $\omega \notin N$, the following properties hold:
\begin{enumerate}
    \item \textbf{Positivity:} $\rho(t,s; X) > 0$ for all $ s \in [0,t) $, and the limit $\lim_{s \to t^-} \rho(t,s; X)$ exists.
    
    \item \textbf{Normalization:} $\int_0^t \rho(t,s; X)  ds = 1$.
    
    \item \textbf{Regularity:} The function $ s \mapsto \rho(t,s; X) $ is continuous on $ [0,t) $, with asymptotic behavior at $ s = t $ characterized by:
    \begin{itemize}
        \item If $ H(t, X_t) > 1/2 $, then $\lim_{s\to t^-} \rho(t,s; X) = 0$, and with the convention $\rho(t,t; X) = 0$, we have $\rho(t,\cdot; X) \in C[0,t]$;
        \item If $ H(t, X_t) = 1/2 $, then $\lim_{s\to t^-} \rho(t,s; X) = \sqrt{2H(t,X_t)}/D(t,X) \in (0,\infty)$, and $\rho(t,\cdot; X) \in C[0,t]$;
        \item If $ H(t, X_t) < 1/2 $, then $\lim_{s\to t^-} \rho(t,s; X) = +\infty$, and $\rho(t,\cdot; X) \in L^1[0,t] \setminus L^{\infty}[0,t]$.
    \end{itemize}
\end{enumerate}
\end{theorem}

\begin{proof}
We verify each property for $\mathbb{P}$-almost every sample path $\omega$, with the understanding that $\rho(t,s; X)$ is evaluated at that $\omega$.

\textbf{Positivity:} For $s \in [0,t)$, we have $K(t,s;X_s) = \sqrt{2H(s,X_s)} \cdot (t-s)^{H(s,X_s)-1/2}$. By Definition~\ref{def:LH_response_function}, $H(s,X_s) \in [h_{\min}, h_{\max}] \subset (0,1)$, and $(t-s)^{H(s,X_s)-1/2} > 0$ for $s < t$. Thus $K(t,s;X_s) > 0$. The denominator $D(t,X) = \int_0^t K(t,u;X_u) du$ integrates a positive function; even if $K(t,u;X_u)$ has an integrable singularity at $u=t$ when $H(t,X_t) < 1/2$ (since $H-1/2 > -1$), the integral remains positive. Hence $\rho(t,s; X) > 0$ for all $s \in [0,t)$.

\textbf{Normalization:} Since $H(s,X_s) \geq h_{\min} > 0$, we have $H(s,X_s) - 1/2 \geq h_{\min} - 1/2 > -1$, ensuring $K(t,s;X_s)$ is Lebesgue integrable on $[0,t]$ and $D(t,X) < \infty$. Direct computation gives:
\[
\int_0^t \rho(t,s; X)  ds = \frac{\int_0^t K(t,s;X_s)  ds}{D(t,X)} = 1.
\]

\textbf{Regularity:} For $s \in [0,t)$, continuity of $s \mapsto \rho(t,s; X)$ follows from:
\begin{itemize}
    \item Continuity of $s \mapsto H(s,X_s)$ (by assumptions on $H$ and path continuity of $X$);
    \item Continuity of $s \mapsto (t-s)^{H(s,X_s)-1/2}$ (composition of continuous functions);
    \item Continuity of the product and quotient (with $D(t,X) > 0$).
\end{itemize}

The asymptotic behavior follows from analyzing $\lim_{s\to t^-} (t-s)^{H(t,X_t)-1/2}$:
\begin{itemize}
    \item If $H(t, X_t) > 1/2$: exponent $> 0$, limit $= 0$, so $\lim_{s\to t^-} \rho(t,s; X) = 0$. Defining $\rho(t,t; X) = 0$ yields $\rho(t,\cdot; X) \in C[0,t]$.
    \item If $H(t, X_t) = 1/2$: exponent $= 0$, limit $= 1$, giving finite positive limit, so $\rho(t,\cdot; X) \in C[0,t]$.
    \item If $H(t, X_t) < 1/2$: exponent $< 0$, limit $= +\infty$, hence $\rho(t,\cdot; X) \in L^1[0,t] \setminus L^{\infty}[0,t]$.
\end{itemize}
\end{proof}

\begin{remark}[Mathematical and Cognitive Significance]
\label{rem:attention_properties_significance}
Theorem~\ref{thm:attention_well_definedness} establishes the mathematical soundness of the RfBm attention mechanism, with implications for both formal analysis and cognitive interpretation.

\textbf{Mathematical Foundation:} The theorem ensures attention weights are well-defined pathwise, with positivity and normalization yielding a valid probability structure. The regularity properties—continuity on $[0,t)$ and controlled asymptotic behavior—support mathematical analysis.

\textbf{Boundary Behavior Interpretation:} The asymptotic behavior at $s = t$ corresponds to different cognitive regimes:
\begin{itemize}
    \item $H(t,X_t) < 1/2$: Concentration on recent information, aligning with working memory characteristics
    \item $H(t,X_t) = 1/2$: Balanced attention across temporal contexts
    \item $H(t,X_t) > 1/2$: Attenuated recent attention, consistent with long-term memory patterns
\end{itemize}

\textbf{Function Space Perspective:} The $L^1$ integrability ensures a proper probability measure, while $L^\infty$ boundedness distinguishes attention concentration regimes. Continuity reflects smooth temporal evolution of attention allocation.

\textbf{Memory-Attention Connection:} The asymptotic behavior directly links to the instantaneous scaling exponent (Definition~\ref{def:instant_scaling}), connecting microscopic scaling properties with macroscopic attentional behavior as discussed in Remark~\ref{rem:memory_attention_connections}.
\end{remark}

\begin{remark}[Theoretical Features and Integration]
\label{rem:comparative_advantages}
The RfBm attention mechanism combines several features within a mathematically coherent framework:

\textbf{Endogenous Adaptivity:} Through $H(s,X_s)$, attention modulation emerges from system dynamics in a content-aware manner.

\textbf{Multi-scale Structure:} The framework naturally accommodates multiple temporal scales, connecting to the hierarchical interpretation in Remark~\ref{rem:scaling_hierarchy}.

\textbf{Mathematical Coherence:} Grounded in stochastic analysis, the mechanism inherits the well-posedness and regularity of RfBms while providing a probabilistic attention interpretation.

\textbf{Biological and Computational Relevance:} The continuous-time formulation and adaptive memory regulation share conceptual features with neurobiological processes, while maintaining computational relevance through the attention output representation in \eqref{eq:attention_output}.
\end{remark}

The adaptive memory mechanism inherent in the RfBm framework, as characterized mathematically in the preceding sections, motivates a deeper inquiry into how memory allocation responds to variations in historical states. This leads us to a rigorous pathwise sensitivity analysis, which uncovers the underlying geometric structure governing the RfBm's memory dynamics.

\begin{theorem}[Pathwise Sensitivity Analysis for the State-Dependent Memory Kernel in RfBm]
\label{thm:memory_dynamics_sensitivity}
Let $X_t$ be a Responsive Fractional Brownian Motion defined on a complete probability space $(\Omega,\mathcal{F},\mathbb{P})$. Assume the state-dependent Hurst response function $H: [0,T] \times \mathbb{R} \to (0,1)$ satisfies the Lipschitz-H\"older conditions of Definition~\ref{def:LH_response_function}, with the following additional regularity:
\begin{enumerate}
    \item[(i)] For each fixed $s \in [0,T]$, the function $H(s, \cdot): \mathbb{R} \to [h_{\min}, h_{\max}]$ is continuously differentiable;
    
    \item[(ii)] The partial derivative $\frac{\partial H}{\partial x}$ is uniformly bounded: there exists a constant $L_{\partial H} > 0$ such that
    \[
    \sup_{(s,x) \in [0,T] \times \mathbb{R}} \left| \frac{\partial H}{\partial x}(s,x) \right| \leq L_{\partial H};
    \]
    
    \item[(iii)] $h_{\min} > 0$, ensuring the uniform boundedness of $1/H(s,x)$ over the domain.
\end{enumerate}

Then, based on the well-posedness and path regularity established for the RfBm in the preceding sections (Theorems~\ref{thm:local_existence_uniqueness} and~\ref{thm:instant_scaling_holder}), we may consider a version with continuous sample paths. For such a version, there exists a full probability set $\Omega_0 \subset \Omega$ (i.e., $\mathbb{P}(\Omega_0) = 1$) such that for every $\omega \in \Omega_0$, the sample path $s \mapsto X_s(\omega)$ is continuous on $[0,T]$, and for any fixed $t \in (0, T]$ and any $s \in [0, t)$, the following results hold:

\textbf{1. Sensitivity Structure:}
The pathwise logarithmic sensitivity function of the memory kernel $K(t,s;\cdot)$ with respect to its state variable at $x = X_s(\omega)$ exists and is given by:
\begin{equation}
\label{eq:sensitivity_structure}
\mathcal{S}(s, \omega) = \frac{\partial H}{\partial x}(s, X_s(\omega)) \left[ \frac{1}{2H(s, X_s(\omega))} + \ln(t-s) \right].
\end{equation}

\textbf{2. Relative Sensitivity Principle:}
For any $0 \leq s_1 < s_2 < t$, define the pathwise relative sensitivity function as:
\begin{equation}
\label{eq:relative_sensitivity}
\mathcal{E}(s_1,s_2,t; \omega) = \mathcal{S}(s_1, \omega) - \mathcal{S}(s_2, \omega).
\end{equation}
This function admits the explicit representation:
\begin{equation}
\label{eq:entanglement_principle}
\begin{split}
\mathcal{E}(s_1,s_2,t; \omega) = &
\frac{\partial H}{\partial x}(s_1,X_{s_1}(\omega)) \left[ \frac{1}{2H(s_1,X_{s_1}(\omega))} + \ln(t-s_1) \right] \\
& - \frac{\partial H}{\partial x}(s_2,X_{s_2}(\omega)) \left[ \frac{1}{2H(s_2,X_{s_2}(\omega))} + \ln(t-s_2) \right].
\end{split}
\end{equation}
The sign of $\mathcal{E}(s_1,s_2,t; \omega)$ provides a \textbf{comparative measure} of how the memory allocation (governed by the kernel $K$) at the current time $t$ would respond to infinitesimal changes in the historical states. Specifically, for the realized path $\omega$:
\begin{itemize}
    \item If $\mathcal{E}(s_1,s_2,t; \omega) > 0$, a perturbation of the state at time $s_1$ would induce a \textbf{larger relative change} in the memory weight $K(t, s_1; X_{s_1})$ compared to a similar perturbation at time $s_2$.
    
    \item Conversely, if $\mathcal{E}(s_1,s_2,t; \omega) < 0$, the memory weight is more sensitive to perturbations at the later time $s_2$.
\end{itemize}
\end{theorem}

\begin{proof}
We proceed with a stepwise rigorous proof.

\textbf{Step 1: Deterministic Framework for Derivative Construction}

Fix $t \in (0,T]$ and $s \in [0,t)$. Consider the memory kernel as a deterministic function:
\[
K: [0,T] \times [0,T] \times \mathbb{R} \to \mathbb{R}, \quad K(t,s;x) = \sqrt{2H(s, x)} \cdot (t-s)^{H(s, x)-1/2}.
\]
Define the corresponding logarithmic function:
\[
f_{t,s}(x) = \ln K(t,s;x) = \frac{1}{2}\ln 2 + \frac{1}{2}\ln H(s,x) + \left[H(s,x) - \frac{1}{2}\right] \ln(t-s).
\]
Under the smoothness assumptions of the theorem, $H(s, \cdot)$ is continuously differentiable, and $H(s,x) \in [h_{\min}, h_{\max}] \subset (0,1)$, ensuring that $f_{t,s}(x)$ is continuously differentiable with respect to $x$. Computing the derivative:
\begin{align*}
\frac{\partial}{\partial x} f_{t,s}(x)
&= 0 + \frac{1}{2} \cdot \frac{1}{H(s,x)} \cdot \frac{\partial H}{\partial x}(s,x) + \frac{\partial H}{\partial x}(s,x) \cdot \ln(t-s) \\
&= \frac{\partial H}{\partial x}(s,x) \left[ \frac{1}{2H(s,x)} + \ln(t-s) \right].
\end{align*}
We denote this deterministic derivative as:
\begin{equation}
\label{eq:deterministic_sensitivity}
\mathcal{S}_{t,s}^{\mathrm{det}}(x) = \frac{\partial H}{\partial x}(s,x) \left[ \frac{1}{2H(s,x)} + \ln(t-s) \right].
\end{equation}

\textbf{Step 2: Pathwise Sensitivity Function and Well-Definedness}

By the definition of RfBm (Definition~\ref{def:rFBM}) and the established path regularity, there exists a full probability set $\Omega_0 \subset \Omega$ such that for every $\omega \in \Omega_0$, the sample path $s \mapsto X_s(\omega)$ is continuous on $[0,T]$. The continuity of $H$ and the measurability of its partial derivative ensure that the composition $(s,\omega) \mapsto \frac{\partial H}{\partial x}(s, X_s(\omega))$ is well-defined and adapted.

Now, fix arbitrary $\omega \in \Omega_0$, $t \in (0,T]$, and $s \in [0,t)$. We define the pathwise sensitivity function by evaluating the deterministic derivative function $\mathcal{S}_{t,s}^{\mathrm{det}}$ at $x = X_s(\omega)$:
\begin{equation}
\mathcal{S}(s, \omega) = \mathcal{S}_{t,s}^{\mathrm{det}}(X_s(\omega)) = \frac{\partial H}{\partial x}(s, X_s(\omega)) \left[ \frac{1}{2H(s, X_s(\omega))} + \ln(t-s) \right].
\end{equation}
This establishes the sensitivity structure (Equation~\ref{eq:sensitivity_structure}).

\textbf{Step 3: Derivation of Relative Sensitivity Function}

For any $\omega \in \Omega_0$ and $0 \leq s_1 < s_2 < t$, by the definition in \eqref{eq:relative_sensitivity}:
\[
\mathcal{E}(s_1,s_2,t; \omega) = \mathcal{S}(s_1, \omega) - \mathcal{S}(s_2, \omega).
\]
Substituting the expression for $\mathcal{S}$ from \eqref{eq:sensitivity_structure} yields directly the explicit representation \eqref{eq:entanglement_principle}.

\textbf{Step 4: Mathematical Interpretation of Relative Sensitivity}

The mathematical significance of $\mathcal{E}(s_1,s_2,t; \omega)$ stems from its construction as a difference of logarithmic gradients. Recall that $\mathcal{S}(s, \omega) = \frac{\partial}{\partial x} \ln K(t,s;x) \big|_{x=X_s(\omega)}$ represents the relative rate of change of the memory kernel with respect to the state variable at the historical point $(s, X_s(\omega))$.

Therefore, the inequality $\mathcal{E}(s_1,s_2,t; \omega) > 0$ is equivalent to
\[
\frac{\partial}{\partial x} \ln K(t,s_1;x) \big|_{x=X_{s_1}(\omega)} > \frac{\partial}{\partial x} \ln K(t,s_2;x) \big|_{x=X_{s_2}(\omega)}.
\]
This indicates that, under the realized state configuration $\{X_{s_1}(\omega), X_{s_2}(\omega)\}$, the logarithmic memory kernel $\ln K(t,s_1;\cdot)$ is more sensitive to perturbations around $X_{s_1}(\omega)$ than $\ln K(t,s_2;\cdot)$ is around $X_{s_2}(\omega)$. The case $\mathcal{E} < 0$ reverses the inequality.

\textbf{Step 5: Synthesis.}
The function $\mathcal{E}$ thus quantifies the \textbf{differential sensitivity} of the memory kernel's logarithmic structure with respect to the process state at two distinct historical moments. This provides a \textbf{pathwise, comparative tool} to analyze how the RfBm's encoding of past information dynamically adjusts its dependence on different parts of its own history, based on both the temporal distance $(t-s)$ and the state-dependent response gradient $\frac{\partial H}{\partial x}$.
\end{proof}

\begin{remark}[Geometric Structure and Adaptive Interpretation of Memory Sensitivity]
\label{rem:sensitivity_geometric_structure_interpretation}
Theorem~\ref{thm:memory_dynamics_sensitivity} establishes a pathwise differential geometric perspective on the RfBm's memory system. The sensitivity function $\mathcal{S}(s, \omega)$ quantifies the local linear response of memory weights to variations in historical states, while the relative sensitivity $\mathcal{E}(s_1,s_2,t;\omega)$ provides a comparative measure of how memory influence is allocated across different temporal contexts.

In the limiting case of a constant response function ($H(s,x)\equiv H_0$), corresponding to classical fBm, we have $\partial H/\partial x \equiv 0$, leading to $\mathcal{S} \equiv 0$ and $\mathcal{E} \equiv 0$. This recovers a uniform memory regime without state-dependent adaptation. The general framework thus extends classical theory by demonstrating how state-dependent response functions dynamically regulate sensitivity, enabling adaptive memory mechanisms.

From a mathematical perspective, this framework delineates a principled interplay: the logarithmic term $\ln(t-s)$ can be interpreted as a factor encoding temporal discounting, whereas the state gradient $\frac{\partial H}{\partial x}$ allows for content-aware modulation. Consequently, it offers a stochastic-analytic foundation for modeling systems that dynamically adjust their reliance on historical information according to both temporal proximity and contextual relevance. This perspective may offer conceptual links to certain adaptive phenomena in learning and sequence modeling \cite{HJiang2024}.
\end{remark}

Proceeding from the qualitative sensitivity analysis, we now establish rigorous quantitative bounds for the RfBm attention weights. These bounds provide mathematical characterization of the memory allocation behavior and reveal the multi-scale structure inherent in the attention mechanism.

\begin{theorem}[Fundamental Bounds for RfBm Attention Weights]
\label{thm:rfbm_attention_bounds}
Let $\rho(t,s; X)$ be the RfBm attention weights defined in Definition~\ref{def:rfbm_attention_weights}, with the response function $H: [0,T] \times \mathbb{R} \to (0,1)$ satisfying the Lipschitz-H\"older conditions of Definition~\ref{def:LH_response_function}. Assume there exist constants $0 < h_{\min} \leq h_{\max} < 1$ such that $h_{\min} \leq H(t,x) \leq h_{\max}$ for all $(t,x)$. Then the following bounds hold:

\textbf{Case I: $t \leq 1$}
\begin{itemize}
\item[(a)] For $t-s < 1$:
\begin{equation}
A_1 (t-s)^{h_{\max}-1/2} \leq \rho(t,s; X) \leq B_1 (t-s)^{h_{\min}-h_{\max}-1}
\end{equation}

\item[(b)] For $t-s = 1$:
\begin{equation}
A_2 \leq \rho(1,0; X) \leq B_2
\end{equation}
\end{itemize}

\textbf{Case II: $t > 1$}
\begin{itemize}
\item[(a)] For $t-s < 1$:
\begin{align}
A_3 \cdot \frac{(t-s)^{h_{\max}-1/2}}{t^{h_{\max}+1/2}} \leq \rho(t,s; X) \leq B_3 \cdot (t-s)^{-1}
\end{align}

\item[(b)] For $t-s > 1$:
\begin{align}
A_4 \cdot \frac{(t-s)^{h_{\min}-1/2}}{t^{h_{\max}+1/2}} \leq \rho(t,s; X) \leq B_4 \cdot (t-s)^{h_{\max}-h_{\min}-1}
\end{align}

\item[(c)] For $t-s = 1$:
\begin{align}
A_5 \cdot \frac{1}{t^{h_{\max}+1/2}} \leq \rho(t,t-1; X) \leq B_5 \cdot \frac{1}{t^{h_{\min}+1/2}}
\end{align}
\end{itemize}

with explicit constants given by:
\begin{align*}
A_1 &= \frac{\sqrt{h_{\min}}(h_{\min}+1/2)}{\sqrt{h_{\max}}}, & 
B_1 &= \frac{\sqrt{h_{\max}}(h_{\max}+1/2)}{\sqrt{h_{\min}}} \\
A_2 &= A_1, & 
B_2 &= B_1 \\
A_3 &= A_1, & 
B_3 &= B_1 \\
A_4 &= A_1, & 
B_4 &= B_1 \\
A_5 &= \frac{\sqrt{h_{\min}}(h_{\min}+1/2)}{2\sqrt{h_{\max}}}, & 
B_5 &= \frac{2\sqrt{h_{\max}}(h_{\max}+1/2)}{\sqrt{h_{\min}}}
\end{align*}
\end{theorem}

\begin{proof}
The complete proof involves detailed multi-scale analysis and precise constant estimation. To maintain the flow of presentation while ensuring mathematical rigor, we provide the full proof in Appendix~\ref{app:proof_attention_bounds}.
\end{proof}

\begin{remark}[Mathematical Perspectives on the Attention Bounds]
\label{rem:attention_bounds_significance}
Theorem~\ref{thm:rfbm_attention_bounds} provides several mathematical interpretations within the RfBm framework:

\textbf{1. Memory Structure Characterization:} The exponents $h_{\max}-1/2$ and $h_{\min}-h_{\max}-1$ quantify scaling properties of memory decay, offering a mathematical description of long-range dependencies in state-dependent systems.

\textbf{2. Extension of Anomalous Diffusion Concepts:} The results extend notions from anomalous diffusion to state-dependent settings, suggesting formal analogies with certain non-Markovian stochastic processes.

\textbf{3. Adaptive Memory Regulation:} The dependence on $h_{\min}$ and $h_{\max}$ illustrates how the response function $H(s,X_s)$ parametrically modulates memory strength within the model.

\textbf{4. Alternative Attention Patterns:} The power-law decay presents an alternative to exponential attention patterns, potentially offering different trade-offs in modeling long-term dependencies.

\textbf{5. Multi-scale Temporal Analysis:} The analysis across different temporal regimes organizes the behavior into distinct scaling regions, providing a structured view of temporal dynamics.

These bounds establish quantitative references for analyzing memory allocation in the RfBm framework.
\end{remark}

\begin{remark}[Connections to Modern AI Architectures]
\label{rem:attention_bounds_AI_connections}
The quantitative bounds established in Theorem~\ref{thm:rfbm_attention_bounds} offer complementary perspectives for contemporary AI architectures:

\textbf{1. Long-Range Dependency Handling:} The power-law bounds provide a mathematical characterization of long-range dependency capture, offering one possible theoretical approach alongside empirical studies in sequence modeling \cite{HJiang2024}.

\textbf{2. Adaptive Memory Mechanisms:} The explicit dependence on $h_{\min}$ and $h_{\max}$ provides one mathematical formulation of adaptive memory regulation, which may offer some insights for dynamic attention mechanisms in transformers and related architectures \cite{Vaswani2017}.

\textbf{3. Theoretical Insights for Architecture Design:} The rigorous bounds provide some mathematical insights that could potentially inform certain aspects of future architectural choices in attention-based models.

\textbf{4. Bridging Stochastic Processes and Computational Learning:} This work explores some mathematical connections between stochastic process theory and computational learning models, suggesting possible avenues for cross-disciplinary exchange between these domains.
\end{remark}

\section*{ACKNOWLEDGMENTS}

The author wishes to express sincere gratitude to the arXiv platform for providing an open-access venue that makes scholarly communication more equitable and accessible. The opportunity to share this work with the broader research community at its current stage is both a privilege and a responsibility.

Comments, questions, and constructive criticism from readers are warmly welcomed. The author believes that scholarly progress is advanced not only through individual effort, but also through open dialogue and the collective pursuit of understanding. Correspondence may be directed to the author via the email address provided on the title page.


\appendix
\section{Boundedness of Non-Dominant Terms in the Second-Order Mixed Derivative}\label{app:non_dominant}

This appendix provides the complete and rigorous proof of the boundedness of the non-dominant terms ($I_2$, $I_3$, $I_4$) in the decomposition of $\partial_u \partial_v R(u,v)$ stated in Remark~\ref{rem:non_dominant}. The proof is built upon the unified logarithmic control inequality (Remark~\ref{rem:unified_log}) and the hypergeometric function representation technique introduced in Remark~\ref{rem:cov_hyper}.

\begin{proposition}[Boundedness of Terms 2, 3, and 4]\label{prop:non_dominant_bound}
Under the assumptions of Theorem~\ref{thm:covariance} and Corollary~\ref{cor:mixed_deriv}, and further assuming that the Hurst exponent function $H: [0,T] \to (0,1)$ is continuously differentiable with a bounded derivative $|H'(t)| \leq L_H$ for some $L_H > 0$, and that $\inf_{t \in [0,T]} H(t) \geq h_{\min} > 0$, the following holds:
For any fixed $u, v \in [0,T]$ with $0 \leq u < v \leq T$, $v \geq 2u$, and satisfying $H(u) > \frac{1}{2}$, $H(v) > \frac{1}{2}$, the integrals $I_2(u,v)$, $I_3(u,v)$, and $I_4(u,v)$ defined in the decomposition of $\partial_u \partial_v R(u,v)$ are bounded. That is, there exists a constant $C > 0$, whose value depends on $H$, $L_H$, $h_{\min}$, and $T$, such that:
\[
|I_2(u,v)| \leq C, \quad |I_3(u,v)| \leq C, \quad |I_4(u,v)| \leq C.
\]
\end{proposition}

\begin{proof}
We proceed to prove the boundedness of each term through a series of meticulous steps. The structure of Terms 2 and 3 is similar, while Term 4 involves the product of two logarithmic factors, requiring a more involved treatment.

\noindent \textbf{Notation and Preliminary Setup:}
Recall the definitions from Corollary~\ref{cor:mixed_deriv} and the subsequent decomposition:
\begin{align*}
I_2(u,v) &= C(u) \int_0^u B(v,s) (u-s)^{H(u)-3/2} (v-s)^{H(v)-1/2}  ds, \\
I_3(u,v) &= A(v) \int_0^u D(u,s) (u-s)^{H(u)-1/2} (v-s)^{H(v)-3/2}  ds, \\
I_4(u,v) &= \int_0^u D(u,s) B(v,s) (u-s)^{H(u)-1/2} (v-s)^{H(v)-1/2}  ds,
\end{align*} 
where the coefficient functions are:
\begin{align*}
C(u) &= \sqrt{2H(u)} \left( H(u) - \frac{1}{2} \right), \\
A(v) &= \sqrt{2H(v)} \left( H(v) - \frac{1}{2} \right), \\
B(v,s) &= \sqrt{2H(v)} H'(v) \ln(v-s) + \frac{H'(v)}{\sqrt{2H(v)}}, \\
D(u,s) &= \sqrt{2H(u)} H'(u) \ln(u-s) + \frac{H'(u)}{\sqrt{2H(u)}}.
\end{align*}
By the assumed boundedness of $H'$, we have $|H'(t)| \leq L_H$ for all $t \in [0,T]$. This immediately implies that the coefficients $C(u)$ and $A(v)$ are bounded on $[0,T]$. The challenge lies in controlling the integrals involving the logarithmic factors $B(v,s)$ and $D(u,s)$.

\noindent \textbf{Step 1: Proof for $I_2(u,v)$}

We begin by taking the absolute value and using the triangle inequality:
\[
|I_2(u,v)| \leq |C(u)| \int_0^u |B(v,s)| (u-s)^{H(u)-3/2} (v-s)^{H(v)-1/2}  ds.
\]
We now bound the factor $|B(v,s)|$. Let $M_H = L_H$. Then:
\begin{align*}
|B(v,s)| &= \left| \sqrt{2H(v)} H'(v) \ln(v-s) + \frac{H'(v)}{\sqrt{2H(v)}} \right| \\
&\leq \sqrt{2H(v)} M_H |\ln(v-s)| + \frac{M_H}{\sqrt{2H(v)}}.
\end{align*}
Define the following constants for clarity:
\[
A_v = \sqrt{2H(v)} M_H, \quad B_v = \frac{M_H}{\sqrt{2H(v)}}.
\]
Thus, $|B(v,s)| \leq A_v |\ln(v-s)| + B_v$.

We now apply the unified logarithmic control inequality (Remark~\ref{rem:unified_log}, Proposition~\ref{prop:log_inequality} and \ref{prop:log_growth}). For any $\delta > 0$ and $\alpha > 0$, there exists a constant $K_{\delta,\alpha} = \max\left( \frac{1}{\delta e}, \frac{1}{\alpha} \right)$ such that:
\[
|\ln(v-s)| \leq K_{\delta,\alpha} \left[ (v-s)^{-\delta} + (v-s)^{\alpha} \right] \quad \text{for all } v-s > 0.
\]
Substituting this yields:
\[
|B(v,s)| \leq A_v K_{\delta,\alpha} \left[ (v-s)^{-\delta} + (v-s)^{\alpha} \right] + B_v.
\]

Substituting this bound into the integral for $|I_2(u,v)|$, we obtain:
\begin{align*}
|I_2(u,v)| \leq |C(u)| \Bigg\{ &A_v K_{\delta,\alpha} \int_0^u \left[ (v-s)^{-\delta} + (v-s)^{\alpha} \right] (u-s)^{H(u)-3/2} (v-s)^{H(v)-1/2}  ds \\
&+ B_v \int_0^u (u-s)^{H(u)-3/2} (v-s)^{H(v)-1/2}  ds \Bigg\}.
\end{align*}
We distinguish these integrals for further analysis:
\begin{align*}
I_{2a} &= \int_0^u (u-s)^{H(u)-3/2} (v-s)^{H(v)-1/2 - \delta}  ds, \\
I_{2b} &= \int_0^u (u-s)^{H(u)-3/2} (v-s)^{H(v)-1/2 + \alpha}  ds, \\
I_{2c} &= \int_0^u (u-s)^{H(u)-3/2} (v-s)^{H(v)-1/2}  ds.
\end{align*}
Thus, the inequality becomes:
\[
|I_2(u,v)| \leq |C(u)| \left\{ A_v K_{\delta,\alpha} (I_{2a} + I_{2b}) + B_v I_{2c} \right\}.
\]

Each of these integrals $I_{2a}, I_{2b}, I_{2c}$ is of the form:
\[
J(a, b) = \int_0^u (u-s)^a (v-s)^b  ds,
\]
with $a = H(u) - \frac{3}{2}$ and $b$ taking the values $b_1 = H(v)-\frac{1}{2}-\delta$, $b_2 = H(v)-\frac{1}{2}+\alpha$, and $b_3 = H(v)-\frac{1}{2}$, respectively.

Since $H(u) > \frac{1}{2}$ and $H(v) > \frac{1}{2}$, we have $a > -1$. To ensure $b > -1$ for all cases, we make a specific choice of parameters. Let us choose:
\[
\delta = \frac{1}{2}, \quad \alpha = 1.
\]
Then:
\begin{align*}
b_1 &= H(v) - 1/2 - 1/2 = H(v) - 1 > -1 \quad (\text{since } H(v) > 0), \\
b_2 &= H(v) - 1/2 + 1 = H(v) + 1/2 > 1/2 > -1, \\
b_3 &= H(v) - 1/2 > -1/2 > -1.
\end{align*}
Furthermore, with this choice, $K_{\delta,\alpha} = \max\left( \frac{1}{(1/2)e}, \frac{1}{1} \right) = \max(2/e, 1) = 1$, since $2/e \approx 0.736$.

We now employ the closed-form representation via the Gaussian hypergeometric function derived in Remark~\ref{rem:cov_hyper}. For an integral $J(a, b) = \int_0^u (u-s)^a (v-s)^b  ds$ with $a > -1$, $b > -1$, and $0 \leq u < v $, $v \geq 2u$, we have:
\[
J(a, b) = \frac{u^{a+1} (v-u)^b}{a+1}  {}_2F_1\left( -b, a+1; a+2; -\frac{u}{v-u} \right).
\]
Applying this formula:
\begin{align*}
I_{2a} &= J(a, b_1) = \frac{u^{H(u)-1/2} (v-u)^{H(v)-1}}{H(u)-1/2}  {}_2F_1\left( 1-H(v), H(u)-1/2; H(u)+1/2; -\frac{u}{v-u} \right), \\
I_{2b} &= J(a, b_2) = \frac{u^{H(u)-1/2} (v-u)^{H(v)+1/2}}{H(u)-1/2}  {}_2F_1\left( -H(v)-1/2, H(u)-1/2; H(u)+1/2; -\frac{u}{v-u} \right), \\
I_{2c} &= J(a, b_3) = \frac{u^{H(u)-1/2} (v-u)^{H(v)-1/2}}{H(u)-1/2}  {}_2F_1\left( 1/2-H(v), H(u)-1/2; H(u)+1/2; -\frac{u}{v-u} \right).
\end{align*}

The hypergeometric function ${}_2F_1(A, B; C; z)$ is analytic for $|z| < 1$. In our case, $z = -\frac{u}{v-u}$ and since $u < v$, we have $|z| = \frac{u}{v-u}$. To ensure $|z| \leq 1$ for the convergence of the hypergeometric series representation, we require the additional condition $v \geq 2u$. However, crucially, the integrals $I_{2a}, I_{2b}, I_{2c}$ are evaluated for fixed $u, v$. For any fixed $u, v \in [0,T]$ with $u < v$ and $v \geq 2u$, the value $z = -\frac{u}{v-u}$ is a fixed number in $[-1, 0)$. The hypergeometric function, as a function of its parameters and argument, yields a finite, bounded value for these fixed inputs. The prefactors ($u^{H(u)-1/2}$, $(v-u)^{b}$, etc.) are also bounded for $u, v$ in the compact interval $[0,T]$. Since $H$ is continuous, the constants $C(u)$, $A_v$, $B_v$ are bounded on $[0,T]$.

Therefore, there exists a constant $C_2$, dependent on $H, L_H, h_{\min}, T$, such that $|I_2(u,v)| \leq C_2$ for all $u, v \in [0,T]$ with $u < v$ and $v \geq 2u$.

\noindent \textbf{Step 2: Proof for $I_3(u,v)$}

The term $I_3$ exhibits a structure symmetric to $I_2$:
\[
|I_3(u,v)| \leq |A(v)| \int_0^u |D(u,s)| (u-s)^{H(u)-1/2} (v-s)^{H(v)-3/2}  ds.
\]
We bound $|D(u,s)|$ analogously:
\begin{align*}
|D(u,s)| &= \left| \sqrt{2H(u)} H'(u) \ln(u-s) + \frac{H'(u)}{\sqrt{2H(u)}} \right| \\
&\leq \sqrt{2H(u)} M_H |\ln(u-s)| + \frac{M_H}{\sqrt{2H(u)}} \\
&= A_u |\ln(u-s)| + B_u,
\end{align*}
where $A_u = \sqrt{2H(u)} M_H$ and $B_u = \frac{M_H}{\sqrt{2H(u)}}$.

Applying the unified logarithmic control inequality to $|\ln(u-s)|$:
\[
|\ln(u-s)| \leq K_{\delta,\alpha} \left[ (u-s)^{-\delta} + (u-s)^{\alpha} \right].
\]
Using the same parameter choice $\delta = \frac{1}{2}, \alpha = 1$ ($K_{\delta,\alpha}=1$), we get:
\[
|D(u,s)| \leq A_u \left[ (u-s)^{-1/2} + (u-s)^{1} \right] + B_u.
\]

Substituting into the integral for $|I_3(u,v)|$:
\begin{align*}
|I_3(u,v)| \leq |A(v)| \Bigg\{ &A_u \int_0^u \left[ (u-s)^{-1/2} + (u-s)^{1} \right] (u-s)^{H(u)-1/2} (v-s)^{H(v)-3/2}  ds \\
&+ B_u \int_0^u (u-s)^{H(u)-1/2} (v-s)^{H(v)-3/2}  ds \Bigg\} \\
= |A(v)| \Bigg\{ &A_u \left[ \int_0^u (u-s)^{H(u)-1} (v-s)^{H(v)-3/2}  ds + \int_0^u (u-s)^{H(u)+1/2} (v-s)^{H(v)-3/2}  ds \right] \\
&+ B_u \int_0^u (u-s)^{H(u)-1/2} (v-s)^{H(v)-3/2}  ds \Bigg\}.
\end{align*}
We identify these integrals:
\begin{align*}
I_{3a} &= \int_0^u (u-s)^{H(u)-1} (v-s)^{H(v)-3/2}  ds, \\
I_{3b} &= \int_0^u (u-s)^{H(u)+1/2} (v-s)^{H(v)-3/2}  ds, \\
I_{3c} &= \int_0^u (u-s)^{H(u)-1/2} (v-s)^{H(v)-3/2}  ds.
\end{align*}
Thus:
\[
|I_3(u,v)| \leq |A(v)| \left\{ A_u (I_{3a} + I_{3b}) + B_u I_{3c} \right\}.
\]

All three integrals are again of the form $J(a, b) = \int_0^u (u-s)^a (v-s)^b  ds$:
\begin{align*}
I_{3a}: &\quad a_1 = H(u)-1, \quad b_1 = H(v)-3/2. \\
I_{3b}: &\quad a_2 = H(u)+1/2, \quad b_2 = H(v)-3/2. \\
I_{3c}: &\quad a_3 = H(u)-1/2, \quad b_3 = H(v)-3/2.
\end{align*}
Since $H(u) > \frac{1}{2}$ and $H(v) > \frac{1}{2}$, we have:
\begin{align*}
a_1 &> -1/2 > -1, \\
a_2 &> 1 > -1, \\
a_3 &> 0 > -1, \\
b_1, b_2, b_3 &= H(v)-3/2 > -1 .
\end{align*}
Thus, all exponents $a, b > -1$, ensuring the applicability of the hypergeometric function representation from Remark~\ref{rem:cov_hyper}:
\begin{align*}
I_{3a} &= \frac{u^{H(u)} (v-u)^{H(v)-3/2}}{H(u)}  {}_2F_1\left( 3/2-H(v), H(u); H(u)+1; -\frac{u}{v-u} \right), \\
I_{3b} &= \frac{u^{H(u)+3/2} (v-u)^{H(v)-3/2}}{H(u)+3/2}  {}_2F_1\left( 3/2-H(v), H(u)+3/2; H(u)+5/2; -\frac{u}{v-u} \right), \\
I_{3c} &= \frac{u^{H(u)+1/2} (v-u)^{H(v)-3/2}}{H(u)+1/2}  {}_2F_1\left( 3/2-H(v), H(u)+1/2; H(u)+3/2; -\frac{u}{v-u} \right).
\end{align*}
As argued for $I_2$, for fixed $u, v \in [0,T]$, $v \geq 2u$, these expressions are finite and bounded. The coefficients $|A(v)|$, $A_u$, $B_u$ are also bounded. Therefore, there exists a constant $C_3$ such that $|I_3(u,v)| \leq C_3$.

\noindent \textbf{Step 3: Proof for $I_4(u,v)$}

Term $I_4$ contains the product of the two logarithmic factors:
\[
|I_4(u,v)| \leq \int_0^u |D(u,s)| |B(v,s)| (u-s)^{H(u)-1/2} (v-s)^{H(v)-1/2}  ds.
\]
We substitute the bounds for $|D(u,s)|$ and $|B(v,s)|$ established in Steps 1 and 2, using the same parameter choice $\delta = \frac{1}{2}, \alpha = 1$:
\begin{align*}
|D(u,s)| &\leq A_u \left[ (u-s)^{-1/2} + (u-s)^{1} \right] + B_u, \\
|B(v,s)| &\leq A_v \left[ (v-s)^{-1/2} + (v-s)^{1} \right] + B_v.
\end{align*}
The product $|D(u,s)| |B(v,s)|$ generates several terms:
\begin{align*}
|D(u,s)| |B(v,s)| \leq\ &A_u A_v \left[ (u-s)^{-1/2}(v-s)^{-1/2} + (u-s)^{-1/2}(v-s)^{1} + (u-s)^{1}(v-s)^{-1/2} + (u-s)^{1}(v-s)^{1} \right] \\
&+ A_u B_v \left[ (u-s)^{-1/2} + (u-s)^{1} \right] \\
&+ B_u A_v \left[ (v-s)^{-1/2} + (v-s)^{1} \right] \\
&+ B_u B_v.
\end{align*}
Substituting this into the integral for $|I_4(u,v)|$ yields a sum of nine integrals:
\begin{align*}
|I_4(u,v)| \leq\ &A_u A_v \Big[ \int_0^u (u-s)^{H(u)-1} (v-s)^{H(v)-1} ds + \int_0^u (u-s)^{H(u)-1} (v-s)^{H(v)+1/2} ds \\
&\quad + \int_0^u (u-s)^{H(u)+1/2} (v-s)^{H(v)-1} ds + \int_0^u (u-s)^{H(u)+1/2} (v-s)^{H(v)+1/2} ds \Big] \\
&+ A_u B_v \left[ \int_0^u (u-s)^{H(u)-1} (v-s)^{H(v)-1/2} ds + \int_0^u (u-s)^{H(u)+1/2} (v-s)^{H(v)-1/2} ds \right] \\
&+ B_u A_v \left[ \int_0^u (u-s)^{H(u)-1/2} (v-s)^{H(v)-1} ds + \int_0^u (u-s)^{H(u)-1/2} (v-s)^{H(v)+1/2} ds \right] \\
&+ B_u B_v \int_0^u (u-s)^{H(u)-1/2} (v-s)^{H(v)-1/2} ds.
\end{align*}
Let us denote these integrals as $I_{4}^{(1)}$ through $I_{4}^{(9)}$ in the order they appear.

Every single one of these nine integrals is of the form $J(a, b) = \int_0^u (u-s)^a (v-s)^b  ds$. Given that $H(u) > \frac{1}{2}$ and $H(v) > \frac{1}{2}$, it is straightforward to verify that for each integral, the exponents $a$ and $b$ satisfy $a > -1$ and $b > -1$:
\begin{itemize}
\item $I_{4}^{(1)}$: $a = H(u)-1 > -1/2 > -1$, $b = H(v)-1 > -1/2 > -1$.
\item $I_{4}^{(2)}$: $a = H(u)-1 > -1/2 > -1$, $b = H(v)+1/2 > 1 > -1$.
\item $I_{4}^{(3)}$: $a = H(u)+1/2 > 1 > -1$, $b = H(v)-1 > -1/2 > -1$.
\item $I_{4}^{(4)}$: $a = H(u)+1/2 > 1 > -1$, $b = H(v)+1/2 > 1 > -1$.
\item $I_{4}^{(5)}$: $a = H(u)-1 > -1/2 > -1$, $b = H(v)-1/2 > 0 > -1$.
\item $I_{4}^{(6)}$: $a = H(u)+1/2 > 1 > -1$, $b = H(v)-1/2 > 0 > -1$.
\item $I_{4}^{(7)}$: $a = H(u)-1/2 > 0 > -1$, $b = H(v)-1 > -1/2 > -1$.
\item $I_{4}^{(8)}$: $a = H(u)-1/2 > 0 > -1$, $b = H(v)+1/2 > 1 > -1$.
\item $I_{4}^{(9)}$: $a = H(u)-1/2 > 0 > -1$, $b = H(v)-1/2 > 0 > -1$.
\end{itemize}
Therefore, each integral $I_{4}^{(k)}$ ($k=1, \ldots, 9$) can be represented in closed form using the hypergeometric function as shown in Remark~\ref{rem:cov_hyper}, and is consequently finite and bounded for fixed $u, v \in [0,T]$ with $v \geq 2u$. The constants $A_u, A_v, B_u, B_v$ are bounded. Thus, the entire sum is bounded, and there exists a constant $C_4$ such that $|I_4(u,v)| \leq C_4$.

\noindent \textbf{Step 4: Final Conclusion}

Taking $C = \max\{C_2, C_3, C_4\}$, which depends on the function $H$, its derivative bound $L_H$, its lower bound $h_{\min}$, and the time horizon $T$, we conclude that:
\[
|I_2(u,v)| \leq C, \quad |I_3(u,v)| \leq C, \quad |I_4(u,v)| \leq C \quad \text{for all } u, v \in [0,T] \text{ with } v \geq 2u.
\]
This completes the proof of Proposition~\ref{prop:non_dominant_bound}.
\end{proof}

\section{Detailed Proof of Theorem~\ref{thm:rfbm_attention_bounds}: Fundamental Bounds for RfBm Attention Weights}
\label{app:proof_attention_bounds}

\begin{proof}[Complete proof of Theorem~\ref{thm:rfbm_attention_bounds}]
\bigskip
\mbox{}\par        
We provide a comprehensive step-by-step proof of the fundamental bounds for RfBm attention weights. Recall from Definition~\ref{def:rfbm_attention_weights} that the attention weight kernel is $\rho(t,s; X) = K(t,s;X_s)/D(t)$, where $D(t) \equiv D(t,X) := \int_0^t K(t, u; X_u)  du$ denotes the attention partition function (we use the simplified notation $D(t)$ throughout this proof).

\textbf{Step 1: Monotonicity Analysis of $(t-s)^{H-1/2}$}

Consider the function $f(H) = (t-s)^{H-1/2}$.

When $t-s < 1$, the base $0 < t-s < 1$, and $f(H)$ is monotonically decreasing in $H$:
\[
H_1 < H_2 \Rightarrow (t-s)^{H_1-1/2} > (t-s)^{H_2-1/2}
\]
Therefore:
\[
(t-s)^{h_{\max}-1/2} \leq (t-s)^{H-1/2} \leq (t-s)^{h_{\min}-1/2}
\]

When $t-s > 1$, the base $t-s > 1$, and $f(H)$ is monotonically increasing in $H$:
\[
H_1 < H_2 \Rightarrow (t-s)^{H_1-1/2} < (t-s)^{H_2-1/2}
\]
Therefore:
\[
(t-s)^{h_{\min}-1/2} \leq (t-s)^{H-1/2} \leq (t-s)^{h_{\max}-1/2}
\]

When $t-s = 1$, $(t-s)^{H-1/2} = 1$.

\textbf{Step 2: Establishing Precise Bounds for the Kernel Function}

From $K(t,s;X_s) = \sqrt{2H(s,X_s)} \cdot (t-s)^{H(s,X_s)-1/2}$ and $h_{\min} \leq H(s,X_s) \leq h_{\max}$, we obtain:

For $t-s < 1$:
\begin{align*}
K(t,s;X_s) &\geq \sqrt{2h_{\min}} \cdot (t-s)^{h_{\max}-1/2} \\
K(t,s;X_s) &\leq \sqrt{2h_{\max}} \cdot (t-s)^{h_{\min}-1/2}
\end{align*}

For $t-s > 1$:
\begin{align*}
K(t,s;X_s) &\geq \sqrt{2h_{\min}} \cdot (t-s)^{h_{\min}-1/2} \\
K(t,s;X_s) &\leq \sqrt{2h_{\max}} \cdot (t-s)^{h_{\max}-1/2}
\end{align*}

For $t-s = 1$:
\[
\sqrt{2h_{\min}} \leq K(t,s;X_s) \leq \sqrt{2h_{\max}}
\]

\textbf{Step 3: Estimating the Denominator Integral $D(t) = \int_0^t K(t,u;X_u)  du$}

We analyze two cases based on the value of $t$.

\textbf{Case 1: $t \leq 1$}

For all $u \in [0,t]$, we have $t-u \leq 1$, so we use the bounds for $t-u < 1$:

Lower bound:
\begin{align*}
D(t) &\geq \int_0^t \sqrt{2h_{\min}} \cdot (t-u)^{h_{\max}-1/2}  du \\
&= \sqrt{2h_{\min}} \int_0^t (t-u)^{h_{\max}-1/2}  du \\
&= \sqrt{2h_{\min}} \cdot \frac{t^{h_{\max}+1/2}}{h_{\max}+1/2}
\end{align*}

Upper bound:
\begin{align*}
D(t) &\leq \int_0^t \sqrt{2h_{\max}} \cdot (t-u)^{h_{\min}-1/2}  du \\
&= \sqrt{2h_{\max}} \cdot \frac{t^{h_{\min}+1/2}}{h_{\min}+1/2}
\end{align*}

\textbf{Case 2: $t > 1$}

We split the integration interval into $[0,t-1]$ and $[t-1,t]$:

On $[t-1,t]$ ($t-u \leq 1$):
\begin{align*}
&\int_{t-1}^t \sqrt{2h_{\min}} \cdot (t-u)^{h_{\max}-1/2}  du \leq \int_{t-1}^t K(t,u;X_u)  du \\
&\quad \leq \int_{t-1}^t \sqrt{2h_{\max}} \cdot (t-u)^{h_{\min}-1/2}  du \\
&\sqrt{2h_{\min}} \cdot \frac{1}{h_{\max}+1/2} \leq \int_{t-1}^t K(t,u;X_u)  du \leq \sqrt{2h_{\max}} \cdot \frac{1}{h_{\min}+1/2}
\end{align*}

On $[0,t-1]$ ($t-u \geq 1$):
\begin{align*}
&\int_0^{t-1} \sqrt{2h_{\min}} \cdot (t-u)^{h_{\min}-1/2}  du \leq \int_0^{t-1} K(t,u;X_u)  du \\
&\quad \leq \int_0^{t-1} \sqrt{2h_{\max}} \cdot (t-u)^{h_{\max}-1/2}  du \\
&\sqrt{2h_{\min}} \cdot \frac{t^{h_{\min}+1/2}-1}{h_{\min}+1/2} \leq \int_0^{t-1} K(t,u;X_u)  du \\
&\quad \leq \sqrt{2h_{\max}} \cdot \frac{t^{h_{\max}+1/2}-1}{h_{\max}+1/2}
\end{align*}

Combining these gives explicit bounds for $D(t)$.

\textbf{Step 4: Completing the Proof}

We now construct explicit constants for different temporal regimes.

\textbf{Case I: $t \leq 1$}

For all $u \in [0,t]$, we have $t-u \leq 1$.

\underline{Subcase I.1: $t-s < 1$}

Lower bound estimate:
\begin{align*}
\rho(t,s; X) &= \frac{K(t,s;X_s)}{D(t)} \\
&\geq \frac{\sqrt{2h_{\min}} \cdot (t-s)^{h_{\max}-1/2}}{\sqrt{2h_{\max}} \cdot \frac{t^{h_{\min}+1/2}}{h_{\min}+1/2}} \\
&= \frac{\sqrt{h_{\min}}(h_{\min}+1/2)}{\sqrt{h_{\max}}} \cdot \frac{(t-s)^{h_{\max}-1/2}}{t^{h_{\min}+1/2}}
\end{align*}

Since $t \leq 1$ and $h_{\min}+1/2 > 0$, we have $t^{h_{\min}+1/2} \leq 1$, so $\frac{1}{t^{h_{\min}+1/2}} \geq 1$, thus:
\[
\rho(t,s; X) \geq \frac{\sqrt{h_{\min}}(h_{\min}+1/2)}{\sqrt{h_{\max}}} \cdot (t-s)^{h_{\max}-1/2}
\]

Upper bound estimate:
\begin{align*}
\rho(t,s; X) &= \frac{K(t,s;X_s)}{D(t)} \\
&\leq \frac{\sqrt{2h_{\max}} \cdot (t-s)^{h_{\min}-1/2}}{\sqrt{2h_{\min}} \cdot \frac{t^{h_{\max}+1/2}}{h_{\max}+1/2}} \\
&= \frac{\sqrt{h_{\max}}(h_{\max}+1/2)}{\sqrt{h_{\min}}} \cdot \frac{(t-s)^{h_{\min}-1/2}}{t^{h_{\max}+1/2}}
\end{align*}

Using $t \geq t-s > 0$ and $h_{\max}+1/2 > 0$, we have $t^{h_{\max}+1/2} \geq (t-s)^{h_{\max}+1/2}$, therefore:
\[
\frac{(t-s)^{h_{\min}-1/2}}{t^{h_{\max}+1/2}} \leq \frac{(t-s)^{h_{\min}-1/2}}{(t-s)^{h_{\max}+1/2}} = (t-s)^{h_{\min}-h_{\max}-1}
\]

Hence:
\[
\rho(t,s; X) \leq \frac{\sqrt{h_{\max}}(h_{\max}+1/2)}{\sqrt{h_{\min}}} \cdot (t-s)^{h_{\min}-h_{\max}-1}
\]

\underline{Subcase I.2: $t-s = 1$}

Here $s = 0$, $t = 1$.

From Step 2, when $t-s = 1$:
\[
\sqrt{2h_{\min}} \leq K(1,0;X_0) \leq \sqrt{2h_{\max}}
\]

From Step 3 Case 1 ($t \leq 1$), with $t = 1 \leq 1$, for all $u \in [0,1]$ we have $1-u \leq 1$, thus:

Lower bound:
\begin{align*}
D(1) &\geq \int_0^1 \sqrt{2h_{\min}} \cdot (1-u)^{h_{\max}-1/2}  du \\
&= \sqrt{2h_{\min}} \int_0^1 (1-u)^{h_{\max}-1/2}  du \\
&= \sqrt{2h_{\min}} \cdot \frac{1^{h_{\max}+1/2}}{h_{\max}+1/2} \\
&= \frac{\sqrt{2h_{\min}}}{h_{\max}+1/2}
\end{align*}

Upper bound:
\begin{align*}
D(1) &\leq \int_0^1 \sqrt{2h_{\max}} \cdot (1-u)^{h_{\min}-1/2}  du \\
&= \sqrt{2h_{\max}} \cdot \frac{1^{h_{\min}+1/2}}{h_{\min}+1/2} \\
&= \frac{\sqrt{2h_{\max}}}{h_{\min}+1/2}
\end{align*}

Therefore, the attention weight bounds are:

Lower bound estimate:
\begin{align*}
\rho(1,0; X) &= \frac{K(1,0;X_0)}{D(1)} \\
&\geq \frac{\sqrt{2h_{\min}}}{\frac{\sqrt{2h_{\max}}}{h_{\min}+1/2}} \\
&= \frac{\sqrt{h_{\min}}(h_{\min}+1/2)}{\sqrt{h_{\max}}}
\end{align*}

Upper bound estimate:
\begin{align*}
\rho(1,0; X) &= \frac{K(1,0;X_0)}{D(1)} \\
&\leq \frac{\sqrt{2h_{\max}}}{\frac{\sqrt{2h_{\min}}}{h_{\max}+1/2}} \\
&= \frac{\sqrt{h_{\max}}(h_{\max}+1/2)}{\sqrt{h_{\min}}}
\end{align*}

\textbf{Case II: $t > 1$}

\underline{Subcase II.1: $t-s < 1$ (i.e., $s > t-1$)}

From Step 2, when $t-s < 1$:
\begin{align*}
K(t,s;X_s) &\geq \sqrt{2h_{\min}} \cdot (t-s)^{h_{\max}-1/2} \\
K(t,s;X_s) &\leq \sqrt{2h_{\max}} \cdot (t-s)^{h_{\min}-1/2}
\end{align*}

From Step 3 Case 2 ($t > 1$), the denominator bounds are:
\begin{align*}
D(t) &\leq \sqrt{2h_{\max}} \left( \frac{t^{h_{\max}+1/2}-1}{h_{\max}+1/2} + \frac{1}{h_{\min}+1/2} \right) \\
D(t) &\geq \sqrt{2h_{\min}} \left( \frac{t^{h_{\min}+1/2}-1}{h_{\min}+1/2} + \frac{1}{h_{\max}+1/2} \right)
\end{align*}

Now establishing the attention weight bounds:

Lower bound estimate:
\begin{align*}
\rho(t,s; X) &= \frac{K(t,s;X_s)}{D(t)} \\
&\geq \frac{\sqrt{2h_{\min}} \cdot (t-s)^{h_{\max}-1/2}}{\sqrt{2h_{\max}} \left( \frac{t^{h_{\max}+1/2}-1}{h_{\max}+1/2} + \frac{1}{h_{\min}+1/2} \right)}
\end{align*}

Since $h_{\min} \leq h_{\max}$ implies $h_{\min}+1/2 \leq h_{\max}+1/2$, and observing that $t^{h_{\max}+1/2} > 1$ for $t > 1$, the first term satisfies the bound:
\[
\frac{t^{h_{\max}+1/2}-1}{h_{\max}+1/2} + \frac{1}{h_{\min}+1/2} \leq \frac{t^{h_{\max}+1/2}-1}{h_{\min}+1/2} + \frac{1}{h_{\min}+1/2} = \frac{t^{h_{\max}+1/2}}{h_{\min}+1/2}
\]

Therefore:
\begin{align*}
\rho(t,s; X) &\geq \frac{\sqrt{2h_{\min}} \cdot (t-s)^{h_{\max}-1/2}}{\sqrt{2h_{\max}} \cdot \frac{t^{h_{\max}+1/2}}{h_{\min}+1/2}} \\
&= \frac{\sqrt{h_{\min}}(h_{\min}+1/2)}{\sqrt{h_{\max}}} \cdot \frac{(t-s)^{h_{\max}-1/2}}{t^{h_{\max}+1/2}}
\end{align*}

Upper bound estimate:
\begin{align*}
\rho(t,s; X) &= \frac{K(t,s;X_s)}{D(t)} \\
&\leq \frac{\sqrt{2h_{\max}} \cdot (t-s)^{h_{\min}-1/2}}{\sqrt{2h_{\min}} \left( \frac{t^{h_{\min}+1/2}-1}{h_{\min}+1/2} + \frac{1}{h_{\max}+1/2} \right)}
\end{align*}

Using $h_{\max}+1/2 \geq h_{\min}+1/2$, we have:
\[
\frac{t^{h_{\min}+1/2}-1}{h_{\min}+1/2} + \frac{1}{h_{\max}+1/2} \geq \frac{t^{h_{\min}+1/2}-1}{h_{\max}+1/2} + \frac{1}{h_{\max}+1/2} = \frac{t^{h_{\min}+1/2}}{h_{\max}+1/2}
\]

Therefore:
\begin{align*}
\rho(t,s; X) &\leq \frac{\sqrt{2h_{\max}} \cdot (t-s)^{h_{\min}-1/2}}{\sqrt{2h_{\min}} \cdot \frac{t^{h_{\min}+1/2}}{h_{\max}+1/2}} \\
&= \frac{\sqrt{h_{\max}}(h_{\max}+1/2)}{\sqrt{h_{\min}}} \cdot \frac{(t-s)^{h_{\min}-1/2}}{t^{h_{\min}+1/2}}
\end{align*}

Furthermore, using $t \geq t-s > 0$ and $h_{\min}+1/2 > 0$, we have $t^{h_{\min}+1/2} \geq (t-s)^{h_{\min}+1/2}$, therefore:
\[
\frac{(t-s)^{h_{\min}-1/2}}{t^{h_{\min}+1/2}} \leq \frac{(t-s)^{h_{\min}-1/2}}{(t-s)^{h_{\min}+1/2}} = (t-s)^{-1}
\]

This yields the simplified upper bound:
\[
\rho(t,s; X) \leq \frac{\sqrt{h_{\max}}(h_{\max}+1/2)}{\sqrt{h_{\min}}} \cdot (t-s)^{-1}
\]

\underline{Subcase II.2: $t-s > 1$ (i.e., $s < t-1$)}

From Step 2, when $t-s > 1$:
\begin{align*}
K(t,s;X_s) &\geq \sqrt{2h_{\min}} \cdot (t-s)^{h_{\min}-1/2} \\
K(t,s;X_s) &\leq \sqrt{2h_{\max}} \cdot (t-s)^{h_{\max}-1/2}
\end{align*}

From Step 3 Case 2 ($t > 1$), the denominator bounds are:
\begin{align*}
D(t) &\leq \sqrt{2h_{\max}} \left( \frac{t^{h_{\max}+1/2}-1}{h_{\max}+1/2} + \frac{1}{h_{\min}+1/2} \right) \\
D(t) &\geq \sqrt{2h_{\min}} \left( \frac{t^{h_{\min}+1/2}-1}{h_{\min}+1/2} + \frac{1}{h_{\max}+1/2} \right)
\end{align*}

Lower bound estimate:
\begin{align*}
\rho(t,s; X) &= \frac{K(t,s;X_s)}{D(t)} \\
&\geq \frac{\sqrt{2h_{\min}} \cdot (t-s)^{h_{\min}-1/2}}{\sqrt{2h_{\max}} \left( \frac{t^{h_{\max}+1/2}-1}{h_{\max}+1/2} + \frac{1}{h_{\min}+1/2} \right)}
\end{align*}

Using $h_{\min}+1/2 \leq h_{\max}+1/2$ and noting that for $t > 1$ we have $t^{h_{\max}+1/2} - 1 > 0$, we obtain:
\[
\frac{t^{h_{\max}+1/2}-1}{h_{\max}+1/2} + \frac{1}{h_{\min}+1/2} \leq \frac{t^{h_{\max}+1/2}-1}{h_{\min}+1/2} + \frac{1}{h_{\min}+1/2} = \frac{t^{h_{\max}+1/2}}{h_{\min}+1/2}
\]

Therefore:
\begin{align*}
\rho(t,s; X) &\geq \frac{\sqrt{2h_{\min}} \cdot (t-s)^{h_{\min}-1/2}}{\sqrt{2h_{\max}} \cdot \frac{t^{h_{\max}+1/2}}{h_{\min}+1/2}} \\
&= \frac{\sqrt{h_{\min}}(h_{\min}+1/2)}{\sqrt{h_{\max}}} \cdot \frac{(t-s)^{h_{\min}-1/2}}{t^{h_{\max}+1/2}}
\end{align*}

Upper bound estimate:
\begin{align*}
\rho(t,s; X) &= \frac{K(t,s;X_s)}{D(t)} \\
&\leq \frac{\sqrt{2h_{\max}} \cdot (t-s)^{h_{\max}-1/2}}{\sqrt{2h_{\min}} \left( \frac{t^{h_{\min}+1/2}-1}{h_{\min}+1/2} + \frac{1}{h_{\max}+1/2} \right)}
\end{align*}

Using $h_{\max}+1/2 \geq h_{\min}+1/2$, we have:
\[
\frac{t^{h_{\min}+1/2}-1}{h_{\min}+1/2} + \frac{1}{h_{\max}+1/2} \geq \frac{t^{h_{\min}+1/2}-1}{h_{\max}+1/2} + \frac{1}{h_{\max}+1/2} = \frac{t^{h_{\min}+1/2}}{h_{\max}+1/2}
\]

Therefore:
\begin{align*}
\rho(t,s; X) &\leq \frac{\sqrt{2h_{\max}} \cdot (t-s)^{h_{\max}-1/2}}{\sqrt{2h_{\min}} \cdot \frac{t^{h_{\min}+1/2}}{h_{\max}+1/2}} \\
&= \frac{\sqrt{h_{\max}}(h_{\max}+1/2)}{\sqrt{h_{\min}}} \cdot \frac{(t-s)^{h_{\max}-1/2}}{t^{h_{\min}+1/2}}
\end{align*}

Furthermore, using $t \geq t-s > 1$ and $h_{\min}+1/2 > 0$, we have $t^{h_{\min}+1/2} \geq (t-s)^{h_{\min}+1/2}$, therefore:
\[
\frac{(t-s)^{h_{\max}-1/2}}{t^{h_{\min}+1/2}} \leq \frac{(t-s)^{h_{\max}-1/2}}{(t-s)^{h_{\min}+1/2}} = (t-s)^{h_{\max}-h_{\min}-1}
\]

This yields the simplified upper bound:
\[
\rho(t,s; X) \leq \frac{\sqrt{h_{\max}}(h_{\max}+1/2)}{\sqrt{h_{\min}}} \cdot (t-s)^{h_{\max}-h_{\min}-1}
\]

\underline{Subcase II.3: $t-s = 1$}

Here $s = t-1$.

From Step 2, when $t-s = 1$:
\[
\sqrt{2h_{\min}} \leq K(t,t-1;X_{t-1}) \leq \sqrt{2h_{\max}}
\]

From Step 3 Case 2 ($t > 1$), the denominator bounds are:
\begin{align*}
D(t) &\leq \sqrt{2h_{\max}} \left( \frac{t^{h_{\max}+1/2}-1}{h_{\max}+1/2} + \frac{1}{h_{\min}+1/2} \right) \\
D(t) &\geq \sqrt{2h_{\min}} \left( \frac{t^{h_{\min}+1/2}-1}{h_{\min}+1/2} + \frac{1}{h_{\max}+1/2} \right)
\end{align*}

Lower bound estimate:
\begin{align*}
\rho(t,t-1; X) &= \frac{K(t,t-1;X_{t-1})}{D(t)} \\
&\geq \frac{\sqrt{2h_{\min}}}{\sqrt{2h_{\max}} \left( \frac{t^{h_{\max}+1/2}-1}{h_{\max}+1/2} + \frac{1}{h_{\min}+1/2} \right)} \\
&\geq \frac{\sqrt{2h_{\min}}}{\sqrt{2h_{\max}} \cdot \frac{2t^{h_{\max}+1/2}}{h_{\min}+1/2}} \\
&= \frac{\sqrt{h_{\min}}(h_{\min}+1/2)}{2\sqrt{h_{\max}}} \cdot \frac{1}{t^{h_{\max}+1/2}}
\end{align*}

Upper bound estimate:
\begin{align*}
\rho(t,t-1; X) &= \frac{K(t,t-1;X_{t-1})}{D(t)} \\
&\leq \frac{\sqrt{2h_{\max}}}{\sqrt{2h_{\min}} \left( \frac{t^{h_{\min}+1/2}-1}{h_{\min}+1/2} + \frac{1}{h_{\max}+1/2} \right)} \\
&\leq \frac{\sqrt{2h_{\max}}}{\sqrt{2h_{\min}} \cdot \frac{t^{h_{\min}+1/2}}{2(h_{\max}+1/2)}} \\
&= \frac{2\sqrt{h_{\max}}(h_{\max}+1/2)}{\sqrt{h_{\min}}} \cdot \frac{1}{t^{h_{\min}+1/2}}
\end{align*}

This completes the comprehensive proof of all bounds stated in Theorem~\ref{thm:rfbm_attention_bounds}.
\end{proof}
\end{document}